\documentclass[12pt]{article}
 \usepackage[utf8]{inputenc}
 \usepackage{graphicx}
 \usepackage{latexsym}
 \usepackage{amsmath}
 \usepackage{verbatim}
 \usepackage{amssymb}
 \usepackage{mathrsfs}
 \usepackage{amsmath,amssymb}
 \DeclareMathAlphabet      {\mathbf}{OT1}{cmr}{bx}{n}
 \DeclareFontFamily{OT1}{pzc}{}
 \DeclareFontShape{OT1}{pzc}{m}{it}%
 {<-> s * [1.15] pzcmi7t}{}
 \DeclareMathAlphabet{\mathpzc}{OT1}{pzc}{m}{it}
 \usepackage{enumitem}
 \usepackage{ dsfont }
 \usepackage{ mathrsfs }

 \usepackage{amsbsy }
 \usepackage{graphics}

 \usepackage{amsthm}
 \newtheorem{thm}{Theorem}
 \newtheorem{theorem}{Theorem}[section]

 \newtheorem{lemma}[theorem]{Lemma}
 \newtheorem{corollary}[theorem]{Corollary}

 \newtheorem{definition}[theorem]{Definition}

 \newtheorem{prop}[theorem]{Proposition}

 \newtheorem{remark}{Remark}[section]

 \usepackage{romannum}
 \usepackage[toc,page,title,titletoc,header]{appendix}
 \usepackage{hyperref}
 \usepackage{amscd}
 \usepackage[all,cmtip]{xy}
  \usepackage{bm}

 \def\ack{\section*{Acknowledgements}%
 	\addtocontents{toc}{\protect\vspace{6pt}}%
 	\addcontentsline{toc}{section}{Acknowledgements}%
 }
 
 \usepackage{abstract}

\begin{document}   \small
 	\pagenumbering{arabic}

\def\ack{\section*{Acknowledgements}%
	\addtocontents{toc}{\protect\vspace{6pt}}%
	\addcontentsline{toc}{section}{Acknowledgements}%
}

	\pagenumbering{arabic}
	\title{Closed-open morphisms on  periodic Floer homology}
	\author{{ Guanheng Chen}}
	\date{}
	%\begin{titlingpage}
	\maketitle
	%Department of Mathematics, University of Adelaide \\
	\thispagestyle{empty}
	\begin{abstract}
	  In this paper, we investigate homomorphisms from periodic Floer
homology (PFH) to quantitative Heegaard Floer homology. We call these
homomorphisms closed-open morphisms. Under certain assumptions on the
 Lagrangian link, we first follow R. Lipshitz's idea to give a cylindrical formulation of the quantitative Heegaard Floer homology. Then, we construct
 the closed-open morphisms from PFH to quantitative Heegaard Floer
 homology. Moreover,  we show that these morphisms are non-vanishing.  As
 an application, we deduce a relation between the PFH spectral invariants
 and the link spectral invariants.
	\end{abstract}
	\tableofcontents
	\section{Introduction and Main results}
Throughout the paper, we assume that $(\Sigma, \omega)$ is a connected, closed symplectic surface
such that $\int_{\Sigma} \omega =1$. Given a \textbf{Hamiltonian function} $H: [0,1] \times \Sigma \to \mathbb{R}$, we define the
\textbf{Hamiltonian vector field} $X_H$ by $\omega(X_H, \cdot) =d_{\Sigma} H$. Let $\varphi_H^t$
be the flow generated by
$X_H$. Set  $\varphi_H : =\varphi_H^1$.  A symplectomorphism $\varphi$ is called \textbf{Hamiltonian} if  $\varphi=\varphi_H^1$ for
some $H$. The set of all Hamiltonian symplectomorphisms is denoted by  $Ham(\Sigma, \omega). $

For  $\varphi \in Ham(\Sigma, \omega)$, there are two distinct Floer theories associated with it: \textbf{periodic Floer homology} (abbreviated as PFH)   $\widetilde{PFH}(\Sigma, \varphi, d)$ and \textbf{quantitative Heegaard Floer homology}  $HF(\operatorname{Sym}^d \varphi(\underline{\Lambda}), \operatorname{Sym}^d \underline{\Lambda})$, where $d$ is an integer and $\underline{\Lambda} \subset \Sigma$ is a link (a disjoint union of
embedded circles). The former is introduced by M. Hutchings \cite{H1, H3} which is a sister
version of a more well-known invariant called embedded contact homology (abbreviated
as ECH) (see \cite{H4} for example). It is well defined due to M. Hutchings and C. Taubes
\cite{HT1, HT2}. The later is Lagrangian Floer homology of symmetric product of a monotone link on a surface which is shown to be well defined and computed by D. Cristofaro-Gardiner, V. Humili\`{e}re, C. Mak, S. Seyfaddini and I. Smith \cite{CHMSS}. Moreover, these two Floer homologies are independent of $\varphi$.   Denote the abstract groups  respectively by  $\widetilde{PFH}(\Sigma,  d)$ and  $HF(\operatorname{Sym}^d \underline{\Lambda})$.   For our purpose, we will take $d$ to be the number of components of  $\underline{\Lambda}.$

Two types
of numerical invariants emerge respectively from these two Floer theories called \textbf{PFH
spectral invariants} \cite{CHS1, EH, CPZ} and \textbf{link spectral invariants} \cite{CHMSS}. More precisely, they are
maps
\begin{equation*}
\begin{split}
&c^{pfh}_d : C^{\infty}(S^1\times \Sigma) \times \widetilde{PFH}(\Sigma, d) \to \{-\infty\} \cup \mathbb{R},\\
&c^{link}_{\underline{\Lambda}} : C^{\infty}([0,1]\times \Sigma) \times  HF( \operatorname{Sym}^d\underline{\Lambda}) \to \{-\infty\} \cup  \mathbb{R}.
\end{split}
\end{equation*}
Both spectral invariants satisfy a kind of “Weyl law” called Calabi property, originally
proposed by M. Hutchings (Remark 1.12 of \cite{CHS1}). It is an analogue of the “volume property” for ECH \cite{CHR}. The Calabi property for PFH spectral invariants is first verified for a
special class of Hamiltonian functions on the two-sphere \cite{CHS1} by D. Cristofaro-Gardiner,
V. Humili\`{e}re, and S. Seyfaddini. The general case is proved by O. Edtmair and Hutchings \cite{EH}, also by D. Cristofaro-Gardiner, R. Prasad and B. Zhang independently \cite{CPZ}.
The assumption about existence of U-sequences in \cite{EH} is guaranteed by \cite{CPZ1}. For the
link spectral invariants, the Calabi property is proved by D. Cristofaro-Gardiner, V.
Humili\`{e}re, C. Mak, S. Seyfaddini and I. Smith \cite{CHMSS}. The power of the spectral invariants
and the Calabi property has already been demonstrated in the study of various topics
such as the simplicity of the $C^0$-closure of $Ham(\Sigma, \omega)$ \cite{CHS1, CHMSS}, the $C^{\infty} $-closing lemma
\cite{EH, CPZ}, and the Hofer geometry of $Ham(\Sigma, \omega)$\cite{CHS2}, etc.

The purpose of this paper is to investigate the relation between these two types of
spectral invariants. The motivation arises naturally from the observation that these
two kinds of spectral invariants share many common properties. Moreover, the PFH
spectral invariants are more closely related to dynamics as PFH counts periodic orbits
of a Hamiltonian vector field, while the link spectral invariants tend to be easier to
compute due to the Lagrangian control property and the flexibility in selecting the
links. Consequently, elucidating their relation should have some potential applications
in the future. The principal outcome of this paper is that we obtain an estimate (\ref{eq68})
$$c^{link}_{\underline{\Lambda}}(H, \mathbf{1}_{\underline{\Lambda}}) \le c_d^{pfh}(H, \mathfrak{e}),$$
 where  $\mathbf{1}_{\underline{\Lambda}} $ and $\mathfrak{e}$ are   distinguished  classes of quantitative Heegaard Floer homology and PFH, respectively, referred to as the unit.

To prove the above result, we try to construct a homomorphism from PFH to the
quantitative Heegaard Floer homology 
called a \textbf{closed-open morphism}. It is used to
relate the domains of the spectral invariants. This homomorphism is an analogy of the
one defined by P. Albers \cite{A} which is a bridge between symplectic Floer homology and
Lagrangian Floer homology.

Following P. Albers’s idea, the closed-open morphism should be defined by counting holomorphic curves in a certain symplectic cobordism. However, a notable difficulty arises in this construction: PFH is defined by counting holomorphic curves in
a 4-manifold while the quantitative Heegaard Floer homology is defined by counting holomorphic curves in a manifold of higher dimension. This dimensional disparity
presents a significant obstacle in directly relating the two concepts. To overcome this
problem, our strategy is to follow R. Lipshitz’s idea to reformulate the quantitative Heegaard Floer homology \cite{RL1}. That is to define an intermediate invariant $HF(\Sigma, \varphi(\underline{\Lambda}), \underline{\Lambda})$ by
counting holomorphic curves in a 4-manifold $\mathbb{R} \times [0,1] \times \Sigma$ with Lagrangian boundary
conditions determined by  the  link. Furthermore, we  show that  $HF(\Sigma, \varphi(\underline{\Lambda}), \underline{\Lambda})$ is  canonically isomorphic to the quantitative
Heegaard Floer homology   $HF(\operatorname{Sym}^d \varphi(\underline{\Lambda}), \operatorname{Sym}^d \underline{\Lambda})$ (see Theorem \ref{thm1}). Then we construct the  homomorphism (Theorem \ref{thm2})
\begin{equation*}
\mathcal{CO}(\underline{\Lambda}):  \widetilde{PFH}(\Sigma, d) \to HF(\Sigma, \underline{\Lambda}),
\end{equation*}
from PFH to $HF(\Sigma,  \underline{\Lambda})$. From the view point of tautological correspondence, the closed-open morphisms we constructing are reformulation of the ones
defined by P. Albers \cite{A} (Remark \ref{remark2}). Therefore, it is reasonable to anticipate that
the closed-open morphisms preserve the units. We confirm this point in our result. In particular, the closed-open morphisms are non-vanishing.

Following the framework in \cite{CHMSS}, we define a numerical invariant
 \begin{equation}\label{eq23}
c^{hf}_{\underline{\Lambda}} : C^{\infty}([0,1]\times \Sigma) \times  HF(\Sigma, \underline{\Lambda}) \to \{-\infty\} \cup  \mathbb{R}.
\end{equation}
from the reformulation called \textbf{HF spectral invariant}, which is equivalent to the
link spectral invariant (Corollary \ref{corollary1}). We show that the HF spectral
invariants are controlled by the PFH spectral invariants (\ref{eq68}). More about the
properties of the HF spectral invariants and an enhancement (in the case of sphere)
of the relation between PFH spectral and HF spectral invariants will appear in the
author’s subsequent paper \cite{GHC3}.

\textbf{Comparing with the existing results: } Most of the results presented in our
paper in fact are not new tales. They have already appeared in the work of many
predecessors. What we have done is to reconstruct these results within our settings so
that they can be used to compare the PFH spectral invariants and the link spectral
invariants. Consequently, it is worth demonstrating the main differences between our
results and the previous work.
Our reformulation of the quantitative Heegaard Floer homology basically the same
as Lipshitz \cite{RL1}, yet with a closer alignment to Chapter 4 of \cite{VPK} in the details, using the
concepts such as HF curves and ECH index, etc. There is a main difference between our
reformulation with Lipshitz is that we cannot rule out the bubbles simply by topological
reasons. Thus, we need to calculate the contribution of the bubbles to the ECH index.
This is also the reason why we need a stronger assumption on the link than  \cite{CHMSS} (see
Remark \ref{remark10}).
The closed-open morphism between PFH and Heegaard Floer homology is defined
by V. Colin, P. Ghiggini, and K. Honda when they prove the equivalence between ECH
and Heegaard Floer homology \cite{VPK}. As an intermediate step, they show that the closed-open morphism is an isomorphism between the Heegaard Floer homology and a version
of PFH. They consider the curves satisfying a novelty Lagrangian boundary where
all the components of the link intersect at a common point, and they need to spend
a significant amount of time on the transversality and compactness. Our situation
is much simpler, the Lagrangians are pairwise disjoint, and we only need to use the
SFT compactness, Hutchings-Taubes's obstruction gluing \cite{HT1, HT2} and the classical transversality argument. In fact, the construction
is closer to the argument in Chapter 6 of \cite{VPK}, where they construct the morphism
from Heegaard Floer homology to PFH, called open-closed morphism. Furthermore,
the morphisms defined in \cite{VPK} are quasi-isomorphisms, but this may not be true in our
case, even when the link has a single component. What we hope is that the closed-open morphisms are non-vanishing. We prove this by using a classical idea: we take
the Hamiltonian function to be a small Morse function and compute this special case.

\paragraph{Coefficient} Throughout this paper, all the Floer theories are defined with $\mathbb{F} =\mathbb{Z}/2\mathbb{Z}$
coefficient.

\begin{ack}
The author would like to express sincere gratitude to the anonymous referee for his/her valuable feedback and insightful comments, which have greatly improved the quality of this manuscript. 
\end{ack}

\subsection{Twisted periodic Floer homology}
We begin our paper by reviewing those aspects of periodic Floer homology that we
need.  For more details, please refer to \cite{H3, H4, LT}.
	
	%\subsection{Basic definition}
	%\begin{definition}
	%Let $\pi_W: W \to B$ be a surface fibration over a circle or a  Lefschetz fibration. Given a $2$-form  $\omega_W \in \Omega^2(W)$, $\omega_W$ is called admissible if $d\omega_W=0$  and $\omega_W$ is fiberwise nondegenerate.
	%\end{definition}
	%Note that the  admissible $2$-form  $\omega_W$ gives a decomposition   $TW = TW^{hor} \oplus  TW^{vert}$ of the tangent bundle of $W$, where $TW^{vert}= \ker \pi_{W*}$ and $TW_{w}^{hor}=\{ a \in TW_w \vert   \omega_W(a, b)=0  \mbox{ } \forall  b \in  TW_w^{vert} \}$.
	
	Given $\varphi \in Ham(\Sigma, \omega)$, we define the \textbf{mapping torus} by
	\begin{equation*}
		Y_{\varphi}: =[0,1 ] \times \Sigma/\sim, \ (1, x) \sim (0, \varphi(x)).
	\end{equation*}
	Obviously, $\pi: Y_{\varphi} \to S^1$ is a surface  bundle over the circle.  Let $\xi:=\ker \pi_*$ denote the vertical bundle.
	The volume form $\omega$ and the vector field $\partial_t$ descend  to a  closed two form $\omega_{\varphi}$ and  a vector field $R$  on $Y_{\varphi}$ respectively. Fix a Hamiltonian function $H$ generating $\varphi$. Then,  we have the following global
trivialization
	\begin{equation} \label{eq11}
		\begin{split}
			\Psi_H: &S_t^1 \times  \Sigma  \to Y_{\varphi_H}\\
			&(t, x) \to (t, (\varphi_H^t)^{-1}(x)).
		\end{split}
	\end{equation}
	Moreover,  $\Psi_H^*(\omega_{\varphi}) =\omega + dH \wedge dt$ and $(\Psi_H)_*^{-1}(R) =\partial_t +X_H$.

	A \textbf{periodic orbit}  is a smooth map $\gamma: \mathbb{R}_{\tau} / q \mathbb{Z} \to Y_{\varphi} $ satisfying the ODE $\partial_{\tau} \gamma =R \circ \gamma$ for some $q>0$. The number  $q$ is called \textbf{period or degree} of the periodic orbit.  It is easy to show that $q$ is the intersection number $[\gamma] \cdot [\Sigma]$.
	
	 We say that a periodic orbit $\gamma$ is  \textbf{non-degenerate} if $1$ is not an eigenvalue of the linearized return map.  It is called elliptic if the eigenvalues of the linearized return map are on the unit circle, positive hyperbolic if the eigenvalues of  are real positive numbers, and negative hyperbolic if the eigenvalues   are real negative numbers. The symplectomorphism $\varphi$ is called \textbf{$d$-nondegenerate} if   every closed orbit with degree at most $d$ is nondegenerate. We assume that the   symplectomorphism is $d$-nondegenerate throughout. 	%\begin{equation*}
	%\gamma: \mathbb{R}/{d \mathbb{Z}} \to Y, \partial_t \gamma =R \circ \gamma(t),
	%\end{equation*}
	%The Reeb vector field $R$ is a section of $TY^{hor}$ such that $\pi_*(R)=\partial_t$. A periodic orbit is a smooth map $\gamma: \mathbb{R}_{\tau} / q \mathbb{Z} \to Y $ satisfying the ODE $\partial_{\tau} \gamma =R \circ \gamma$ for some $q>0$.   The number $q$  is called period or degree of $\gamma$.
	
	An \textbf{orbit set} $\alpha=\{(\alpha_i, m_i) \}$ is a finite collection  of periodic orbits, where $\alpha_i$ are distinct, nondegenerate, irreducible embedded periodic orbits and $m_i$  are positive integers.
	An orbit set $\alpha$ is called a  \textbf{PFH generator} if $m_i=1$ whenever $\alpha_i$ is a hyperbolic orbit.   Sometime  we write an orbit set using multiplicative notation $\alpha=\Pi_i \alpha_i$; we allow  $\alpha_i =\alpha_j$ for $i\ne j$.
	% The intersection number $[\alpha] \cdot [fiber]$ is called the degree of $\alpha$.

	%We define two kind of elliptic orbits which is used to define  the cobordism maps on PFH by holomorphic curves.
The following terminology is introduced by Hutchings  \cite{H5} which is used to define
the cobordism maps on periodic Floer homology in some special cases.
	\begin{definition} (see \cite{H5} Definition 4.1) \label{definition2}
		Fix  $d>0$.  Let  $\gamma$ be an embedded elliptic orbit with degree $q \le d$.
		\begin{itemize}
			\item
			$\gamma$ is called $d$-positive elliptic if the rotation number $\theta  $ is in $ (0, \frac{q}{d}) \mod 1$.
			\item
			$\gamma$ is called $d$-negative elliptic if the rotation number $\theta $ is in $ ( -\frac{q}{d},0) \mod 1$.
		\end{itemize}
	\end{definition}

\paragraph{Holomorphic curves and holomorphic currents}
Let $$(X, \Omega_X):=(\mathbb{R}_s \times Y_{\varphi},   \omega_{\varphi} +ds \wedge dt )$$ be  the symplectization of $Y_{\varphi}$.    An  almost complex structure is called \textbf{admissible} if $J(\partial_s) =R$, $J$ preserving $\xi$ and $J \vert_{\xi}$ is compatible with $\omega \vert_{\xi}$.   The set of admissible almost complex structures is denoted by $\mathcal{J}(Y_{\varphi}, \omega_{\varphi})$.

Fix   $J \in \mathcal{J}(Y_{\varphi}, \omega_{\varphi})$.
 A holomorphic curve is a map $u: (\dot{F}, j) \to (X, J) $ satisfying
the Cauchy-Riemann equation $J\circ  du =du \circ j$, where $(\dot{F}, j)$ is a punctured Riemann
surface without boundary. If the energy $\int u^*\omega_{\varphi}$ is finite, then each puncture of $u$
is either  positively or negatively asymptotic to a periodic orbit. Let  $\gamma$ be a periodic orbit
with degree $q$. We say $u$ is ``positively asymptotic to $\gamma$'' if there are holomorphic
coordinates $[R_0, \infty)_{\sigma} \times \mathbb{R}_{\tau}/ q\mathbb{Z} $ near each puncture such that  $\lim_{\sigma \to \infty} \pi_{\mathbb{R}} \circ u =+\infty$  and $\lim_{\sigma\to\infty} \pi_{Y_{\varphi}} \circ u =\gamma$.
 ``Negatively asymptotic to $\gamma$" is defined analogously with $\sigma \to -\infty$.
To distinguish with the HF curves defined later, we call a holomorphic curve in $(X, \Omega_X)$
\textbf{a PFH curve.}

A holomorphic curve is called \textbf{simple} if it does not factor as
$$(\dot{F}, j) \xrightarrow{\phi}  (\dot{F}', j') \xrightarrow{v} (X, J)$$
where $\phi$ is a branched cover of degree 2 or more, and $v$ is another $J$-holomorphic curve.
When a curve $u$ is simple, it is determined just by its image $C=u(\dot{F})$ in $X$. Thus,
sometimes we use the image $C$ to denote a simple holomorphic curve.
There is a similar concept called holomorphic current which is originated from
Taubes when he defines the Gromov invariant.

\begin{definition} [Section 3.1 of \cite{H4}]
A \textbf{holomorphic current} $\mathcal{C}= \sum_a d_a C_a$  from $\alpha$ to $\beta$ is a finite formal sum of  holomorphic curves such that  $\mathcal{C}$ is asymptotic to $\alpha$ as $s\to \infty$ and asymptotic to $\beta$ as $s \to -\infty$ in current sense, where  $C_a$  are distinct, irreducible,  simple  $J$-holomorphic curves with finite energy $\int_{C_a} \omega_{\varphi} < \infty$ and $ d_a$  are positive integers.
\end{definition}
By the standard theory of holomorphic curves (see Chapter 7, 8 of \cite{Wen}), there exists a Baire subset of admissible almost complex structures $\mathcal{J}^{reg}(Y_{\varphi}, \omega_{\varphi})$ such that the
moduli space of simple $J$-holomorphic curves is a manifold for $\mathcal{J}^{reg}(Y_{\varphi}, \omega_{\varphi})$. We
call an almost complex structure in  $\mathcal{J}^{reg}(Y_{\varphi}, \omega_{\varphi})$ a \textbf{generic almost complex structure}.
When discussing the almost complex structures of  $\mathbb{R} \times [0,1] \times \Sigma$ and the closed-open cobordism later, we will also use the term ``generic"  to describe those structures such that
the simple holomorphic curves are Fredholm regular.

%Let  $\mathcal{M}_i^J(\alpha, \beta, Z)$ denote the moduli space of holomorphic currents from $\alpha$ to $\beta$ with ECH index $I=i$ and relative homology class $Z$.

	\paragraph{The ECH index}
	Given orbit sets $  \alpha=\{(\alpha_i, m_i) \}$ and $\beta= \{ ( \beta_j, {n_j}) \}$ on $Y_{\varphi}$, define   $H_2(Y_{\varphi}, \alpha,\beta)$ to be the set of   $2$-chains $Z$ in $Y_{\varphi}$ such that $\partial Z = \sum_i m_i \alpha_i- \sum_j n_j \beta_j$, modulo the boundary of $3$-chains. Note that  $H_2(Y_{\varphi}, \alpha ,\beta)$ is an affine space over $H_2(Y_{\varphi}, \mathbb{Z})\cong \mathbb{Z}[\Sigma] \oplus (H_1(S^1, \mathbb{Z}) \otimes H_1(\Sigma, \mathbb{Z}))$. An element in   $H_2(Y_{\varphi}, \alpha ,\beta)$ is called a \textbf{relative homology class}.
	
	Given $Z \in H_2(Y_{\varphi}, \alpha ,\beta)$ and  trivializations    $\tau$ of $ \xi \vert_{\alpha}$ and $\xi \vert_{\beta}$,  the ECH index is defined by
	\begin{equation} \label{eq28}
		I(\alpha, \beta, Z) := c_{\tau}(\xi \vert_Z) + Q_{\tau}(Z) + \sum_i  \sum\limits_{p=1}^{m_i} CZ_{\tau}(\alpha_i^{p})- \sum_j \sum\limits_{q=1}^{n_j} CZ_{\tau}(\beta_j^{q}),
	\end{equation}
	where $c_{\tau}(Z)$ and $Q_{\tau}(Z)$ are respectively the relative Chern number and the relative self-intersection number  (see \cite{H2, H4}), and $CZ_{\tau}$ is the Conley-Zehnder index.   The ECH index also can be defined  for punctured holomorphic curves in a general symplectic cobordism \cite{H2}.  %The  ECH index $I$  depends only on orbit sets $\alpha$, $\beta$ and the  relative homology class  $Z$.
	
\paragraph{Fredholm index}	Let $u: \dot{F} \to X$ be a $J$-holomorphic curve from $\alpha= \{(\alpha_{i}, m_i)\}$ to $\beta= \{ (\beta_{j}, n_j)\}$. For each $i$, let $k_i$ denote the number of ends of $u$ at $\alpha_{i}$, and let $\{p_{ia}\}^{k_i}_{a=1}$ denote their multiplicities. Likewise,
	for each $j$, let $l_j$ denote the number of ends of $u$ at $\beta_{j}$, and let $\{q_{jb}\}^{lj}_{b=1}$ denote their multiplicities. Then the Fredholm index of $u$ is defined by
	\begin{eqnarray} \label{eq29}
		\operatorname{ind} u: = -\chi(C) + 2 c_{\tau}(u^*\xi) + \sum\limits_i \sum\limits_{a=1}^{k_i} CZ_{\tau} (\alpha^{p_{ia}}_{i}) -  \sum\limits_j \sum\limits_{b=1}^{l_j}  CZ_{\tau} (\beta^{q_{jb}}_{j}).
	\end{eqnarray}

\paragraph{$J_0$ index}
Besides the ECH index, Hutchings also introduces another topological index $J_0$ which measures the Euler characteristic of the holomorphic curves \cite{H2}. Given $Z \in H_2(Y_{\varphi}, \alpha ,\beta)$, the $J_0$ index is defined by
	\begin{equation}
		J_0(\alpha, \beta, Z): = -c_{\tau}(\xi \vert_Z) + Q_{\tau}(Z) + \sum_i  \sum\limits_{p=1}^{m_i-1} CZ_{\tau}(\alpha_i^{p})- \sum_j \sum\limits_{q=1}^{n_j-1} CZ_{\tau}(\beta_j^{q}).
	\end{equation}

It is useful to know that holomorphic currents in the symplectization of $Y_{\varphi}$ have
nonnegative $J_0$ index.
\begin{lemma} \label{lem27}
Let $J \in \mathcal{J}(Y_{\varphi}, \omega_{\varphi})$  be an admissible almost complex structure in  the symplectization of $\mathbb{R} \times Y_{\varphi}$. Let $\mathcal{C} \in \mathcal{M}^{J}(\alpha_+, \alpha_-)$ be a $J$-holomorphic current in $\mathbb{R} \times Y_{\varphi}$ without closed component. Then $J_0(\mathcal{C}) \ge 0$.
\end{lemma}
\begin{proof}
Let $C$ be a  simple $J$-holomorphic curve. By Corollary 6.11 of \cite{H2}, we have
\begin{equation*}
J_0(C) \ge 2(g(C)-1+\delta(C)) + |\alpha_+| + |\alpha_-|,
\end{equation*}
where $|\alpha_{\pm}|$ are  the quantities  in Definition 6.4 of \cite{H2}. By definition, $|\alpha_{\pm}| \ge 1$ if $\alpha_{\pm}$ is nonempty.
Assume that $C$ is not closed. Thanks to the fibration structure of $\mathbb{R}\times Y_{\varphi}$, $C$ has at least one positive end and one negative end. Thus, $|\alpha_+|, |\alpha_-| \ge 1$. In particular, $J_0(C) \ge 0$.

If $C=\mathbb{R} \times \gamma$ is a trivial cylinder, then $J_0(mC)=0$ for any $m\ge 1$ by definition.

For a general holomorphic current  $\mathcal{C}$, write $\mathcal{C} =\mathcal{C}_0 \cup \mathcal{C}_1$, where $\mathcal{C}_0$ consists of trivial cylinders and $\mathcal{C}_1$ does not contain trivial cylinder component.  Now $J$ is $\mathbb{R}$-invariant. We  may  assume the all the components in $\mathcal{C}_1$ are distinct by the translation trick in \cite{H1}.
More specifically, as a holomorphic current,  $mC=\sum C_i$
, where $C_i$
is $s_i$-translation
of $C$ and $s_i \ne s_j$ for $i\ne j$. Then we regard $C_i$ as distinct holomorphic curves and apply Proposition 6.14  of \cite{H2}, we have $J_0(\mathcal{C}) \ge 0.$
\end{proof}
\begin{remark}
In contrast, there may be holomorphic plane in $(\mathbb{R} \times Y, d(e^s \lambda))$ with $J_0=-1$ in  contact cases. Furthermore, the result may not true in the cobordism case.
The reason is the same as that the ECH index can be negative in cobordism. When the
multiple covered curve $mC$ appears, the results in \cite{H1, H2, H5} is not enough to show
that  $J_0(mC) \ge 0$.
\end{remark}

	%\subsection{Almost complex structure}
	
	%The symplectic form on $W$ is given by $\Omega_W=\omega_W+\pi^*\omega_B$, where $\omega_B$ is a large volume form of $B$. %Set $\omega=\omega_X \vert_{Y}$.  Assume that there is a neighborhood $U=[-\epsilon, 0]\times Y$ of $Y$ such that $\omega_X \vert_U=\omega$, then
	% We can define the symplectic completion $(\overline{W}, \omega_W)$ by adding a cylindrical end. (See section 2.3 of \cite{GHC}).

	%\begin{definition}
	% A holomorphic current $\mathcal{C}=\sum_a d_a C_a$ is called embedded if $d_a=1$ for any $a$ and $C_a$'s are pairwise disjoint embedded holomorphic curves.
	%\end{definition}
	
	\paragraph{Partition conditions}
Given a simple periodic orbit $\gamma$ and a holomorphic curve $u$,
suppose that $u$ has positive ends at covers of $\gamma$ whose total covering multiplicity is $m$.
The multiplicities of these covers form a partition of the integer $m$, denoted by $p^+_{\gamma}(u)$.
For example, if $u$ has two positive ends at $\gamma$, one positive end at the double cover of $\gamma$,
and one positive end at the triple cover of $\gamma$, with no other ends at $\gamma$, then  $m=7$ and
$p^+_{\gamma}(u)= (1, 1, 2, 3)$. Likewise, we define the partition $p^-_{\gamma}( u)$ for negative ends.

	\begin{definition}
		Let $u: \dot{F} \to X$ be a $J$-holomorphic curve from $\alpha=\{(\alpha_i, m_i)\}$ to $\beta=\{(\beta_j, n_j)\}$ without $\mathbb{R}$-invariant cylinders.    We  say that $u$ satisfies the \textbf{ECH partition conditions} if  $p^{+}_{\alpha_i}(u)=p^{+}(\alpha_i, m_i)$ and $p^{-}_{\beta_j}(u)=p^{-}(\beta_j, n_j)$, where $p^{+}(\alpha_i, m_i)$ and $p^{-}(\beta_j, n_j)$ are defined in Definition 4.3. of \cite{H1} (also see Section 3.9 of \cite{H4}).
	\end{definition}

	\begin{definition}
		A \textbf{connector}   $u: \dot{F} \to X$  is a union of branched covers  of trivial cylinders.   A connector is called  trivial if it is a union of unbranched covers of trivial cylinders; otherwise, it is called non-trivial.
	\end{definition}
	
	%In this article, we don't use the precise definition of $p_{+}(m,\gamma)$ and $p_{-}(m,\gamma)$.
The precise definitions of $p^{+}(m_i,\alpha_i)$ and $p^{-}(n_j,\beta_j)$ are not required in this paper. 	We  only  need   following facts about the connectors and ECH partition conditions.
	\begin{enumerate} [label=\textbf{F.\arabic*}]
		\item \label{fact1}
	Let $\gamma$ be an elliptic orbit. There are no nontrivial $\operatorname{ind}u=0$ connectors $u$ such that $p^+_{\gamma}(u)= p^+(\gamma, m)$.
Likewise, there are no nontrivial $\operatorname{ind}u=0$ connectors $u$ such
that $p^-_{\gamma}(u)= p^-(\gamma, m)$  (see exercise 3.14 of \cite{H4}).
		\item \label{fact2}
		Suppose that  $u$ is a simple holomorphic curve.  Then
\begin{equation}
I(u) \ge \operatorname{ind} u+ 2\delta(u),
\end{equation}
where $\delta(u)$   is a sign count of the double points and singularities of $u$. Equality
holds only if $u$ satisfies the ECH partition conditions.   For more details, see
Theorem 1.7 of \cite{H1} or Theorem 4.15. of \cite{H2}.
		This property also holds for the PFH-HF curves defined later; see Theorem 5.6.9 of \cite{VPK}.
		
		\item \label{fact3}
		The connector has nonnegative Fredholm index (see Lemma 1.7 of \cite{HT1}).

	\end{enumerate}

	\paragraph{Definition of periodic Floer homology }
	Now let us return to the definition of PFH.    Fix abase point $\mathbf{x} =(x_1,...,x_d) \in \Sigma^{\times d}$.  Define a reference 1-cycle by $$\gamma_H^{\mathbf{x}}: =\Psi_H(S^1 \times \mathbf{x}),$$ where $\Psi_H$ is the trivialization (\ref{eq11}). We require that $d> g$ throughout.

The pair  $(\varphi, [\gamma_H^{\mathbf{x}}])$ is called \textbf{monotone} if there exists a constant  $\rho \in \mathbb{R}$ such that
\begin{equation*}
c_1(\xi) +2 PD([\gamma_H^{\mathbf{x}}]) =\rho[\omega_{\varphi}].
\end{equation*}
The monotonicity  of $(\varphi, [\gamma_H^{\mathbf{x}}])$  implies that $I(Z_1) - I(Z_2) =\rho \int_{Z_1-Z_2} \omega_{\varphi} $ (Proposition 1.6
of \cite{H1}). This guarantees the finiteness of the classes in $H_2(Y_{\varphi}, \alpha, \beta)/ \ker\omega_{\varphi}$ with fixed
ECH index. For $\varphi = \varphi_H \in Ham(\Sigma, \omega) $, the pair $(\varphi, [\gamma_H^{\mathbf{x}}])$ always satisfies the above
monotone condition. The reason is that using the trivialization (\ref{eq11}), we have
\begin{equation*}
c_1(\xi) =c_1(T\Sigma),  PD([\gamma_H^{\mathbf{x}}]) =d[\Sigma],  [\omega_{\varphi}] =[\omega] \in   H^2(\Sigma, \mathbb{R}) \cong \mathbb{R}.
\end{equation*}
	
	The PFH chain complex $\widetilde{PFC}(\Sigma, \varphi,  d)$   is a free module  generated by  $(\alpha, [Z])$, where $\alpha$ is a PFH generator  with degree $d$, %Two relative homology classes $Z_1, Z_2$ are equivalent if $\int_{Z_1} \omega_{\varphi} +\eta J_0(Z_1)= \int_{Z_2} \omega_{\varphi} +\eta J_0(Z_2)$.
and $[Z]$ is an equivalence class in $ H_2(Y_{\varphi}, \alpha, \gamma_H^{\mathbf{x}}) /\ker \omega_{\varphi}$.   Note that  $ H_2(Y_{\varphi}, \alpha, \gamma_H^{\mathbf{x}}) /\ker \omega_{\varphi}$ is an affine space over $\mathbb{Z}[\Sigma]$. Fix a generic $J \in \mathcal{J}(Y_{\varphi}, \omega_{\varphi})$.  Let  $\mathcal{M}_i^J(\alpha, \beta, Z)$ denote the moduli space of holomorphic currents from $\alpha$ to $\beta$ with ECH index $I=i$ and relative homology class $Z$. The PFH differential     is defined by
	\begin{equation*}
		\partial_J (\alpha, [Z]): =\sum_{(\beta, [Z'])}  \left( \# {\mathcal{M}}^J_{1}(\alpha, \beta, Z-Z')/\mathbb{R}  \right) (\beta, [Z'])
	\end{equation*}
	Hutchings shows that ${\mathcal{M}}^J_{1}(\alpha, \beta, Z-Z')/\mathbb{R}$   consists of finitely many points (Theorem
1.8 of \cite{H1}), thus the counting is well-defined. Also, the right hand side of the above equation is a finite sum thanks to the monotonicity of    $(\varphi, [\gamma_H^{\mathbf{x}}])$.
	The obstruction gluing  argument  in \cite{HT1, HT2} shows that $\partial_J^2=0$.   The homology of $(\widetilde{PFC}(\Sigma, \varphi,  d ),  \partial_J)$ is called \textbf{twisted periodic Floer homology}, denoted by $\widetilde{PFH}_*(\Sigma, \varphi, d)_J$.
	By  Lee and Taubes's isomorphism \cite{LT}, we know that $\widetilde{PFH}_*(\Sigma, \varphi,
	d)_J$ is independent of the choices of $J$ and   Hamiltonian isotropy  of $\varphi$.

\begin{remark}
We require  $d>g(\Sigma)$ throughout because we will take $d$ to be the number of components of a link $\underline{\Lambda}$, which is  strictly greater than $g(\Sigma)$.

Another reason is that when $d \le g(\Sigma)$, one need to define PFH by using $(\omega_{\varphi}+ds \wedge dt)$-tame almost complex structure (see  (1.6) of \cite{LT}).    Some extract works  are needed for defining  the PFH cobordism maps and prove the invariance of the closed-open morphisms.
\end{remark}
	
	%Therefore, there is abstract group  $\widetilde{PFH}_*(\Sigma, \gamma_0)$ such that for every $\varphi_H$, we have a canonical isomorphism
	%	\begin{equation} \label{eq38}
	%\mathfrak{j}_H: \widetilde{PFH}_*(\Sigma, \varphi_H, \gamma_0) \to  \widetilde{PFH}_*(\Sigma, \gamma_0).
	%\end{equation}

	%Since we assume that $\varphi$ is a Hamiltonian symplecticmorphism, it is easy to check that $[\gamma_0]$ is negative monotone with respective to $[\omega_{\varphi}]$ in the sense of Definition 1.1 of \cite{LT}. In this case, $$
	% We use $PFH_*(Y, \omega, Q)=\oplus_{\Gamma \cdot [fiber] =Q} PFH_*(Y, \omega, \Gamma)$.

	\paragraph{Cobordism maps}
  Given two Hamiltonian functions $H$ and $ G$,  there is a canonical
isomorphism called a \textbf{PFH continuous morphism}
\begin{equation*}
\mathfrak{I}_{H, G} : \widetilde{PFH}_*(\Sigma, \varphi_{H}, d) \to \widetilde{PFH}_*(\Sigma, \varphi_{G},  d)
\end{equation*}
satisfying the following properties:
\begin{equation} \label{eq66}
\mathfrak{I}_{H, H}=\operatorname{Id},  \mbox{ and }    \mathfrak{I}_{H_2, H_3} \circ \mathfrak{I}_{H_1, H_2} =\mathfrak{I}_{H_1, H_3} 
\end{equation}
 for Hamiltonian functions $H_1, H_2, H_3$.
These isomorphisms are defined as follows:

A (complete) \textbf{symplectic cobordism}   from $(Y_{\varphi_+}, \omega_{\varphi_+})$ to $(Y_{\varphi_-}, \omega_{\varphi_-})$ is a symplectic manifold $(X, \Omega_X)$ such that
\begin{equation} \label{eq19}
		(X\setminus K, \Omega_X) \cong ([R_0, \infty) \times, Y_{\varphi_+},   \omega_{\varphi_{+}} + ds \wedge dt ) \cup  ((-\infty, -R_0] \times, Y_{\varphi_-},   \omega_{\varphi_{-}} + ds \wedge dt ),
	\end{equation}
	where $K\subset X$ is a compact subset.
Given a symplectic cobordism $(X, \Omega_X)$  from $(Y_{\varphi_+}, \omega_{\varphi_+})$ to $(Y_{\varphi_-}, \omega_{\varphi_-})$ and a relative homology class $Z_{ref} \in H_2(X, \gamma_+^{ref}, \gamma_-^{ref})$,
	%\begin{equation} \label{eq19}
		%(X-K, \Omega_X) \cong ([R_0, \infty) \times, Y_{\varphi_+},   \omega_{\varphi_{+}} + ds \wedge dt ) \cup  ((-\infty, -R_0] \times, Y_{\varphi_-},   \omega_{\varphi_{-}} + ds \wedge dt ),
	%\end{equation}
	%where $K\subset X$ is a compact subset.  Fix a reference relative homology class $Z_{ref} \in H_2(X, \gamma_0, \gamma_1)$.  Then $(X, \Omega_X, Z_{ref})$ induces a homomorphism
then they induce a homomorphism (Theorem 1 of \cite{GHC},  also see
Appendix \ref{appendixB})
	\begin{equation*}
		PFH_{Z_{ref}}^{sw}(X, \Omega_X): \widetilde{PFH}(\Sigma_+, \varphi_{+}, d_+)  \to  \widetilde{PFH}(\Sigma_-, \varphi_{-}, d_-).
	\end{equation*}
	Such a homomorphism is called a \textbf{cobordism map}. % Due to certain technical issues, the cobordism maps cannot be defined by holomorphic curves so far.
 For our purpose, we only need
the following symplectic cobordism:
	\begin{equation} \label{eq9}
		X=\mathbb{R} \times S^1 \times \Sigma,  \  \omega_{X} =\omega +dH_{s, t} \wedge dt,  \mbox{\ and \ } \Omega_X= \omega_X+ ds \wedge dt,
	\end{equation}
	where $H_{s, t} =\chi(s) H_t + (1-\chi(s)) G_t$,  where $\chi$ is a  cutoff function  such that $\chi(s)=1$ for $s \gg 1$ and $\chi(s)=0$ for $s\ll 0$. Then $(X, \Omega_{X})$ gives  a   symplectic cobordism  from $(Y_{\varphi_H}, \omega_{\varphi_H})$ to $(Y_{\varphi_G}, \omega_{\varphi_G})$ under the trivialization  (\ref{eq11}). Take the reference homology
class to be $Z_{ref} = [\mathbb{R} \times S^1 \times \mathbf{x}] \in H_2(X, \gamma_H^{\mathbf{x}}, \gamma_G^{\mathbf{x}})$. Then we define
	\begin{equation*}
		 \mathfrak{I}_{H, G}: =PFH_{Z_{ref}}^{sw}(X, \Omega_X).
	\end{equation*}
The properties (\ref{eq66}) follow from the holomorphic curve axioms and composition rule
of the PFH cobordism maps (see Theorem 1 of \cite{GHC} and Theorem \ref{thm3}).

The superscript ``$sw$" in  $PFH_{Z_{ref}}^{sw}(X, \Omega_X)$ indicates that the cobordism map is
defined by the Seiberg-Witten cobordism maps \cite{KM} via the isomorphism “SWF=PFH”
in \cite{LT}. The story of Seiberg-Witten equations with non-exact perturbation can be found in Chapter 30, 31, 32 of   \cite{KM}. Some review of these theory also can be found in
\cite{GHC, LT} and Appendix \ref{appendixA}. The idea of PFH cobordism maps is the same as the ECH cobordism maps defined by Hutchings and Taubes \cite{HT}. Their differences are summary in Remark 3.11 of \cite{CHS1}. Even these maps are defined by Seiberg-Witten equations, they
ensure that the existence of holomorphic curves in $X$ when the cobordism maps are
nontrivial. These crucial properties are called \textbf{holomorphic axioms.} We put the
precise statements in the Appendix \ref{appendixB} (see Theorem \ref{thm3}).

An $\Omega_X$-compatible almost complex structure $J_X$ is called  \textbf{admissible} if it satisfies
	\begin{enumerate}
		\item
		$J_X$ agrees with some  admissible almost complex structures   $J_{\pm} \in \mathcal{J}(Y_{\varphi_{\pm}}, \omega_{\varphi_{\pm}}) $ over the cylindrical ends;
		\item
		(Compatible with the fibration) $J_X$ preserves $TX^{vert}$ and $TX^{hor}$, where $TX^{vert}:=\ker (\pi_X)_*$ and $TX^{hor}$ is the $\Omega_X$-orthogonal complement.
	\end{enumerate}
	Fix a generic admissible almost complex structure $J_X$.  Suppose that $\varphi_{H_+}$  and  $\varphi_{H_-}$  satisfy
certain technical assumptions (described by \ref{assumption1} and \ref{assumption2}  later).  According to  Theorem 2 in \cite{GHC}, we can define the  cobordism map
	\begin{equation*}
		PFH_{Z_{ref}}(X, \Omega_X)_{J_X}: \widetilde{PFH}(\Sigma, \varphi_{H_+}, d)  \to  \widetilde{PFH}(\Sigma, \varphi_{H_-},  d).
	\end{equation*}
	alternatively  by  counting embedded $J_X$-holomorphic curves with $I=0$. More precisely,  $PFH_{Z_{ref}}(X, \Omega_X)_{J_X}$ is induced by the following chain map
 	\begin{equation*}
		PFC_{Z_{ref}}(X, \Omega_X)_{J_X}(\alpha_+, [Z_+]) =\sum_{(\alpha_-, [Z_-]), I(Z) =0} \#\mathcal{M}^{J_X}(\alpha_+, \alpha_-, Z)(\alpha_-, [Z_-])
	\end{equation*}
where $Z= Z_+\#Z_{ref} \#(-Z_-)$.
 In this case, we remove the superscript ``$sw$" to indicate that the map is defined by counting holomorphic curves.   Moreover, the holomorphic curve definition coincides with the Seiberg-Witten definition (see Theorem 3 of \cite{GHC}).

\begin{remark}
In Appendix \ref{appendixB}, the PFH cobordism maps are defined by using a larger class of almost complex structures called cobordism-admissible almost complex structure (defined in Appendix \ref{appendixA}). There is no deference between these two kinds of almost complex structures when defining PFH cobordism maps for the symplectic cobordisms (\ref{eq9}). However, using the admissible almost complex structures, we obtain a better estimate for the energy of holomorphic currents (see (\ref{eq6})).
\end{remark}

\paragraph{Invariance  of PFH}
The construction of the  PFH   does not depend on the choice
of base points.  For completeness, we briefly recall the canonical
identification between the groups defined using different base points, and
afterwards suppress this dependence from the notation.

Let $\mathbf{x}$ and $\mathbf{x}'$ be two choices of base points in
$\mathrm{Sym}^d\Sigma$.  Consider the relative homology class
$$
Z_{\mathbf{x},\mathbf{x}'}
:= \Psi_H(S^1 \times \eta)
\in H_2(Y_{\varphi_H}, \gamma_H^{\mathbf{x}},\gamma_H^{\mathbf{x}'})
$$
where $\eta:[0,1]\to \mathrm{Sym}^d\Sigma$ is a path with
$\eta(0)=\mathbf{x}'$ and $\eta(1)=\mathbf{x}$.
It induces an isomorphism
\begin{equation}\label{eq34}
	\Psi^{pfh}_{H, \mathbf{x}, \mathbf{x}'}: \widetilde{PFH}_*(\Sigma, \varphi_H,  d, \mathbf{x})  \to \widetilde{PFH}_*(\Sigma, \varphi_H, d, {\mathbf{x}'})
	\end{equation}
defined by 	 sending $(\alpha, [Z]) $ to $(\alpha,[Z+ Z_{\mathbf{x}, \mathbf{x}'}])$.

For two pairs $(H,\mathbf{x})$ and $(G,\mathbf{x}')$, the groups
$\widetilde{PFH}_*(\Sigma,\varphi_H,d,\mathbf{x})$ and 
$\widetilde{PFH}_*(\Sigma,\varphi_G,d,\mathbf{x}')$ are canonically
identified via 
\[
\Psi^{pfh}_{G,\mathbf{x},\mathbf{x}'} \circ \mathfrak{I}_{H,G},
\]
where $\mathfrak{I}_{H,G}$ is the continuation map.  Let
$\widetilde{PFH}_*(\Sigma,d)$ be the direct limit of the directed system
\[
\Bigl\{
\bigl(\widetilde{PFH}_*(\Sigma,\varphi_{H_i},d,\mathbf{x}_i),\;
\Psi^{pfh}_{H_j,\mathbf{x}_i,\mathbf{x}_j}\circ \mathfrak{I}_{H_i,H_j}
\bigr)
\Bigr\}.
\]
Then every $(H,\mathbf{x})$ gives a canonical isomorphism
\begin{equation} \label{eq71}
\mathfrak{j}_H:
\widetilde{PFH}_*(\Sigma,\varphi_H,d,\mathbf{x})
\longrightarrow
\widetilde{PFH}_*(\Sigma,d)
\end{equation}
which is compatible with the identifications above.  This shows that  PFH is independent of the choice of base point, and from
now on we suppress base points from the notation entirely.

\paragraph{PFH spectral invariants}
  The PFH complex carries  a natural action functional, 
 which  induces a filtration on the complex and gives us the PFH-spectral invariants \cite{CHS1, EH}.
We review them as follows.

	Let $H$ be a Hamiltonian function such that $\varphi=\varphi_H$ is $d$-nondegenerate.   Given a generator $(\alpha, [Z])$, its \textbf{action} is defined by
	$$\mathbb{A}_{H} (\alpha, [Z]) := \int_{Z} \omega_{\varphi_H} + \int_{S^1} H_t(\mathbf{x})dt, $$
where  $H_t(\mathbf{x}) :=\sum_{i=1}^dH_t(x_i).$
 	Define a  subset  $\widetilde{PFC}^L(\Sigma, \varphi_H,  d) $ generated by the anchored PFH generators   $(\alpha, [Z])$ with $\mathbb{A}_H  <L$.   It is easy to see this is a subcomplex. Its homology is  denoted by $\widetilde{PFH}^L(\Sigma, \varphi_H, d) $. The natural inclusion induces a homomorphism
		\begin{equation*}
	\mathfrak{i}_L: \widetilde{PFH}^L(\Sigma, \varphi_H,  d) \to  \widetilde{PFH}(\Sigma, \varphi_H,  d).
	\end{equation*}

\begin{definition}	\label{definition5}
 	Fix a nonzero  class $ \sigma \in \widetilde{PFH}(\Sigma,  d)$.  The  \textbf{PFH spectral invariant} at $\sigma$ is defined by
	\begin{equation*}
		c_d^{pfh}(H, \sigma):= \inf\{L \in \mathbb{R}\vert  \mathfrak{j}_H^{-1}( \sigma) \mbox{ belongs to the  image of }   \mathfrak{i}_L  \}.
	\end{equation*}
In the case that  $\varphi_H$ is degenerate,  we find a sequence $\{H_n\}_{n=1}^{\infty}$ such that $H_n$ converges in $C^{\infty}$ to $H$ and each $\varphi_{H_n}$ is nondegenerate. Then we define
$$	c_d^{pfh}(H, \sigma):= \lim_{n\to \infty} 	c_d^{pfh}(H_n, \sigma),$$
The limit is well defined by Corollary 4.4 of \cite{EH}.
\end{definition} 

If $\mathbf{x}$ and $\mathbf{x}'$ are two choices of base points,  note that we have 
	\begin{equation} \label{eq21}
		\mathbb{A}_{H}(\Psi^{pfh}_{H,\mathbf{x},\mathbf{x}'}(\alpha,Z))
= \mathbb{A}_H(\alpha,Z),
	\end{equation}
so the isomorphism (\ref{eq34}) preserves the filtration. 
As  in  Proposition 4.2 of \cite{EH},   this implies that $c_d^{pfh}(H,\sigma)$ is
independent of the base point.
 Combined with (\ref{eq34}) and (\ref{eq21}), we conclude
that the filtered PFH groups are likewise independent of the base point, a
fact already reflected in our notation.

	\subsection{Quantitative Heegaard Floer homology} \label{Section1.2}
 In this section, we review the quantitative Heegaard Floer homology defined in \cite{CHMSS}. We
 begin with introduce the definition of the link.

Let  $\underline{\Lambda} =\sqcup_{i=1}^d \Lambda_i$ be a disjoint union of simple  closed curves in $\Sigma$. We call  $\underline{\Lambda}$ a
link on $\Sigma$. \textbf{Throughout this paper, we fix a constant  $\eta \ge 0$ and assume that
the link $\underline{\Lambda} $  satisfies the following properties:}
	\begin{enumerate} [label=\textbf{A.\arabic*}]
		\item \label{assumption5}
		$\underline{\Lambda}$ consists of $k \ge 1$ disjoint contractible  components $\sqcup_{i=1}^k\Lambda_i$  and $g$ non-contractible
components. There exists a decomposition of $\Sigma$ as a sphere adding $g$ 1-handles. In
each handle, we have exactly one non-contractible component $\Lambda_i$.
		Consequently, $d=k+g$.
		\item \label{assumption3}
		Let $\Sigma \setminus \underline{\Lambda} =\cup_{i=1}^{k+1} \mathring{B}_k$, where $\mathring{B}_i$ is a disk for $1 \le i\le k$ and $\mathring{B}_{k+1} $ is a planar domain. Let $B_i$ denote the closure of $\mathring{B}_i$ in $\Sigma$. For $1\le i\le k $,	each circle $\Lambda_i$ is the boundary of   $B_i$.
		\item \label{assumption6}
		$\mathring{B}_i \cap \mathring{B}_j =\emptyset$ for $i\ne j$. %, where $\mathring{B}_i$ denote the interior of $B_i$.
		\item \label{assumption4}
		For $ 0\le i<j \le k$, we have  $\int_{B_i} \omega =\int_{B_j} \omega=\lambda$. Also, $\lambda=2\eta(2g+k-1)+ \int_{B_{k+1}} \omega$.
	\end{enumerate}
We call the link satisfying the above conditions \textbf{$\eta$-admissible}. See   examples in Figure \ref{figure4}.

\begin{figure} [h]
    \centering
    \includegraphics[width=1\textwidth,  height=0.2\textheight]{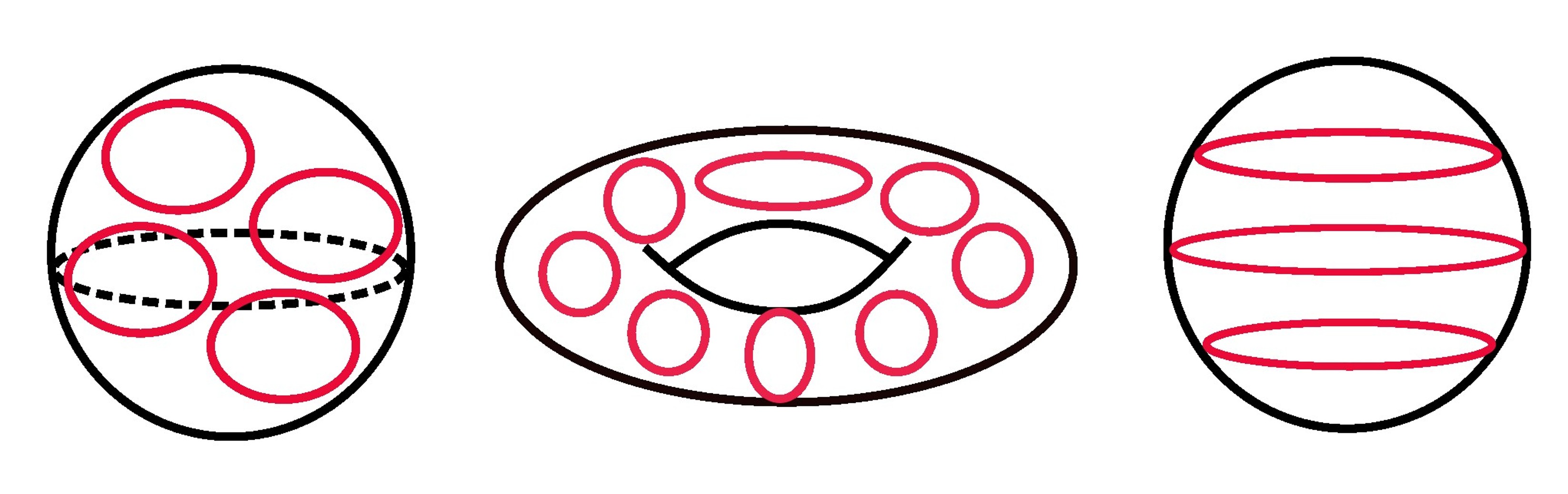}
    \caption{The red circles are the links. Here the first two pictures are examples
 of $\eta$-admissible links. The third one is a non-example. However, it satisfies the
 monotone conditions in Definition 1.12 of \cite{CHMSS}.}
  \label{figure4}
\end{figure}

	Let $\mathbb{M}:=\operatorname{Sym}^d \Sigma$ be the $d$ symmetric product of           $ \Sigma$. The product symplectic form $\omega^{\times d }$  on $(\Sigma)^{\times d}$ descends to  a singular K\"{a}hler  current $\omega_{\mathbb{M}}$ on $\mathbb{M}$. $\omega_{\mathbb{M}}$ is smooth away from the diagonal $\Delta$.   Fix   a small  neighbourhood  $V$  of $\Delta$.  According   Perutz's result (Section 7 of \cite{P}), we obtain a smooth K\"{a}hler  form $\omega_V$ on $\mathbb{M}$ such that
	\begin{itemize}
		\item
		$[\omega_V]=[\omega_{\mathbb{M}}] \in H^2(\mathbb{M}; \mathbb{R})$;
		\item
		$\omega_V=\omega_{\mathbb{M}}$ on ${\mathbb{M}} \setminus V$.
	\end{itemize}
	Then $\operatorname{Sym}^d\varphi_H(\underline{\Lambda})$ and $\operatorname{Sym}^d\underline{\Lambda}$ are $\omega_V$-Lagrangian submanifolds. In fact, they are monotone Lagrangian submanifolds in the sense of Lemma 4.19 of \cite{CHMSS}. %(see Remark 4.22 of \cite{CHMSS}).
The Hamiltonian symplectomorphism $\varphi_H$ is called \textbf{nondegenerate (with respect to $\underline{\Lambda}$) } if $\operatorname{Sym}^d\varphi_H(\underline{\Lambda})$   intersects  $\operatorname{Sym}^d\underline{\Lambda} $ transversely.  We assume that $\varphi_H$ is nondegenerate; unless otherwise stated.
	
	Fix a base point $\textbf{x}\in \operatorname{Sym}^d\underline{\Lambda}$. Define a reference chord
\begin{equation*}
\textbf{x}_H(t): =(\operatorname{Sym}^d\varphi_{H})  \circ (\operatorname{Sym}^d\varphi_{H}^{t})^{-1}(\mathbf{x}).
\end{equation*}
For any $\textbf{y} \in  \operatorname{Sym}^d \varphi_H(\underline{\Lambda}) \cap \operatorname{Sym}^d \underline{\Lambda}$,  a capping is a smooth map $\hat{\mathbf{y}}: [0,1]_s \times [0,1 ]_t \to  \mathbb{M}$ such that $\hat{\mathbf{y}}(1,t) =\textbf{x}_H(t) $, $\hat{\mathbf{y}}(0,t) =\textbf{y}$, and $\hat{\mathbf{y}}(s, i) \in  \operatorname{Sym}^d \varphi_H^{1-i}( \operatorname{Sym}^d \underline{\Lambda})$, $i \in \{0, 1\}$. Two cappings $\hat{\mathbf{y}}_0$ and $\hat{\mathbf{y}}_1$ are equivalent if $\textbf{y}_0=\textbf{y}_1$ and $\int_{\hat{\mathbf{y}}_0} \omega_V +\eta \Delta \cdot \hat{\mathbf{y}}_0=\int_{\hat{\mathbf{y}}_1} \omega_V+\eta \Delta \cdot \hat{\mathbf{y}}_1$.   The equivalent class of the capping is denoted by $\mathcal{S}$.
	
	\begin{definition} [Definition 13.1  of \cite{RL1}, \cite{CHMSS}]
		A path of  almost complex structures $\{{J}_t\}_{t\in [01,]}$ on $\operatorname{Sym}^d \Sigma$ is called  a  quasi-nearly-symmetric almost complex structure if it satisfies
		\begin{enumerate}
			\item
			${J}_t$ is $\omega_{\mathbb{M}}$-compatible  over $\operatorname{Sym}^d \Sigma \setminus V$.
			\item
			For each $t$,  there exists a complex structure $j_t$ on $\Sigma$ such that ${J}_t =\operatorname{Sym}^d(j_t)$ over $V$. 	
		\end{enumerate}
	\end{definition}
	
	Let $CF (\operatorname{Sym}^d\varphi_{H}(\underline{\Lambda}), \operatorname{Sym}^d\underline{\Lambda})$ be a free module generated by $(\textbf{y}, \mathcal{S})$.    Fix a generic  quasi-nearly-symmetric almost complex structure $\{J_t\}_{t\in [0,1]}$. Let $\mathcal{M}_1^{\{J_t\}}(\textbf{y}_+, \textbf{y}_-)$ be the moduli space of  holomorphic curves  $u: \mathbb{R}_s \times [0,1]_t \to \mathbb{M}$ satisfying
	\begin{equation*}
		\left\{  	\begin{split}
			&\partial_s u + J_t(u) \partial_t u  =0 \\
			&u(s, 0) \in \operatorname{Sym}^d\varphi_{H}(\underline{\Lambda}), \ u(s,1)\in \operatorname{Sym}^d\underline{\Lambda}\\
			& \lim_{s\to \pm\infty}u(s,t)=\textbf{y}_{\pm}\\
			&\operatorname{ind} u=1.
		\end{split} \right.
	\end{equation*}
	The differential is defined by
	\begin{equation} \label{eq48}
		m_1(\textbf{y}_+, \mathcal{S}_+) := \sum_{\textbf{y}_-}   \sum_{u \in \mathcal{M}_1^{\{J_t\}}(\textbf{y}_+, \textbf{y}_-)  / \mathbb{R}}  \varepsilon(u) (\textbf{y}_-, [u]\#\mathcal{S}_+).
	\end{equation}
	The \textbf{quantitative  Heegaard Floer homology} $HF_{*}(\operatorname{Sym}^d \varphi_H(\underline{\Lambda}), \operatorname{Sym}^d\underline{\Lambda})$ is the homology of  $(CF(\operatorname{Sym}^d\varphi_{H}(\underline{\Lambda}), \operatorname{Sym}^d\underline{\Lambda}), m_1)$. Under the assumptions \ref{assumption5}, \ref{assumption3} , \ref{assumption6} and \ref{assumption4},  the Lagrangian submanifolds $\operatorname{Sym}^d \varphi_H(\underline{\Lambda})$ and $\operatorname{Sym}^d\underline{\Lambda}$ are monotone in the sense of  Lemma 4.21  of \cite{CHMSS} (also see (\ref{eq65})).   By Lemma 6.6 in \cite{CHMSS}, the homology is well defined.
	
	The  quantitative  Heegaard Floer homology  is independent of the choices of Hamiltonian function and almost complex structure.  	 Moreover, Lemma 6.10 in \cite{CHMSS} shows that
	\begin{equation}\label{eq17}
		HF_{*}( \operatorname{Sym}^d \varphi_H(\underline{\Lambda}), \operatorname{Sym}^d\underline{\Lambda})\cong H^*(\mathbb{T}^d, \mathbb{F}[T^{-1}, T]),
	\end{equation}
where $\mathbb{T}^d$ is the $d$ dimensional torus and $T$ is a formal variable.

\begin{remark} \label{remark10}
Our assumptions on the links are different from the ones in  \cite{CHMSS}. Especially, if $\Sigma$ is not the sphere, then $\underline{\Lambda}$ does not satisfy the   planarity conditions  in Definition 1.12 of \cite{CHMSS}.

In Remark  \ref{remark8}, we  know that  $\operatorname{Sym}^d \varphi_H(\underline{\Lambda})$ and $\operatorname{Sym}^d\underline{\Lambda} $ are monotone in the sense of  Lemma 4.21  of \cite{CHMSS}. Therefore, the Floer homology  $HF_{*}( \operatorname{Sym}^d \varphi_H(\underline{\Lambda}), \operatorname{Sym}^d\underline{\Lambda})$ is still well defined.   Following  the same argument  in  Section 5 of \cite{CHMSS} to compute the disk potential, one can  prove the isomorphism (\ref{eq17}). Alternatively,  we can obtain the same result by applying    Mak-Trifa's  K$\ddot{u}$nneth formula (Theorem 6 of \cite{MT}).

 The reason using these links  is that we
need to compute the contribution of the bubbles to the ECH index. In the $\eta$-admissible
case, the homology classes of the bubbles basically are linear combination of disks and
$[\Sigma]$. In this case, the contribution of a disk to the ECH index  is computable (Lemma
\ref{lem6}). See Remark \ref{remark7} for more discussion on the computation of the ECH index. Later,
we need to choose a suitable Morse function on $\Sigma$ to compute the closed-open morphisms. Another  merit from $\eta$-admissible is that
we can draw down the Morse function in a concrete picture (Figure  \ref{figure2}).

\end{remark}

\paragraph{Invariance  of quantitative  Heegaard Floer homology}
	For two different choices of base points $\textbf{x}$,  $\textbf{x}'$, we  construct an isomorphism
\begin{equation*}
	\Psi^{hf}_{H, {\textbf{x}, \textbf{x}'}}: HF_{*}( \operatorname{Sym}^d\varphi_H(\underline{\Lambda}), \operatorname{Sym}^d\underline{\Lambda}, \textbf{x}) \to   HF_{*}( \operatorname{Sym}^d\varphi_H(\underline{\Lambda}), \operatorname{Sym}^d\underline{\Lambda}, \textbf{x}').
	\end{equation*}
as  follows.  Take a path  $\eta: [0,1] \to \operatorname{Sym}^d \underline{\Lambda}$ such that $\eta(0) =\mathbf{x}$ and $\eta(1) =\mathbf{x}'$. Define a
 surface  $u_{\mathbf{x}, \mathbf{x}'} (s, t) : =\left(s, t, (\operatorname{Sym}^d \varphi_H)\circ (\operatorname{Sym}^d \varphi_H^t)^{-1}) (\eta(s)\right)$. Note that $u_{\mathbf{x}, \mathbf{x}'} (s, 0) \subset \operatorname{Sym}^d \varphi_H(\underline{\Lambda})$ and  $u_{\mathbf{x}, \mathbf{x}'} (s, 1) \subset \operatorname{Sym}^d \underline{\Lambda} $.
 Then,  we define a map  $\psi_{\mathbf{x}, \mathbf{x}'} $
 at the chain level by sending  $(\mathbf{y}, \mathcal{S})$ to   $(\mathbf{y}, [u_{\mathbf{x}, \mathbf{x}'} ] \# \mathcal{S})$. One can verify directly that   $\psi_{\mathbf{x}, \mathbf{x}'} $
 is a chain map from the definition. Since  $\psi_{\mathbf{x}, \mathbf{x}'} $
is just a shifting of the generators’s
 cappings via  $u_{\mathbf{x}, \mathbf{x}'}$, $\psi_{\mathbf{x}, \mathbf{x}'} $
is an isomorphism at the chain level. Let  $\Psi^{hf}_{H, \mathbf{x}, \mathbf{x}'} : =(\psi_{\mathbf{x}, \mathbf{x}'})_*$
 be the induced map at  the homological level. The above discussion implies that   $\Psi^{hf}_{H, \mathbf{x}, \mathbf{x}'} $ is an
 isomorphism.

Let $\mathbb{I}_{H, G}: HF_{*}( \operatorname{Sym}^d\varphi_H(\underline{\Lambda}), \operatorname{Sym}^d\underline{\Lambda}) \to HF_{*}( \operatorname{Sym}^d\varphi_G(\underline{\Lambda}), \operatorname{Sym}^d\underline{\Lambda})$ be the \textbf{continuous morphism} on the Lagrangian Floer homology.   It is
 defined by the formula (\ref{eq48}) but replacing the index one curves by index zero holomorphic strips, and replacing the Lagrangian submanifolds  by a Lagrangian cobordism determined
 by $H$ and $G$.
 From the definition, we know that  $\Psi_{H, \mathbf{x}, \mathbf{x}'}^{hf}$
satisfies the following diagram:
 \begin{equation} \label{eq41}
 	\begin{CD}
				HF_{*}( \operatorname{Sym}^d \varphi_H(\underline{\Lambda}), \operatorname{Sym}^d \underline{\Lambda}, \textbf{x})  @>\Psi_{H, \mathbf{x}, \mathbf{x}'}^{hf}>> HF_{*}( \operatorname{Sym}^d \varphi_H(\underline{\Lambda}), \operatorname{Sym}^d \underline{\Lambda}, \textbf{x}')  \\
				@VV \mathbb{I}_{H, G}V @VV  \mathbb{I}_{H, G}V\\
				HF_{*}( \operatorname{Sym}^d \varphi_G(\underline{\Lambda}), \operatorname{Sym}^d \underline{\Lambda}, \textbf{x})  @>\Psi_{G, \mathbf{x}, \mathbf{x}'}^{hf}>> HF_{*}( \operatorname{Sym}^d\varphi_G(\underline{\Lambda}), \operatorname{Sym}^d \underline{\Lambda}, \textbf{x}').			
				\end{CD}
				\end{equation}
Therefore, we have an abstract Floer group $HF(\operatorname{Sym}^d\underline{\Lambda})$ and a canonical isomorphism
	\begin{equation} \label{eq38}
		 \textbf{j}_H:  HF_{*}(\operatorname{Sym}^d \varphi_H(\underline{\Lambda}), \operatorname{Sym}^d\underline{\Lambda}, \textbf{x}) \to HF(\operatorname{Sym}^d \underline{\Lambda} ).
	\end{equation}

	 There is a canonical class $\mathbf{1}_{\underline{\Lambda}} \in HF(\operatorname{Sym}^d \underline{\Lambda})$   called \textbf{unit}. It is the unit with
 respective to the quantum product. We defer the definition of $\mathbf{1}_{\underline{\Lambda}}  $   to Section \ref{section6.3} because
 we do not need it at the current stage. Also, we will give a concrete description when $H$
 is a small Morse function therein.

	\paragraph{Action functional and link spectral invariants}
	Given a generator $(\textbf{y}, \mathcal{S})$, its \textbf{action} is defined by
	\begin{equation}\label{eq46}
		\mathfrak{A}^{\eta}_H(\textbf{y}, \mathcal{S})  :=- \int_{\mathcal{S}} \omega_V + \int_0^1\operatorname{Sym}^d H_t( \mathbf{x})dt-\eta \Delta \cdot \mathcal{S},
	\end{equation}
where $\Delta \cdot \mathcal{S}$ is the algebraic intersection number with the diagonal.
	%We omit the subscript ``\textbf{x}" when the context is clear.
	
	Let $  CF^L(\operatorname{Sym}^d\varphi_H(\underline{\Lambda}), \operatorname{Sym}^d\underline{\Lambda}) $   be the complex generated by the cappings with $\mathfrak{A}^{\eta}_H<L$.   Similar to  the PFH case, for $L'<L$,   the inclusion induces a homomorphism
$$   \textbf{i}_{L, L'}: HF^{L}(\operatorname{Sym}^d\varphi_H(\underline{\Lambda}), \operatorname{Sym}^d\underline{\Lambda}) \to HF^{L'}(\operatorname{Sym}^d\varphi_H(\underline{\Lambda}), \operatorname{Sym}^d\underline{\Lambda})$$
 on the chain complex (see Lemma 6.12 of \cite{CHMSS}). We denote $\mathbf{i}_{L, \infty}$ by  $\mathbf{i}_{L}$.

 Fix a class $a  \ne 0 \in  HF(\operatorname{Sym}^d\underline{\Lambda})$. The \textbf{link spectral invariant at $a$} is defined by
	\begin{equation*}
		c^{link}_{\underline{\Lambda}}(H,   a ):= \inf\{L \in  \mathbb{R}\vert   \mathbf{j}_{H}^{-1}( a) \mbox{ belongs to the image of }  \textbf{i}_L \}.
	\end{equation*}
\begin{remark}
To ensure that the link spectral invariants take the same form as the
 PFH spectral invariants, the definition here differs from the one in  \cite{CHMSS} by a factor of $d$. More precisely, we have $c^{link}_{\underline{\Lambda} } (H, a) = d c_{\underline{\Lambda}, \eta}(H, a)$, where  $ c_{\underline{\Lambda}, \eta}(H, a)$ denotes the link
 spectral invariant defined in \cite{CHMSS}.
\end{remark}

\begin{lemma} \label{lem32}
 The link spectral invariant is independent of the choice of the base point $\mathbf{x}$.
\end{lemma}
\begin{proof}
Similar to the case of PFH, it suffices to show that   
\begin{equation} \label{eq49}
\mathfrak{A}^{\eta}_H(\textbf{y}, \mathcal{S})  =\mathfrak{A}^{\eta}_H( \psi_{\mathbf{x}, \mathbf{x}'} ( \textbf{y}, \mathcal{S}) ),
\end{equation}
then the lemma is true.

To prove the above equality, reintroduce
\begin{equation*}
u_{\mathbf{x}, \mathbf{x}'} (s, t) : = \left(s, t, (\operatorname{Sym}^d \varphi_H)\circ (\operatorname{Sym}^d \varphi_H^t)^{-1} (\eta(s))\right)=[u_1(s,t),..., u_d(s,t)],
\end{equation*}
consists of $d$ disjoint union of strips  $u_i(s,t) =(s, t, \varphi_H \circ (\varphi_H^t)^{-1}(\eta_i(s)))$, where $\eta_i: [0,1] \to \Lambda_i$ such that $\eta_i(0) =x_i$ and $\eta_i(1) =x_i'.$  Therefore, $u_{\mathbf{x}, \mathbf{x}'} \cdot\Delta =0$
and
\begin{equation} \label{eq13}
\begin{split}
\int_{[0,1] \times [0,1]}u_{\mathbf{x}, \mathbf{x}'}^*\omega_V&= \sum_{i=1}^d\int_{[0,1] \times [0,1]} \omega(X_H(\eta_i(s)), \partial_s \eta_i ) ds \wedge dt\\
&=  \sum_{i=1}^d\int_{[0,1] \times [0,1]}  d_{\Sigma} H_t(\partial_s \eta_i(s)) ds \wedge dt\\
& =\int_0^1 H_t(\mathbf{x}) dt - \int_0^1 H_t(\mathbf{x}') dt.
\end{split}
\end{equation}
where  $H_t(\mathbf{x})$ is short for $\sum_{i=1}^d H_t(x_i)$  and similar for $H_t(\mathbf{x}')$. Therefore, (\ref{eq49})  is true.

\end{proof}

\subsection{Statements of main results} 

 In this section, we give the precise statement about the results mentioned at the be
ginning of this paper.

As mentioned at the introduction, we define an invariant $HF(\Sigma, \varphi_H(\underline{\Lambda}), \underline{\Lambda})_J$ using R. Lipshitz’s idea \cite{RL1}. Roughly speaking, the chain complex is generated by the intersection points $\operatorname{Sym}^d \varphi_H(\underline{\Lambda}) \cap \operatorname{Sym}^d \underline{\Lambda}$ and certain cappings depending on the base point $\mathbf{x}$.
The differential is defined by counting punctured holomorphic curves in $\mathbb{R} \times [0,1] \times \Sigma$
 with Lagrangian boundary condition $\mathbb{R} \times \{0\} \times \varphi_H(\underline{\Lambda}) \cup \mathbb{R} \times \{1\} \times \underline{\Lambda}$, and the punctures are asymptiotic to the intersection points. The construction is very similar to the usual Lagrangian Floer theory, the main differences are that now we have several Lagrangian submanifolds  rather than a single one, and the topological type of the curves can be more complicated than a strip.

As $HF(\Sigma,  \varphi_H(\underline{\Lambda}), \underline{\Lambda})_J$  is a reconceptualization of the quantitative Heegaard Floer homology, it has various parallel structures to those of its counterparts. These include the following structures:
\begin{itemize}
\item
For different choices of Hamiltonian functions $H, G$ and almost complex structures $J_H, J_G$, we have a canonical isomorphism
\begin{equation*}
I_{H, G}: HF(\Sigma, \varphi_H(\underline{\Lambda}), \underline{\Lambda})_{J_H} \to HF(\Sigma, \varphi_G(\underline{\Lambda}), \underline{\Lambda})_{J_G}
\  (\mbox{Proposition }  \ref{lem30}).
\end{equation*}
 We still call it a continuous morphism for convenience.

\item
 The continuous morphisms tell us that the homology group is independent of the
 choice of Hamiltonian function and almost complex structure. In fact, it is also
 independent of the base point (\ref{eq22}). Thus, we have an abstract group $HF(\Sigma, \underline{\Lambda})$
 and a canonical isomorphism $$j_H : HF(\Sigma, \varphi_H(\underline{\Lambda}), \underline{\Lambda})  \to HF(\Sigma, \underline{\Lambda}).$$

\item
There is a special class $e_{\underline{\Lambda}} \in HF(\Sigma, \underline{\Lambda})$ which is corresponding to the unit  $\mathbf{1}_{\underline{\Lambda}} \in HF(\operatorname{Sym}^d \underline{\Lambda})$ (Definition \ref{definition4}). Moreover,  $e_{\underline{\Lambda}}  \ne 0$ in our situation.

\item
There is an action functional  $\mathcal{A}_H^{\eta}$  on the chain complex.  $\mathcal{A}^{\eta}_H$ is defined by the integration of the capping plus a term that measures the topological complexity of
 the capping (see (\ref{eq7})). We define the filtered version $HF(\Sigma, \varphi_H(\underline{\Lambda}), \underline{\Lambda})$ as the other
 two Floer theories, ie., the homology of the subcomplex generated by generators
 with  $\mathcal{A}^{\eta}_H<L$. For $L < L'$, we have a homomorphism
\begin{equation*}
 i_{L,L'}: HF^L(\Sigma, \varphi_H(\underline{\Lambda}), \underline{\Lambda}) \to HF^{L'}(\Sigma, \varphi_H(\underline{\Lambda}), \underline{\Lambda})
 \end{equation*}
induced by the inclusion.

\item
Using the filtration in the last bullet, we have the HF spectral invariants (\ref{eq23}).
\end{itemize}

The first result we got is the equivalence between  $HF(\Sigma, \varphi_H(\underline{\Lambda}), \underline{\Lambda})$  and the quantitative Heegaard Floer homology. Note that this also implies the non-trivialness of
$ HF(\Sigma, \varphi_H(\underline{\Lambda}), \underline{\Lambda})$.
	\begin{thm} \label{thm1}
		Fix an admissible link $\eta$ and a  base point $ {\mathbf{x}} \in \operatorname{Sym}^d \underline{\Lambda}$. Given a nondegenerate  Hamiltonian symplectomorphism $\varphi_H $,  then there is an isomorphism
		$$\Phi^L_H: HF_*^L(\Sigma, \varphi_H(\underline{\Lambda}), \underline{\Lambda}) \to HF^L_{*}( \operatorname{Sym}^d \varphi_H(\underline{\Lambda}), \operatorname{Sym}^d\underline{\Lambda})$$
for any $L \in \mathbb{R}$. Moreover, the isomorphisms satisfy the following properties:
\begin{enumerate}
\item
(Compatible with the action filtration) Given $L<L' \in \mathbb{R}$, we have the following diagram:
 \begin{equation}
 	\begin{CD}
				HF_*^L(\Sigma, \varphi_H(\underline{\Lambda}), \underline{\Lambda}) @>\Phi_H^L>> HF_{*}^L( \operatorname{Sym}^d \varphi_H(\underline{\Lambda}), \operatorname{Sym}^d\underline{\Lambda})  \\
				@VV i_{L, L'}V @VV  \mathbf{i}_{L,L'}V\\
				HF_*^{L'}(\Sigma, \varphi_H(\underline{\Lambda}), \underline{\Lambda})  @>\Phi_H^{L'}>> HF_{*}^{L'}( \operatorname{Sym}^d \varphi_H(\underline{\Lambda}), \operatorname{Sym}^d\underline{\Lambda}).			
				\end{CD}
				\end{equation}
Then $\Phi_H:=\lim_{L \to \infty} \Phi_H^L: 	HF(\Sigma, \varphi_H(\underline{\Lambda}), \underline{\Lambda})  \to HF_{*}( \operatorname{Sym}^d \varphi_H(\underline{\Lambda}), \operatorname{Sym}^d\underline{\Lambda}) $
is an isomorphism. In particular, we have  $	HF(\Sigma, \varphi_H(\underline{\Lambda}), \underline{\Lambda}) \cong H_*(\mathbb{T}^d, \mathbb{F}[T^{-1}, T]).$
  \item
(Compatible with the continuous morphisms) Let $H, G$ be two Hamiltonian functions such that $\varphi_H, \varphi_G$ are nondegenerate. Then we have 
\begin{equation}
 	\begin{CD}
				HF_*(\Sigma, \varphi_H(\underline{\Lambda}), \underline{\Lambda})_{J_H}  @>\Phi_H>> HF_{*}( \operatorname{Sym}^d \varphi_H(\underline{\Lambda}), \operatorname{Sym}^d\underline{\Lambda})_{\mathbb{J}_H}  \\
				@VV I_{H, G}V @VV  \mathbb{I}_{H, G}V\\
				HF_*(\Sigma, \varphi_G(\underline{\Lambda}), \underline{\Lambda})_{J_G} @>\Phi_G>> HF_{*}( \operatorname{Sym}^d \varphi_G(\underline{\Lambda}), \operatorname{Sym}^d\underline{\Lambda})_{\mathbb{J}_G}.			
				\end{CD}
				\end{equation}
In particular, $\Phi_H$ descends to an isomorphism 
\begin{equation*}
  \Phi:  HF_*(\Sigma,  \underline{\Lambda}) \to HF_*(\operatorname{Sym}^d \underline{\Lambda}). 
\end{equation*}

\item
(Preserve the units) We have 
\begin{equation*}
\Phi(e_{\underline{\Lambda}}) = \mathbf{j}_H\circ \Phi_H \circ j_H^{-1}(e_{\underline{\Lambda}}) =\mathbf{1}_{\underline{\Lambda}}.
\end{equation*}
\end{enumerate}
\end{thm}

 Using Theorem \ref{thm1}, it is easy to deduce the following corollary.
\begin{corollary} \label{corollary1}
For an $\eta$-admissible link $\underline{\Lambda}$ and $a  \in  HF(\Sigma, \underline{\Lambda})$, we have
\begin{equation}  \label{eq43}
c^{link}_{\underline{\Lambda}}(H, \Phi(a)) =c^{hf}_{\underline{\Lambda}}(H, a).
\end{equation}

\end{corollary}

  The next result provides the existence of the homomorphisms from twisted PFH to $	HF_{*}( \Sigma, \varphi_H(\underline{\Lambda}), (\underline{\Lambda})$  with desired properties. These homomorphisms are the closed-open  morphisms mentioned at the introduction.
	\begin{thm} \label{thm2}
		Fix a  $\eta$-admissible link $\underline{\Lambda}$ and a base point $\mathbf{x} \in \operatorname{Sym}^d \underline{\Lambda}$. For any nondegenerate Hamiltonian function $H$,   there exists a  homomorphism
$${\mathcal{CO}}(\underline{\Lambda}, H): \widetilde{PFH}(\Sigma, \varphi_H, d) \to HF(\Sigma, \varphi_H(\underline{\Lambda}), \underline{\Lambda}) $$
satisfying  the following properties:
		\begin{itemize}
			\item
			(Invariance) For any Hamiltonian functions $H$ and $G$, we have the following diagram:
			$$\begin{CD}
				\widetilde{PFH}_*(\Sigma,  \varphi_H, d) @> {\mathcal{CO}}(\underline{\Lambda}, H)>> HF(\Sigma, \varphi_H(\underline{\Lambda}), \underline{\Lambda}) \\
				@VV \mathfrak{I}_{H, G}V @VV I_{H, G}V\\
				\widetilde{PFH}_*(\Sigma,  \varphi_G, d)  @>{\mathcal{CO}}(\underline{\Lambda}, G) >> HF(\Sigma, \varphi_G(\underline{\Lambda}), \underline{\Lambda}).
			\end{CD}$$
As a result,   	$\mathcal{CO}(\underline{\Lambda}, H)$ descends to homomorphism
\begin{equation*}
\mathcal{CO}(\underline{\Lambda}) : \widetilde{PFH}_*(\Sigma, d) \to HF_{*}(\Sigma, \underline{\Lambda}). 
\end{equation*}
			
			\item
			(Non-vanishing)	 There is a non-zero class $\mathfrak{e}  \in \widetilde{PFH}_*(\Sigma,d)$ (Definition \ref{definition6})  such that
			$$\mathcal{CO}(\underline{\Lambda}) (\mathfrak{e})= e_{\underline{\Lambda}} \ne 0.$$
	%	where   $j^{\mathbf{x}}_H$ is the canonical isomorphism  (\ref{eq37}).
			In particular, the closed-open morphism is non-vanishing.

\item (Decreasing spectral invariants) Assume  that the link $ \underline{\Lambda}$ is 0-admissible.
Suppose that there are nonzero classes   $\sigma \in \widetilde{PFH}(\Sigma, d)$ and  $a\in HF(\Sigma, \underline{\Lambda})$
 such that  $\mathcal{CO}(\underline{\Lambda}) (\sigma)=a$.
Then,  for any Hamiltonian function $H$, we have
\begin{equation*} 
c_{\underline{\Lambda}}^{hf}(H, a) \le  c_{d}^{pfh}(H, \sigma).
\end{equation*}
In particular,   we have 
\begin{equation}\label{eq68}
	c^{link}_{\underline{\Lambda}}(H,   \mathbf{1}_{\underline{\Lambda}})= c_{\underline{\Lambda}}^{hf}(H, e_{\underline{\Lambda}}) \le  c_{d}^{pfh}(H, \mathfrak{e}).
	\end{equation}
		\end{itemize}
	\end{thm}
	The class $\mathfrak{e}$ plays the role of unit in PFH; see Remark  \ref{remark4}.   Therefore, the  third  statement of Theorem \ref{thm2} says that the closed-open morphisms preserve the units.

	\begin{remark}  \label{remark2}
		By analogy with  Theorem  \ref{thm1}, it is conjectured that periodic Floer homology
 is isomorphic to symplectic Floer homology of $(\operatorname{Sym}^d\Sigma, \operatorname{Sym}^d\varphi)$ via the tautological correspondence.  The closed-open morphisms in
 Theorem \ref{thm2} should be regarded as an alternative formulation of the usual closed-open
 morphisms defined in  \cite{A}.  Moreover,  the estimate (\ref{eq68})  corresponds to Corollary 7.3 in \cite{CHMSS}.
	\end{remark}

	\section{A Heegaard Floer type homology}
	In this section, we construct a Heegaard Floer-type homology following Lipshitz’s approach \cite{RL1}. To begin with, we recall the essential concepts including HF curves, relative homology classes and the ECH index, introduced by V. Colin, P. Ghiggini, and
 K. Honda \cite{VPK}. All definitions in this section are the same as those in Section 4 of \cite{VPK},  except that our parameter $d$ is no longer related to the genus of the surface any more.

\begin{definition}
	Fix an admissible link $\underline{\Lambda}$ and a   nondegenerate  Hamiltonian symplectomorphism $\varphi_H \in Ham(\Sigma, \omega)$. A \textbf{Reeb chord}  of $\varphi_H$ is   a union of paths  $$ {\mathbf{y}}=[0,1] \times \{y_1,..., y_d\} \subset [0,1] \times \Sigma$$  so that $y_i\in \varphi_H(\Lambda_i)\cap \Lambda_{\sigma(i)},$ where $\sigma$  is a permutation of  $\{1, ... ,d\}$.
	\end{definition}
Note that there is a one-to-one correspondence between Reeb chords and intersection points  $\operatorname{Sym}^d \varphi_H(\underline{\Lambda}) \cap  \operatorname{Sym}^d \underline{\Lambda}$.  Thus, we do not distinguish these two concepts in
 the rest of the paper.

	Consider the symplectic manifold $M:= \mathbb{R}_s \times [0,1]_t \times \Sigma $  with the product  symplectic form $\Omega:=\omega+ds \wedge dt$. Let $L_0 :=\mathbb{R} \times \{0\} \times \varphi_H(\underline{\Lambda})$ and $L_1 :=\mathbb{R} \times \{1\} \times \underline{\Lambda}$. Note that they are Lagrangian submanifolds of $(M, \Omega)$.

\paragraph{Almost complex structure}
	An $\Omega$-compatible almost complex structure $J$ on $M$ is called \textbf{admissible} if it satisfies the following properties:
	\begin{enumerate}
		\item
		$J$ is $\mathbb{R}_s$-invariant.
		\item
		$J(\partial_s) = \partial_t$.
		\item
		$J$ preserves the vertical bundle of $M$.
	\end{enumerate}
	The space of admissible almost complex structures is denoted by $\mathcal{J}_M$.

	\paragraph{HF curves}
	Now we define the ``open" holomorphic  curves in $M$.
	%Let $\check{M}$, $\check{L}_0$ and  $\check{L}_1$ be the compactification of $M$, $L_0$ and $L_1$ respectively.
	Given two Reeb chords $\textbf{y}_{\pm}$, define $Z_{\mathbf{y}_+, \mathbf{y}_-} : =L_0 \cup L_1 \cup (\{\infty\} \times \mathbf{y}_+) \cup (\{-\infty\} \times  \mathbf{y}_-)$.
	\begin{definition} \label{definition3}(Definition 4.3.1 of  \cite{VPK})
		Let $(F,j)$ be a Riemann surface (possibly disconnected) with boundary marked points $\{p_{i}^{\pm}\}_{i=1}^d$. Let $\dot{F} =F \setminus  \{p_{i}^{\pm}\}_{i=1}^d$. Each irreducible component of $\dot{F}$ has at least one puncture. A \textbf{d-multisection} is a smooth map $u: (\dot{F}, \partial \dot{F}) \to (M,Z_{\mathbf{y}_+, \mathbf{y}_-}) $ such that
		\begin{enumerate}
			\item
			$u(\partial {\dot{F}}) \subset L_0, L_1$.  For each $1 \le i\le d$, $u^{-1}(L_0^i)$,  $u^{-1}(L_1^i)$ consists of exactly one component of $\partial \dot{F}$, where $L_0^i = \mathbb{R} \times \{0\} \times \varphi_H(\Lambda_i)$ and $L_1^i = \mathbb{R} \times \{1\} \times \Lambda_i$.

			\item
			$u$ is asymptotic to $ {\mathbf{y}}_{\pm}$ as $s \to \pm\infty $.
			
			%\item
		%	$\int_{\dot{F}} u^* \omega < \infty$.
		\end{enumerate}
Given $J \in \mathcal{J}_M$, 	a $J$-holomorphic $d$-multisection $u: (\dot{F}, \partial \dot{F}, j) \to (M, Z_{\mathbf{y}_+, \mathbf{y}_-} ) $  is called a ($J$-holomorphic) \textbf{HF curve}. The number $d$ is called the \textbf{degree} of the HF curve.
	\end{definition}

\begin{remark} \label{remark6}
Under like the PFH curves, we require that each irreducible component of $\dot{F}$ contains at least one puncture. This excludes the possibility that an HF curve contains an irreducible component entirely within a fiber or that the whole curve is contained within a fiber. As a result, an HF curve has at least  one positive end and one negative end.

One could  define a concept called ``generalized HF curves" to include these possibilities.  In Lemma \ref{lem6}, we show that each of these  exceptional components (curves inside a fiber) contributes at least two to the ECH index.
To define the differential $\partial$ of $ HF(\Sigma, \varphi_H(\underline{\Lambda}), \underline{\Lambda}) $, there is no difference between  using ``HF curves'' and ``generalized HF curves'', because we only consider the holomorphic curves with $I=1$. We only need a special case of   ``generalized HF curves''  when we prove $\partial ^2 =0$. To simplify  the concepts, we  use   Definition \ref{definition3}.

However, if a sequence of  HF curves converges to a broken holomorphic curve $\mathbf{u}$ in the sense of SFT  \cite{FYHKE}, then each level of $\mathbf{u}$ is a ``generalized HF curve''. See Lemma \ref{lem31} details.
\end{remark}
	
	Let $H_2(M, \textbf{y}_+, \textbf{y}_-)$ be the set of continuous maps $u: (\dot{F}, \partial \dot{F}) \to (M, Z_{\textbf{y}_+, \textbf{y}_-} ) $ satisfying the conditions $1),2)$,   modulo the equivalence relation  $\sim$.  Here  $u_1 \sim u_2  $ if and only if the compactifications $\check{u}_1$ and $\check{u}_2$ are equivalent in $H_2(M, Z_{\textbf{y}_+, \textbf{y}_-} ; \mathbb{Z})$.
	An element in $H_2(M, {\textbf{y}_+, \textbf{y}_-} )$  is called  \textbf{a relative homology class}. % One can replace the Reeb chords in the above definition by $\textbf{x}_H$, where $\textbf{x}_H(t) = (\varphi^t_H)^{-1} \circ \varphi_H(\textbf{x})$ and $\textbf{x}$ is a base point in $Sym^d \Lambda$. We will use this version of relative homology class later.
	
	Fix a base point $z_j $ in each connected component of $\Sigma \setminus ( \underline{\Lambda}\cup \varphi_H(\underline{\Lambda}))$. For each  $A\in H_2(M, {\textbf{y}_+, \textbf{y}_-} )$, we have a well-defined intersection number
	\begin{equation} \label{eq20}
		n_{z_j}(A) : = \#(A \cap (\mathbb{R} \times [0,1] \times \{z_j\})).
	\end{equation}
	A relative homology class $A$ is called \textbf{positive} if $n_{z_j}(A) \ge 0$ for each $j$. Note that if $A$ admits an HF curve, then $A$ is positive.

	\paragraph{Fredholm index}
	To define the Fredholm index of an HF curve $u$, we first need to
 fix a trivialization of $u^* T \Sigma$ as follows.  Fix a non-singular vector field $v$ along each component of  $\underline{\Lambda}$.  Then $(\varphi_*(v), j \varphi_*(v))$ and $(v, j(v))$ give trivializations on  $T\Sigma \vert_{L_0}$ and $T\Sigma  \vert_{L_1}$ respectively.  We extend the trivialization arbitrarily along $\textbf{y}_{\pm}$.  Such a trivialization is denoted by $\tau$.
	
	Define a real line bundle $\mathcal{L}$ over $\partial F$ as follows. We set $\mathcal{L} \vert_{\partial \dot{F}} = u^*T\varphi(\underline{\Lambda}), u^*T \underline{\Lambda}$. Extend $\mathcal{L}$ along $\partial F \setminus \partial \dot{F}$ by rotating  in counterclockwise direction from  $u^*T \varphi(\underline{\Lambda})$ to  $u^*T \underline{\Lambda}$ by the minimum account possible. Then $(u^*T\Sigma, \mathcal{L})$ forms  a bundle pair over $\partial  F$. With respect to the trivialization $\tau$, we have  well-defined
	\begin{enumerate}
		\item
		Maslov index: $\mu_{\tau}(u)= \mu(u^* T\Sigma, \mathcal{L}, \tau)$;
		\item
		relative Chern number: $ c_1(u^*T\Sigma, \tau). $
	\end{enumerate}
A fact is that  $2c_1(u^*T\Sigma, \tau) + \mu_{\tau}(u)$ is independent of $\tau$.
	
	The \textbf{Fredholm index} of an HF curve  is defined by
	\begin{equation*}
		\operatorname{ind} u := -\chi(F) +d +2c_1(u^*T\Sigma, \tau) + \mu_{\tau}(u).
	\end{equation*}
	
	\paragraph{ECH index}
	 The concept of ECH index also can be generalized to the current setting.  Given $A \in H_2(M, \textbf{y}_+, \textbf{y}_-)$,   an oriented immersed surface   $C \subset  M$ is a  \textbf{$\tau$-trivial representative} of $A$ if
	\begin{enumerate}
		\item
		$C$ intersects the fibers positively along $\partial C$;
		\item
		$\pi_{[0,1] \times \Sigma } \vert_C$ is an embedding near infinity; %$\partial C \cap (\{-1, 1\} \times [0,1] \times \Sigma)$;
		\item
		$C$ satisfies the $\tau$-trivial conditions in the sense of Definition 4.5.2 in \cite{VPK}.
	\end{enumerate}
	
	Given a class  $A \in H_2(M, \textbf{y}_+, \textbf{y}_-)$ with a  $\tau$-trivial representative $C$, we  define the \textbf{relative self-intersection number} as follows. Let $\nu_C $ be the normal bundle of   $C$.
	Let $\psi \in \Gamma(\nu_C)$ be a generic section such that $\psi \vert_{\partial C} =J \tau$, where $J$ is an admissible almost complex structure.  Let ${C}'$ be a push off of $ {C}$ along the direction $\psi$. Then   define
	$$Q_{\tau}(A):=\# ( {C} \cap  {C}').$$
	
	The \textbf{ECH index} (see Definition 4.5.11 in \cite{VPK}) is defined by
	\begin{equation*}
		I(A) :=c_1(T\Sigma \vert_A, \tau) + Q_{\tau }(A) + \mu_{\tau}( A).
	\end{equation*}
	
 The Fredholm index and ECH index satisfy the following relation. It is an analogy
 of Hutchings’s ECH inequality \cite{H1, H2}.
	\begin{theorem}[Theorem 4.5.13  of \cite{VPK}] \label{thm4}
		%By the relative adjunction formula,
		Fix an admissible almost complex structure $J$ (not necessarily generic). Then for a $J$-holomorphic  HF curve $u$, we have
		\begin{equation*}
			I(u) = \operatorname{ind} u +2 \delta(u),
		\end{equation*}
		where $\delta(u)$ is a sign count of the double points and singularities  of $u$.
		If $u$ is Fredholm regular, then   $I(u) \ge 0$.   Furthermore,  $I (u)= 0$ if and only if  $u$ is a union of  trivial strips.
	\end{theorem}

	\subsection{Moduli space of HF curves}
	Fix a relative homology class $A \in H_2(M, {\textbf{y}_+, \textbf{y}_-} )$. Let $\mathcal{M}^J({\textbf{y}_+, \textbf{y}_-} , A)$ denote the moduli space of $J$-holomorphic  HF curves with relative homology class $A$.

	\paragraph{Bubbling analysis}
 As mentioned at the introduction of our paper, the main difference with Lipshitz’s work  is that we also need to compute the contribution from the bubbles. In this subsection, we study the possible bubbles in the compactification of the moduli space of HF curves and their contribution to ECH index.
	%By the exact sequence $$0\to H_2(X) \to H_2(X, Z_{y_+, y_-}) \to H_1(Z_{y_+, y_-}) \to 0,$$ we have $ H_2(X, Z_{y_+, y_-})  = H_2(\Sigma) \oplus H_1(\Lambda) \oplus H_1( \phi(\Lambda))\oplus \mathbb{Z}^d$.
	%Let $\Sigma -\Lambda = \sqcup_{i=1}^{d+1} B_i$, where $B_i$ is a disk for $1\le i \le d$ and $B_{d+1}$ is a plannr domain with $d$ boundaries.

	By definition, $H_2(M, {\textbf{y}_+, \textbf{y}_-} )$ is an affine space over  $H_2([0,1] \times \Sigma,  \{0\} \times \varphi_{H}(\underline{\Lambda}) \cup \{1\} \times \underline{\Lambda}; \mathbb{Z})$.  By the exact sequence for relative homology,
  \begin{equation} \label{eq42}
  \begin{split}
  &0 \to  H_2([0,1] \times \Sigma, \mathbb{Z}) \to  H_2([0,1] \times \Sigma,  \{0\} \times \varphi_{H}(\underline{\Lambda}) \cup \{1\} \times \underline{\Lambda}; \mathbb{Z}) \to \\
   &H_1(  \{0\} \times \varphi_{H}(\underline{\Lambda}), \mathbb{Z}) \oplus  H_1(  \{1\} \times \underline{\Lambda}, \mathbb{Z})\xrightarrow{(i_0)_*-(i_1)_*} H_1([0,1] \times \Sigma, \mathbb{Z}) \to...
  \end{split}
  \end{equation}
  we know that $H_2([0,1] \times \Sigma,  \{0\} \times \varphi_{H}(\underline{\Lambda}) \cup \{1\} \times \underline{\Lambda}; \mathbb{Z}) \cong \mathbb{Z} [\Sigma] \oplus \ker ((i_0)_*-(i_1)_*)$. Note that  $\ker (i_0)_*-(i_1)_*$  consists
 of the contractible components in  $\underline{\Lambda}$  and $\varphi(\underline{\Lambda})$.
    Thus, $H_2([0,1] \times \Sigma,  \{0\} \times \varphi_{H}(\underline{\Lambda}) \cup \{1\} \times \underline{\Lambda}; \mathbb{Z})$ is spanned by classes $[\Sigma]$, $\{ \{0\} \times \varphi_H(B_i)\}_{i=1}^{k}$,  $\{ \{1\} \times B_i\}_{i=1}^{k}$.  To simplify
 the notation, we write $[\{0\} \times \varphi_H(B_i)]$  and     $[\{1\} \times B_i]$ as $[\varphi_H(B_i)]$ and  $[B_i]$  respectively.

	Therefore, for  $A, A' \in H_2(M, \textbf{y}_+, \textbf{y}_-)$,  we have
	\begin{equation*}
		A'-A=m[\Sigma]+\sum_{i=1}^k \left( c_i[B_i] +c_i'[\varphi_H(B_i)] \right) .
	\end{equation*}
	Since $\sum_{i=1}^{k+1}[B_i]= \sum_{i=1}^{k+1} [\varphi_H(B_i)] =[\Sigma]$, the above equation can be written alternatively as
	\begin{equation*}
		A'-A=\sum_{i=1}^{k+1} \left(c_i [B_i] + c_i' [\varphi_H(B_i') ] \right).
	\end{equation*}

	\begin{lemma} \label{lem6}
		Let $A, A' \in H_2(M,  {\mathbf{y}}_+,  {\mathbf{y}}_-)$  be positive homology classes  such that
\begin{equation*}
A'-A=m[\Sigma]+\sum_{i=1}^{k+1} \left( c_i[B_i] +c_i'[\varphi_H(B_i)] \right).
\end{equation*}
Then we have
		\begin{equation} \label{eq4}
			I(A')=I(A) + \sum_{i=1}^{k+1} (2c_i + 2c_i') +2m(k+1).
		\end{equation}
	\end{lemma}
	\begin{proof}
 The proof is divided into three steps. The first two steps is to prove the special
 case that $I(A+c_i B_i) =I(A) +2c_i$.  To this end, we construct a representative of $A+ B_i$ in Step 1, and then we compute the relative Chern number, relative self intersection number, and Maslov index of such a representative in Step 2. In Step 3, we use the above results and the definition of the ECH index to compute the general case.

\begin{itemize}
\item[\textbf{Step 1: }]
		Let $u$ be a $\tau$-trivial  representative  of $A$.  % Now we suppose that $c_{k+1}=m=0$ first.
As $u$ is $\tau$-trivial on the ends,  we assume that the ends of   $u$ are just trivial strips  by
 modifying the ends of $u$ but without changing the relative homology class. Let $\mathcal{E}_i=[-1, \infty) \times [0,1]_t \times y_i$ be a positive  end of $u$ and $y_i \in \varphi_H(\Lambda_i)$ for $1 \le i\le  k$.
		
 We perform a modification on  $\mathcal{E}_i$ as follows. Let $I=[0,1]$. Define  a path $p_0: I_s \to I_s \times  \{0\}\times \varphi_H( B_i)$ by  $p_0(s)=(s, 0, r=1, \theta_0 + 2\pi f(s))$, where $(r, \theta)$ is the polar coordinates of $\varphi_H( B_i)$, $f(s)$ is a  non-decreasing    function such that $f=0$ near $s=0$ and $f=1$ near $s=1$, and $(1, \theta_0)$ is the polar coordinates of $y_i$.   Note that $p_0$ wraps the boundary   $\partial \varphi_H(B_i)$ counterclockwise one time. Define  $p_1: I_s \to I_s \times \{1\} \times \varphi_H(B_i)$ by $p_1(s)=(s, 1, 1, \theta_0)$.  Let $p_-(t) =(0 ,t, 1, \theta_0)$ and  $p_+(t) =(1 ,t, 1, \theta_0)$. Then we have  an embedded map $v: I_s \times I_t \to  I_s \times I_t \times \varphi_H(B_i)$ such that
		\begin{equation*}
			\begin{split}
				&v(s,t ) =(s,t, v_{B}(s,t));\\
				&v(s,0)=p_0, \ v(s,1)=p_1;\\
				&v(0,t)=p_-, \ v(1, t)=p_+.
			\end{split}
		\end{equation*}
		Such a map $v$ exists   because  the topology of $I_s \times I_t \times \varphi_H(B_i)$ is trivial, the loop $p_0+p_+-p_1-p_-$ bounds a disk.
		We  replace a segment $[0,1 ]_s\times[0,1]_t \times y_i $ of the end $\mathcal{E}_i$ by $v$, the result is called $u'$.   Note that $[u']=[u] + [\varphi_{H}(B_i)]$.

\item[\textbf{Step 2:}]
Let $\psi \in \Gamma(u^* T\Sigma)$  be a generic section such that $\psi=\tau$ along $u(\partial \dot{F}) $, where $\tau$ is induced by a non-vanishing vector field on $\varphi_H(\Lambda_i)$.  We use the notation $\tau$ to denote the trivialization of $u'^* T\Sigma$  induced by the same vector field  on   $\varphi_H(\Lambda_i)$. Also, we choose $\tau$ such that it agrees with the one on $u^* T\Sigma$ outside the segment $v$.  Then the the real line bundle $\mathcal{L}$   still   is a constant along $u'(\partial \dot{F})$ with respect to $\tau$. Therefore, we have
\begin{equation} \label{eq50}
\mu_{\tau}(u') =\mu_{\tau}(u).
\end{equation}
		Take a generic  section $\psi' \in \Gamma({u'}^* T\Sigma)$ such that $\psi'=\psi$ outside the region $v$ and $\psi' \vert_{\partial v}=\tau \vert_{\partial v}$.
		%By the construction, $\psi'$ is still a constant along $u'(\partial\dot{F})$ with respect to the trivialisation $\tau$.
		Note that $\psi'^{-1}(0) =\psi^{-1}(0) $ outside  the segment $v$, but $\psi'$ has extra zeros inside  $v$. The zeros $\#(\psi'^{-1}(0) \vert_v) =\operatorname{wind}(\psi'\vert_{\partial v} ) $, where
		$\operatorname{wind}(\psi'\vert_{\partial v} )$ is the winding number of $\psi' \vert_{\partial v}$ in $v_B^*T \varphi_H(B_i)$. We have  $\#(\psi'^{-1}(0))=1  $ because $p_0$ wraps  $\varphi_H(\partial B_i)$ one time. As a result, we have
\begin{equation} \label{eq51}
c_{1}(u'^* T\Sigma,\tau)=c_{1}(u^* T\Sigma, \tau) +1.
\end{equation}		
		Note that after the modification, the ends of $u'$ still intersect the fibers positively. Therefore, we can take ${u'}^*T\Sigma$ to be the normal bundle. As before, we have $c_{1}(N_{u'}, J\tau)= c_{1}(N_{u}, J\tau) +1$. %Since $\varphi_H(\mathring{B}_i) \cap \varphi_H(\mathring{B}_j) =\emptyset $  for $1\le i \ne j \le k  +1$
		Since the other positive  ends of $u'$ are $[-1, \infty) \times [0,1] \times \{y_j\}$, where $y_j \in \varphi_H(\Lambda_j)$ ($j\ne i$),  and the image of $v$ is contained in $I_s \times I_t \times \varphi_H(B_i)$, the modification does not create any new double points. Hence, $\delta(u')=\delta(u)$. By Lemma 4.5.8 of \cite{VPK}, we have
\begin{equation}\label{eq52}
Q_{\tau}(u')=Q_{\tau}(u) +1 .
\end{equation}

		% Since $\tau$ trivialized  $T\Sigma$ over $\phi_H(B_i)$, the modification above will not introduce new zeros of the section $\psi$. Then $c_{\tau}(u'^*T\Sigma)=c_{\tau}(u^*T\Sigma)$ and $c_{\tau}(N_{u'})=c_{\tau}(N_u)$.  Sice $B_i$ are disjoint and the image of $v_B$ is contained in $B_i$, the modification does not introduce any new self-intersection points. Hence, $\delta(u')=\delta(u)$ and $Q_{\tau}(u')=Q_{\tau}(u)$. Finally, $\mu_{\tau}(\partial u')=\mu_{\tau}(\partial u) +2$ since  $\tau$ wrap the circle $\phi_H(\Lambda_i)$ one more time along $\partial u'$  than along $\partial u$.
		
		%Let $u$ be a $\tau$--trivial representative  of $A$. Suppose that $m=0$ first. Then we can represent $A$ by attaching $u$ with $c_i$ disjoint copies of $B_i$ or $\phi_H(B_i)$ to the boundaries of $u$.   As $u$ is $\tau$ trivial, we can assume that the ends of $u$ are just trivial strips. Also, $B_i$ or $\phi_H(B_i)$  is attached to the ends. Then we smooth the intersection point as the follows: We cut $\phi(B_i)$ along the red line in the following picture and then slice the the boundary of $\phi_H(B_i)$ along $\mathbb{R}_s \times \{1\} \times \partial \phi_H(B_i)$ such that the red line is move to $\partial u$.
		
		%\begin{figure}
		%      \begin{center}
		%      \includegraphics[width=10cm, height=10cm]{pic.png}
		%     \end{center}
		%\end{figure}
		We perform this operation $c_i \ge 0$ times  for $\varphi_H(B_i)$.  Then
\begin{equation*}
  I(A+c_i [\varphi_H(B_i)] ) = I(A) +2c_i.
\end{equation*}
If $c_i<0$ but  $A+c_i [\varphi_H(B_i)]$ still is positive,  then the above argument implies
\begin{equation*}
I(A)=  I(A+c_i [\varphi_H(B_i)]  -c_i [\varphi_H(B_i)] ) =I(A+c_i [\varphi_H(B_i)] ) -2c_i.
\end{equation*}
 Thus, the argument still work for  $c_i<0$.
For each $B_i$,  $1\le i\le k$, we  perform a similar operation. But in this case, we need to take $p_0$ to be a constant path and take $p_1$ to be a path wrapping  $\{1\} \times \partial B_i$ one time.  Then we get
		% We perform this operation for each copy of $B_i$ or $\phi(B_i)$.  The result is denoted by $u'$, then it is  a $\tau$--trivial representative of $A+ [\sum_i c_iB_i] +[\sum_i c_i'\phi(B_i)]$. The index of $u'$ is still well defined, and
		% $$indu'=ind u + \sum c_i(ind B_i +1) + \sum c_i' (ind \phi_H(B_i) +1).  $$ It is easy to check that $ind B_i =ind \phi_H(B_i)=1$.   Since $B_i$ doesn't intersect other $\Lambda_j$, hence attaching $B_i$, $\phi_H(B_i)$ to $u$ will not create any self intersection points. By the relative adjunction formula, we have
		\begin{equation*}
			I(A +\sum_{i=1}^k \left( c_i[B_i]  + c_i' [\varphi_H(B_i)]\right)) =I(A)+ \sum_{i=1}^k (2c_i + 2c_i').
		\end{equation*}

\item[\textbf{Step 3: }]	
		If $m\ne 0$, we first modify $u$ to be a $d$-multisection  $u'$ representing the class $A+ \sum_{i=1}^k\left( c_i [B_i] + c_i' [\varphi_H(B_i)] \right)$. Let $u''$ be a union of $u'$ with  $m$ copies of $\Sigma$. By definition of the ECH index, we have
\begin{equation*}
  \begin{split}
    I(u'')&=I(u')+I(m\Sigma) +2m\# (u' \cap \Sigma) \\
    &=I(u') + m\langle c_1(TM), \Sigma  \rangle + m^2\Sigma \cdot \Sigma  + 2m\#(u' \cap \Sigma)\\
    &=I(u')+ 2m(d-g+1)=I(u')+2m(k+1).
   \end{split}
\end{equation*}
Finally, to figure out the contribution of $B_{k+1}$, note that
\begin{equation}\label{eq64}
I(u+c[B_{k+1}]) +2ck = I(u+ c[B_{k+1}] + \sum_{i=1}^kc [B_i]) =I(u+ c[\Sigma]) =I(u) +2c(k+1).
\end{equation}
 Therefore,  $I(u+c[B_{k+1}]) = I(u) +2c$.   In sum, each $[B_i]$ ($1\le i\le k+1$)  contributes  $2$ to the ECH index, similarly for  $\varphi_H([B_i])$ ($1\le i\le k+1$).
		%Let $u$ is a node curve which can be written as $u=u_0 \cup \cup_i u_i$, where $u_0$ is a HF curve and $u_i$ is a $c_i$--branched covered of the disk $B_i$ with one boundary point attached to the boundary of $u_0$. The we define the ECH index of $u$ by $I(u)=I([u_0]+ \sum_i c_i [B_i])= I(u_0) + \sum_i 2c_i$.

\end{itemize}
	\end{proof}
	
\begin{remark} \label{remark7}
 The ECH index in \cite{VPK} is defined by using a  $\tau$-trivial representative.
 So strictly speaking, the ECH index cannot be calculated by using $u_0 \cup u_1$ directly, where $u_0$ is a $d$-multisection and $u_1$ is a bubble in a fiber, because $u_0\cup u_1$ is not a $\tau$-trivial representative. This is why in the proof of Lemma \ref{lem6}, we need to construct
 a representative of $A+[B_i]$ first. The construction is easy when $B_i$ is a disk, but  the author does not figure out the general case when he wrote the paper. Recently, J. Chaidez, O. Edtmair, L. Wang, Y. Yao, Z. Zhao develop the intersection theory for surfaces with Lagrangian boundary and generalize the ECH index \cite{CEWYZ}. Using their results, we could compute  $I(A+ [B_i])$ in terms of the ECH index of $A, B_i$ and their intersection number, without constructing a nice representative of $A + [B_i]$. Then the
 results of this paper may be generalized to the monotone link defined in \cite{CHMSS}.
\end{remark}

	\begin{lemma}\label{lem7}
		Fix an admissible almost complex structure $J \in \mathcal{J}_M$.  Let $u:F \to M$ be an irreducible $J$-holomorphic curve  with Lagrangian boundary conditions $u(\partial F) \subset \mathbb{R} \times \{0\} \times \varphi_H(\underline{\Lambda})$ or  $u(\partial F) \subset \mathbb{R} \times \{1\} \times  \underline{\Lambda} $,    provided that  $\partial F \ne \emptyset$. If  $u(F)$ is compact, then $u$ is contained in  a  fiber.
 \begin{comment}
 If $\partial F =\emptyset$, then it is a branched covering of $\Sigma$. If $\partial F \ne \emptyset$ and $u(\partial F) \subset \{s\} \times \{1\} \subset \Lambda_i$ or $\{s\} \times \{0\} \subset \varphi_H(\Lambda_i)$ for some $s \in \mathbb{R}$, then we have the following possibilities:
		\begin{enumerate}
			\item
			$u$ is a branched covering of $\{s\} \times \{1\} \times B_i$ or $\{s\} \times \{0\} \times \varphi_H(B_i)$   for   $1\le i\le k$;
			\item
			$u$ is a branched covering of $\{s\} \times \{1\} \times B^c_i$ or $\{s\} \times \{0\} \times \varphi_H(B^c_i)$   for  $1\le i\le k$, where $B_i^c $ is the closure of $\Sigma \setminus B_i$;
			\item
			$u$ is   a branched covering  of   $\{s\} \times \{0\} \times \Sigma  \setminus \cup_{i=1}^l\varphi_H(\Lambda_{k_j}) \cup _{j=1}^{l'}$     or  $\{s\} \times \{1\} \times \Sigma -  \Lambda_i$, $k+1\le i\le d$.  In these  two cases, the homology class of $u$ is $m[\Sigma]$.
		\end{enumerate}
\end{comment}
The  homology class $[u] \in H_2(\Sigma, \underline{\Lambda}, \mathbb{Z})$  is a linear combination of $\cup_i^k[B_i], \cup_{i=1}^k[\varphi_H(B_i)]$ and $[\Sigma]$.
	\end{lemma}
	\begin{proof}
		Note that $\pi \circ u : F \to \mathbb{R} \times [0,1]$ is holomorphic because $\pi$ is complex linear.  The image $\operatorname{Im} (\pi \circ u)$ is a compact subset by the assumption. Then the open mapping theorem implies that $\pi \circ u$ is a constant, i.e., $u$ lies inside a fiber. %It is easy to get the rest of conclusions.
	Then $[u] \in H_2(\Sigma, \underline{\Lambda}, \mathbb{Z})$ or  $[u] \in H_2(\Sigma, \varphi_H(\underline{\Lambda}), \mathbb{Z})$.  The computation in (\ref{eq42}) implies that   $[u]$  is a linear combination of $\cup_i^k[B_i], \cup_{i=1}^k[\varphi_H(B_i)]$ and $[\Sigma]$.
		%In the case that  $u(\partial F) \subset \{s\} \times \{1\} \subset \Lambda_i$ or $\{s\} \times \{0\} \subset \varphi_H(\Lambda_i)$ for some $i$
	\end{proof}

	\begin{comment}

	\begin{enumerate} [label=\textbf{Observation. }]
	\item  \label{observation}
	The   bubbles in the moduli space $\overline{\mathcal{M}^J}(y_+, y_-)$ must lie in the fiber and it is a branch covered of $B_i$ or $\phi_H(B_i)$.  Moreover,  homology class of the bubble is $c_iB_i + m_i[\Sigma]$ for some $c_i, m_i \ge 0$. Therefore, formula \ref{eq4} implies that each bubble contributes the index at least 2
	\end{enumerate}

	\end{comment}
%\begin{remark}
%Note that the definition of HF--curve allows  $u$ to contain  closed irreducible components (branched covering of the fiber).  However, for a generic almost complex structure, Lemma \ref{lem6} implies that the ECH index of $u$ is at least  $2(k+1)$.
%\end{remark}	

\subsection{$J_0$ index for HF curves}
	 In this subsection, we follow Hutchings’s approach to define the $J_0$ index for HF curves.
 There is a similar concept called $J_+$ index defined by C. Kutluhan, G. Matic, J. Van Horn-Morris, and A. Wand \cite{KMHW}.  In Lemma \ref{lem28}, we know that the  $J_0$ index is essentially the Euler characteristics of HF curves, while the  $J_+$ index also involves terms related
 to the Reeb chords $ \mathbf{y}_{\pm}$ (see (2-4) of \cite{KMHW}).

 Let $A \in H_2(M, \mathbf{y}_+, \mathbf{y}_-)$. We define the  \textbf{$J_0$ index } by
\begin{equation*}
J_0(A):=-c_1(T\Sigma \vert_A, \tau) + Q_{\tau }(A).
\end{equation*}
The properties of $J_0$ index are summarized in the following lemma.
\begin{lemma} \label{lem28}
The $J_0$  index satisfies the following properties:
\begin{enumerate}
\item
Let $u: \dot{F} \to M$ be an irreducible HF curve, then
\begin{equation*}
J_0(u)=-\chi(F) + d+ 2\delta(u).
\end{equation*}

\item
Let $u=\cup_a u_a$ be an HF curve and each $u_a$ is irreducible. Then
\begin{equation*}
J_0(u) = \sum_a J_0(u_a) +2 \sum_{a \ne b}  \#(u_a  \cap u_b).
\end{equation*}

\item
Let $u$ be an  HF curve, then $J_0(u) \ge 0$. Furthermore, equality holds if and only if $u$ is a disjoint union of holomorphic  strips.
\item
Let $A, A' \in H_2(M,  {\mathbf{y}}_+,  {\mathbf{y}}_-)$ be positive relative homology  classes. Suppose that $A'-A=m[\Sigma]+\sum_{i=1}^{k+1} \left( c_i[B_i] +c_i'[\varphi_H(B_i)] \right)$. Then
\begin{equation*}
J_0(A')=J_0(A) + 2(m + c_{k+1} +c_{k+1}')(d+g-1).
\end{equation*}
\end{enumerate}
\end{lemma}	
\begin{proof}
\begin{itemize}
\item
The item $(1)$ follows directly from the relative adjunction formula (Lemma 4.5.9) in \cite{VPK}.
\item
Without loss of generality, assume  that $u=u_0 \cup u_1$ has two irreducible components. Since   Chern number is additive and  the relative self-intersection is quadratic, we have
\begin{equation*}
\begin{split}
J_0(u) &=  -c_1(T\Sigma \vert_{u_0\cup u_1}, \tau) +Q_{\tau} (u_0 \cup u_1)\\
&= -c_1(T\Sigma \vert_{u_0}, \tau) +Q_{\tau} (u_0)  -  c_1(T\Sigma \vert_{u_1}, \tau)+ Q_{\tau} (u_1)  + 2\#(u_0 \cap u_1)\\
&= J_0(u_0) +  J_0(u_1) + 2\#(u_0 \cap u_1).
\end{split}
\end{equation*}
%Note that the HF curves have at least one boundary component.
\item
Let $u=\cup_a u_a $ be an HF curve and each $u_a: \dot{F}_a \to M$ is irreducible.  By definition, $F_a$ has at least one boundary. By the first bullet,
\begin{equation}\label{eq53}
J_0(u_a) = -\chi(F_a) + d_a +2\delta(u_a) = 2g(F_a) -2 +\#\partial F_a + d_a +2\delta(u_a) \ge 0,
\end{equation}
where $d_a \ge 1$ is the degree of $u_a$.  By the second item and intersection positivity of holomorphic curves, we have
\begin{equation*}
J_0(u)=\sum_a J_0(u_a) + \sum_{a\ne b}  2\#(u_a \cap u_b) \ge 0.
\end{equation*}

 If $J_0(u) =0$, then the above inequality implies that $J_0(u_a) =0$ and $ \#(u_a \cap u_b)=0$. By (\ref{eq53}), $J_0(u_a) =0$ only when $d_a =1$ and $\#\partial F_a =1$.  In other words,  $u_a$ is a holomorphic strip.  Conversely, if $u$ is a disjoint union of holomorphic strips, then the first and second bullets of the lemma imply that $J_0(u)=0$.

 \item
 Let $u$ be a  $\tau$-representative  of $A$.  By the construction in Lemma \ref{lem6},  we  obtain a   $\tau$-trivial representative  $u'$ with relative homology class  $[u'] =A+ \sum_{i=1}^{k} \left(c_i[B_i] + c_i'[\varphi_H(B_i)] \right)$.  By the computations (\ref{eq50}), (\ref{eq51}) and (\ref{eq52}),  we have $J_0(u') = J_0(u)$. In other words, $J_0(A+ \sum_{i=1}^{k} \left(c_i[B_i] + c_i'[\varphi_H(B_i)] \right))  =J_0(A)$.  A geometric interpretation of this formula is that adding a disk does not change the topology  of a $d$-multisection.

 We now compute contribution of $m[\Sigma]$ to the $J_0$ index.  By definition,
\begin{equation}
\begin{split}
J_0(A + m[\Sigma]) &= -c_1(T\Sigma \vert_{A+m[\Sigma]}, \tau) +Q_{\tau} (A+m[\Sigma])  \\
&=- c_1(T\Sigma \vert_{A}, \tau) -mc_1(T\Sigma \vert_{[\Sigma]}) + Q_{\tau} (A) + m^2[\Sigma]\cdot[\Sigma] + 2m\#(A\cap \Sigma)\\
&=J_0(A) -m\chi(\Sigma)  + 2md \\
&= J_0(A) +2m(d+g-1).
\end{split}
\end{equation}
Finally, the contribution of $c_{k+1}[B_{k+1}]$ follows from
\begin{equation*}
 \begin{split}
     J_0(A+ c_{k+1}[B_{k+1}])  &=J_0(A+ c_{k+1}[B_{k+1}] + \sum_{i=1}^k c_{k+1}[B_i]) \\
     &=J_0(A + c_{k+1}[\Sigma])=  J_0(A) +2c_{k+1}(d+g-1).
   \end{split}
\end{equation*}
The contribution of  $c'_{k+1}[\varphi_H(B_{k+1})]$ follows from the same argument.
\end{itemize}
\end{proof}

The following lemma is an analogy of Lemma 4.21 in \cite{CHMSS}. It illustrates that the admissible link is monotone.
\begin{lemma}
Let $A, A' \in H_2(M,  {\mathbf{y}}_+,  {\mathbf{y}}_-)$. Then  we have
\begin{equation}  \label{eq8}
		%I(A)-I(A') =2 (k+1) \left(\int_{A} \omega -\int_{A'} \omega \right).
\left(\int_{A} \omega -\int_{A'} \omega \right) + \eta \left(J_0(A)-J_0(A')  \right) = \frac{\lambda}{2}  \left(I(A)-I(A')  \right). 
\end{equation}
	
\end{lemma}
\begin{proof}
Suppose that  $A'-A=m[\Sigma]+\sum_{i=1}^{k+1} \left( c_i[B_i] +c_i'[\varphi_H(B_i)] \right) + \sum_{j=k+1}^d [w_j]$. By the same computation as (\ref{eq13}), we have $\int w_j^*\omega =0$. Then
\begin{equation*}
\int_{A} \omega -\int_{A'} \omega= \sum_{i=1}^{k} (c_i+ c_i') \lambda + (c_{k+1} + c_{k+1}')\int_{B_{k+1}} \omega +  m.
\end{equation*}
Combining the above equation with Lemmas \ref{lem6},\ref{lem28} and the assumption \ref{assumption4}, then we get the result.
\end{proof}

	\subsection{Transversality and compactness} \label{Section2.3}
With the preparations of the bubbling analysis in the last subsection, now we 
 prove the compactness of the moduli space with lower index. First, we deal with the
 transversality in the following lemma.
	\begin{lemma} \label{lem23}
		There is a Baire subset $\mathcal{J}^{reg}_M \subset \mathcal{J}_M$ such that for $J \in  \mathcal{J}^{reg}_M $,  each HF curve $u \in \mathcal{M}^J(  {\mathbf{y}}_+,  {\mathbf{y}}_-, A)$    is Fredholm regular. In particular, $I(u) \ge \operatorname{ind} u \ge 0.$
	\end{lemma}
	\begin{proof}
	 Just note that the HF curves cannot be multiply covered because of the asymptotic behaviours. Then the proof is standard. For the details, please refer to Section   3 of \cite{RL1} and Lemma 4.7.2 of \cite{VPK}.
	\end{proof}
Similar to  the case  of  PFH, we call an almost complex structure in $J^{reg}_M$
 a \textbf{generic} almost  complex structure.

 Before we move on, let us briefly review the Gromov compactness in \cite{FYHKE}. A \textbf{broken holomorphic curve} is a chain of holomorphic curves $\mathbf{u}=\{u^1, ..., u^N\}$ in $M$  such
 that the positive ends of  $u^i$ agree with the negative ends of  $u^{i+1}$. Moreover, none of  $u^i$  only consists of trivial strips.  The index $i$ is called the \textbf{level} of $u^i$.  Given a sequence of
 HF curves $\{u_n : (\dot{F}_n, j_n) \to M\}_{n=1}^{\infty}$ with fixed topological type and uniformly bounded
 energy, then the sequence converges to a broken holomorphic curve in the sense of SFT \cite{FYHKE}.   Roughly speaking, the convergence means that the domains   $\{  (\dot{F}_n, j_n)  \}_{n=1}^{\infty}$ converge
 to a nodal surface  $(F, j)$  (Riemann surface with double points) by collapsing some  ``special arcs” or ``special curves”, and then $\{u_n \circ \varphi_n \}_{n=1}^{\infty}$  converges $C^{\infty}_{loc}$ to $\mathbf{u}$  (modulo
 out the $\mathbb{R}$ translation) away from the double points, where $\varphi_n: (\dot{F}, j) \to  (\dot{F}_n, j_n)$ are
 certain identifications.    Let $F_i$ and $F_j$ be two irreducible components of  $ {F}$. As $n \to \infty$,  if the distance of the images of $u_n \circ \varphi_n \vert_{F_i}$ and  $u_n \circ \varphi_n \vert_{F_i}$
tends to infinity in $M$,   then
 the limits lies inside two different copy of $M$.   In this case,   $u_n \circ \varphi_n \vert_{F_i}$ and  $u_n \circ \varphi_n \vert_{F_i}$  converge to two curves in different level. Conversely, if their distances are bounded
 for all $n$,  then   $u_n \circ \varphi_n \vert_{F_i}$ and  $u_n \circ \varphi_n \vert_{F_i}$ converge to two curves in the same level.  On
 the other hand, some irreducible component  $u_n \circ \varphi_n \vert_{F_i}$
may converges to a curve with
 compact image in $M$. Such a curve is called a \textbf{bubble.} The degenerate phenomenon is shown in Figure \ref{figure5} and Figure \ref{figure3}.

\begin{figure}[h]
    \centering
    \includegraphics[width=0.8\textwidth,  height=0.25\textheight]{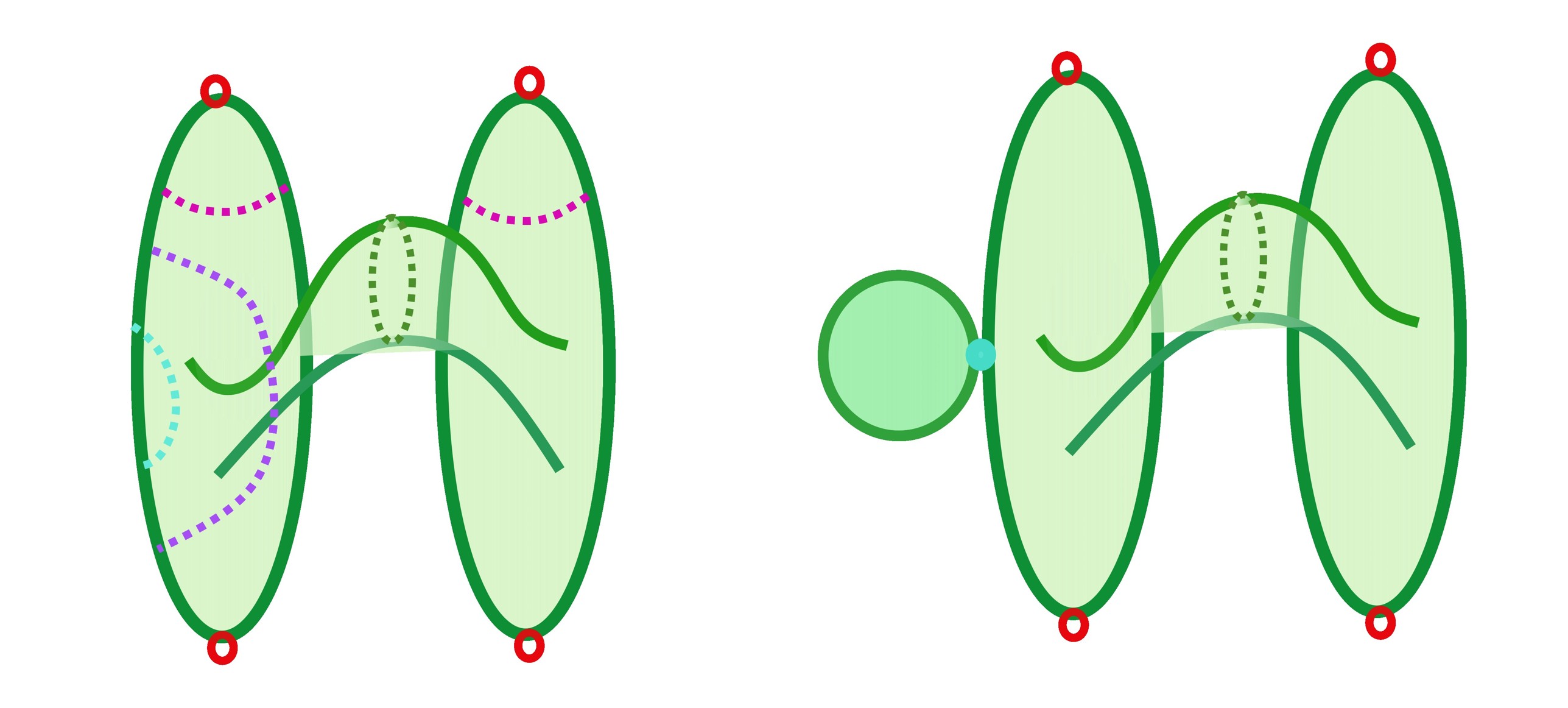}
    \caption{ The small red circles are the boundary punctures. If the holomorphic
 curve degenerates along the blue dashed arc, then a disk bubble is created as
 shown in the figure on the right.}
    \label{figure5}
\end{figure}

\begin{figure} [h]
    \centering
    \includegraphics[width=1\textwidth,  height=0.3\textheight]{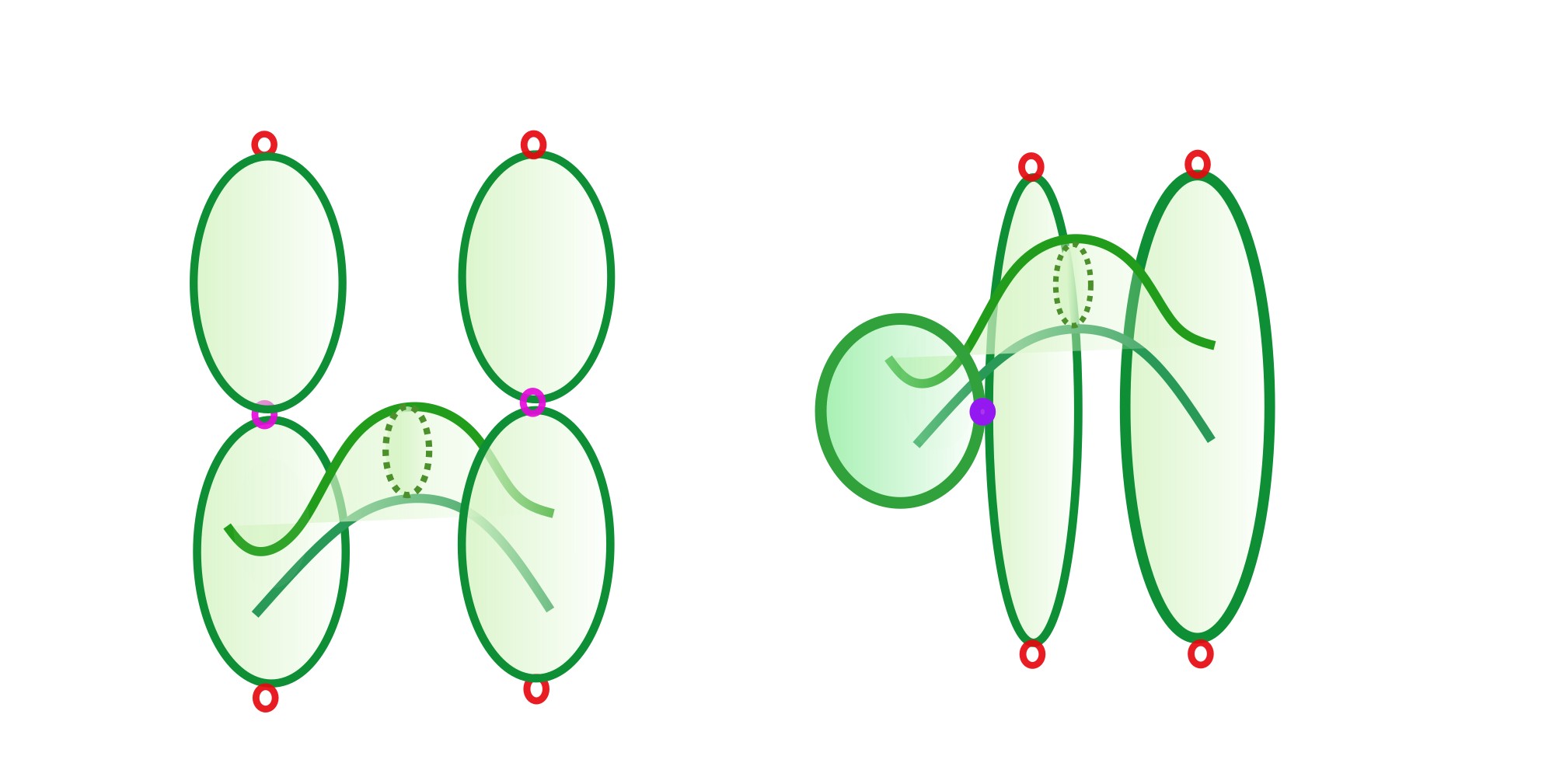}
    \caption{  If the curve degenerates along the pink dashed arcs in Figure 2, then
 the curve splits into two levels as depicted in the figure to the left. If the curve
 degenerates along the purple dashed arcs in Figure \ref{figure5}, then an extra boundary is
 created as shown in the picture on the right.}
    \label{figure3}
\end{figure}

\begin{lemma} \label{lem31}
 Let  $\{u_n\}_{n=1}^{\infty}$ be a sequence of HF curves with homology class $A$. Then $\{u_n\}_{n=1}^{\infty}$  converges to a broken holomorphic curve $\mathbf{u}= (u^1,...,u^N)$. Moreover, the
 non-compact component of each level is an HF curve (possibly with interior nodes).
\end{lemma}
\begin{proof}
By Lemma \ref{lem28}, the Euler characteristics of $\{\dot{F}_n\}_{n=1}^{\infty}$ have a uniform upper bound. Thus, we  assume that the topological type of $\{\dot{F}_n\}_{n=1}^{\infty}$  is fixed. The energy of
  $\{u_n\}_{n=1}^{\infty}$  is fixed $\int u_n^*\omega = \int_A \omega$.   Then, we apply the SFT compactness in Section 11.3
 of  \cite{FYHKE} mentioned above.

Recall that some arcs or curves in the domain of $u_n$ degenerate into a point, then
 we get a nodal surface $F= \cup_i F_i$, where $F_i$ are the irreducible components.  If a loop
  $a_n \subset \dot{F}_n$ degenerates to a point, then it creates a closed bubble with homology class $m[\Sigma]$.  If an arc    $a_n \subset \dot{F}_n$ degenerate to a point, then we have two possibilities: 1) One
 end point of $a_n$ lies in some  $u_n^{-1}(\varphi_H(\Lambda_i))$ and the other end point lies in $u_n^{-1}(\Lambda_{\sigma(i)})$. In
 this case, the curves splits into two levels (see left hand side of Figure \ref{figure3} for example).
 2) The end points of $a_n$ lie inside the same component of  $\partial \dot{F}_n$ because the Lagrangian submanifolds  are disjoint.   In the second case, the irreducible components with compact images in  $M$ are the bubbles described in Lemma \ref{lem7}. The noncompact irreducible components are supposed to be HF curves. However, a priori, some extra Lagrangian boundaries without punctures may created after the Gromov compactness, which violates the first
 condition in Definition \ref{definition3} (as shown in the right hand side of Figure \ref{figure3}). Suppose that $F_i$ is an irreducible component such that $u^j(F_i)$ is noncompact, and $\partial_0 F_i $ is a boundary
 component without punctures such that $u^j( \partial_0 F_i) \subset \mathbb{R} \times \{1\} \times \Lambda_i$ $ (\mathbb{R} \times \{0\} \times \varphi_H(\Lambda_i))$. The existence of such a component $F_i$  is  ruled out by the following reasons.

 Let $p\in \partial_0 F_i$ and $z$ be the holomorphic coordinate around $p$. Since $\pi: M \to \mathbb{R} \times [0,1]$  is holomorphic, $\pi\circ u^j(z) =z^k.$ Because the HF curves $\{u_n\}_{n=1}^{\infty}$ lie inside $M$, so does $\mathbf{u}$.  If $k > 1$,
 part of the image of $u^j$ goes outside $M$, then we have a contradiction. If $k=0$, then the
 open mapping theorem implies that  $\pi\circ u^j $ is a constant.  Then the image of $u^j$ cannot be
 non-compact, we have a contradiction again. Therefore, $k=1$.   In particular, $u^j(\partial_o F_i)$
 intersects $\Sigma$ positively at arbitrary $p \in \partial_o F_i$. On the other hand, $ [u^j(\partial_0F_i)]\cdot [\Sigma] = 0$. This leads
 to a contradiction. In sum, there is no extra boundary created for the non-compact components and the non-compact components of $u^j$ are HF curves.
\end{proof}

	\begin{lemma} \label{lem1}
		Suppose that $I(A)=1$. Then for a generic $J \in \mathcal{J}^{reg}_M$, the moduli space   $\mathcal{M}^J( {\mathbf{y}}_+,  {\mathbf{y}}_-, A)/ \mathbb{R}$ is  compact. In particular, $ \mathcal{M}^J( {\mathbf{y}}_+,  {\mathbf{y}}_-, A)/ \mathbb{R}$ is a  set of finite  points.
	\end{lemma}
	\begin{proof}
By  Lemma \ref{lem28},  $\mathcal{M}^J(\textbf{y}_+, \textbf{y}_-, A)/\mathbb{R}$ is a zero dimensional manifold. Thus, it suffices to show that it is compact.

		Let $\{u_n: \dot{F}_n \to M \}_{n=1}^{\infty}$ be a sequence of HF curves in  $\mathcal{M}^J(\textbf{y}_+, \textbf{y}_-, A)$.   By Lemma \ref{lem31},  $\{u_n\}_{n=1}^{\infty}$ converges to    a broken holomorphic curve $\mathbf{u}=\{u^1,...,u^N\}$.   Each level $u^i$ is a union of HF curves and  bubbles.  By Lemma \ref{lem28} , we know that the homology
 classes of the bubbles are linear combination of   $\cup_{i=1}^{k+1}[B_i], \cup_{i=1}^{k+1}[\varphi_H(B_i)]$,   and $[\Sigma]$.  Let  $u^i_{\star}$
be the   component of $u^i$ by removing the bubbles. Then
\begin{equation*}
[u^i]=[u^i_{\star}] + \sum_{j=1}^{k+1} \left(c_{ij} [B_j] + c_{ij}'[\varphi_H(B_j)] \right) +m_i[\Sigma]
\end{equation*}
	 for some $c_{ij}, c_{ij}', m_i \in \mathbb{Z}$.   The part $\sum_{j=1}^{k+1} \left(c_{ij} [B_j] + c_{ij}'[\varphi_H(B_j)] \right)$  comes from the bubbles
 with Lagrangian boundary conditions, and $m_i[\Sigma]$ comes from the closed bubbles. Let $v$     be a bubble component with homology class  $\sum_{i=1}^{k+1} a_i [B_i].$ Let $z_j \in \mathring{B}_j$.  By definition, $n_{z_j} (B_i) =\delta_{ij}$, where $n_{z_j}$ are  defined in (\ref{eq20}). By the intersection positivity of holomorphic curves, we have  $n_{z_j}(v) \ge 0$.  On the other hand, $n_{z_j}(v)=a_j$. Thus, $a_j \ge 0$.  Since $c_{ij}$
 are the sum of $a_j$ corresponding to different bubbles, it is nonnegative. By the same  reasons, we have  $c_{ij}' \ge 0$.  If $v$ is a closed bubble, then $[v] = a[\Sigma]$ for some $a\in \mathbb{Z}$. Since the
 energy of $v$ is nonnegative, we have $\int v^*\omega = a \ge  0$.  Hence, $m_i  \ge 0$.

	 By Lemma \ref{lem6}, we have
\begin{equation} \label{eq1}
\begin{split}
1=I(A) = I(\mathbf{u})&=  I([u^1]) +... +  I([u^N])\\
&=\sum_{i=1}^N \left( I(u^i_{\star}) + 2\sum_{j=1}^{k+1} (c_{ij} +c_{ij}') +2m_i(k+1)\right).
\end{split}
\end{equation}
 By Theorem \ref{thm4}, we have $I(u_{\star}^i) \ge  \operatorname{ind}  u_{\star}^i \ge 1$. Hence, $c_{ij} =c'_{ij} =m_i=0$. In other words,  the
 bubbles cannot exist. Moreover, $N = 1$ and $\mathbf{u} = u^1=u^1_{\star}$; otherwise, the right hand side of
 (\ref{eq1}) is strictly greater than one. Again by Theorem \ref{thm4},  $ \operatorname{ind}  u^1 = I(u^1) = 1, \delta(u^1) = 0$,
 and $u^1$ is embedded.
		
	\end{proof}

	\begin{lemma} \label{lem2}
		Suppose that $I(A)=2$. Then for a generic $J$,  $\partial \mathcal{M}^J( {\mathbf{y}}_+,  {\mathbf{y}}_-, A)/ \mathbb{R}$ consists of the following two types of broken holomorphic curves:
		\begin{itemize}
			\item
			$\mathbf{u}=(u_1, u_2)$, where $u_1 \in  \mathcal{M}^J( {\mathbf{y}}_+,  {\mathbf{y}}_0, A_1)/ \mathbb{R}$ and $u_1 \in \mathcal{M}^J( {\mathbf{y}}_0,  {\mathbf{y}}_-, A_2)/ \mathbb{R}$. Here $A_1$ and $A_2$ satisfy $A=A_1\#A_2$ and $I(A_1)=I(A_2)=1$.
			\item
			$\mathbf{u}=u_1\cup u_2$, where $u_1$ is a union of trivial holomorphic strips and $u_2$ is a simple index 1 holomorphic disk whose image is $\{s\} \times \{1\} \times B_i$ or $\{s\} \times \{0\} \times \varphi_H(B_i)$ for some $s$ and  $1\le i\le k +1$.      In particular, $ {\mathbf{y}}_+= {\mathbf{y}}_-$ in this case.
		\end{itemize}
	\end{lemma}
	
	\begin{proof}
		 By the same argument in Lemma \ref{lem1}, we still have Equation (\ref{eq1}). As $I(u_{\star}^i) \ge 0$,   $m_i=0$. If there exists some $u_{\star}^i$ with $I=1$, then $c_{ij}=c_{ij}'=0$ as well.  The Equation (\ref{eq1}) becomes
$$2=I(\mathbf{u}) = \sum_{i=1}^N I(u_{\star}^i). $$
Hence, there is another level with $I = 1$ and all the other levels have $I = 0$. The levels
 with $I = 0$ are just union of trivial strips, which violates the stable conditions of the
 holomorphic building (see Page 832 of \cite{FYHKE}). Hence, they cannot exist. Then we get the
 first conclusion.
		
		If $I(u^i_{\star})=0$ for all $i$, then $u_{\star}^i$ is just a union of trivial strips.  By  Equation (\ref{eq1}), either  $c_{ij}=1$ or $c_{ij}'=1$  for some $(i,j)$ and all the  other $c_{ij}, c_{ij}'. m_i$ are zero.    Thus,  $u$ can be written as a union of trivial strips $u_1$ and a fiber bubble $u_2$, and the
 homology class of $u_2$ is $[B_i]$ or $[\varphi_H(B_i)]$  $1\le i\le k+1$. The  image  $u_2$ is  $\{s\} \times \{1\} \times B_i$ or $\{s\} \times \{0\} \times \varphi_H(B_i)$.  Since  $u_2$  the
 homology class of $u_2$ is either $[B_i]$ or $[\varphi_H(B_i)]$,  $u_2$ cannot be branched covered.
		
	\end{proof}

	\subsection{A Heegaard Floer type homology} \label{section3.2}
 With the preparations from the previous section, we now define the cylindrical formulation of the quantitative Heegaard Floer homology.
	
	Fix a base point $\textbf{x}=(x_1, ..., x_d)$, where $x_i \in \Lambda_i$. As before, we  fix  a reference chord  $$\textbf{x}_H(t):=\varphi_H \circ (\varphi_H^t)^{-1} (\textbf{x})$$ from $\varphi_H(\underline{\Lambda})$ to $\underline{\Lambda}.$ Let $\textbf{y}$ be a Reeb chord.   Replacing the Reeb chord  $\textbf{y}_+$ in Definition \ref{definition3} by $\textbf{x}_H$,  define the space of  equivalence  classes of $d$-multisections from $\textbf{x}_H$ to $\textbf{y}$, denoted by $H_2(M, \textbf{x}_H, \textbf{y})$.
A capping of $\textbf{y}$ is  a relative homology  class $A \in H_2(M, \textbf{x}_H, \textbf{y})$. Two cappings $A_1, A_2$ are equivalent if and only if  $$\int_{A_1} \omega + \eta J_0(A_1)=\int_{A_2} \omega + \eta J_0(A_2). $$ Let $[A]$ denote the equivalence   class.
	
	% Let $\Gamma=\{\sum_{a_i}t^{\lambda_i} \vert a_i \in \mathbb{C}, \lambda_i \to \infty \}$ be the universal Novikov ring.
	The chain complex  is a free module  generated by the Reeb chords and cappings, i.e.,
	\begin{equation*}
		CF(\Sigma, \varphi_H(\underline{\Lambda}), \underline{\Lambda})=\oplus  \mathbb{F}  (\textbf{y}, [A]).
	\end{equation*}
	By the monotone assumption  (\ref{assumption4}),   all the classes $[B_i]$ and $[\varphi_H(B_i)]$ are equivalent, written as $B$.  It induces  a $\mathbb{Z}$-action on $CF(\Sigma, \varphi_H(\underline{\Lambda}), \underline{\Lambda}) $  by  $(\mathbf{y}, [A]) \to (\mathbf{y},[A\#B])$.
	
	Fix a generic $J \in \mathcal{J}^{reg}_M$. We define the  differential by
	\begin{equation} \label{eq54}
		d_J (\textbf{y}_+, [A_+])=\sum_{A \in H_2(M, \textbf{y}_+, \textbf{y}_-), I(A)=1} \#\left(\mathcal{M}^J(\textbf{y}_+, \textbf{y}_-, A)/\mathbb{R} \right) (\textbf{y}_-, [A_+ \#A]),
	\end{equation}
	%where $E(A) =\int_A \omega$ is the energy.
	By Lemma \ref{lem1},  the counting $\#\left(\mathcal{M}^J(\textbf{y}_+, \textbf{y}_-, A)/\mathbb{R} \right)$  is well defined.  Moreover,  (\ref{eq8}) implies  that  the classes  $A$ with $I(A)=1$ satisfy  $\int_A \omega + \eta J_0(A) =\kappa$, where $\kappa$ is a constant.  If  $\mathcal{M}^J(\textbf{y}_+, \textbf{y}_-, A) \ne \emptyset$, by Lemma \ref{lem28},  we have $J_0(A) \ge 0$.  Then, $0 \le \int_A \omega \le \kappa$.  Therefore, the right hand side of (\ref{eq54}) only consists of finitely many terms.
	
	According to  Lemma \ref{lem2} and the gluing argument in \cite{RL1}, we have
	\begin{equation*}
	d_J^2(\textbf{y}, [A])=\sum_{i=1}^{k+1}\left(\#\mathcal{M}^{J_1}(B_i) - \#\mathcal{M}^{J_0}(\varphi_H(B_i)\right) (\textbf{y},[A \# B]),
	\end{equation*}
	where $\mathcal{M}^{J_1}(B_i)$ and  $\mathcal{M}^{J_0}(\varphi_H(B_i))$ are moduli spaces of holomorphic curves  in $\Sigma$ with homology classes  $[B_i]$ and  $[\varphi_H(B_i)]$ respectively.  Obviously, the Hamiltonian symplectomorphism $\varphi^{-1}_H$ gives a diffeomorphism  between  $\mathcal{M}^{J_0}(\varphi_H(B_i))$  and   $\mathcal{M}^{J_1}(B_i)$. Therefore, $d^2_J =0.$

	In sum, the  homology of $(CF(\Sigma, \varphi_H(\underline{\Lambda}), \underline{\Lambda}), d_J)$  is well defined, denoted by  $ HF(\Sigma, \varphi_H(\underline{\Lambda}), \underline{\Lambda})_J$.

%	Since $\Lambda$ is pairwise disjoint,   likewise,  we  can define such an isomorphism  on   $HF(Sym^d\varphi_{H} (\Lambda), Sym^d \Lambda, \textbf{x})$.
	
\begin{comment}	
		\paragraph{Unit} When $H=H_{\varepsilon}$ be the Hamiltonian function  defined in Section \ref{section6},  then the unit    $ e^{\textbf{x}}_{H_{\varepsilon}} \in HF_*(\Sigma, \varphi_H, \Lambda, \textbf{x})_J$ is defined to be one in $HF(Sym^d\varphi_{H} (\Lambda), Sym^d \Lambda, \textbf{x})$ via the isomorphism $\Psi_{H_{\varepsilon}}$ in Theorem \ref{thm1}, i.e., $e^{\textbf{x}}_{H_{\varepsilon}}: =\Psi^{-1}(1^{\textbf{x}}_{{H_{\varepsilon}}})$.    Then $e^{\textbf{x}}_{H_{\varepsilon}}$ induces a class $e\in HF(\Sigma, \Lambda, \textbf{x}). $
		%As long as $0< \varepsilon  \ll 1 $  is sufficiently small, the class $e^{\textbf{x}}_{H_{\varepsilon}} $ is independent of the choice of $\varepsilon$   because  $1^{\textbf{x}}_{{H_{\varepsilon}}}$ is  represented by the maximum of $H_{\varepsilon}$.  See Remark \ref{remark2}.
		
		% For any other $H$, we have
		%$$e_H^{\textbf{x}}:=  I(H_{\varepsilon}, H  )_*(e^{\textbf{x}}_{H_{\varepsilon}}),$$
		%where
		%$  I(H_{\varepsilon}, H  )_* $ is the cobordism map defined in Section \ref{section3}.
		 In particular, we have $e  \ne 0$.   For the other choice of base point, we define the unit via the isomorphism in (\ref{eq22}).
\end{comment}
%	By the definition,  it is easy to check that the isomorphism also preserves the unit.

	\subsection{Invariance} \label{section3}
 The purpose of this section is to show that the homology $ HF(\Sigma, \varphi_H(\underline{\Lambda}), \underline{\Lambda})_J$ is independent of the choices of several data. The following Lemma \ref{lem10} and Proposition
\ref{lem30} are the same as Lemmas 6.1.1, 6.1.2, and 6.1.3 in  \cite{CHT}, where V. Colin, K. Honda,
 and Y. Tian prove the invariance of the higher dimensional Heegaard Floer homology.
 The construction of the Heegaard Floer homology in \cite{CHT} is essentially the same as
 $ HF(\Sigma, \varphi_H(\underline{\Lambda}), \underline{\Lambda})_J$. The difference is that the symplectic manifolds and Lagrangian submanifolds  in
 \cite{CHT} are required to be exact in order to rule out the bubbles, while here we rule out the
 bubbles using Lemma \ref{lem6}. Thus, ignore the effects from bubbles, one can image that
 the proof of Proposition \ref{lem30} is identical to those of \cite{CHT}.

 In the proof of Lemma \ref{lem10}, we use a construction that differs from \cite{CHT}. The
 advantage of this construction is that we can describe the reference relative homology
 in a clearer way.
	
	\begin{lemma} \label{lem10}[Lemma 6.1.1 of \cite{CHT}]
		Given Hamiltonian functions $H_- $ and $ H_+$, 	there exists a triple $( \Omega,  L_0, L_1)$ on $M$ such that
		\begin{enumerate}
			\item
			$\Omega$ is a symplectic form such that $\Omega =\omega + ds \wedge dt$ when $|s|  \ge  R_0$, where $R_o >0$ is a constant.
			\item
			$L_1= \mathbb{R} \times \{1\} \times \underline{\Lambda}$,   $L_0 \subset \mathbb{R} \times \{0\} \times \Sigma$   are $d$ disjoint union of cylinders.
 Moreover,
			\begin{equation*}
				\begin{split}
					& L_0 \vert_{s \ge R_0} =\mathbb{R} \vert_{s\ge R_0} \times \{0\} \times
					\varphi_{H_+}(\underline{\Lambda}),\\
					&L_0 \vert_{s \le -R_0} = \mathbb{R} \vert_{s\le -R_0} \times \{0\} \times \varphi_{H_-}(\underline{\Lambda}).
				\end{split}
			\end{equation*}
			\item
			$L_0$ and $L_1$ are disjoint union of  Lagrangian submanifolds with respective to $\Omega$.
		\end{enumerate}
	\end{lemma}
	\begin{proof}
	Let $\chi(s): \mathbb{R}_s  \to \mathbb{R} $ be a nondescreasing  cutoff function such that
\begin{equation*}
\chi(s)=
\begin{cases}
0&\mbox{if $s \le -R_0$}\\
1.  &\mbox{if $s \ge R_0$.}
\end{cases}
\end{equation*}
Let  $H^s := \chi(s)H_+ + (1- \chi(s)) H_-$.
Define a diffeomorphism
\begin{equation}  \label{eq44}
\begin{split}
F: &\mathbb{R} \times [0,1] \times \Sigma \to   \mathbb{R} \times [0,1] \times \Sigma\\
&(s, t, x) \to (s, t,  \varphi_{H^s} \circ (\varphi_{H^{s}}^{t})^{-1}(x)).
\end{split}
\end{equation}
 Let
\begin{equation} \label{eq56}
\begin{split}
&L_0  :=F(\mathbb{R} \times \{0\} \times \underline{\Lambda}), L_1  :=F(\mathbb{R} \times \{1\} \times \underline{\Lambda} )  =\mathbb{R} \times \{1\} \times \underline{\Lambda},\\
&\omega_E:=(F^{-1})^*(\omega + d(H^s_t dt)) \mbox{ and } \Omega_E =\omega_E +ds \wedge dt.
\end{split}
\end{equation}
$\Omega_E$ is symplectic as long as we choose $R_0$ large enough.  Then $L_0 \subset \mathbb{R} \times \{0\} \times \Sigma$ and $L_1 \subset \mathbb{R} \times \{1\} \times \Sigma$ are disjoint  union of  $\Omega_{E}$-Lagrangian submanifolds   such that
\begin{equation*}
\begin{split}
&L_0 \vert_{s \ge R_0} = \mathbb{R}_{s \ge R_0} \times    \{0\} \times \varphi_{H_+}(\underline{\Lambda}),  \\
&L_0 \vert_{s \le  -R_0} = \mathbb{R}_{s \le R_0} \times  \{0\} \times \varphi_{H_-}(\underline{\Lambda})) .
\end{split}
\end{equation*}

	\begin{comment}
	Let $K:[0,1]_t \times \Sigma$ be a Hamiltonian function. Take a cut-off function  $\chi_1:  \mathbb{R}_s \to  \mathbb{R}$ such that $\chi_1 =1 $ for $s \ge R_0$ and $\chi_1 =0$ for $s \le -R_0$. 	Let $G_s=\dot{\chi_1}(s) K_{\chi_1}$ be a Hamiltonian with support in $|s| \le R_0$.  Note that $\varphi_G^s =\varphi_K^{\chi_1(s)}$.  Let $\chi_2: [0,1]_t \to [0,1]$ be another cut--off function so that $\chi_2=1 $ near $t=0$ and $\chi_2=0$ near $t=1$. Define $G_{s, t} (x) =\chi_2(t) G_s(x)$.  Let $X_{G_{s,t}}$ be the Hamiltonian vector field of $G_{s, t}$. For each fixed $t$, we define a flow $\psi_{s,t}$ by $\partial_s \psi_{s,t}= X_{G_{s,t}} \circ \psi_{s,t}$  with initial value $\psi_{s, t}=Id$ for $s \le -1$.
		
		Let $L_0=\cup_s \psi_{s,0}(\varphi_H(\Lambda)) \subset \mathbb{R} \times \{0\} \times \Sigma  \subset M$ and $L_1 =\mathbb{R} \times \{1\} \times \Lambda$.  Note that $L_0 \vert_{s \ge R_0} = \mathbb{R}_{s \ge R_0}  \times \varphi_K \circ \varphi_H(\Lambda)$ and $L_0 \vert_{s \le R_0} = \mathbb{R}_{s \le R_0}  \times  \varphi_H(\Lambda)$.
		
		Let $\Omega= \omega + dG_{s,t}\wedge ds + ds \wedge dt$. If $R_0$ is sufficiently large,  then $\Omega$ is still a symplectic form.  It is easy to check that $L_1,L_0$ are Lagrangian submanifolds.
		\end{comment}
	\end{proof}

\begin{definition}
	The triple  $(\Omega, L_0, L_1)$ is called a \textbf{Lagrangian cobordism} from $(\varphi_{H_+}(\underline{\Lambda}),\underline{\Lambda})$ to $(\varphi_{H_-}(\underline{\Lambda}), \underline{\Lambda})$.
	\end{definition}

	\begin{prop} \label{lem30}
		The   homology  $HF(\Sigma, \varphi_H(\underline{\Lambda}), \underline{\Lambda})_J$ is independent of the choices of $H$ and $J$.
	\end{prop}
	\begin{proof}
		 Since the proof is the same as Section 6.1 in  \cite{CHT}, we just outline the main
 argument.
		Fix  two pairs $(H_-, J_-)$ and $(H_+, J_+)$, where $H_{\pm}$ are Hamiltonian functions and $J_{\pm} \in \mathcal{J}_M^{reg}$.  We  define a cobordism map  from $HF(\Sigma, \varphi_{H_+}(\underline{\Lambda}), \underline{\Lambda})_{J_+}$  to $HF(\Sigma, \varphi_{H_-}(\underline{\Lambda}), \underline{\Lambda})_{J_-}$  as follows.

 Take the triple  $ (\Omega, L_0, L_1)$ provided by Lemma \ref{lem10}. % Let $J_+$ and $J_- $ be two generic admissible almost complex structures.
		Let $J$ be a  $\Omega$-tame almost complex structure on $M$ such that
		\begin{enumerate}
			\item
			$J$ agrees with $J_+$ at the positive end ($s \ge  R_0$) and with  $J_-$ at the negative end ($s\le -R_0$), where $J_{\pm} \in \mathcal{J}^{reg}_M;$
			\item
			$\pi: M\to \mathbb{R} \times [0,1]$ is complex linear;
\item
The $J$-holomorphic  HF curves are Fredholm regular.
		\end{enumerate}
		
%Since we can find a family of Lagrangian cobordism $(\Omega_{\tau}, L_0^{\tau}, L_1^{\tau})_{\tau\in [0,1]}$	 such that  $(\Omega_{\tau}, L_0^{\tau}, L_1^{\tau}) \vert_{\tau=0}$ is the Lagrangian cobordism from $(\varphi_{H_+}(\Lambda), \Lambda)$ to $(\varphi_{H_-}(\Lambda), \Lambda)$  	 and  $(\Omega_{\tau}, L_0^{\tau}, L_1^{\tau}) \vert_{\tau=1}$ is the trivial Lagrangian cobordism from  $(\varphi_{H_-}(\Lambda), \Lambda)$ to itself.  This will induce an isomorphism between $H_2(M, \textbf{x}_{H_+} ,\textbf{x}_{H_-})$ and $H_2(M, \textbf{x}_{H_-} ,\textbf{x}_{H_-})$.
%Fix a reference homology class $A_{ref} \in H_2(M, \textbf{x}_{H_+} ,\textbf{x}_{H_-})$ such that it corresponds to the class in  $H_2(M, \textbf{x}_{H_-} ,\textbf{x}_{H_-})$ represented by $\mathbb{R} \times \mathbf{x}_{H_-}$.
		\begin{comment}
		Fix a reference homology class $A_{ref} \in H_2(M, \textbf{x}_{H_+} ,\textbf{x}_{H_-})$ satisfying the conditions:
\begin{enumerate}
\item
Let $u_{ref}$ be a representative of $A_{ref}$. Then there exists an open set $U$ such that $\pi_{\Sigma}(u_{ref}) \subset \Sigma-U$.
\item
$\pi_{\Sigma}(\partial u_{ref})$ doesn't bound  a nontrivial class in  $H_2(\Sigma, L_z^i) $, where $L_z = L_0 \cap (\{z\} \times  \Sigma) $ or   $L_z = L_1 \cap (\{z\} \times \Sigma) $, $z \in \mathbb{R} \times (\{0\}\cup \{1\})$ and $L_z^i$ is a component of  $L_z$ that bounds a disk ($1\le i \le k$).
\end{enumerate}
\end{comment}
Let $F$ be the diffeomorphism (\ref{eq44}). Let $A_{ref} \in H_2(M, \textbf{x}_{H_+} ,\textbf{x}_{H_-})$  be a class represented by $F(\mathbb{R} \times [0,1]\times \mathbf{x})$.
Then we  define a homomorphism
		\begin{equation*}
			CF_{A_{ref}}(\Omega, L_0, L_1)_J : HF(\Sigma, \varphi_{H_+}(\underline{\Lambda}), \underline{\Lambda})_{J_+} \to HF(\Sigma, \varphi_{H_-}(\underline{\Lambda}), \underline{\Lambda})_{J_-}
		\end{equation*}
		by counting $I=0$ $J$-holomorphic HF curves    in $(M, \Omega)$  with boundary on $L_0$ and $L_1$.    The compactness of the moduli space can be guaranteed by the same argument of
 Lemma \ref{lem1} and the index calculation (\ref{eq1}).
		
		Consider a sequence of $I=1$    HF curves  $\{u_n\}_{n=1}^{\infty}$ in $(M, \Omega)$  with boundary on $L_0$ and $L_1$. As $n\to \infty$,  the  same argument  in Lemma \ref{lem2} shows that  $\{u_n\}_{n=1}^{\infty}$ converges to a broken curve $\mathbf{u}=(u_1, u_2)$.  One level is an HF curve with $I=1$ and another level is an HF curve in the Lagrangian cobordism with $I=0$. The gluing argument (see Appendix in \cite{RL1}) shows that  $CF_{A_{ref}}(\Omega, L_0, L_1)_J$ is a chain map. Therefore, it induces a homomorphism
			\begin{equation*}
			HF_{A_{ref}}(\Omega, L_0, L_1)_J :HF(\Sigma, \varphi_{H_+}(\underline{\Lambda}), \underline{\Lambda})_{J_+} \to HF(\Sigma, \varphi_{H_-}(\underline{\Lambda}), \underline{\Lambda})_{J_-}.
		\end{equation*}
Moreover, the usual homotopy argument shows  that $HF_{A_{ref}}(\Omega, L_0, L_1)_J$  is  unchanged  if  we deform  $(L_0, L_1, \Omega, J)$ over a compact subset of $M$. Therefore, $HF_{A_{ref}}(\Omega, L_0, L_1)_J$
 only depends on $(H_+, J_+) $ and $(H_-, J_-)$.  To simplify the notation, we denote $HF_{A_{ref}}(\Omega, L_0, L_1)_J$ by $I_{H_+, H_-}$.
		
		Let $(M_{s_+}, \Omega, L_0, L_1)$ and $(M_{s_-}, \Omega', L'_0, L_1)$ be  Lagrangian cobordisms  from $(\varphi_{H_+}(\underline{\Lambda}), \underline{\Lambda})$ to $(\varphi_{H_0} (\underline{\Lambda}), \underline{\Lambda})$ and  from $(\varphi_{H_0}(\underline{\Lambda}), \underline{\Lambda})$ to $(\varphi_{H_-}(\underline{\Lambda}), \underline{\Lambda})$ respectively, where $s_{\pm}$ denote the $\mathbb{R}$-coordinates.   Define  a family of Lagrangian  cobordisms   by
		\begin{equation*}
			(M, \Omega_R, L_R^0, L_1) : = (M_{s_+}, \Omega, L_0, L_1)\vert_{s_+ \ge -R} \cup_{s_+=-R \sim s_-=R}  (M_{s_-}, \Omega', L'_0, L_1) \vert_{s_- \le R}
		\end{equation*}
		% $$L_0^R:=L_0 \vert_{s \ge -R_0} \cup ([0, R] \times \{0\} \times \phi_{H_0}(\Lambda) ) \cup L_0' \vert_{s\le R_0}.$$  We glue $\Omega$ with $\Omega'$ in the obvious way, the result is denoted by $\Omega_R$. Then $L_0^R$ is a Lagrangian with respect to $\Omega_R$.
		Note that $(L_R^0, L_R^1)$ are $\Omega_R$-Lagrangian submanifolds. %The cobordism map induced by $(L_R^0, L_1)$ is denoted by $I(\Omega_R, L^0_R\cup L_1)_*$.
		By the  neck stretching  and homotopy arguments, we have
		\begin{equation} \label{eq36}
		I_{H_+, H_-} =HF_{A_{ref}}(\Omega_R, L^0_R,  L^1_R)_* =I_{H_0, H_-}  \circ I_{H_+, H_0}.
		\end{equation}
		When  $(H_+, J_+) =(H_-, J_-)$,  we can choose the Lagrangian cobordism $(\Omega, L_0, L_1)$  to be   $\mathbb{R}$-invariant.  Then,
 take the almost complex structure to be in  $\mathcal{J}_M$.  As a result, the only $ I = 0$ HF curves
 are union of the trivial strips (see Theorem \ref{thm4}). Hence,  $I_{H, H}=\operatorname{Id}$.  Then  (\ref{eq36})    implies that  $I_{H_+, H_-} $ is an isomorphism.
	\end{proof}

 The map $I_{H_+, H_-}$
defined in the proof of Proposition \ref{lem30} is called a \textbf{continuous morphism.}

	\paragraph{Different choice of base points}
	Let $\textbf{x}'=(x'_1, ..., x'_d)$ be another base point, where $x_i' \in \Lambda_i$.  As in Section \ref{Section1.2}, we take a special relative homology class  $A_{\textbf{x}, \textbf{x}'}$ as
 follows:  Take a path $\eta :\sqcup_i [0,1]_{s_i}  \to  \underline{\Lambda}$ such that $\eta(0)=\textbf{x}$ and $\eta(1)=\textbf{x}'$.  Define  $u(s, t) :=(s, t, \varphi_H \circ (\varphi_{H}^{t})^{-1}(\eta(s)) $ and  $A_{\textbf{x}, \textbf{x}'}=[u] \in H_2(M, \textbf{x}_H', \textbf{x}_H)$. Then $u$ induces an isomorphism
	\begin{equation} \label{eq22}
		\Psi_{H, {\textbf{x}, \textbf{x}'}}:    HF_*(\Sigma, \varphi_H(\underline{\Lambda}),  \underline{\Lambda},   \textbf{x})_{J}  \to HF_*(\Sigma, \varphi_H(\underline{\Lambda}),   \underline{\Lambda}, \textbf{x}')_{J}
	\end{equation}
	by sending $(\mathbf{y}, [A])$ to $(\mathbf{y}, [A\#A_{\textbf{x}, \textbf{x}'}])$.
 By the same reasons as in (\ref{eq41}), the continuous maps and the maps  $\Psi_{H, {\textbf{x}, \textbf{x}'}}$
 are compatible in the sense that
 \begin{equation}
 	\begin{CD}
				HF_*(\Sigma, \varphi_H(\underline{\Lambda}),  \underline{\Lambda},   \textbf{x})_{J}   @>\Psi_{H, \mathbf{x}, \mathbf{x}'}>> HF_*(\Sigma, \varphi_H(\underline{\Lambda}),  \underline{\Lambda},  \textbf{x}')_{J}   \\
				@VV I_{H, G}V @VV  I_{H, G}V\\
				HF_*(\Sigma, \varphi_G(\underline{\Lambda}),  \underline{\Lambda},    \textbf{x})_{J}  @>\Psi_{G, \mathbf{x}, \mathbf{x}'}>> HF_*(\Sigma, \varphi_G(\underline{\Lambda}),  \underline{\Lambda},  \mathbf{x}')_{J} .			
				\end{CD}
				\end{equation}
By the above diagram, (\ref{eq36}) and  (\ref{eq22}), we define $HF(\Sigma, \underline{\Lambda})$ to be the direct limit
 with respect to the continuous morphisms  and $	\Psi_{H, {\textbf{x}, \textbf{x}'}}$.   Therefore, we have a canonical isomorphism
\begin{equation} \label{eq37}
j_H: HF(\Sigma, \varphi_H(\underline{\Lambda}),  \underline{\Lambda},   \mathbf{x}) \to HF(\Sigma,  \underline{\Lambda}).
\end{equation}

\subsection{ Action functional and HF spectral invariant}
We define the \textbf{action functional} on the chain complex   by
	\begin{equation} \label{eq7}
		\mathcal{A}^{\eta}_H(\textbf{y},[A]):=-\int_{A} \omega +  \int_0^1 H_t(\textbf{x})dt -\eta J_0(A).
	\end{equation}
Later, we will see that the perturbation term $\eta J_0$ is corresponding to the $\eta \Delta \cdot$ in (\ref{eq46}).

Similar to  the cases   of  PFH and quantitative Heegaard Floer homology, we define a submodule  $ CF^L(\Sigma, \varphi_H(\underline{\Lambda}), \underline{\Lambda})$ generated by $(\mathbf{y}, [A])$ with $\mathcal{A}^{\eta}_H(\textbf{y},[A])<L. $

\begin{lemma}
The differential decreases the action functional, ie.,  $ \partial( CF^L(\Sigma, \varphi_H(\underline{\Lambda}), \underline{\Lambda})) \subset  CF^L(\Sigma, \varphi_H(\underline{\Lambda}), \underline{\Lambda}). $
\end{lemma}
\begin{proof}
Suppose that $\langle  \partial (\mathbf{y}_+, [A_+]), (\mathbf{y}_-, [A_-]) \rangle \ne 0$. Then we have an HF curve $u \in \mathcal{M}^J(\mathbf{y}_+, \mathbf{y}_-, (-A_+) \#A_-)$.  By definition, we have
	\begin{equation*}
		\mathcal{A}^{\eta}_H(\mathbf{y}_+, [A_+]) -\mathcal{A}^{\eta}_H(\mathbf{y}_-,[A_-])  =\int u^*\omega  + \eta J_0(u).
	\end{equation*}
By Lemma \ref{lem28}, the right hand side of the above equation is nonnegative. Hence, $ \mathcal{A}^{\eta}_H(\mathbf{y}_-,[A_-]) \le \mathcal{A}^{\eta}_H(\mathbf{y}_+, [A_+]) <L$. 

\end{proof}

 Therefore, we get a filtration on the chain complex and can define the spectral
 invariant as before. Let
 $$c^{hf}_{\underline{\Lambda} }: C^{\infty}([0,1]\times \Sigma) \times HF(\Sigma, \underline{\Lambda}) \to \mathbb{R}$$ denote the spectral invariant defined by   $HF(\Sigma, \underline{\Lambda})$ and $\mathcal{A}^{\eta}_H.$ By the same argument as
 in Lemma  \ref{lem32}, the spectral invariant is independent of the base point $\mathbf{x}$.

	 %the spectral invariant defined by  $ HF_*(\Sigma, \varphi_H, \Lambda, \textbf{x})$  agrees with the one defined by  $HF(Sym^d\varphi_{H} (\Lambda), Sym^d \Lambda, \textbf{x})$.

	\section{Equivalence of two formulations}
In this section, we prove the promised isomorphism in Theorem  \ref{thm1}. The main idea is the same as \cite{RL1} what Lipshitz done for the usual Heegaard Floer homology. It may
 be feasible to copy the argument in \cite{RL1}  to prove Theorem  \ref{thm1}. But here we provide
 an alternative way to compare the ECH index of HF curves and Fredholm index of
 holomorphic sections in   $\mathbb{R} \times[0,1] \times  \operatorname{Sym}^d \Sigma$.

	\subsection{Index comparison}
 To compare the index between HF curve and holomorphic section, we first need to construct a nice representative of  $A\in H_2(M, \textbf{y}_+, \textbf{y}_-)$ which is  like a holomorphic curve.  Lipshitz needs such a  representative in his argument as well.  We apply his arguments in  Lemma 4.2' and Lemma 4.9' of  \cite{RL2} to our setting directly.    Then we have the following lemma.
	\begin{lemma} \label{lem8}
		Fix a positive relative homology class $A\in H_2(M, {\mathbf{y}}_+,  {\mathbf{y}}_-)$. Then we have  a    representative $u: \dot{F} \to M $ of $A+[\Sigma]$ with the following   properties:
		\begin{enumerate}
			\item
			$u$ is holomorphic in a neighborhood of $\partial \dot{F}$;
			\item
			$\pi \circ u: \dot{F} \to \mathbb{R} \times [0,1 ]$ is a branched  covering;
			\item
			$u$ is holomorphic near the branch points of $\pi \circ u$;
			
			\item
			$u$ is embedded except for finitely many double points.
		\end{enumerate}
	\end{lemma}
	\begin{proof}
		Let $A \in H_2(M, \textbf{y}_+, \textbf{y}_-)$ be a positive relative homology  class. It associates a domain $\sum_jn_{z_j}(A)D_j$, where $n_{z_j}(A) \ge 0$ is the intersection number  in (\ref{eq20})  and $\{D_j\}$ is the complement of $\underline{\Lambda}  \cup \varphi_H(\underline{\Lambda})$.
		
The proof is the same as \cite{RL2}. So here we only briefly outline the key steps. The
 construction consists of the following steps:  First, we follow Lemma 4.2' of \cite{RL2} to  glue $\{D_j\}$ together according to the coefficients $\{n_{z_j}(A)\}$. This gives us a surface $F$ with corners   and a map  $u_{\Sigma}: F \to \Sigma$. Secondly, we modify  $u_{\Sigma}: F \to \Sigma$ such that
		\begin{enumerate}
			\item
			$F$ has conners $\{p_i^{\pm}\}_{i=1}^d$ and $u_{\Sigma}(p_i^{\pm}) =y_i^{\pm}$.
			\item
			The conners are acute.
			\item
			Each component of  $\partial F$  is mapped to some $\varphi_H(\Lambda_i)$ or $\Lambda_i$ so that each   $\varphi_H(\Lambda_i)$, $\Lambda_i$, $1\le i \le d$, is used exactly once.
		\end{enumerate}
		All the modifications in this step are local. Therefore, the argument in Lemma 4.2' of \cite{RL2} can be applied to our case directly.
		The surface $F$ maybe disconnected. We can glue it with $\Sigma$ so that the surface $F$ is connected (see the proof of Lemma 4.9' in \cite{RL2} (Page 27)). Moreover, the new map  $u_{\Sigma}: F \to \Sigma$  still satisfies the mentioned properties.

 Finally, we   use  the argument in Lemma 4.9' of \cite{RL2} to construct a branched  covering  $u_{\mathbb{D}} : F  \to \mathbb{R} \times [0,1]$ as follows:  1) Let $U = \partial F \times [0, \epsilon)$ be a neighborhood of $\partial F$. One
 can choose a holomorphic map $u^1_{\mathbb{D}}$ from $U $ to $\mathbb{R} \times [0,1]$ such that $u^1_{\mathbb{D}}$  is a $d$-fold covering, mapping the arcs in
 $u_{\Sigma}^{-1}(\varphi_H(\underline{\Lambda}))$  to  $\mathbb{R} \times \{0\}$ and the arcs in $u_{\Sigma}^{-1}(\underline{\Lambda})$   to $\mathbb{R} \times \{1\}$. 2) Collapsing  the circles $\partial F \times \{\frac{\epsilon}{2}\}$ yields  a surface consisting of a
 nonempty union of disks (one for each component of $\partial F$) and a closed surface $F'$. Take $u^2_{\mathbb{D}}: F' \to \mathbb{S}^2$
 to be a $d$-branched covering. Finally, splicing  $u^1_{\mathbb{D}}$ and  $u^2_{\mathbb{D}}$  together
gives the desired
 branched covering $u_{\mathbb{D}} : F \to \mathbb{R} \times [0,1].$
		
		Define $u:=u_{\mathbb{D}} \times u_{\Sigma} : F \to M$. We  perturb $u$ such that it is  embedded except for finitely many double points.  Then  $u: F \to M$ satisfies all the requirements.
	\end{proof}
	Let $u$ be the $d$-multisection provided by Lemma \ref{lem8}. Note that $u$ intersects each fiber $(s,t ) \times \Sigma$ with $d$ points (counting  multiplicities) because $\pi \circ u$ is a branched covering, where $\pi: M \to \mathbb{R} \times [0, 1]$ is the projection.   These $d$ intersection points give us an element $S(s,t)  \in  \operatorname{Sym}^d \Sigma$. Then we  define a section $s_u$ on $\mathbb{R}_s \times [0,1]_t \times \operatorname{Sym}^d \Sigma$ by sending $(s,t)$ to $S(s,t)$. The section $s_u$ is called the \textbf{tautological correspondence} of $u$.  Note that the intersections of  $s_u$ with  the diagonal arise in  the following two ways:
	\begin{enumerate}
		\item
		Branch points of $\pi \circ u$.
		\item
		Double points of $u$.
	\end{enumerate}

 The index formula of  $s_u$ is
 \begin{equation*}
   \operatorname{ind} s_u = 2c_1(s_u^*T \operatorname{Sym}^d \Sigma, \tau_d) + \mu_{\tau_d}(s_u),
 \end{equation*}
  where $\tau_d$ is a trivialization of  $s_u^*T \operatorname{Sym}^d \Sigma$  that will be explained later. Let $z \in \Sigma \setminus (\varphi_H(\underline{\Lambda}) \cup\underline{\Lambda})$.  Define
  \begin{equation}\label{eq55}
    \mathfrak{n}_z(\mathcal{S}):= \#\left( \mathcal{S} \cap  \{z\} \times  \operatorname{Sym}^{d-1} \Sigma\right)
  \end{equation}
  for  $ \mathcal{S}  \in H_2(\mathbb{M}, \mathbf{y}_+, \mathbf{y}_-)$, where  $ \{z\} \times \operatorname{Sym}^{d-1} \Sigma : =\{[z_1,...,z_d] \in \operatorname{Sym}^d\Sigma \vert z_i =z \mbox{ for some } i\}$. The main goal of this section is to prove the following proposition.

	\begin{prop}\label{lem5}
		Let $u $ and $s_u$ be the above.  Then we have
\begin{equation*}
I(u) = \operatorname{ind} s_u, J_0(u)=\Delta \cdot s_u,   \mbox{ and }  n_z(u) = \mathfrak{n}_z(s_u),
\end{equation*}
where $z\in \Sigma \setminus (\varphi_H(\underline{\Lambda}) \cup  \underline{\Lambda})$, and  $n_z(u),$ $\mathfrak{n}_z(s_u)$ are respectively the intersection numbers (\ref{eq20}) and  (\ref{eq55}). 
	\end{prop}
 The strategy of the proof of Proposition \ref{lem5} is simply to compute the $c_1(s_u^*T \operatorname{Sym}^d \Sigma, \tau_d)$, $ \mu_{\tau_d}(s_u)$	 in terms of  $c_1(u^*T\Sigma, \tau), \mu_{\tau}(u) $ and $\delta(u)$.  We will choose a suitable trivialization  such that  $\mu_{\tau_d}(s_u)= \mu_{\tau}(u) $.  By definition,   $c_1(s_u^*T \operatorname{Sym}^d \Sigma, \tau_d)$ is a counting of the zeros
 of a generic section.  Roughly speaking, we will choose such a section to be the symmetric product of a section in $u^*T\Sigma$. Then, we get a relation between  $c_1(s_u^*T \operatorname{Sym}^d \Sigma, \tau_d)$
 and  $c_1(u^*T\Sigma, \tau)$ (Lemma  \ref{lem16}).
	
 We first clarify the definition of Maslov index  $\mu_{\tau_d}(s_u)$  and the trivialization $\tau_d$ as
 follows: Define a real subbundle $\mathcal{L}_d$ of $s_u^*T \operatorname{Sym}^d\Sigma$ by
 \begin{equation*}
\mathcal{L}_d \vert_{\partial s_u} = s_u^* \left(T(\mathbb{R} \times \{0\} \times \operatorname{Sym}^d \varphi_H(\underline{\Lambda}))  \cap T\mathbb{M}\right)\cup  \left(s_u^*T(\mathbb{R} \times \{1\} \times \operatorname{Sym}^d\underline{\Lambda}) \cap T \mathbb{M} \right).
 \end{equation*}
  We extend $\mathcal{L}_d$ along  $\{\pm \infty \} \times [0,1]$    by rotating in counterclockwise direction from $s_u^*T \operatorname{Sym}^d \varphi_H(\underline{\Lambda})$  to $s_u^*T \operatorname{Sym}^d  \underline{\Lambda}$ by the minimum account.   Recall that $\mathcal{L}$ is a real bundle over $\partial F$ such that  $\mathcal{L}\vert_{\partial \dot{F}} =u^*T\varphi_H(\underline{\Lambda}), u^*T\underline{\Lambda}$  and it is given by rotating in counterclockwise direction from $u^*T\varphi_H(\underline{\Lambda})$ to  $u^*T\underline{\Lambda}$  by the minimum account along $\partial F \setminus \partial \dot{F}$. Therefore, the real subbundle  $\mathcal{L}_d =\operatorname{Sym}^d  \mathcal{L}$.

  Fix a trivialization   $\tau$  of $u^* T \Sigma$  as before. Let $\mathbf{x}=[x_1, ...,x_d] \in  \operatorname{Sym}^d \varphi_H (\underline{\Lambda})$ or $\operatorname{Sym}^d \underline{\Lambda}$. Since $\{\Lambda_i\}_{i=1}^d$ are pairwise disjoint,  we identify $T_{\mathbf{x}} \operatorname{Sym}^d \Sigma$ with $T_{x_1} \Sigma \oplus,...,\oplus  T_{x_d}  \Sigma$.   %, $T_{[x] }Sym^d \varphi_H(\Lambda) $ with $T_{x_1} \varphi_H(\Lambda_1) \times \cdots  \times T_{x_d} \varphi_H( \Lambda_d) $,
	%and   $T_{[x] }Sym^d  \Lambda $ with $T_{x_1} \Lambda_1 \times \cdots  \times T_{x_d}  \Lambda_d $.
	Therefore,
	$\tau$ induces a trivialization $\tau_d$  of  $ T \operatorname{Sym}^d \Sigma$  along  $\operatorname{Sym}^d  \varphi_H (\underline{\Lambda})$  and  $\operatorname{Sym}^d \underline{\Lambda}$. Since the  Reeb chords $\{\textbf{y}=  \cup_i  [0,1] \times y_i \vert y_i \in \varphi_H(\Lambda_i) \cap \Lambda_{\sigma(i)}\}$ are pairwise disjoint,  $\tau$ also induces trivializations  of $T \operatorname{Sym}^d  \Sigma \vert_{ \{\pm \infty\} \times \textbf{y}_{\pm} }$ (still dented by $\tau_d$) similarly. The Maslov index term in the index formula is
\begin{equation*}
  \mu_{\tau_d}(s_u):= \mu(s_u^*T \operatorname{Sym}^d \Sigma, \mathcal{L}_d, \tau_d).
\end{equation*}
 The following lemma is obtained directly from the definition and  the direct sum property of the Maslov index.

	\begin{lemma}  \label{lem15}
		We have $\mu(u^*T\Sigma,  \mathcal{L}, \tau) =\mu(s_u^*T \operatorname{Sym}^d \Sigma, \mathcal{L}_d, \tau_d)$. In other words, $\mu_{\tau}(u) =\mu_{\tau_d}(s_u)$.
	\end{lemma}
	
	The next step is to compare the  relative Chern number of $u$ and $s_u$.  We take a  generic smooth section of $\psi \in \Gamma(u^*T\Sigma)$ such that $\psi =\tau$ along  $\partial {F}$.    We  choose  $\psi$ such it  satisfies the following condition: Let $q$ be a  branch point of  $\pi\circ u: F \to \mathbb{R} \times [0,1]$.  We   identify a neighbourhood  of $u(q)$ with $D_w \times D_z$,  where $z$ is the holomorphic coordinate of the fiber and $w$ is the coordinate  pull back from the base.  %Under this identification, we choose $\psi = \partial_z$ over $D_w \times D_z $.
	Similarly, if $q$ is a double point, then we can take the coordinates $(w, z) $  around $u(q)$ as the above. We take $\psi=\partial_z$ in terms of these coordinates.  In particular, $\psi$  has  no  zeros near the branch points and the double points.

	Away from the diagonal, the section $\psi$ induces a section $\psi_d $ over $(\Lambda^{\max}T \operatorname{Sym}^d \Sigma )^{\otimes 2} \vert_{s_u}$ in the following  way:
	$$\psi_d([x_1,..., x_d]) := (\psi(x_1) \wedge \cdots \wedge \psi(x_d)) \otimes  (\psi(x_1) \wedge \cdots \wedge \psi(x_d) ).$$
 Note that the right hand side of the above formula is independent of the order  of  $(x_1, ..., x_d)$.

	Now we extend $\psi_d$ over the whole $s_u$ by the following way: Let $q $  be a branch point  of $\pi \circ u: F \to \mathbb{R} \times [0,1]$. Assume that the degree of $q$ is $d$ firstly.   Reintroduce the coordinates $D_w \times D_z$ in the above. Then $\operatorname{Image} u \cap (\pi^{-1}(w)) = [z_1(w), ..., z_d(w) ] \in \operatorname{Sym}^d \mathbb{C}_z$  near $u(q)$.
	%To extend  $\psi_d$ over the branch points, we firstly need to express $\psi_d$ in term of coordinate  $\{\sigma_1, \dots \sigma_d\}$.
	When $w \ne 0$, then  $\{z_1(w), ...,z_d(w)\}$ are  distinct points in $\mathbb{C}_z$.   For fixed $w\ne 0$,  we use  $z_i$  to denote the coordinate near $z_i(w)$.    Since $[z_1(w), ..., z_d(w) ] \in  \operatorname{Sym}^d  \mathbb{C}$ is away from diagonal, we  identify its neighbourhood with $\mathbb{C}_{z_1} \times \dots \mathbb{C}_{z_d}$.    Then $\partial_z \vert_{z_i(w)} =\partial_{z_i}$.
	Let $$\sigma_1(z_1 \dots z_d) =\sum_i z_d, \dots, \sigma_d(z_1 \dots z_d) =\Pi_i z_i$$  be the  elementary symmetric functions.  These functions are coordinates of $ \operatorname{Sym}^d  \mathbb{C}$. We have
	\begin{equation*}
		\begin{split}
			\psi \vert_{z_i(w)} = \partial_z \vert_{z_i(w)}  = \partial_{z_i} =  \sum_j \frac{\partial \sigma_j}{\partial z_i} \partial_{\sigma_j} .
		\end{split}
	\end{equation*}
	It is well known that the Jacobian  of  the elementary symmetric functions is $\Delta = \Pi_{i<j} (z_i -z_j)$.  Therefore, we have
	\begin{equation} \label{eq10}
		\begin{split}
			\psi_d \vert_{[z_1(w),...,z_d(w)]} = (\partial_z \vert_{z_1(w)} \wedge \dots \wedge \partial_z \vert_{z_d(w)} )^{\otimes 2} =( \partial_{z_1} \wedge  \dots \wedge \partial_{z_d})^{\otimes 2} =  \Delta^2(w)(\partial_{\sigma_1} \wedge  \dots \wedge \partial_{\sigma_d} )^{\otimes 2} ,
		\end{split}
	\end{equation}
where  $ \Delta(w) = \Pi_{i<j} (z_i(w) -z_j(w)) $ for $w\ne 0$.
	%Therefore, if $p$ is a branch point with order $d$, then we can identify a neighbourhood  of $p$ with $D_w \times D_z$  and    neighbourhood  of $s_C(\pi(p))$ with $D_w \times Sym^d D_z $  as before.   By above discussion, $\psi_d = \Delta \vert_{s_C(w) } \partial_{\sigma_1} \wedge  \dots \wedge \partial_{\sigma_d}  $ over $ \{\frac{1}{2} \le  |w| \le 1\} \times Sym^dD_z$.
	We extend the function $\Delta^2(w)$ generically    over the whole  $D_w$, still denoted by $\Delta^2(w)$.   We set $\psi_d$ over  $D_w \times  \operatorname{Sym}^d D_z $ to be the most right hand side of (\ref{eq10}).

	For the general case, the construction  is  similar.  Let $q_1, ...,q_N$ be branch points of $\pi \circ u: F \to \mathbb{R} \times [0,1]$ such that $\pi \circ u(q_k) =p$. Let $d_k$ denote  the degree of $q_k$. Then $d=\sum_{k=1}^Nd_k$.    Let $(w, z_k) $ be the local coordinates around $u(q_k)$ as before.  A neighbourhood  of $[\pi^{-1}(p) \cap \operatorname{Image} u] \in  \mathbb{R} \times [0,1] \times  \operatorname{Sym}^d \Sigma$ can be identified with  $D_w \times  \operatorname{Sym}^{d_1} D_{z_1} \times, ..., \times  \operatorname{Sym}^{d_N} D_{z_N}$.   By the discussion above, the vector field $\partial_{z_k}$ gives  raise vectors $ \{ \sum_r\frac{\partial{\sigma_r^k}}{\partial z^k_i} \partial_{\sigma^k_r} \}_{i=1}^{d_k}$   over $\{ \delta \le |w| \le 1\} \times  \operatorname{Sym}^{d_k} D_{z_k}$. Then
	$$\psi_d = (\Pi_k \Delta^2_k(w)) \Pi_k (\partial_{\sigma_1^k} \wedge \dots \partial_{\sigma_{d_k}^k} )^{\otimes 2}$$   over $\{ \delta  \le |w| \le 1\} \times  \operatorname{Sym}^{d_1} D_{z_1} \times \dots  \times \operatorname{Sym}^{d_N} D_{z_N}$.
	We   extend $\Delta^2_k(w)$ generically  over $D_w$ and set   $\psi_d$ to be  $  (\Pi_k \Delta^2_k(w)) \Pi_k (\partial_{\sigma_1^k} \wedge \dots \partial_{\sigma_{d_k}^k} )^{\otimes 2}$ over the whole   $D_w  \times  \operatorname{Sym}^{d_1} D_{z_1} \times \dots  \times \operatorname{Sym}^{d_N} D_{z_N}$. If $u(q)$ is a double point, then we extend $\psi_d$ over a neighbourhood of $s_u(\pi(u(q)))$ in the above way.

	\begin{lemma} \label{lem16}
We have
 $$2c_1( s_u^*T \operatorname{Sym}^d\Sigma, \tau_d) =\# \psi_d^{-1}(0) =2c_1(u^* T\Sigma, \tau) +  b + 2\delta(u), $$ where $b$  is the sum over all the branch points of the order of multiplicity minus
		one  and $\delta(u)$ is the sign count  of the double points.
	\end{lemma}
	\begin{proof}
By construction, the section $\psi_d =\tau_d^{\otimes 2}$ along the boundary $\partial s_u$ and the ends. By definition,   we have $2c_1( s_u^*T \operatorname{Sym}^d\Sigma, \tau_d) =\# \psi_d^{-1}(0) . $

		The contribution of  $ \# \psi_d^{-1}(0)$ comes from two parts $\psi^{-1}_{d, 1}(0) $ and $\psi^{-1}_{d, 2}(0) $, where $\psi_{1, d}$ is the  restriction of $\psi_d$ away from diagonal and $\psi_{d ,2}$ is the restriction of $\psi_d$ near the diagonal.
		
	We first consider the case that away from diagonal.    By definition,  for any $p \in \mathbb{R} \times [0,1]$ such that $\psi_{1, d}(s_u(p))=0$,  then we have $\# \psi_{1, d}^{-1}(0)\vert_p = 2\sum_i \#\psi^{-1}(0) \vert_{q_i}$,  where $\pi^{-1}(p) \cap \operatorname{Image} u =\{q_1,...,q_d\}$.  Therefore,  we obtain $\#\psi_{1, d}^{-1}(0)=2\# \psi^{-1}(0) $ away from the double points and branch points. Since $\psi$ has no zeros  near the branch points and double points, we have
		$$\#\psi_{1, d}^{-1}(0)=2\#\psi^{-1}(0)=2c_1(u^*T\Sigma, \tau).$$
		
		Let $q$ be a branch point of  $\pi \circ u:  F \to \mathbb{R} \times [0,1]$. To simplify the notation, we assume that the degree of $q$ is $d$. Reintroduce   the   coordinates $D_w \times D_z$ around $u(q)$.
		%\begin{equation*}
		%\begin{split}
		%2\Delta_d \vert_{|w|=\epsilon} = wind \Delta^2_d \vert_{|w|=\epsilon}.
		%\end{split}
		%\end{equation*}
		Under coordinates  $D_w \times D_{z}$ around  $u(q)$,  we write $u$ as $(x^{d}, f(x))$, where $x$ is the holomorphic  coordinate on $F$. %Since $T^{0,1} =<\partial_{\bar{w}} + b \partial_{z_k}, \partial_{\bar{z}_k}>$ near $u(q)$,  $\bar{\partial}_{\bar{x}} f = d b\bar{x}^d$.  The standard trick
		%deduce that $f(x)=e^{g(\bar{x}^{d})} h(x)$, where $g$ is a function such that
		Since $u$ is holomorphic near the branch points, we know that $f$ is a holomorphic function. % we  $\partial_{\bar{w}} g =b$ and $h$ is a holomorphic function.
		$u$ is embedded near $u(q)$ implies that  $f(x) =c_1 x + O(x^2)$ and $c_1 \ne 0$.   The contribution of  $\psi_{2, d}^{-1}(0)$ near $u(q)$ is the same as the winding number of $\Delta^2 \vert_{|w|=\epsilon}$. Let $x_1(w),..., x_{d}(w)  $ be the  distinct branches  of $w^{\frac{1}{d}}$.   Then  we have
		\begin{equation*}
			\Delta^2(w) =c_0w^{d-1} +o(|w|^{d-1})
		\end{equation*}
		with $c_0 \ne 0$.  Hence, $\operatorname{wind} \Delta^2 \vert_{|w|=\epsilon} = d-1$. For the general case, the argument is similar.
		
		If $u(q)$ is a double point, then locally $\operatorname{Image} u$ is a graph of two functions $f, g$.  Then $\Delta^2(w)=(f(w)-g(w) )^2$ under the local coordinates. We may assume $f(w)=aw$ and $g(w) =b w$ if the double points is positive. Otherwise, we assume that  $f(w)=a\bar{w}$ and $g(w) =b \bar{w}$.  Hence, each  positive (negative) double point contribute $2(-2)$ to the $\# \psi_{2, d}^{-1}(0)$.

		In sum, $\# \psi_{2, d}^{-1}(0) = b + 2\delta(u)$ and this finishes  the proof of the lemma.
	\end{proof}
	
	\begin{proof}[ Proof of Proposition \ref{lem5}]
		By the Riemann-Hurwitz formula (see Corollary 3.2 of \cite{RL1}),  we have $e(F) =de(D) -b$, where $e(F), e(D)$ are the Euler measures. We know that $e(F) =\chi(F) -\frac{d}{2} $ and $e(D) =\chi(D) - \frac{1}{2} =\frac{1}{2}$  (see Page 974 of \cite{RL1}).  By    Lemmas \ref{lem15}, \ref{lem16}, we have
		\begin{equation*}
			\begin{split}
				 \operatorname{ind} s_u =&2c_1( s_u^*T \operatorname{Sym}^d\Sigma, \tau_d) + \mu_{\tau_d}(s_u) \\
				&=  b+ 2c_1(u^*T\Sigma, \tau) + \mu_{\tau}(u) + 2\delta(u) \\
				&=-\chi(F) + d + 2c_1(u^*T\Sigma, \tau) + \mu_{\tau}(u) + 2\delta(u)\\
				&= \operatorname{ind} u+ 2 \delta(u) =I(u).
			\end{split}
		\end{equation*}

By the argument in Lemma 4.9' of \cite{RL2}, we can arrange that all the branch points  of   $u$ are order 2.  Then we have $\#(\Delta \cap s_u) = b +2 \delta(u)$.  On the other hand, $\chi(F) =d-b$ by the Riemann-Hurwitz formula.  Therefore, by Lemma \ref{lem28}, we have
\begin{equation*}
		 \Delta \cdot s_u = -\chi(F) +d +2\delta(u) =J_0(u).
		\end{equation*}

Finally, we compare the intersection numbers  $n_z(u)$ and $\mathfrak{n}_z(s_u).$ Note that for any
 two $z_1$, $z_2$ in the same component of $\Sigma \setminus(\varphi_H(\underline{\Lambda}) \cup \underline{\Lambda})$, we have $n_{z_1}(u)=n_{z_2}(u)$ and  $\mathfrak{n}_{z_1}(s_u)=\mathfrak{n}_{z_2}(s_u). $  Therefore, we can choose $z$ such that the intersection points are
 away from the branch points and doubles. As a result, the intersection points of $s_u$ and $\mathbb{R} \times [0,1] \times \{z\}\times  \operatorname{Sym}^{d-1}\Sigma$ are away from the diagonal. Then one can check
$ n_z(u)$ = $\mathfrak{n}_z(s_u)$  directly from the construction of $s_u$.
	\end{proof}

	\subsection{Proof of Theorem \ref{thm1}}
	In this subsection, we use the tautological correspondence to prove Theorem \ref{thm1}.
	\begin{lemma} \label{lem19}
		%Fix $\textbf{y}_0 \in \varphi_H(\Lambda) \cap \Lambda$.
		The tautological correspondence induces a $\mathbb{Z}$-module isomorphism
		\begin{equation*}
			\Psi_* : H_2(M, \mathbf{y}_+, \mathbf{y}_{-}) \to H_2( \operatorname{Sym}^d\Sigma, \mathbf{y}_+, \mathbf{y}_{-}).
		\end{equation*}	
 Furthermore, for an HF curve  $u \in \mathcal{M}^J(\mathbf{y}_+, \mathbf{y}_-)$, we have $\Psi_*([u])=[s_u]$, where $s_u$ is the tautological correspondence of $u$.
	\end{lemma}
	\begin{proof}
	%To begin with, we first construct an isomorphism
	%\begin{equation*}
	%		\Psi_* : H_2(M, \textbf{y}_+, \textbf{y}_{-}) \to H_2(Sym^d\Sigma, \textbf{y}_+, \textbf{y}_{-}).
	%	\end{equation*}	
		Fix a positive relative homology class $A\in H_2(M, \textbf{y}_+, \textbf{y}_-)$. 	Let $u_0$ be a representative provided by  Lemma \ref{lem8} with $A_0=[u_0]$.  Then, we define
\begin{equation*}
  \Psi_*(A_0) = [s_{u_0}] \in H_2( \operatorname{Sym}^d\Sigma, \mathbf{y}_+, \mathbf{y}_-)
\end{equation*}
 to be the class of the tautological correspondence.

%Then the tautological correspondence gives a class $\Psi_*[A+\Sigma]=[s_u] \in H_2(Sym^d\Sigma, \textbf{y}_+, \textbf{y}_-).$
		
		Recall that  $H_2(M, \textbf{y}_+, \textbf{y}_-)$ is an affine space over $\mathbb{Z} \langle [\varphi_H(B_i)], [B_i] \rangle _{i=1}^{k+1}$. On the other hand,  $H_2( \operatorname{Sym}^d\Sigma, \textbf{y}_+, \textbf{y}_-)$ is an affine space over $H_2( \operatorname{Sym}^d\Sigma,  \operatorname{Sym}^d\varphi_H(\underline{\Lambda}) )$ and $H_2( \operatorname{Sym}^d\Sigma,  \operatorname{Sym}^d \underline{\Lambda})$.  We extend $\Psi_*$ linearly   by $$\Psi_*(A_0+ \sum_{i=1}^{k+1} c_i [B_i] +  \sum_{i=1}^{k+1} c'_i [\varphi_H(B_i)]) = \Psi_*(A_0) + \sum_{i=1}^{k+1} c_i \Psi_*([B_i]) + \sum_{i=1}^{k+1} c_i' \Psi_*[\varphi_H(B_i)]). $$  The definition of $\Psi_*([B_i]) $  and $ \Psi_* ([\varphi_H(B_i)]) $ are  given in the next paragraph.
		
		For each $B_i$, we define the class $\Psi_*([B_i]) \in H_2( \operatorname{Sym}^d\Sigma,  \operatorname{Sym}^d \underline{\Lambda} ) $ as follows: For $1\le i \le k$, define a surface $F=\sqcup_{j=1}^d \mathbb{D}_j$.  Let  $u_{\Sigma}^i: F \to \Sigma$  be a map such that $u_{\Sigma}^i \vert_{\mathbb{D}_j}$ is a constant at $\Lambda_j$ when $j\ne i$ and $u_{\Sigma}^i   \vert_{\mathbb{D}_i} $ is a biholomorphism to $B_i$. Take $u^i_{\mathbb{D}} : F \to \mathbb{D}$ to be the $d$-fold trivial covering.  In the case that $i=k+1$, let $F=B_{k+1} \cup \cup_{j=k+1}^d\mathbb{D}_j$.  Take  $u^{k+1}_{\Sigma}: F\to \Sigma$  such that  $u^{k+1}_{\Sigma} \vert_{B_{k+1}}$  is a biholomorphism     and  $u^{k+1}_{\Sigma} \vert_{\mathbb{D}_{j}}$ is a constant at $\Lambda_j$ for $k+1 \le j\le d$. Take  $u^{k+1}_{\mathbb{D}}: F\to \mathbb{D} $ to be a   $d$-fold branched covering such that $u^{k+1}_{\mathbb{D}} \vert_{B_{k+1}} $ is a $k$-fold branched covering and $u^{k+1}_{\mathbb{D}} \vert_{\cup_{j=k+1}^d\mathbb{D}_j}$ is the trivial covering. Define $u^i := u_{\mathbb{D}}^i \times  u_{\Sigma}^i:(F, \partial F) \to \mathbb{D} \times \Sigma$.  Then,  we obtain a map $s_{u^i}: (\mathbb{D}, \partial \mathbb{D}) \to ( \operatorname{Sym}^d\Sigma,  \operatorname{Sym}^d \underline{\Lambda})$  via the tautological correspondence.  Define $\Psi_*([B_i])$ to be the homology class of $s_{u^i}$.   The construction is similar for $[\varphi_H(B_i)]$.
		
		By  Lemma 4.10 in \cite{CHMSS},  we have
		\begin{equation*}
			\begin{split}
				&  H_2( \operatorname{Sym}^d\Sigma,  \operatorname{Sym}^d\varphi_H(\underline{\Lambda}) ) \cong \mathbb{Z} \langle [\varphi_H(B_i)] \rangle _{i=1}^{k+1}\\
				&   H_2( \operatorname{Sym}^d\Sigma,  \operatorname{Sym}^d\underline{\Lambda}) \cong \mathbb{Z} \langle  [B_i] \rangle _{i=1}^{k+1}.
			\end{split}
		\end{equation*}
Therefore, $\Psi_*$ is an isomorphism.

Let $u$ be an HF curve.  We now prove the statement that $\Psi_*([u]) =[s_u]$.  Suppose
 that $[u] =A_0 + \sum_{i=1}^{k+1} \left(c_i [B_i]+ c_i' [\varphi_H(B_i)] \right)$.  By definition, we have
 \begin{equation*}
   \Psi_*([u]) =\Psi_*(A_0) + \sum_{i=1}^{k+1} \left(c_i \Psi_*([B_i])+ c_i' \Psi_*( [\varphi_H(B_i)] )\right).
 \end{equation*}
  Write the homology class of $s_u$ as  $\Psi_*(A_0) + \sum_{i=1}^{k+1} \left(b_i \Psi_*([B_i])+ b_i' \Psi_*( [\varphi_H(B_i)] )\right)$ for some $b_i, b_i' \in \mathbb{Z}. $  Fix $z \in \mathring{B}_i $ such that it does not lie inside $\underline{\Lambda}$, $\varphi_H(\underline{\Lambda})$ and $\varphi_H(B_j)$ for $1\le j \le k$.  This is feasible because $\varphi_H$ is nondegenerate. Note that we must have $z \in \varphi_H(\mathring{B}_{k+1}).$  From the construction of $\Psi_*([B_i])$, we have
 \begin{equation*}
\mathfrak{n}_z(  \Psi_*([B_j])) = n_z(B_j) = \delta_{ij} \mbox{ and }  \mathfrak{n}_z(  \Psi_*([\varphi_H(B_j)])) = n_z(\varphi_H(B_j)) = \delta_{jk+1}
 \end{equation*}
 for $1\le j\le k+1$. Then
 \begin{equation*}
   c_i+ c'_{k+1} =n_z(u)-n_z(A_0) \mbox{ and } b_i +b_{k+1}' =\mathfrak{n}_z(s_u) - \mathfrak{n}_z(\Psi_*(A_0)).
 \end{equation*}
 By Proposition \ref{lem5},  we have  $n_z(u) = \mathfrak{n}_z(s_u) $ and $n_z(A_0) = \mathfrak{n}_z(\Psi_*(A_0)).$ As a result,  $   c_i+ c_{k+1}' = b_i +b_{k+1}'$. By the similar argument, we have    $   c_i'+ c_{k+1} = b_i' +b_{k+1}$. These
 equations implies that
 \begin{equation}\label{eq58}
  \begin{split}
  [s_u]&=  \Psi_*(A_0) + \sum_{i=1}^{k+1} \left(b_i \Psi_*([B_i])+ b_i' \Psi_*( [\varphi_H(B_i)] )\right)\\
  &=   \Psi_*(A_0) + \sum_{i=1}^{k+1} \left(c_i \Psi_*([B_i])+ c_i' \Psi_*( [\varphi_H(B_i)] )\right) \\
  &+  (c'_{k+1}-b'_{k+1}) \sum_{i=1}^{k+1} \Psi_*([B_i])+ (c_{k+1} -b_{k+1})) \sum_{i=1}^{k+1}   \Psi_*( [\varphi_H(B_i)] )\\
  &=\Psi_*([u]) +  (c'_{k+1}-b'_{k+1} + c_{k+1} -b_{k+1}) \Psi_*([\Sigma]) \\
  &=\Psi_*([u]).
    \end{split}
 \end{equation}

\begin{comment}
We construct a reference class $A_{\textbf{y}_0} \in H_2(M, \textbf{y}_0, \textbf{y}_{0H})$ as follows:  Let $\chi: \mathbb{R} \to \mathbb{R}$ be a non-increasing  cut off function such that $\chi=0$ when $s \ge 1$ and $\chi=1$ when $s\le0$.  Define a section $u_i:\mathbb{R}_s \times [0,1]_t \to M$ by
		\begin{equation*}
		u_i(s, t) =(s, t, (\varphi^{\chi(s)t}_H)^{-1}(y_i)).
		\end{equation*}
We define 	$A_{\textbf{y}_0} \in H_2(M, \textbf{y}_0, \textbf{y}_{0H})$ to be the class represented by $\cup_i u_i$.     Obviously, the  tautological correspondence of $\cup_i u_i$ gives a class $\Psi_*(A_{\textbf{y}_0}) \in H_2(\mathbb{M}, \textbf{y}_0, \textbf{y}_{0H})$.

Using $A_{\textbf{y}_0} $,  any class $A \in H_2(M, \textbf{y}, \textbf{y}_{0H}) $ can be written as $A=A_0\#A_{\textbf{y}_0}$. Then we define  $\Psi_*(A) =\Psi_*(A_0) \# \Psi_*(A_{\textbf{y}_0})$.  By the above discussion, this is an isomorphism.
\end{comment}
%we can identify $H_2(M, \textbf{y}, \textbf{y}_{0H})$ with $H_2(M, \textbf{y}, \textbf{y}_{0})$.

	\end{proof}

The next lemma generalizes  Proposition \ref{lem5} to any  relative homology class.
	\begin{lemma} \label{lem9}
		Let $ A\in H_2(M, \mathbf{y}_+, \mathbf{y}_-)$ be a relative homology class. Then
		\begin{equation*}
			\begin{split}
				&I(A) =  \operatorname{ind}  \Psi_*(A), J_0(A)= \Psi_*(A) \cdot \Delta,\\
&n_z(A) = \mathfrak{n}_z(\Psi_*(A)) \mbox{\ and \ } \int_{A} \omega  =\int_{\Psi_*(A)} \omega_V.
			\end{split}
		\end{equation*}	
	\end{lemma}
	\begin{proof}
		By Remark \ref{remark8},  $\mu(\Psi_*([B_i])) =2$ and  $\mu(\Psi_*([\varphi_H(B_i)])) =2$ for $1\le i \le k+1$. Therefore, we obtain
		\begin{equation*}
\begin{split}
			\operatorname{ind}\Psi_*(A) =& \operatorname{ind}\left( \Psi_*(A_0) + \sum_{i=1}^{k+1} \left(c_i \Psi_*([B_i])+ c_i' \Psi_*( [\varphi_H(B_i)] )\right) \right)\\
=& \operatorname{ind} \Psi_*(A_0) + \sum_i^{k+1}c_i \mu(\Psi_*([B_i])) + \sum_i^{k+1}c'_i \mu(\Psi_*([\varphi_H(B_i)]))\\
   =&  \operatorname{ind}  \Psi_*(A_0) +2\sum_{i=1}^{k+1}(c_i+ c_i'),
\end{split}
		\end{equation*}
where $A_0$ is the class defined in Lemma \ref{lem19}.
		Combine the above equation with  Proposition \ref{lem5} and  Lemma \ref{lem6}; then we get the first identity.

By Proposition \ref{lem5},  to prove $J_0(A) =\Psi_*(A) \cdot \Delta$, it suffices  to check that $J_0(A')-J_0(A)=\Delta\cdot(\Psi_*(A') -\Psi_*(A) )$.  By Lemma 4.16 in \cite{CHMSS} and Remark \ref{remark8},   we have
\begin{equation*}
\Delta \cdot  \Psi_*([B_i]) = 0  \mbox{  for $1\le i\le k$, and } \Delta \cdot  \Psi_*([B_{k+1}]) = 2(2g+k-1) =2(d+g-1).
\end{equation*}
  Combine these equations  with Lemma \ref{lem28}; then we obtain the result.

  By Proposition \ref{lem5},   $n_z(A_0) = \mathfrak{n}_z(\Psi_*(A_0))$.   It follows from the construction,one can check that
  $n_z([B_i]) = \mathfrak{n}_z(\Psi_*([B_i]))$ and   $n_z([\varphi_H(B_i)]) = \mathfrak{n}_z(\Psi_*([\varphi_H(B_i)]))$. Therefore,
  \begin{equation*}
  \begin{split}
 n_z(A) &=n_z(A_0)  +\sum_{i=1}^{k+1} c_in_z([B_i])+ \sum_{i=1}^{k+1} c'_in_z([\varphi_H(B_i)]) \\
 & =   \mathfrak{n}_z(\Psi_*(A_0))  +\sum_{i=1}^{k+1} c_i   \mathfrak{n}_z(\Psi_*([B_i]))+ \sum_{i=1}^{k+1} c'_i   \mathfrak{n}_z(\Psi_*([\varphi_H(B_i)])).
 \end{split}
\end{equation*}
		
	 To prove the last statement,  we first need to know how to recover $u$ from $s_u$.  Write $s_u=(s,t, u')$ and  $\dot{\mathbb{D}} =\mathbb{R} \times [0,1]$.  Lift  $u' $ to $\tilde{u}': \tilde{F} \to (\Sigma)^{\times d}$ satisfying the following  diagram
 		\begin{equation}  \label{eq18}
			\xymatrix{
				\tilde{F} \ar[d]^{ \tilde{u}_{\mathbb{D}}} \ar[r]^{\tilde{u}' } &  (\Sigma)^{\times d} \ar[d]^{ p }\\
				\dot{\mathbb{D}} \ar[r]^{u'} &  \operatorname{Sym}^d \Sigma}
		\end{equation}
 where $ \tilde{F}=\{(x, y) \in \dot{\mathbb{D}} \times  (\Sigma)^{\times d}: u'(x)=p(y)\}$, and $\tilde{u}_{\mathbb{D}}:   \tilde{F} \to \dot{\mathbb{D}}$ is a covering map
 with degree $d!$.  Let  $S_{d-1}$ be the permutation group fixing the first factor of $(\Sigma)^{\times d}$.  Note that   $S_{d-1}$ also acts on $\tilde{F}$. Let $\pi_1$ be the projection of $(\Sigma)^{\times d}$ to its first factor. Then $\pi_{\Sigma} \circ u = \pi_1 \circ  u'/S_{d-1}$.
		
		By the diagram  (\ref{eq18}) and the fact that $\tilde{u}_{\mathbb{D}}$  is a degree  $d!$-covering, we have
		\begin{equation*}
			\begin{split}
				\int_{\tilde{F}} \tilde{u}^* \omega^{\times d } = \int_{\tilde{F}} \tilde{u}^* p^*\omega_{\mathbb{M}} = \int_{\tilde{F}} \tilde{u}^*_{\mathbb{D}} u'^* \omega_{\mathbb{M}} =d!\int_{\dot{\mathbb{D}}}  u'^* \omega_{\mathbb{M}}=d!\int_{\dot{\mathbb{D}}}  s_u^* \omega_{\mathbb{M}}.
			\end{split}
		\end{equation*}
		On the other hand, by the fact that $\tilde{F} \to \dot{F}=\tilde{F}/S_{d-1}$ is a $(d-1)!$-covering, we have
		\begin{equation*}
			\begin{split}
				(d-1)! \int_{\dot{F}} u^* \omega=\int_{\tilde{F}} \tilde{u}^* \pi_1^* \omega = \frac{1}{d}\int_{\tilde{F}} \tilde{u}^* \omega^{\times d }.
			\end{split}
		\end{equation*}
		The above two equations imply that $\int_{\dot{F}} u^* \omega = \int_{\dot{\mathbb{D}}} s_u^* \omega_{\mathbb{M}}$.
		
		By the definition of $\Psi_*([B_i])$ in Lemma \ref{lem19} and the similar argument, we have  $\int_{B_i}  \omega= \int_{\Psi_*([B_i])} \omega_{\mathbb{M}}$.  As the energy is additivity, we get the second statement of the lemma.
	\end{proof}

	\begin{remark} \label{remark8}
The isomorphism in Lemma \ref{lem19} and  the proof of Proposition \ref{lem5} tell us that the Lagrangian submanifold $\operatorname{Sym}^d\underline{\Lambda}$ is monotone in the sense of
\begin{equation}  \label{eq65}
  \omega_V(\mathcal{S}) + \eta  \langle  PD(\Delta), \mathcal{S} \rangle  = \frac{\lambda}{2}\mu(\mathcal{S}),
\end{equation}
for $\mathcal{S} \in H_2(\operatorname{Sym}^d \Sigma, \operatorname{Sym}^d\underline{\Lambda}).$ Let $u^i: (F_i, \partial F_i) \to \mathbb{D} \times \Sigma$  and $s_{u^i}: (\mathbb{D}, \partial \mathbb{D}) \to (\operatorname{Sym}^d \Sigma, \operatorname{Sym}^d\underline{\Lambda})$ be the maps defined in Lemma \ref{lem19}.  The homology classes   of $\{s_{u^i}\}_{i=1}^{k+1}$ generate $ H_2(\operatorname{Sym}^d \Sigma, \operatorname{Sym}^d\underline{\Lambda}).$ It suffices to verifies (\ref{eq65}) for $s_{u^i}$.  Since $s_{u^i}$ are   tautological correspondence
 of the disks $B_i$ for $1\le i \le k$, the computations are the same as  Lemma 4.17 and Lemma 4.18 of \cite{CHMSS}.

 We  only consider $i=k+1$. By the same argument, the conclusions in Proposition \ref{lem5}, Lemma \ref{lem15} and Lemma \ref{lem16} still hold for $s_{u^{k+1}}$.  Since $\tau$ is induced by  non-vanishing vector  fields on $\underline{\Lambda}$,  then
  \begin{equation*}
    \begin{split}
  & \mu_{\tau_d}(s_{u^{k+1}}) = \mu_{\tau}(u^{k+1}) = 0, \\
&c_1({u^{k+1}}^*T\Sigma, \tau) =c_1((u^{k+1}_{\Sigma})^*T\Sigma, \tau) =\chi(B_{k+1}).
\end{split}
 \end{equation*}
Also, note that $u^{k+1} \vert_{B_{k+1}}$ intersects $u^{k+1} \vert_{\cup_{j=g}^d \mathbb{D}_j}$ at $g$ points, and $u^{k+1}$ has no other self-intersection. Hence, $\delta(u^{k+1}) =g$. By Riemann-Hurwitz formula,   $$b=d\chi(\mathbb{D}) - \chi(F_{k+1}) = k-\chi(B_{k+1}).$$ In sum, we have
 \begin{equation*}
    \begin{split}
   \mu(s_{u^{k+1}}) =&2c_1(s_{u^{k+1}}^*T\operatorname{Sym}^d\Sigma, \tau_d)  +  \mu_{\tau_d}(s_{u^{k+1}}) \\
=&  2c_1((u^{k+1})^*T\Sigma, \tau) +b +2\delta(u)\\
    =&2\chi(B_{k+1}) + k-\chi(B_{k+1}) + 2g\\
    =&2, \\
  \mbox{ and }\#(s_{u^{k+1}} \cap \Delta) =&b+ 2\delta(u^{k+1}) \\
    =&  k-\chi(B_{k+1}) +2g\\
     =&k +2g+k-2 +2g \\
     =&2(2g+k-1) =2(d+g-1).
    \end{split}
 \end{equation*}
One can show that $\int (s_{u^i})^*\omega_V = \int_{B_i} \omega$ by the argument in Lemma \ref{lem9}. Therefore, (\ref{eq65}) is true for all $s_{u^i}$.
\end{remark}

	\begin{proof} [Proof of Theorem \ref{thm1} Part A]

		First of all, let us identify the moduli space of HF-curves and the moduli space of holomorphic sections in $\mathbb{R} \times [0,1] \times \operatorname{Sym}^d\Sigma$.

		Since $J \in \mathcal{J}_M$ is $\mathbb{R}$-invariant, it is determined by a path of complex structures $\{j_t\}_{t \in [0,1]}$ on $\Sigma$. Therefore, each $J\in \mathcal{J}_M$ gives raise to a   path of quasi-nearly-symmetric almost complex structure $\mathbb{J}_t := \operatorname{Sym}^d j_t $.   We extend it to be an almost complex structure $\mathbb{J}$ on $\mathbb{R} \times [0,1] \times  \operatorname{Sym}^d \Sigma$ by setting $\mathbb{J}(\partial_s)=\partial_t$.  If  $u$ is a $J$-holomorphic HF curve, then $s_u$ is a $\mathbb{J}$-holomorphic section.

		Let $J$ be a generic almost complex structure such that the curves in  $\mathcal{M}^J(\textbf{y}_+, \textbf{y}_-, A)$ and $\mathcal{M}^{\mathbb{J}}(\textbf{y}_+, \textbf{y}_-, \Psi_*(A))$ are Fredholm regular.
		Fix a positive  class $A \in H_2(M, \textbf{y}_+, \textbf{y}_-)$ with $I(A)=1$.  The  tautological correspondence  provides a map
		$$\Psi: \mathcal{M}^J(\textbf{y}_+, \textbf{y}_-, A) \to \mathcal{M}^{\mathbb{J}}(\textbf{y}_+, \textbf{y}_-, \Psi_*(A))$$
 by sending $u$ to $s_u.$
 Obviously, $\Psi$ is injective.  By Lemma \ref{lem9}, the curves in  $ \mathcal{M}^{\mathbb{J}}(\textbf{y}_+, \textbf{y}_-, \Psi_*(A))$ have Fredholm index 1.
		
		Conversely,  given a holomorphic section  $s=(s, t, u') \in  \mathcal{M}^{\mathbb{J}}(\textbf{y}_+, \textbf{y}_-, \Psi_*(A))$ with $\operatorname{ind} s=1$,   by the same argument in Proposition 13.2 of \cite{RL1},  we lift  $u'$ to a map $\tilde{u}': \tilde{F} \to (\Sigma)^{\times d}$ satisfying the diagram (\ref{eq18}).
Moreover, the maps in the diagram are holomorphic.  Recall the $S_{d-1}$-action in Lemma \ref{lem9}.  Then $\tilde{u}' $ is $S_{d-1}$-equivariant.  The map  $\pi_1 \circ \tilde{u}'  $ descends to a map $u_{\Sigma}: \dot{F}  \to \Sigma$, where $\dot{F}:=\tilde{F}/ S_{d-1}$. Let $u_{\mathbb{D}} : = \tilde{u}_{\mathbb{D}} /S_{d-1} : \dot{F} \to \mathbb{D}$. Then we define the inverse $u=\Psi^{-1}(s)$ by  $u:= u_{\mathbb{D}} \times u_{\Sigma} : \dot{F} \to M$. 
		
		%We can construct an inverse $\Psi^{-1}(u')$  to get a HF--curve  $u=\Psi^{-1}(u') $.
		By the argument in   Section 13 of \cite{RL1}, we have $\Psi \circ \Psi^{-1}= \operatorname{Id}$.    Again by Lemma \ref{lem9},   we have $I(u)= \operatorname{ind} s =1 $ and $[s]=\Psi_*([u]) =\Psi_*(A)$.  Hence,  the  map $\Psi$ is 1-1 onto.

	 The rest of task is to identify the cappings. 	We extend the isomorphism in Lemma  \ref{lem19} to
\begin{equation*}
			\Psi_* : H_2(M, \mathbf{x}_{H}, \mathbf{y}) \to H_2( \operatorname{Sym}^d \Sigma, \mathbf{x}_H, \mathbf{y})
\end{equation*}	
as follows.  Fix $\mathbf{y}_0 \in \varphi_H(\underline{\Lambda}) \cap \underline{\Lambda}.$
Let $\mathcal{S}_0 \in  H_2( \operatorname{Sym}^d\Sigma, \mathbf{x}_H, \textbf{y}_{0})$ be a class represented by a capping $s: \mathbb{R}_s \times [0,1]_t \to ( \operatorname{Sym}^d\Sigma,  \operatorname{Sym}^d \varphi_H(\underline{\Lambda}) \cup  \operatorname{Sym}^d \underline{\Lambda})$.  Because the diagonal $\Delta$ is codimension 2, we assume that $s$ intersects $\Delta$ at  finitely  many points.  Also, assume that $s$ is holomorphic near the intersection points.  Using the construction in diagram (\ref{eq18}), we can construct a  surface  $u=\Psi^{-1}(s): \dot{F}  \to M$. Then $u$ represents a relative homology class $A_0 \in H_2(M, \mathbf{x}_{H}, \mathbf{y}_0) $. Then define  $\Psi_*(A\#A_0) = \Psi_*(A)\# \mathcal{S}_0$, where $A  \in H_2(M, \textbf{y}_0, \textbf{y}) $.

We claim that  $\mathcal{A}_H^{\eta}(\mathbf{y}, A) =\mathfrak{A}_H^{\eta}(\mathbf{y}, \Psi_*(A)).$  By Lemma \ref{lem9}, it suffices to show that   $\int_{A_0} \omega = \int_{\mathcal{S}_0} \omega_V$ and $J_0(A_0)=\Delta \cdot \mathcal{S}_0$. Fix another  class $\mathcal{S}_0' \in  H_2( \operatorname{Sym}^d\Sigma, \textbf{y}_{0},  \textbf{x}_H)$. Let $A_0' \in H_2(M, \mathbf{y}_0, \mathbf{x}_{H}) $ be the corresponding class  via the  tautological correspondence. Then $A_0' \#A_0 \in H_2(M, \mathbf{y}_0, \mathbf{y}_0)$ and  $\mathcal{S}_0' \#\mathcal{S}_0 \in H_2( \operatorname{Sym}^d\Sigma, \mathbf{y}_0, \mathbf{y}_0)$.    By Lemma \ref{lem9}, we have $\Delta \cdot \mathcal{S}_0 + \Delta \cdot \mathcal{S}_0 '  = J_0(A_0) + J_0(A_0')$.   Without loss of generality, assume that $A_0'\#A_0$ is represented by surface $u_0$ satisfying the conclusion of   Lemma \ref{lem8}.  Recall the proof of Proposition  \ref{lem5},  $\Delta \cdot \mathcal{S}_0' \#\mathcal{S}_0$ comes from the double points and branch points of $u_0$. Since $\mathbf{x}_H$ is a union of pairwise disjoint path, there is no contribution to  $\Delta \cdot \mathcal{S}_0' \#\mathcal{S}_0$ near $\mathbf{x}_H$.   For the same reasons, a generic section $\psi$ of $ T \operatorname{Sym}^d \Sigma \vert_{\mathcal{S}_0' \#\mathcal{S}_0}$ has no zeros near $\mathbf{x}_H$.  Hence, we must have $\Delta \cdot \mathcal{S}_0  = J_0(A_0) $ and $ \Delta \cdot \mathcal{S}_0'  = J_0(A_0')$.
By the
 same argument in Lemma \ref{lem9}, we also have $\int_{A_0} \omega = \int_{\mathcal{S}_0} \omega_V. $

With the above preparations, define
\begin{equation*}
\begin{split}
\Phi^L_H:  CF^L_*(\Sigma, \varphi_H(\underline{\Lambda}),   & \  \underline{\Lambda}) \to CF_*^L ( \operatorname{Sym}^d\varphi_{H}(\underline{\Lambda}),  \operatorname{Sym}^d\underline{\Lambda})\\
&(\textbf{y}, [A]) \to (\textbf{y}, [\Psi_*(A)]).
\end{split}
\end{equation*}
  By the above discussion, we know that $\Phi^L_H$ is a chain map and it induces an isomorphism at the  homological level. From the construction
 of $\Phi_H^L$, it is easy to check that first diagram in Theorem \ref{thm1} is true. To prove the reminder statements of Theorem \ref{thm1}, we still need some preparations. We prove them in the next part.

	%	By Lemmas \ref{lem19}, \ref{lem9},  we get a canonical identification between  $CF_*(\Sigma, \varphi_H, \Lambda, \textbf{x})$ and  $CF (Sym^d\varphi_{H}(\Lambda), Sym^d\Lambda, \textbf{x})$.

%$\mathcal{A}^{\eta}_H(\textbf{y},  [A] ) = \mathfrak{A}^{\eta}_H(\textbf{y}, \Psi_*([A]))$.  By the same argument in Lemma \ref{lem9}, we have $\int_{A_0} \omega = \int_{\mathcal{S}_0} \omega_V$.
% If we can find a nice  surface $S$ which is order 2 and holomorphic near branch points as in Lemma , the

 %Combine this with Lemma \ref{lem9}; we know that two action functional are the same.
		%We have proved this in Lemma \ref{lem9}.
	\end{proof}

	To move on, recall the continuous morphism  $I_{H, G}$  is defined by counting HF curves in
 $\mathbb{R} \times [0,1] \times \Sigma$ with Lagrangian boundary $\mathcal{L} = F(\mathbb{R} \times [0,1] \times  \underline{\Lambda})$, where $F$ is defined in  (\ref{eq44}).   The symplectic form is $\Omega_E =\omega_E + ds \wedge dt$ (\ref{eq56}). By a direct computation, we
 have
 \begin{equation*}
   \omega_E= \omega + \operatorname{a} \wedge ds + (1+ \partial_s H^s) ds \wedge dt,
 \end{equation*}
	where $\operatorname{a} $ is a $(s,t)$-dependent 1-form on $\Sigma$ defined by
\begin{equation*}
\operatorname{a} (v):=\omega((\varphi_{H^s})_* \circ (\varphi^t_{H^s})^{-1}_*( \partial_s (\varphi^t_{H^s} \circ \varphi_{H^s}^{-1})), v)
\end{equation*}
	for any $v\in T\Sigma$.  Note that $\operatorname{a} =0$ when $|s| \ge R_0$. Let $C(s)$ be  a function such
 that $C(s) = $C for $ |s|  \le R_0$ and $C(s) = 0$ for $|s| \ge  2R_0$, where $C  \ge  0$ is a large constant.   Replace the symplectic form by $\Omega'_E =\omega_E + (1+C(s)) ds \wedge dt.$ Note that  $\Omega'_E$ still is symplectic and $\mathcal{L}$ is $\Omega'_E$-Lagrangian.   The advantage of this choice is that the almost complex structure $
J=\begin{bmatrix}
j_{\mathbb{D}}& 0 \\
0 &  j_{s,t}
\end{bmatrix}
$ is $\Omega_E'$-tame, where $j_{\mathbb{D}}$ is the standard complex structure on $\mathbb{R} \times [0,1]$ mapping $\partial_s$ to $\partial_t$, and $j_{s,t}$ is a family of complex structures on $\Sigma$.  Also, we require that $j_{s,t}$ is
$s$-independent when $|s| \ge  2R_0$.

  We perform the same construction for $ \operatorname{Sym}^d \Sigma$. Let $\mathbb{H}$, $\mathbb{G}$  be Hamiltonian functions
 which are compatible with  $ \operatorname{Sym}^d H$, $ \operatorname{Sym}^d G$, i.e., they satisfy the following properties:
 \begin{equation}\label{eq70}
    \begin{split}
    & \mathbb{H}= \operatorname{Sym}^dH, \mathbb{G} = \operatorname{Sym}^d G \mbox{  outside the a neighbourhood of the diagonal }\Delta;\\
     &\mathbb{H}, \mathbb{G}  \mbox{ are ($t$-dependent) constant near diagonal}.
    \end{split}
 \end{equation}
The construction of such   $\mathbb{H}, \mathbb{G}$ can be found in Remark 6.8 of \cite{CHMSS}.
	Let $\mathbb{H}^s =\chi(s)\mathbb{H} + (1-\chi(s))\mathbb{G}$. Define a diffeomorphism
 \begin{equation*}
    \begin{split}
     F' : &\mathbb{R} \times [0,1] \times  \operatorname{Sym}^d \Sigma \to \mathbb{R} \times [0,1] \times  \operatorname{Sym}^d \Sigma\\
     &(s, t, \mathbf{x}) \to (s, t, \varphi_{\mathbb{H}^s} \circ  (\varphi_{\mathbb{H}^s}^t)^{-1} (\mathbf{x})).
    \end{split}
 \end{equation*}
 Let
  \begin{equation*}
    \begin{split}
    & \mathbb{L}:= F'(\mathbb{R} \times \{0, 1\} \times  \operatorname{Sym}^d \underline{\Lambda}),\\
     &\Omega := (F')^{-1}(\omega_V + d(\mathbb{H}^s dt) ) + (1 + C(s)) ds \wedge dt.
    \end{split}
 \end{equation*}
 Then $\Omega$ is symplectic and the almost complex structure $
\mathbb{J}=\begin{bmatrix}
j_{\mathbb{D}}& 0 \\
0 &   \operatorname{Sym}^d j_{s,t}
\end{bmatrix}
$ is $\Omega$-tame.  Since $\varphi_{\mathbb{H}^s}^t = \operatorname{Sym}^d \varphi_{H^s}^t$ away from the diagonal, $\mathbb{L} = \operatorname{Sym}^d \mathcal{L}$.

 \begin{proof} [Proof of Theorem \ref{thm1} Part B]
 With the above preparations, we prove the second diagram of Theorem \ref{thm1}. Let  $I_{H,G}$ and $\mathbb{I}_{H,G}$ be the continuous morphisms defined by
 $(\Omega'_E, \mathcal{L},J)$ and $(\Omega, \mathbb{L},\mathbb{J})$ respectively.  It suffices to show that
 \begin{equation*}
    \langle   I_{H, G} (\mathbf{y}_+, [A_+]), (\mathbf{y}_-, [A_-]) \rangle =     \langle \mathbb{I}_{H, G} (\mathbf{y}_+, [\Psi_*(A_+)]), (\mathbf{y}_-, [\Psi_*(A_-)])  \rangle 
 \end{equation*}
The argument is similar to the $\mathbb{R}$-invariant case, we establish a one-to-one correspondence between the $J$-holomorphic HF curves in
$ (\mathbb{R} \times [0,1] \times \Sigma, \mathcal{L})$ and the $\mathbb{J}$-holomorphic
 section in $(\mathbb{R} \times  [0,1] \times   \operatorname{Sym}^d \Sigma, \mathbb{L})$.

Let $u$ be a HF curve in $(\mathbb{R}\times [0,1]\times \Sigma, \mathcal{L})$ with $[u] = A$ and $I(A) = 0$. Note that the
 fibers $\{(s,t)\}\times \Sigma $ are still holomorphic with respect to $J$. Then $u$ intersects  $\{(s,t)\}\times \Sigma $ at $d$  points (counting multiplicities).   We regard these $d$ points as a point in $ \operatorname{Sym}^d\Sigma$,
 and then get a $\mathbb{J}$-holomorphic section $s_u$ in $\mathbb{R} \times  [0,1] \times   \operatorname{Sym}^d \Sigma$. Applying the same argument in Proposition \ref{lem5}, we have $ \operatorname{ind} s_u = 0$ and $J_0(u) = s_u \cdot \Delta$.

  Let $A_+ \in H_2(\Sigma,\mathbf{x}_H, \mathbf{y}_+)$ and $A_- \in H_2(\Sigma, \mathbf{x}_G, \mathbf{y}_-)$ such that $A_+\#[u] \#(-A_-) =
 A_{ref}$, where $A_{ref} = F(\mathbb{R} \times [0, 1] \times \mathbf{x} ).$ Now  we  show that the homology class
 of $s_u$ is determined by $\Psi_*(A_{\pm})$ and $\Psi_*(A_{ref})$, where $\Psi_*(A_{ref}) =[F'(\mathbb{R} \times [0, 1] \times \mathbf{x} )]$
 is the tautological correspondence of $A_{ref}$. Since $H_2( \operatorname{Sym}^d\Sigma,\mathbf{x}_H, \mathbf{x}_G)$  is an affine space of
$ H_2( \operatorname{Sym}^d\Sigma ,  \operatorname{Sym}^d\underline{\Lambda})$ and $ H_2( \operatorname{Sym}^d\Sigma ,  \operatorname{Sym}^d\varphi_{H^s}(\underline{\Lambda}))$, we have
\begin{equation} \label{eq57}
\Psi_*(A_+) \#[s_u] \#(-\Psi_*(A_-)) = \Psi_*(A_{ref}) + \sum_{i=1}^{k+1} c_i [\Psi_*(B_i)] + \sum_{j=1}^{k+1} c'_j [\Psi_*(\varphi_{H^s}(B_i))]
\end{equation}
for some $c_i, c_j'\in \mathbb{Z}$. Note that  the relative homology classes $[\varphi_{H^s}(B_i)] $ and $[\Psi_*(\varphi_{H^s}(B_i))] $
 are independent of $s$.

 Choose $H^s$ be a generic path such that $\mathring{B}_i \cap \varphi_{H^s} (\mathring{B}_{k+1}) \ne \emptyset$ for all $s$ and $1\le i\le k$.
  Let $z(s) : \mathbb{R} \to \mathring{B}_i $ be a path such that $z(s) \in
 \mathring{B}_i \cap \varphi_{H^s}(\mathring{B}_{k+1}) $ for all $s$ ($1\le i\le k+1$). As a result,   $z(s)$ does not lie inside $ \underline{\Lambda}$, $\varphi_{H^s}( \underline{\Lambda} )$ and $\varphi_{H^s}(\mathring{B}_j)$ for $1 \le j \le k$. Define
 intersection numbers
 \begin{equation*}
    \begin{split}
     &n_{z(s)} (u) : =\# (u\cap  \{(s,t, z(s)) \vert (s, t) \in \mathbb{R} \times [0,1]\})\\
      &\mathfrak{n}_{z(s)} (s_u) : =\# (s_u\cap  \{(s,t) \times \{ z(s)\} \times \operatorname{Sym}^{d-1}\Sigma  \vert (s, t) \in \mathbb{R} \times [0,1]\}).
    \end{split}
 \end{equation*}
 The above intersection numbers only depend on the relative homology classes of $u$ and $s_u$.  If we choose $z(s)$ such that $\{(s,t,z(s))|(s,t) \in \mathbb{R} \times  [0,1]\}$ is away from the double
 point of $u$, then one can check that $n_{z(s)}(u) = \mathfrak{n}_{z(s)}(s_u) $ by definition.  We also have
$ n_{z(s)}(A_{ref}) = \mathfrak{n}_{z(s)} (\Psi_*(A_{ref}))$. Applying $ \mathfrak{n}_{z(s)}$ to (\ref{eq57}), we have
 \begin{equation*}
 \mathfrak{n}_{z(s)} (\Psi_*(A_{+})) +  \mathfrak{n}_{z(s)} (s_u) - \mathfrak{n}_{z(s)} (\Psi_*(A_{-})) - \mathfrak{n}_{z(s)} (\Psi_*(A_{ref}))  =c_i + c_{k+1}'.
 \end{equation*}
 On the other hand, by Lemma \ref{lem9}, we get
 \begin{equation*}
    \begin{split}
   &\mathfrak{n}_{z(s)} (\Psi_*(A_{+})) +  \mathfrak{n}_{z(s)} (s_u) - \mathfrak{n}_{z(s)} (\Psi_*(A_{-})) - \mathfrak{n}_{z(s)} (\Psi_*(A_{ref}))  \\
   =&n_{z(s)}(A_+) + n_{z(s)}(u) -n_{z(s)}(A_-) -n_{z(s)}(A_{ref}) \\
   =& n_{z(s)}(A_+\#[u] \# (-A_-)) -n_{z(s)}(A_{ref})= 0.
    \end{split}
 \end{equation*}
  Hence, $c_i +c'_{k+1} = 0$ for $1\le i\le k+1$. By the same argument, we also obtain $c'_i +c_{k+1} = 0$. Therefore, by the same computations as in (\ref{eq58}), we have
 \begin{equation*}
   \Psi_*(A_+) \#[s_u] \#(-\Psi_*(A_-)) =\Psi_*(A_{ref}).
 \end{equation*}

 Let $\mathcal{M}^J(\mathbf{y}_+, \mathbf{y}_-, A)$ be the moduli space of HF curves in $(\mathbb{R} \times [0,1 ] \times \Sigma, \mathcal{L})$, and  $\mathcal{M}^{\mathbb{J}}(\mathbf{y}_+, \mathbf{y}_-, \mathcal{S})$ be the moduli space of holomorphic sections in  $(\mathbb{R} \times [0,1 ] \times\operatorname{Sym}^d \Sigma, \mathbb{L})$, where $\mathcal{S}$ denotes the class $   (-\Psi_*(A_+)) \#\Psi_*(A_{ref}) \#\Psi_*(A_-) $. Then $ \operatorname{ind}  \mathcal{S} = I(A) =0$. We get an injective map
 \begin{equation*}
   \Psi : \mathcal{M}^J(\mathbf{y}_+, \mathbf{y}_-, A) \to  \mathcal{M}^{\mathbb{J}}(\mathbf{y}_+, \mathbf{y}_-, \mathcal{S})
 \end{equation*}
 by sending $u$ to $s_u$.  The inverse of $\Psi$ is constructed by diagram (\ref{eq18}) as before. Then
 $\Psi$ is a bijective map.

 \end{proof}

  For $a \in HF(\operatorname{Sym}^d \underline{\Lambda})$, since the isomorphism $\Phi_H$ is compatible with the continuous  morphisms, the map $$\Phi:=j_H \circ \Phi^{-1}_H \circ \mathbf{j}_H^{-1}: HF(\Sigma,  \underline{\Lambda}) \to HF(\operatorname{Sym}^d  \underline{\Lambda})$$ is independent of the choice of $H$.%, i.e.,
% $j_H \circ \Phi^{-1}_H \circ \mathbf{j}_H^{-1}(a)$ is a fixed class (independent of $H$ and $\mathbf{x}$) in $HF(\Sigma,  \underline{\Lambda})$.

\begin{definition} \label{definition4}
Let  $e_{\underline{\Lambda}}: =\Phi^{-1}(\mathbf{1}_{\underline{\Lambda}})$. We call $e_{\underline{\Lambda}}$ the unit of   $ HF(\Sigma, \underline{\Lambda})$.
\end{definition}
  From the definition, we have  $  \Phi(e_{\underline{\Lambda}}) = \mathbf{1}_{\underline{\Lambda}}$. So far we have finished
 the proof of Theorem \ref{thm1}.
	\section{Closed-open  morphisms}
Instead of proving Theorem \ref{thm2} directly,  the goal of this section and the upcoming sections is to prove the following substitute  theorem.
	\begin{thm} \label{thm5}
		Fix a  $\eta$-admissible link $\underline{\Lambda}$ and a base point $\mathbf{x} \in \operatorname{Sym}^d\underline{\Lambda}$. For a generic
 admissible almost complex structure   $J \in \mathcal{J}_{tame}(W, \Omega_{\varphi_H})$  (Definition \ref{definition1}),   there exists a  homomorphism
$$\widetilde{\mathcal{CO}}(\underline{\Lambda}, H)_J: \widetilde{PFH}(\Sigma, \varphi_H, d)_J \to HF(\Sigma, \varphi_H(\underline{\Lambda}),  \underline{\Lambda} )_J$$
satisfying  the following properties:
		\begin{itemize}
			\item
			(Partial invariance)
			Suppose that $\varphi_H, \varphi_G$ satisfy the following conditions: (see Definition \ref{definition2})
			\begin{enumerate} [label=\textbf{$\spadesuit$.\arabic*}]
				\item  \label{assumption1}
				Each  periodic orbit of $\varphi_{H}$ with degree less than or equal $d$ is either $d$-negative elliptic or hyperbolic.
				
				\item \label{assumption2}
				Each periodic orbit of  $\varphi_{G}$  with degree less than or equal $d$ is either $d$-positive elliptic or hyperbolic.	
			\end{enumerate}
Then for  generic admissible almost complex structures   $J_H \in \mathcal{J}_{tame}(W, \Omega_{\varphi_H})$ and $J_G \in \mathcal{J}_{tame}(W, \Omega_{\varphi_G})$, we have  the following commutative diagram:
			$$\begin{CD}
				\widetilde{PFH}_*(\Sigma,  \varphi_H, d)_{J_H} @> \widetilde{\mathcal{CO}}(\underline{\Lambda}, H)_{J_H}>> HF_{*}(\Sigma, \varphi_H(\underline{\Lambda}), \underline{\Lambda}  )_{J_H} \\
				@VV \mathfrak{I}_{H, G}V @VV I_{H, G}V\\
				\widetilde{PFH}_*(\Sigma,  \varphi_G, d)_{J_G}  @> \widetilde{\mathcal{CO}}(\underline{\Lambda},  G)_{J_G}>> HF_{*}(\Sigma, \varphi_G(\underline{\Lambda}), \underline{\Lambda} )_{J_G}.
			\end{CD}$$

			\item
			(Non-vanishing)	If $\varphi_H$ satisfies the  condition (\ref{assumption1}),   then there is a non-zero class $\mathfrak{e}  \in \widetilde{PFH}_*(\Sigma,  d )$ and    $e_{\underline{\Lambda}} \in HF_*(\Sigma, \underline{\Lambda})$ such that
			$$\widetilde{\mathcal{CO}}(\underline{\Lambda}, H)_J ( \mathfrak{j}_H^{-1}(\mathfrak{e}) )= j_H^{-1}(e_{\underline{\Lambda}}),$$
		where   $\mathfrak{j}_H$ and $j_H$ are respectively  the canonical isomorphisms (\ref{eq71}) and   (\ref{eq37}),  and $e_{\underline{\Lambda}} \ne 0  \in HF_*(\Sigma, \underline{\Lambda})$ is the unit of  HF (see Definition \ref{definition4}). In particular, $\widetilde{\mathcal{CO}}(\underline{\Lambda}, H)_J$ is non-vanishing.
		
		\end{itemize}
	\end{thm}
\begin{remark}
The key difference between   Theorem \ref{thm5} and Theorem \ref{thm2} is the
appearance of the technical assumptions \ref{assumption1} and \ref{assumption2}. However,  Theorem  \ref{thm2} is \textbf{not} a  better result  than Theorem \ref{thm5}.  In fact,  Theorem  \ref{thm2}  is merely a reformulation of  Theorem \ref{thm5}.  The reasons are as follows:

The closed-open morphisms in   Theorem  \ref{thm2} are  defined by
\begin{equation}\label{eq67}
  \mathcal{CO}(\underline{\Lambda}, H): =   I_{H_0, H}  \circ \widetilde{\mathcal{CO}}(\underline{\Lambda}, H_0)_{J_0} \circ \mathfrak{I}_{H, H_0},
\end{equation}
where $H_0$ is a certain fixed Hamiltonian function and $J_0$ is a fixed almost complex structure.  Using this definition, the first     bullet  on    $\mathcal{CO}$  in Theorem \ref{thm2} is  just a consequences   of the properties of  continuous morphisms ((\ref{eq66}) and Proposition of \ref{lem30}).  Provided that  $\widetilde{\mathcal{CO}}(\underline{\Lambda}, H_0)$  maps  $ \mathfrak{j}_{H_0}^{-1}(\mathfrak{e}) $ that $ j_{H_0}^{-1}(e_{\underline{\Lambda}})$, then we obtain the second bullet of Theorem \ref{thm2} . 

To estimate the spectral invariants (\ref{eq68}), we need a holomorphic curve  in a  closed-open symplectic coboridsm $(W_{\varphi_H}, \Omega_{\varphi_H}, L_{\Lambda_H})$ associated with the data $(\underline{\Lambda}, H)$.  However, the definition (\ref{eq67}) is   not sufficient for this purpose  whereas $\widetilde{\mathcal{CO}}$ is, because  $\widetilde{\mathcal{CO}}(\underline{\Lambda}, H)$ is defined by counting holomorphic curves in $(W_{\varphi_H}, \Omega_{\varphi_H}, L_{\Lambda_H})$.  The key observation is that if  $H$ and $H_0$ satisfy  \ref{assumption1} and \ref{assumption2} respectively, then the partial invariance in Theorem \ref{thm2} implies that  ${\mathcal{CO}}(\underline{\Lambda}, H)=\widetilde{\mathcal{CO}}(\underline{\Lambda}, H)$.  Consequently,  we obtain a holomorphic curve in   $(W_{\varphi_H}, \Omega_{\varphi_H}, L_{\Lambda_H})$ and obtain  (\ref{eq68}) for $H$ satisfying \ref{assumption1}. The general case follows from the continuity of the spectral invariants.  Therefore, the underlying proof of Theorem \ref{thm2} in fact also requires  the assumptions \ref{assumption1} and \ref{assumption2}.

\end{remark}

\begin{remark}
The assumptions \ref{assumption1} and \ref{assumption2} come from the holomorphic curve definition of the PFH cobordism maps. Due to certain technical issues mentioned in Section 5.5 of \cite{H4}, the
cobordism maps on PFH cannot yet be defined using holomorphic curves.  At present, the  definition of the PFH cobordism maps relies heavily on Lee-Taubes’s isomorphism ``SWF=PFH” \cite{LT} and the Seiberg-Witten theory \cite{KM}.

However, in the proof of Theorem \ref{thm5}, we indeed require a holomorphic curve definition of the cobordism maps. Assumptions \ref{assumption1} and \ref{assumption2} are introduced to ensure
that the cobordism maps on PFH can be defined in terms of holomorphic curves (see
Theorem 2 of \cite{GHC}). If one could give a holomorphic curve definition for the PFH cobordism maps, then   we believe that   ${\mathcal{CO}}(\underline{\Lambda}, H)=\widetilde{\mathcal{CO}}(\underline{\Lambda}, H)$ for any $H$.
\end{remark}

 In this section, we construct the closed-open morphisms in Theorem \ref{thm5} and prove the partial invariance.
 The construction is very similar to the construction of the PFH cobordism maps. It
 is defined by counting $ I = 0$ holomorphic curves in a ``closed-open” symplectic manifold
 $W$. As pointed out by Hutchings (Section 5.5 of  \cite{H4}), the main difficulty of defining the
PFH cobordism maps is that the appearance of holomorphic curves with negative ECH
 index. These curves violate the compactness and tranvsersality of the moduli space.
 However, in the “closed-open” setting, the HF-ends ensure that the holomorphic curves
 cannot be multiply-covered. A key observation is that if a holomorphic curve in W with
 at least one end, then it has at least one positive end and one negative end. This is due
 of the fibration structure of $W$. By the ECH index inequality, the “closed-open” curves
 have non-negative ECH index. Moreover, as in Lemma \ref{lem6}, the bubbles contribute at
 least 2 to the ECH index; we can rule out them by the index reason. Therefore, we can
 define closed-open morphisms by using the classical techniques.

	%Let $Y=\frac{[0,2] \times \mathbb{S}^2}{(x, 2) \sim (\varphi_H(x), 0)}$ be the mapping torus of $\varphi_H.$ %Replay $H$ by $G(t, x)=\frac{1}{2}H(\frac{t}{2},x). t\in[0,2]$, then $\phi_G^2=\phi_H^1$.  There  is a trivialization $\Psi: \mathbb{R}/2\mathbb{Z} \times \Sigma \to Y$ of $Y$ defined by $\Psi(t,x ) =(t, \phi_t^{-1}(x)). $
	%Obviously, $\pi: Y \to \mathbb{R}/2\mathbb{Z}$ is a bundle over the circle with fiber $\Sigma$. The volume form $\omega$ descend to a fiberwise symplectic form $\omega_{\phi_H}$ on $Y$. The fiber bundle structure can be extended to $\pi: \mathbb{R}_s \times Y \to \mathbb{R}_s \times S^1$.  The natural symplectic form on $\mathbb{R}_s \times Y$ is $\Omega = \omega_{\phi_H} + ds \wedge dt$.
	%Let $B$ be a surface as in the following picture. Define $W_H=\pi^{-1}(B)$. Obviously, $\pi: W_H \to B$ is a sphere bundle over the surface $B$.  The symplectic form $\Omega_H$ on $W_H$ is defined to be the restriction of $\omega_{\varphi_H} + ds\wedge dt$.
	
\subsection{Closed-open curves and their indexes}	 \label{section5}

In this subsection, we review the closed-open cobordism, PFH-HF curves, ECH index, etc., which are largely paraphrased from Chapter 5 of  \cite{VPK}.

\paragraph{Closed-open cobordism} 	First of all, we introduce the closed-open symplectic manifold. Define a surface  $B \subset \mathbb{R} \times (\mathbb{R}/(2 \mathbb{Z}))$  by $B:=  \mathbb{R}_s \times (\mathbb{R}_t/(2 \mathbb{Z})) -B^c$, where $B^c$ is $(-\infty, -2)_s \times [1,2]_t $ with the corners rounded.  The picture of $B$ is  shown in Figure \ref{figure1}.
	\begin{figure}[h]
		\begin{center}
			\includegraphics[width=10cm, height=5cm]{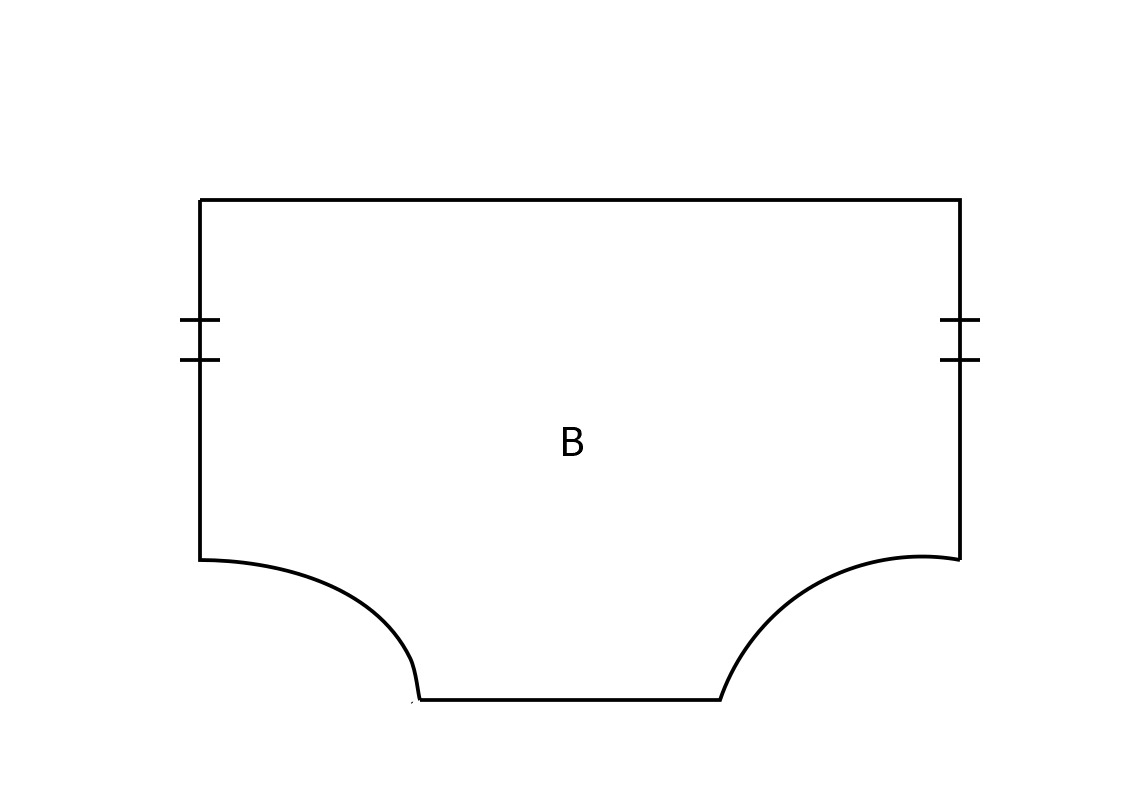}
		\end{center}
		\caption{}
		\label{figure1}
	\end{figure}

	Define the mapping torus $Y_{\varphi_H} :=[0, 2] \times \Sigma / (0, \varphi_H(x)) \sim (2, x)$.  Then  $\pi: \mathbb{R}_s \times Y_{\varphi_{H}} \to \mathbb{R} \times  (\mathbb{R}/(2 \mathbb{Z}))$ is a surface  bundle over the cylinder.
	Define $W_{\varphi_H}:=\pi^{-1}(B)$. Obviously, $\pi_W: W_{\varphi_H} \to B$ is a surface  bundle over the surface $B$.  The symplectic form $\Omega_{\varphi_H}$ on $W_{\varphi_H}$ is defined to be the restriction of $\omega_{\varphi_H} + ds\wedge dt$.
	
	% We define a closed--open cobordism $W_H=\pi^{-1}(B)$ from $Y$ to $[0,1]\times \Sigma$. The symplectic form on $W_H$ is the resctiction of $\Omega$.
	
	%Let $L_{\Lambda}=\Psi(\partial B \times \phi(\Lambda)) \subset W \subset \mathbb{R}_s \times Y$.
	We place a copy of $\underline{\Lambda}$ on the fiber $\pi_B^{-1}(-3,1)$ and take its parallel transport along $\partial B$ using the symplectic connection.  The parallel transport sweeps out a Lagrangian submanifold $L_{\Lambda_H}$ of $(W_{\varphi_H}, \Omega_{\varphi_H})$.  Note that $L_{\Lambda_H}$ consists of $d$ disjoint connected components.  Moreover, we have
	\begin{equation*}
		\begin{split}
			&L_{\Lambda_H} \vert_{s \le -3 \times\{0\}} = \mathbb{R}_{s \le -3} \times \{0\} \times \varphi_H(\underline{\Lambda})\\
			&L_{\Lambda_H} \vert_{s \le -3 \times \{1\}} = \mathbb{R}_{s \le -3} \times \{1\} \times \underline{\Lambda}.
		\end{split}
	\end{equation*}
	The triple $(W_{\varphi_H}, \Omega_{\varphi_H}, L_{\Lambda_H})$ is  called a \textbf{closed-open symplectic  cobordism}.

	\begin{remark}\label{remark1}
		%We use the subscript  $H$ to indicate that  the construction of $W_H$ depends  on  $\phi_H$. However, for any Hamiltonian functions $H$ and $G$,   by using \ref{eq11}, we can construct a diffeomorphism from $W_H$ to $W_G$ preserving the fibration structure.
		By using the trivialization (\ref{eq11}), we can identify $W_{\varphi_H}$ as a  bundle over $\pi^{-1}(B) \subset \mathbb{R} \times S^1 \times \Sigma$, where $\pi: \mathbb{R} \times S^1 \times \Sigma\to \mathbb{R} \times S^1$ is the projection.  Therefore, for any two Hamiltonian functions $H,G$, we have a diffeomorphism from $W_{\varphi_H}$ to $W_{\varphi_G}$ preserving the fibration structure. When  the context is clear, we  suppress the subscript ``$\varphi_H$'' from the notation.
	\end{remark}
	
	\begin{remark} \label{remark5}
		We  define a  slight different trivialization of $W_{\varphi_H}$. The advantage   of the following trivialization is that we   also trivialize the Lagrangian submanifold $L_{\Lambda_H}$ simultaneously. We shift the coordinate by $-\frac{1}{2}$ and define
$$\mathbb{R} \times Y_{\varphi_H} :=\mathbb{R}  \times [-\frac{1}{2}, \frac{3}{2}] \times \Sigma/ (s, -\frac{1}{2}, \varphi_H(x)) \sim (s, \frac{3}{2}, x).$$
		Let $\chi: \mathbb{R} \to \mathbb{R}$ be a nondecreasing cutoff function such that $\chi=1$ when $t \ge 1$ and $\chi=0$ when $t \le 0$.  Define a flow $\varphi_t : = \varphi_H^{\chi(t)}$ and extend it to be $\operatorname{Id}$ for $t \le 0$ and be $\varphi_H$ for $t \ge 1$.  Define a trivialization by
		\begin{equation*}
			\begin{split}
				\Psi_H' :  &  \ \mathbb{R} \times [-\frac{1}{2},  \frac{3}{2}]/ ( -\frac{1}{2}  \sim \frac{3}{2}) \times \Sigma \to \mathbb{R} \times Y_{\varphi_H}\\
				& (s, t, x) \to (s, t, \varphi_H \circ \varphi_t^{-1}(x)).
			\end{split}
		\end{equation*}
		The restriction of   $\Psi_H'$ to  $B \times \Sigma$ gives a trivialization of $W_{\varphi_H}$. Moreover, we have  $(\Psi_H')^* \omega_{\varphi_H} = \omega + d (\dot{\chi} H) \wedge dt$ and $\Psi_H' (\partial B \times   \underline{\Lambda}) =L_{\Lambda_H}$.
	\end{remark}

	\paragraph{PFH-HF curves}
\begin{definition}\label{definition1}
Let $J_{tame}(W,\Omega)$ be the set of $\Omega$-tame almost complex structures $J$
 such that $J$ agrees with the admissible almost complex structures on the ends, and the projection $\pi_W$ is $(J,j_B)$  complex linear.
\end{definition}

	\begin{definition}
		Fix a Reeb chord $\mathbf{y}$ and an orbit set $\alpha$ with degree $d$. Let $(\dot{F}, j)$ be a Riemann surface (possibly disconnected) with   punctures.   Each irreducible component of $\dot{F}$ has at least one puncture. A \textbf{d-multisection} is a smooth map $u: (\dot{F}, \partial \dot{F}) \to (W, L_{\Lambda_H}) $ such that
		\begin{enumerate}
			\item
			$u(\partial {\dot{F}}) \subset L_{\Lambda_H }$. Let $\{L^i_{\Lambda_H}\}_{i=1}^d$ be the connected components of $L_{\Lambda_H}$. For each $1 \le i\le d$,   $u^{-1}(L_{\Lambda_H}^i)$ consists of exactly one component of $\partial \dot{F}$.
			\item
			$u$ is asymptotic to $\mathbf{y}$ as $s \to -\infty $.
			\item
			$u$ is asymptotic to $\alpha$ as $s \to +\infty $.
			
		\end{enumerate}
Given $J \in J_{tame}(W,\Omega)$, a $J$-holomorphic  $d$-multisection  is called a \textbf{PFH-HF} curve.
	\end{definition}
By the same reasons in Remark \ref{remark6}, here we exclude the possibilities that a PFH-HF curve contains an irreducible within fiber.  Consequently, an irreducible  PFH-HF curve has at least one positive end and one negative end.
	
	Fix an orbit set  $\alpha$ and a Reeb chord $\mathbf{y}$. Let  $$Z_{\alpha, \mathbf{y}} : = L_{\Lambda_H} \cup (\{ \infty \} \times \alpha) \cup ( \{- \infty \}  \times \textbf{y}) \subset W.$$
	We denote $H_2(W, \alpha, \mathbf{y})$ the equivalence classes of continuous  maps  $u: (\dot{F}, \partial \dot{F}) \to (W, Z_{\alpha, \textbf{y}}) $ satisfying 1), 2), 3) in  the above definition. Two  maps are equivalent if their difference is trivial in $H_2(W,Z_{\alpha, \textbf{y}}; \mathbb{Z})$.  Then   $H_2(W, \alpha, \textbf{y})$ is an affine space over  $H_2(W, L_{\Lambda_H}; \mathbb{Z})$.   By the trivialization in Remark \ref{remark5} and the following exact sequence
 for relative homology,
 \begin{equation*}
   0 \to H_2(W, \mathbb{Z}) \to H_2(W, L_{\Lambda_H}, \mathbb{Z}) \to H_2(W, L_{\Lambda_H}, \mathbb{Z})  \to H_1(W, \mathbb{Z}) \to ...
 \end{equation*}
 we have  $H_2(W, L_{\Lambda_H}, \mathbb{Z}) \cong (H_1(S^1, \mathbb{Z}) \otimes H_1(\Sigma, \mathbb{Z})) \oplus_{i=1}^{k+1} \mathbb{Z} [B_i]. $ Therefore, the difference of any two relative homology classes can be written as
	\begin{equation*}
		\mathcal{Z}'-\mathcal{Z}=\sum_{i=1}^{k+1} c_i [B_i] +   m[\Sigma]+ [S],
	\end{equation*}
	where $[B_i]$ is the class represented by a symplectic parallel translation of $\{-3\} \times \{1\} \times B_i$, $i=1, ..., k+1$, and $[S]$ belongs to the   $H_1(S^1, \mathbb{Z}) \otimes H_1(\Sigma, \mathbb{Z})$-component of $H_1(Y_{\varphi_H}, \mathbb{Z})$.
	
	\paragraph{Fredholm index}
	Similar as before, we fix a non-singular vector field on $\underline{\Lambda}$ and it induces  a trivialization $\tau $  of $T\Sigma$   along  $\underline{\Lambda}$. We extend $\tau$ by parallel transport along $L_{\Lambda_H}$ and arbitrarily along $\textbf{y}$ and $\alpha$. Define a subbundle $\mathcal{L}$ of $u^*T\Sigma \vert_{\partial {F}}$ as follows. Set $\mathcal{L}: =u^*( TL_{\Lambda_H} \cap T\Sigma)$ along $\partial \dot{F}$. We extend $\mathcal{L}$ over $\partial F \setminus  \partial \dot{F} $ by rotating in  clockwise direction from $T \underline{\Lambda}$ to $T\varphi_H(\underline{\Lambda})$ by minimum amount possible.  Let $\mu_{\tau}(u)$ denote the Maslov index of the bundle pair $(u^*T\Sigma \vert_{\partial F}, \mathcal{L}) $.
	
	The  Fredholm index  is defined by
	\begin{equation}
		\operatorname{ind}u : =-\chi(\dot{F}) + 2c_1(u^* T\Sigma, \tau) + CZ_{\tau}^{ind}(\alpha) +\mu_{\tau}(u),
	\end{equation}
where $CZ_{\tau}^{ind}(\alpha)$ is the 	combination of the  Conley-Zehnder index defined  in  (\ref{eq29}).
	
	\paragraph{ECH index}
	Fix $\mathcal{Z} \in H_2(W, \alpha, \textbf{y})$. Let $u: \dot{F} \to W $ be a $\tau$-trivial representative in the sense of Definition 5.6.1 of \cite{VPK}. We define  the relative self-intersection $Q_{\tau}(\mathcal{Z})$ as before.  Then the ECH index (Definition 5.6.6 \cite{VPK}) is
	\begin{equation}
		I(\mathcal{Z}) : = c_1(TW \vert_{\mathcal{Z}}, \tau) +Q_{\tau}(\mathcal{Z}) + CZ^{ech}_{\tau}(\alpha)+ \mu_{\tau}(\mathcal{Z}),
	\end{equation}
	where $CZ^{ech}_{\tau}(\alpha)$ is the 	combination of the  Conley-Zehnder index defined  in  (\ref{eq28}).

\begin{lemma}  \label{lem38}
The ECH index satisfies the following properties:
\begin{itemize}
\item
(Theorem 5.6.9 of \cite{VPK}) Let $u \in \mathcal{M}^J( \alpha, \mathbf{y})$ be an irreducible  PFH-HF  curve. Then we have
\begin{equation}  \label{eq59}
		\begin{split}
		 &I(u) \ge \operatorname{ind} u  + 2 \delta(u).\\
		\end{split}
	\end{equation}	
 Moreover, equality  holds only if the positive ends of $u$ satisfy  the ECH partition condition.

 \item
 If $u=\cup_a u_a$  is a PFH-HF  curve consisting  of several  (distinct) irreducible components, then
\begin{equation*}
		\begin{split}
		  I(u) \ge \sum_a I(u_a) + 2\sum_{a \ne b} \# (u_a \cap u_b).
		\end{split}
	\end{equation*}
\item
Let $u$ be a PFH-HF  curve. Then     $I(u) \ge 0$  provided that $u$ is Fredholm regular.

\item
Let $\mathcal{Z}', \mathcal{Z} \in H_2(W,  \alpha, \mathbf{y})$  be relative  homology classes  such that $$\mathcal{Z}'-\mathcal{Z}=m[\Sigma]+\sum_{i=1}^{k+1} c_i[B_i] + [S], $$ where $[S] \in H_1(S^1, \mathbb{Z}) \otimes H_1(\Sigma, \mathbb{Z})$.   Then  we have
		\begin{equation} \label{eq5}
\begin{split}
			&I(\mathcal{Z}')=I(\mathcal{Z}) + \sum_{i=1}^{k+1} 2c_i  + 2m(k+1).
\end{split}
		\end{equation}
 \end{itemize}
\end{lemma}
%In this paper, we don't need the details on  ``ECH partition condition''.   For the readers who are interested  in it, please refer to Definition 4.1 of \cite{H1} and Definition 4.13 of \cite{H2}. % The proof of Theorem \ref{thm5} basically is a combination of the relative adjunction formula and Hutchings's analysis in \cite{H2}. We omit the details here.

\begin{proof}
We prove the statements one by one as follows:
\begin{itemize}
\item
 Let  $u$ be an irreducible  PFH-HF  curve.
By definition and the adjunction formula (Lemma 5.6.3  of \cite{VPK}), we have
\begin{equation*}
		\begin{split}
		 &I(u) - \operatorname{ind} u  =2\delta(u) - w_{\tau}(u) +   CZ^{ech}_{\tau}(\alpha) -CZ^{ind}_{\tau}(\alpha), \\
		\end{split}
	\end{equation*}	
where $w_{\tau}(u)$ is  the total writhe of the
braids $u(\dot{F})\cap \{s\} \times Y_{\varphi_H}$ for $s\gg  1$  with respect to $\tau$. See Definition 2.8 of \cite{H2} for its definition.

 By Lemma 4.20  of \cite{H2}, we have $   CZ^{ech}_{\tau}(\alpha)- CZ^{ind}_{\tau}(\alpha) \ge w_{\tau}(u)  $ and equality holds only if $u$ satisfies the ECH partition condition.  This implies the first bullet.

\item
To prove the second statement, without loss of generality, assume that $u=u_0\cup u_1$ has two distinct irreducible components, where   $u_i \in \mathcal{M}^J(\alpha_i, \mathbf{y}_i)$. By Lemma 8.5 of \cite{H1},
\begin{equation}
Q_{\tau}(u_0, u_1) = \# (u_0 \cap u_1) + l_{\tau}(u_0, u_1),
\end{equation}
 where  $l_{\tau}(u_0, u_1)$ is  the total  linking number  of the
braids $u_0(\dot{F}_0)\cap \{s\} \times Y_{\varphi_H}$ and $u_1(\dot{F}_1)\cap \{s\} \times Y_{\varphi_H}$ for $s\gg  1$  with respect to $\tau$  (see Definition 2.9 of \cite{H2}).
 Because  Chern number and Maslov  index are  additive and  the  relative intersection number is quadratic, we have
\begin{equation*}
		\begin{split}
		 &I(u_0 \cup u_1)  - I(u_0) -I(u_1)    =2\#(u_0 \cap u_1)  + 2l_{\tau}(u_0, u_1)  + CZ_{\tau}^{ech}(\alpha)- CZ_{\tau}^{ech}(\alpha_0)  - CZ_{\tau}^{ech}(\alpha_1)       \\
		\end{split}
	\end{equation*}	
By Lemma 4.17 and  Lemma 5.10 of \cite{H2}, we have $ 2l_{\tau}(u_0, u_1)+ CZ_{\tau}^{ech}(\alpha) \ge CZ_{\tau}^{ech}(\alpha_0)  + CZ_{\tau}^{ech}(\alpha_1)   $.  Then we  get the second bullet.

\item
We now show that the ECH index is nonnegative when $u$ is Fredholm regular. By the first bullet and $u$ is Fredholm regular, we have $I(u_a) \ge \operatorname{ind} u_a +2\delta(u_a) \ge 0$.  By the intersection positivity of holomorphic curves, we have $\#(u_a \cap u_b) \ge0$.  Therefore, the third bullet follows from the second bullet.

\item
Let $u$ be a $\tau$-trivial representative of $\mathcal{Z}$. For $1\le i \le k$, using the construction in Step 1 of  Lemma  \ref{lem6},  we modify an end of $u$, denote the result by $u'$ such that $[u']= \mathcal{Z} +[B_i]$.  The computations in Step 2 of Lemma \ref{lem6} show that $I(u') =I(u) +2$.  Therefore, each $[B_i]$ ($1\le i \le k$)  contributes $2$ to the ECH index.

By definition,  we have
\begin{equation*}
\begin{split}
I(\mathcal{Z} + m [\Sigma] + [S]) &=  c_1(TW \vert_{\mathcal{Z}}, \tau  )  + c_1(TW \vert_{m[\Sigma] + [S]}) +Q_{\tau}(\mathcal{Z}) \\
&+ 2 Q_{\tau}(\mathcal{Z},  m[\Sigma] + [S])  + (m[\Sigma] + [S]) \cdot (m[\Sigma] + [S]) +  CZ^{ech}_{\tau}(\alpha)+ \mu_{\tau}(\mathcal{Z})\\
&= I(\mathcal{Z}) + m \chi(\Sigma)+ 2dm  \\
&=I(\mathcal{Z}) + 2m(d-g+1).
\end{split}
\end{equation*}
 By using $\sum_{i=1}^{k+1} [B_i] =[\Sigma]$ and the trick in (\ref{eq64}), we know that adding $[B_{k+1}]$ to $\mathcal{Z}$ increasing the ECH index by $2$.
\end{itemize}
\end{proof}

\begin{comment}
	According to Theorem 5.6.9 of \cite{VPK}, the ECH inequality still holds for the  PFH-HF curves, i.e., we have
	\begin{equation}
		I(u) \ge ind u + 2 \delta(u).
	\end{equation}
	Moreover,   equality holds only if the positive ends of $u $ satisfy the ECH partition conditions. By the above inequality, we have $I(u)  \ge 0$ for a generic $J$.
\end{comment}

\paragraph{$J_0$ index} Again, the concept of $J_0$ index can be generalized to the closed-open setting.  Given $\mathcal{Z} \in H_2(W, \alpha, \textbf{y})$, the $J_0$ index is
	\begin{equation}
		J_0(\mathcal{Z}) : = -c_1(TW \vert_{\mathcal{Z}}, \tau) +Q_{\tau}(\mathcal{Z}) +CZ^{J_0}_{\tau}(\alpha),
	\end{equation}
where $CZ^{J_0}_{\tau}(\alpha) = \sum_i \sum_{p=1}^{m_i-1} CZ_{\tau}(\alpha_i^{p})$.

The following lemma summarizes the properties of the $J_0$ index.
\begin{lemma}\label{lem29}
The $J_0$ index satisfies the following properties:
\begin{itemize}
\item
Let $u \in \mathcal{M}^J(\alpha, \mathbf{y} )$ be an irreducible PFH-HF curve. Then we have
\begin{equation*}
		\begin{split}
       J_0(u)\ge   2(g(F) -1    + \delta(u)) + \#\partial F  + |\alpha|,
		\end{split}
	\end{equation*}	
where $|\alpha| $ is a quantity satisfying $|\alpha| \ge 1$ provided that $\alpha$ is nonempty (see Definition 6.4 of \cite{H2}).

 \item
 If $u=\cup_a u_a$  is a    PFH-HF curve consisting  of several  (distinct) irreducible components, then
\begin{equation*}
		\begin{split}
		  J_0(u) \ge \sum_a J_0(u_a) + \sum_{a \ne b} 2\# (u_a \cap u_b).
		\end{split}
	\end{equation*}
\item
Let $u$ be a  PFH-HF curve.  Then $J_0(u) \ge 0$.

\item
		Let $\mathcal{Z}', \mathcal{Z} \in H_2(W,  \alpha, \mathbf{y})$  be relative  homology classes  such that $$\mathcal{Z}'-\mathcal{Z}=m[\Sigma]+\sum_{i=1}^{k+1} c_i[B_i] + [S], $$ where $[S] \in H_1(S^1, \mathbb{Z}) \otimes H_1(\Sigma, \mathbb{Z})$.   Then  we have
		\begin{equation}
\begin{split}
J_0(\mathcal{Z}')=J_0(\mathcal{Z}) + 2c_{k+1} (d+g-1)+ 2m(d+g-1).
\end{split}
		\end{equation}
 \end{itemize}
\end{lemma}
\begin{proof}
 \begin{itemize}
 \item
By definition and adjunction formula (Lemma 5.6.3  of \cite{VPK}), we obtain
\begin{equation*}
		\begin{split}
		 J_0(u)  &=-\chi(\dot{F})  - w_{\tau}(u)  + 2\delta(u) + CZ^{J_0}_{\tau}(\alpha) \\
&=2(g(F)-1 + \delta(u)) +\#\partial F + \#\Gamma  -  w_{\tau}(u)     + CZ^{J_0}_{\tau}(\alpha),
		\end{split}
	\end{equation*}	
where $\Gamma$ is the set of interior punctures.
By (6.2) of \cite{H2}, we have  $$  \#\Gamma -w_{\tau}(u)  + CZ^{J_0}_{\tau}(\alpha) \ge |\alpha|.$$ Hence, the inequality in the first statement holds.

\item
Assume that $u=u_0\cup u_1$ has two distinct irreducible components, where   $u_i \in \mathcal{M}^J( \alpha_i, \mathbf{y}_i)$.
 Because  Chern number and Maslov  index are  additive,  the  relative intersection number is quadratic, we have
\begin{equation*}
		\begin{split}
		 &J(u_0 \cup u_1)  - J(u_0) -J_0(u_1)    =2\#(u_0 \cap u_1) +2l_{\tau}(u_0, u_1) + CZ_{\tau}^{J_0}(\alpha)  - CZ_{\tau}^{J_0}(\alpha_0)  - CZ^{J_0}_{\tau}(\alpha_1)     \\
		\end{split}
	\end{equation*}	
By Lemma 4.17 and Lemma 6.15 of \cite{H2}, we have $   2l_{\tau}(u_0, u_1)+ CZ_{\tau}^{J_0}(\alpha)  \ge CZ_{\tau}^{J_0}(\alpha_0)  + CZ_{\tau}^{J_0}(\alpha_1) $.  Then we  get the second bullet.

\item
Because a PFH-HF  curve at least one  boundary and $\alpha_a$ are  not empty,  by the first bullet, we have $ J_0(u_a)\ge 0.$ Then $J_0(u) \ge  0$ follows from second bullet and intersection positivity of holomorphic curves.

\item
The proof of the fourth statement is just the same as those in Lemma \ref{lem3}.
 \end{itemize}
\end{proof}

\begin{comment}
Let $u$ be an irreducible  PFH-HF curve. Then
	\begin{equation*}
		J_0(u)\ge 2(g(F) -1    + \delta(u)) + \#\partial F + |\gamma|.
	\end{equation*}
In particular, $J_0(u) \ge 0$. Moreover,  $J_0(u) \ge 0$  still holds if  $u$ has several irreducible components.
\end{lemma}
\begin{proof}
We first assume that $u$ is irreducible. By the Relative adjunction formula (Lemma 5.6.3 of \cite{VPK}), we have
\begin{equation*}
		J_0(u)= 2g(F) -2 + \#\partial F + \# \Gamma_+  -w_{\tau}(u) + \mu_{\tau}^{J_0}(\gamma)+ 2\delta(u),
	\end{equation*}
where $\Gamma_+$ is the positive punctures of $u$. By Equation (6.2) in \cite{H2}, we have $ \# \Gamma_+  -w_{\tau}(u) + \mu_{\tau}^{J_0}(\gamma) \ge |\gamma|. $

If $u$ has multiple components, without loss of generality, $u=u_0\cup u_1$, then by the same argument in \cite{H2}, we have
\begin{equation*}
		J_0(u) \ge J_0(u_0) + J_0(u_1) + 2\#(u_0\cap u_1) + E,
	\end{equation*}
 where $E \ge 0$ is  the quantities defined in \cite{H2}. % What we need to know is that $|\gamma|=|\gamma_0| + |\gamma_1| -E$.
 Hence,  the inequality  still holds and we  have $J_0(u) \ge 0$.

\end{proof}
\end{comment}

	\subsection{Moduli space of PFH-HF curves}
	Fix  $J \in \mathcal{J}_{tame}(W, \Omega_{\varphi_H})$. Let $\mathcal{M}^J(\alpha, \mathbf{y}, \mathcal{Z})$   be the moduli space of $J$-holomorphic PFH-HF
 curves with relative homology class $\mathcal{Z}.$ Apart from  the ECH inequality, the ECH index and Fredholm index also satisfy the following relation. It is an analogy of Proposition 1.6 (c) in \cite{H1}.
	% which come from the restrciction of the admissible almost complex stru ctures on $\mathbb{R} \times Y$. Note that the natural projection $\pi: W \to B$ is still complex linear with respect to $J  \in \mathcal{J}_W$. Fix a generic $J \in \mathcal{J}_W$.

	\begin{lemma} \label{lem20}
	Let  $\alpha=\{(\alpha_i, m_i)\}$ is a  PFH generator.	Let $u \in \mathcal{M}^J(\alpha, \mathbf{y}, \mathcal{Z})$ be a PFH-HF curve.  Then
		\begin{equation*}
			I(u)=\operatorname{ind} u \mod 2.
		\end{equation*}
	\end{lemma}
	\begin{proof}
		By the relative adjunction formula (Lemma 5.6.3 of \cite{VPK}), we have
		\begin{equation*}
			I(u) - \operatorname{ind} u = 2\delta(u) - w_{\tau}(u) +CZ_{\tau}^{ech}(\alpha) -CZ_{\tau}^{ind}(\alpha),
		\end{equation*}
		where $ w_{\tau}(u) $ is  the writhe number defined in Section 3.1 of  \cite{H1}.% and $\mu_{\tau}^{ech}(\gamma)$, $\mu_{\tau}^{indx}(\gamma)$ denote the combination of the  Conley--Zehnder index in formulas (\ref{eq28}), (\ref{eq29}) respectively.
		
		We may assume that the positive ends of $u$ are  $\tau$-trivial in the sense of Definition 2.3 of \cite{H1}. Otherwise,   we modify $u$ such that the positive ends are $\tau$-trivial by gluing
 branched covers of trivial cylinders. Note that such a modification does note change the relative homology class.
		
		Suppose that $u$ has positive ends at a simple periodic orbit $\alpha_i$ with total multiplicity $m_i$.
		Because the positive ends of $u$ are  $\tau$-trivial, $u$ has $m_i$ ends at $\alpha_i$ and  $w_{\tau}(u)=0$.  If $\alpha_i$ is hyperbolic, then $ CZ_{\tau}^{ech}(\alpha_i) =CZ_{\tau}^{ind}(\alpha_i)$ because  $m_i=1$. If $\alpha_i$ is elliptic, then
		$$ CZ_{\tau}^{ech}(\alpha_i) -CZ_{\tau}^{ind}(\alpha_i) =\sum_{q=1}^{m_i} \left( CZ_{\tau}(\alpha^q_i) -CZ_{\tau}(\alpha_i) \right)= even.   $$
		
	\end{proof}

The following lemma is the same as Lemma \ref{lem23}. It asserts that for most of almost
 complex structures the transversality holds for simple PFH-HF curve. These almost complex structures are called generic.	
\begin{lemma} \label{lem33}
There is a Baire subset $J^{reg}_{tame}(W, \Omega) \subset  J_{tame}(W, \Omega)$
 such that for   $J \in J^{reg}_{tame}(W, \Omega)$, any PFH-HF curve $u  \in  \mathcal{M}^J(\alpha,\mathbf{y},\mathcal{Z})$ is Fredholm regular. In particular, we have $ \operatorname{ind}  u  \ge  0$.	
\end{lemma}
\begin{proof}
As mentioned at the beginning, the PFH-HF curve in $W$ has at least one positive
 end and one negative end thanks to the fibration structure of $W$. The appear of the HF-ends ensure that the PFH-HF curves  are  simple. Then we can obtain the Baire subset $J^{reg}_{tame}(W, \Omega)$ by the standard argument in Lemma 9.12 of  \cite{H1}.
\end{proof}

	\begin{lemma} \label{lem3}
		Assume that $I(\mathcal{Z})=0$ and $J \in \mathcal{J}^{reg}_{tame}(W, \Omega_{\varphi_H})$ is a generic  almost complex structure. Then  the moduli space $\mathcal{M}^J(\alpha, \mathbf{y}, \mathcal{Z})$   is a compact  zero-dimensional manifold. In other words,  $\mathcal{M}^J(\alpha, \mathbf{y}, \mathcal{Z})$   is a set of finite points.
	\end{lemma}
	\begin{proof}
Let $\{u_n\}_{n=1}^{\infty} \subset \mathcal{M}^J(\alpha, \mathbf{y}, \mathcal{Z})$  be a sequence of PFH-HF curves. Since the relative
 homology class $\mathcal{Z}$ is fixed, Lemma  \ref{lem29}  implies that we can fix the topological type of $u_n$ by passing to a subsequence.
 By the SFT compactness \cite{FYHKE},   $ \{u_n\}_{n=1}^{\infty}$ converges to a broken holomorphic curve $ \mathbf{u}=\{u^{N_+},..., u^0, ...,u^{N_-}\}$, where  $u^i$ are holomorphic curves in $\mathbb{R} \times Y_{\varphi_H}$ for $i>0$, $u^0$ is a curve in $W$ and $u^j$ are curves in $M$ for $j<0$. Also, the positive ends of $u^j$ agree with the negative ends of $u^{j+1}$. By the
 same argument in Lemma \ref{lem31}, the noncompact irreducible components in $M$ and $W$
 are still HF curves and PFH-HF curves respectively.
 The SFT-compactness in closed-open cobordisms  can be found in 	Section 6.1 and  Section 7.3  of  \cite{VPK}. The broken phenomenon is the same as the one described in Section \ref{Section2.3}.

By the ECH inequality (\ref{eq59}) and the fact \ref{fact2},  the curves in $\mathcal{M}^J(\alpha, \textbf{y}, \mathcal{Z})$  satisfy  $I(u_n) =  \operatorname{ind} u_n =0$ and the ECH partition conditions.  Therefore, the positive ends of $u^{N_+}$ also satisfy the  ECH partition conditions.
		
		Similar as in Lemma \ref{lem7}, by the fact that $\pi_W: W \to B$ is complex linear and the open mapping theorem,   the compact irreducible components of $u^0$ (bubbles) lie inside the fiber of $W$.
\begin{comment}
The bubbles arise in the following two ways:
		\begin{enumerate}
			\item
			Since $\{L_{\Lambda_H}^i\}_{i=1}^d$ are pairwise disjoint, if an irreducible component $v$ of the bubbles  comes from pinching an arc $a$, then the end points of $a$ must lie inside the same component of $L_{\Lambda_H}$.  Then $v$ is a branched covering  of  $\mathcal{P}_{\partial}^{\tau}(\mathcal{B}_i)$ for some $\tau$ and $i$, where $\mathcal{P}_{\partial B}^{\tau} \vert_{\tau \in \mathbb{R}}$ is the parallel transport of the symplectic connection and $\mathcal{B}_i$ is one of the following surfaces:
			\begin{enumerate}
				\item
				$\mathcal{B}_i =\{-3\} \times \{1\} \times B_i$ or $\{-3\} \times \{1\} \times  B_i^c$ for some $1\le i \le k$.
				\item
				$\mathcal{B}_i = \{-3\} \times \{1\} \times (\Sigma -\Lambda_i)$ for some $k+1\le  i\le d$. In this case, the homology class of $\mathcal{B}_i$ is $[\Sigma]$.
			\end{enumerate}
			\item
			If $v$ comes from pinching an interior simple closed curve, then the homology class of $v$ is $m[\Sigma]$ for some $m \ge 0$.
		\end{enumerate}
\end{comment}
 Write
  \begin{equation*}
[u^0]=[u^{0}_{\star}] + m_0[\Sigma]+ \sum_{j=1}^{k+1} c_{0j} [B_j]
\end{equation*}
where $u^{0}_{\star}$ is a PFH-HF curve and the rest are homology classes of the bubbles.   According to Lemma \ref{lem3}, we have
		\begin{equation} \label{eq26}
			I([u^0])=I(u^{0}_{\star}) + 2m_0(k+1) + \sum_{j=1}^{k+1}2c_{0j}.
		\end{equation}
Let $v \subset \pi_W^{-1}(\tau)$ be an irreducible component of bubble, where  $\tau \in \partial B$.  For $z_i \in \mathring{B}_i$,  define
$$n_{z_i}(v): =\#(v \cap \Psi_H'(B \times \{z_i\})), $$
where $\Psi_H'$ is the trivialization in Remark \ref{remark5}. Since  $\partial  \Psi_H'(B \times \{z_i\}) \cap L_{\Lambda_H} = \emptyset$,  $n_{z_i}(v)$ is   well defined and it only depends on the homology
 class  of $[v] =\sum_{i=1}^{k+1} a_i [B_i^{\tau}] \in H_2(\Sigma, L_{\Lambda_H} \cap (\pi_W)^{-1}(\tau), \mathbb{Z})$.  By definition, $n_{z_i}([B_j^{\tau}]) =\delta_{ij}$ and  $ \Psi_H'(B \times \{z_i\})$ intersects the fibers positively transversally.  Since $v$ is holomorphic,
 the orientation of $v$ is the same as the fiber.  As a result,   $ \Psi_H'(B \times \{z_i\})$  intersects $v$
 positively transversally. As in the proof of Lemma \ref{lem1}, we have $a_i = n_{z_i}(v) \ge 0$. Thus, $m_0, c_{0j} \ge 0$.

Let $u_{\star}^i$ denote the component of $u^i$ by removing the bubbles. By the ECH inequalities (Lemma 9.5 of \cite{H1}, (\ref{eq59}) and Theorem \ref{thm4}), we have $I(u^i_{\star} ) \ge 0$.
By (\ref{eq26}) and  Lemma \ref{lem6}, we have
		\begin{equation}\label{eq3}
			\begin{split}
				0=I(\mathcal{Z}) =&I(\mathbf{u})=I(u^{N_+}) +,...,+ I(u^0) +,...,+   I(u^{N_-})\\
				= & \sum_{i\le 0} \left( I(u^{i}_{\star})  + 2\sum_{i=-N_-}^0 \sum_{j=1}^{k+1} (c_{ij} +c_{ij}')\right) + \sum_{i=-N_-}^{N_+} 2m_i(k+1).
			\end{split}
		\end{equation}
By the proof of Lemma \ref{lem6}, we have $c_{ij}, c_{ij}', m_i \ge0$. 	Therefore,  we have  $c_{ij}=c_{ij}'=m_i=0 $; in other words, no fiber bubbles appear.  The inequality (\ref{eq3}) also implies that $I(u^i)=0$ for each  $i$.   The Fredholm  indices of $u^i$ are adding  to zero. By \ref{fact3},  we have $\operatorname{ind}  u^i=0$ for each $i$.  For $i<0$, $u^i$ are trivial strips which are ruled out in the holomorphic buildings. For $i>0$, $u^i$ are connectors with zero Fredholm index.  By  fact \ref{fact1} and  the ECH partition conditions, they must be trivial.  Thus,  these curves   are also  ruled out.  In sum, $\mathcal{M}^J(\alpha, \textbf{y}, \mathcal{Z})$   is compact. By Lemma \ref{lem33},  $\mathcal{M}^J(\alpha, \textbf{y}, \mathcal{Z})$   is a finite set of points.
		
	\end{proof}
	
	Fix a reference class $\mathcal{Z}_{ref} \in H_2(W, \gamma_H^{\mathbf{x}}, \textbf{x}_H)$ represented by  $\Psi_H'(B \times \mathbf{x})$,  where $\Psi_H'$ is defined in Remark \ref{remark5}. We define a homomorphism
	\begin{equation*}
	\widetilde{\mathcal{CO}}(\underline{\Lambda}, H)_J : \widetilde{PFC}(\Sigma, \varphi_H, d) \to CF(\Sigma, \varphi_H(\underline{\Lambda}),    \underline{\Lambda})
	\end{equation*} by
	\begin{equation*}
\widetilde{\mathcal{CO}}(\underline{\Lambda}, H)_J  (\alpha,  [Z]) :=\sum_{(\mathbf{y}, [A]), I(\mathcal{Z})=0} \# \mathcal{M}^J(\alpha, \textbf{y}, \mathcal{Z})(\textbf{y}, [A]),
	\end{equation*}
	where  $A$ is determined by the relation $Z\#\mathcal{Z}_{ref}\#A =\mathcal{Z}$.
	For a generic $J \in J^{reg}_{tame}(W, \Omega_{\varphi_H})$, Lemma \ref{lem3} implies that  the above definition make sense.

	\begin{lemma}
		$\widetilde{\mathcal{CO}}(\underline{\Lambda}, H)_J $  is a chain map, i.e., $d\circ \widetilde{\mathcal{CO}} (\underline{\Lambda}, H)_J =\widetilde{\mathcal{CO}}(\underline{\Lambda}, H)_J \circ \partial.$
	\end{lemma}
	\begin{proof}
		Consider the moduli space  $\mathcal{M}^J(\alpha, \textbf{y}, \mathcal{Z})$ with $I(\mathcal{\mathcal{Z}})=1$, where $\alpha$ is a PFH generator.  % Let $ u=\{u^{N_+} \cdots u^0, \cdots u^{N_-}\}$
		By the ECH index inequality (\ref{eq59}) and Lemma \ref{lem20},  the curves in   $\mathcal{M}^J(\alpha, \textbf{y}, \mathcal{Z})$  have Fredholm index $\operatorname{ind}=1$. The fact \ref{fact2} implies  that the curves  satisfy the ECH partition conditions.  We can repeat the argument in Lemma  \ref{lem3}  and get a broken
 holomorphic curve $\mathbf{u}$ from a sequence of curves in   $\mathcal{M}^J(\alpha, \textbf{y}, \mathcal{Z})$. The inequality (\ref{eq3})
 still holds. Thus,  the bubbles can be ruled out.

		Suppose that $\mathbf{u} =\{ u^{-N_-}, ..., u^0, ...,u^{N_+} \}$ is broken. Then $\sum_i I(u^i) = \sum_i  \operatorname{ind}  u^i=1$.  By the ECH inequalities (Lemma 9.5 of \cite{H1}, (\ref{eq59}, Theorem \ref{thm4}),   $I(u^i) \ge 0$ for $-N_- \le i\le N_+$. The same argument in Lemma \ref{lem3} shows that   $\mathbf{u}$ either consists of
		\begin{itemize}
			\item
			A PFH-curve with $I=\operatorname{ind}=1$ in the top level;

			\item
			The $\operatorname{ind} =0 $ connectors in the middle levels;

 \item
			A PFH-HF curve in $W$ with $I=\operatorname{ind}=0$;

			\item
			There are no negative level;
			
		\end{itemize}
		or
		
		\begin{itemize}
			\item
			There is no positive level;
			
			\item
			A PFH-HF curve  in $W$ with $I=\operatorname{ind}=0$;
			
			\item
			A HF-curve with $I=1$ in the negative level.
		\end{itemize}
		The conclusion follows from Hucthings-Taubes’s obstruction gluing analysis \cite{HT1, HT2}. A
 good summary of the gluing analysis can be found in Section 6.5 of  \cite{VPK}.
		
	\end{proof}

	\subsection{Homotopy invariance}
	Even we only define the closed-open morphism for a tuple  $(W_{\varphi_H}, \Omega_{\varphi_H}, L_{\Lambda_H}, J)$, the construction still holds for a slightly general situation.   If we deform the tuple   $(\Omega_{\varphi_H}, L_{\Lambda_H}, J)$ over a compact subset of $W_{\varphi_H}$ such that the almost complex structure is still generic, then the closed-open morphism  is still well defined by counting $I=0$ PFH-HF curves as before.  The purpose of this subsection is to show that this variation does not change
 the closed-open morphism at the  homological level.

	Let $\{(\Omega_{\tau}, L_{\tau}, J_{\tau}) \}_{\tau \in [0,1]}$ be a family of triples  on $W$ such that
\begin{enumerate} [label=\textbf{B.\arabic*}]
		\item   \label{B1}
		$(\Omega_{\tau}, L_{\tau}, J_{\tau}) \vert_{|s| \ge R_0}=( \Omega_{\varphi_H}, L_{\Lambda_H}, J )$.
		\item \label{B2}
		For each $\tau$, $L_{\tau} \subset \pi_W^{-1}(\partial B)$ is a $\Omega_{\tau}$-Lagrangian submanifold over $\pi^{-1}(\partial B)$ and $L_{\tau} \cap (s,t) \times \Sigma$ is a  link   which  is Hamiltonian isotropic to $\underline{\Lambda}$.
		\item\label{B3}
		For each $\tau$,  $J_{\tau}$ is an $\Omega_{\tau}$-tame almost complex structure such that the projection $\pi_W: W  \to B$ is complex linear with respect to $(J_{\tau},  j_B)$. %Also, $J_{\tau}$ agrees with admissible almost complex structures over the ends.
	\end{enumerate}
	By the previous argument, if $J_{\tau}$ is generic (the curves in $\mathcal{M}^{J_{\tau}}(\alpha, \mathbf{y})$ is Fredholm regular), then we  can   define a homomorphism   $\widetilde{\mathcal{CO}} (W, \Omega_{\tau}, {L_{\tau}})_{J_{\tau}} $ by counting PFH-HF curves with $I=0$.

Fix $J \in \mathcal{J}^{reg}_{tame}(W, \Omega_{\varphi_H})$. 	Let $\mathcal{J}_{path}(W, \Omega_{\tau \in [0,1]}, J)$  denote the set of path of almost
 complex structures satisfying \ref{B1} and \ref{B3}. Consider the moduli space $$\mathcal{M}^{\{J_{\tau}\}}(\alpha, \textbf{y}, \mathcal{Z}): =  \{ (u, \tau) \vert u \in \mathcal{M}^{J_{\tau}} (\alpha, \textbf{y}, \mathcal{Z})\}. $$  Again, thanks to the appearance of HF-ends, a curve in  $\mathcal{M}^{\{J_{\tau}\}}(\alpha, \mathbf{y}, \mathcal{Z}) $ cannot be multiply covered. By the standard Sard-Smale argument, there is a Baire subset $\mathcal{J}^{reg}_{path}(W, \Omega_{\tau \in [0,1]}, J) \subset \mathcal{J}_{path}(W, \Omega_{\tau \in [0,1]}, J)$ such that for  $\{J_{\tau}\}_{\tau \in [0,1]} \in \mathcal{J}^{reg}_{path}(W, \Omega_{\tau \in [0,1]}, J) $, the moduli space $\mathcal{M}^{\{J_{\tau}\}}(\alpha, \textbf{y}, \mathcal{Z})$  is a smooth manifold of expected dimension. We call
$\{J_{\tau}\}_{\tau \in [0,1]}$ a generic family of almost complex structures.

	Before we go on, note that   it is a little  misleading to denote the set of relative homology classes by $H_2(W, \alpha, \mathbf{y})$ because its definition also depend on the Lagrangian boundary.  Now the Lagrangian submanifolds $L_{\tau}$ changes  as  $\tau$ varies,   the group  $H_2(W, \alpha, \mathbf{y})$ should change as well.  But the family $\{(W, L_{\tau})\}_{\tau \in [0,1]}$ also induces an isomorphism between  the groups    $H_2(W, \alpha, \mathbf{y})$ defined by different $\tau$.  Thus, here we use $\tau$-independent notation to denote  $H_2(W, \alpha, \mathbf{y})$ and the relative homology class. % The relative homology class $\mathcal{Z}$ should also change   as  $\tau$ varies.

	\begin{lemma} \label{lem4}
	Suppose that $\{J_{\tau}\}_{\tau \in [0,1]}$ is a generic family of almost complex structures. 	Let $\{(u_n, \tau_n)\}_{n=1}^{\infty}$ be a sequence  of PFH-HF curves in  $\mathcal{M}^{\{J_{\tau}\}}(\alpha, \mathbf{y}, \mathcal{Z}) $.  Then after passing to a subsequence,   $\{(u_n, \tau_n)\}_{n=1}^{\infty}$  converges to a broken holomorphic curve $\mathbf{u}$ in  the sense of SFT.  Moreover, we have the
 following statements:
		\begin{enumerate}
			\item
			If $I(\mathcal{Z}) =-1$, then $\mathbf{u}$ is unbroken.
			\item
			Assume that $\alpha$ is a PFH generator. If $I(\mathcal{Z})=0$ and $\tau \to \tau_*$, then
			\begin{enumerate}
				\item
				If $J_{\tau_*}$ is generic, then  $\mathbf{u}$  is unbroken, i.e.,  $\mathbf{u} \in \mathcal{M}^{J_{\tau_*}} (\alpha, \mathbf{y}, \mathcal{Z})$.
				\item
				If $J_{\tau_*}$ is not generic, then  we have the following two possibilities:
				\begin{itemize}
					\item
					$\mathbf{u} =\{ u^{+}, u^1,..., u^k, u^{-} \}$ , where  $u^+$ is a $J$-holomorphic curve in $\mathbb{R} \times Y_{\varphi_H}$ with $I(u^+) =1$, $u^i $ are  $ \operatorname{ind}=0$ connectors and $u^-$ is a $J_{\tau_*}$-holomorphic PFH-HF curve in $W$ with $I=-1$.
					
					\item
					$\mathbf{u}=\{ u^{+},  u^{-} \}$ , where  $u^+$ is  a  $J_{\tau_*}$-holomorphic PFH-HF curve in $W $ with $I=-1$, and $u^-$ is a  $J$-holomorphic HF-curve with $I=1$.
				\end{itemize}
			\end{enumerate}
		\end{enumerate}
	\end{lemma}
	\begin{proof}
		Let  $u \in \mathcal{M}^{\{J_{\tau}\}}(\gamma, \textbf{y}, \mathcal{Z})$ be a PFH-HF curve  with $I=-1$ or $ I=0$.    The ECH inequality (it still holds even $J_{\tau} $ is not generic) implies that $$-1 \le  \operatorname{ind} u +2 \delta(u) \le I(\mathcal{Z})=-1 \mbox{\ or \ } 0.$$
		If $I(\mathcal{Z}) =-1$, then we have $\operatorname{ind} u=I(u)=-1$. %Moreover, $u$ satisfies the ECH partition conditions.
		If $I(\mathcal{Z}) =0$ and $\alpha$ is a PFH generator, then $\operatorname{ind} u = 0$. Here $ \operatorname{ind} u \ne -1 $ since $\operatorname{ind} u =I(u) \mod 2 $ by Lemma \ref{lem20}.  In any case, the PFH ends of $u$ satisfy the ECH partition conditions and $I(u) = \operatorname{ind} u$.

		The convergence of $\{(u_n, \tau_n)\}_{n=1}^{\infty}$ is
 proved by the same argument as in Lemma  \ref{lem3}.
 The limit is dented by  $ \mathbf{u}=\{  u^{-N_-},..., u^{0},  ..., u^{N_+}\}$,      where $u^i$ are curves in $\mathbb{R} \times Y_{\varphi_H}$ for $i>0$,   and $u^{0}$ is a $J_{\tau_*}$-holomorphic curve in $W$ with Lagrangian boundary conditions $L_{\tau_*}$, and $u^{j}$  are curves in $M$ with Lagrangian boundary conditions for $j<0$.    Moreover, the total ECH index of $\mathbf{u}$ is $I(\mathcal{Z})$.
		
		Write $u^i=u^i_{\star}  \cup u_b^i$, where $u^{0}_{\star}$ is a PFH-HF curve, $u^i_{\star}$ is a PFH-curve for $i>0$  and   $u^i_{\star}$  is  an  HF curve for $i<0$,  and $u_b^i$ consist  of the bubbles in fibers.  Since the path $\{J_{\tau}\}_{\tau \in [0,1]}$ is generic, we have $I(u_{\star}^0) \ge  \operatorname{ind}  u_{\star}^0 \ge -1$.   Also, we have $I(u^i_{\star}) \ge 0$ for $i \ne 0$.   By (\ref{eq3}), each bubble   contributes at least $2$  to the ECH index.   Therefore, when $I(\mathcal{Z}) =0$ or $-1$, the bubbles can be ruled out, i.e., $u^i=u^i_{\star}.$
		
		Suppose that $I(\mathcal{Z}) =-1$.  Then
\begin{equation}\label{eq60}
 -1= I(\mathcal{Z}) =I(\mathbf{u}) =\sum_{i=-N_-}^{N_+} I(u^i).
\end{equation}
 Since the $i\ne 0$ levels have nonnegative ECH index and  $I(u^0) \ge -1$, we must have  $I(u^0) = -1$ and $I(u^i) = 0$ for $i \ne 0$.
Hence, $u^i$ are connectors for $i \ge 1$ and $ u^j$ are trivial strips for $j\le -1$.   Since the Fredholm index is also additivity, i.e.,
 \begin{equation}\label{eq61}
 -1= \operatorname{ind} u_n =\sum_{i=-N_-}^{N_+} \operatorname{ind}  u^i,
\end{equation}
 we have $ \operatorname{ind} u^0 =-1$ and $\operatorname{ind}  u^i=0$ for $i \ne 0$.  The connectors can be ruled out by the ECH partition conditions and the fact \ref{fact1}.
		
		Consider the case that  $I(\mathcal{Z}) =0$.    First, we rule out the bubbles by (\ref{eq3}). Equations
 (\ref{eq60}), (\ref{eq61})  still hold, but the difference is $I(\mathcal{Z}) = \operatorname{ind} u_n =0$.  By the ECH
 inequalities, we have $I(u^0) \ge \operatorname{ind} u^0 \ge -1$ and $I(u^i) \ge \operatorname{ind}  u^i \ge 0$ for $i \ne 0$. Combine these
 inequalities with (\ref{eq60});  we know that that either $\mathbf{u}$ is either unbroken or there is a positive
 or negative level with $I = 1$, $I(u^0) = -1$ and all other levels have $I = 0$.   Suppose that $I(u^{k})=1 $ for some $k > 0$.    Then by Proposition 3.7 of \cite{H4} and (\ref{eq3}), $\operatorname{ind} u^k= 1$
 and all the other positive levels are connectors with $\operatorname{ind}= 0$.   The ECH partition conditions  and the fact \ref{fact1}  imply that $u^j$ are trivial connectors for $j >k$, which are ruled out.  The negative levels are trivial strips, which are ruled out as well.    Then we get the first bullet in the statement (b).
  In the case that $I(u^k) =1 $ for some $k<0$, then $k=-1$ and the positive levels  must be trivial connectors by the ECH partition conditions and the fact \ref{fact1}. We obtain the second bullet in statement (b).
	\end{proof}

	\begin{lemma} \label{lem11}
		Let $\{(\Omega_{\tau}, L_{\tau}, J_{\tau}) \}_{\tau \in [0,1]}$ be a family of  triples  on $W$ satisfying \ref{B1}, \ref{B2} and \ref{B3}.   Suppose that the family of almost complex structures  $\{J_{\tau}\}_{\tau \in [0, 1]}$ is generic. Then there is a chain homotopy
		\begin{equation*}
			K: \widetilde{PFC}(\Sigma, \varphi_H, d)_J \to CF(\Sigma, \varphi_H(\underline{\Lambda}), \underline{\Lambda} )_J
		\end{equation*}
		such that
		\begin{equation*}
		\widetilde{\mathcal{CO}} (W, \Omega_1, L_1)_{J_1} - \widetilde{\mathcal{CO}}(W,\Omega_0, L_0)_{J_0} = K \circ \partial_J + d_J \circ K.
		\end{equation*}
		 In particular, we have $\widetilde{\mathcal{CO}}(W, \Omega_1, L_1)_{J_1} = \widetilde{\mathcal{CO}}(W,\Omega_0, L_0)_{J_0}$ at the homological level.
	\end{lemma}
	\begin{proof}
		Define the  homomorphism $K$ by
		\begin{equation*}
			K(\alpha, Z) :=\sum_{\textbf{y},  I(\mathcal{Z}) =-1} \# \mathcal{M}^{\{J_{\tau}\}}(\alpha, \textbf{y}, \mathcal{Z}) (\textbf{y}, [A]),
		\end{equation*}
		where $A$ is determined by the relation $Z\#\mathcal{Z}_{ref}\#A =\mathcal{Z}$ and $\mathcal{Z}_{ref} =[\Psi_H'(B \times \mathbf{x})]$. By  Lemma \ref{lem4}, the  homomorphism $K$ is well defined.
		
		%Let $\gamma$ be an ECH generator.  The broken curves arising from the moduli space
		By Lemma \ref{lem4} and Hutchings-Taubes's obstruction gluing argument \cite{HT1, HT2}, then we get the result.  Note that the ECH partition conditions ensure that the gluing coefficient is 1.
		
		Here we explain a little more about  the gluing argument. Hutchings and Taubes's original obstruction gluing argument can be modified to glue a PFH-HF curve and a PFH curve with connectors in between, see Section 6.5 of \cite{VPK}.  Let us outline the key steps of  gluing a pair $(u_+, u_-)$ with connectors in between, where $u_+$ is an embedded  PFH-curve with $I=\operatorname{ind}=1$ and $u_-$ is a PFH-HF curve with $I=\operatorname{ind} =0$, and  then we will explain the changes in our situation. The gluing argument includes the following three steps:
		\begin{enumerate}
			\item
			Let $\mathcal{M}$ be the moduli space of $\operatorname{ind}=0$ branched covering  of  the trivial cylinders whose ends are determined by the ECH partition conditions. Given $u_0 \in \mathcal{M}$ and $u_{\pm}$, construct a preglued  curve $u_*: \dot{F}_* \to W$. Let $\psi_{\pm}$ be sections of normal bundles of $u_{\pm}$ and $\psi_{0}$ be a complex function over $u_0$.  Using   these sections,  we can deform $u_*$ to  a new curve  $e_{u_*}(\beta_+ \psi_+, \beta_{0} \psi_{0}, \beta_- \psi_- )$, where $\beta_{\pm}, \beta_0$ are suitable cutoff functions.
			\item
			The deformed curve $e_{u_*}(\beta_+ \psi_+, \beta_{0} \psi_{0}, \beta_- \psi_- )$ is $J$-holomorphic if and only if  it satisfies the following equation
			\begin{equation} \label{eq2}
				\beta_+ \Theta_+ (\psi_+, \psi_{0}) + \beta_{0} \Theta_{0}(\psi_+, \psi_{0}, \psi_-) + \beta_- \Theta_-(\psi_-,  \psi_{0})=0.
			\end{equation}
			Hence, it suffices to solves $\Theta_{\pm} =0$ and $\Theta_0=0$.  Fix a sufficiently small $\psi_0$. Since $u_{\pm}$ are Fredholm regular, one can find solutions  $\psi_{\pm}$ such that $\Theta_{\pm}=0$, where $\psi_{\pm}$ depends on $\psi_0$.  Then, one can reduce solving  $\Theta_0=0$ to finding zeros of  a section $\mathfrak{s}$ of the obstruction bundle $\mathcal{O} \to \mathcal{M}$.
			\item
			Finally, Hutchings and Taubes show that $\mathfrak{s}^{-1}(0)$ only depends on the asymptotic expansion of the ends of $u_0$. Moreover, they give a combinatorial formula for $\#\mathfrak{s}^{-1}(0)$. If $u_{\pm}$ satisfies the ECH partition conditions, then $\#\mathfrak{s}^{-1}(0)=1$.
		\end{enumerate}
		
		% The difference here with ECH=HF I is that our $J_{\tau}$ variant with respect to $\tau$.  Fix a gluing triple $(u_+, u_0, u_-)$, where $u_+$ is an embedded PFH curves with index 1, $u_-$ is an embedded  $J_{\tau_0}$ PFH--HF curve with index $-1$ and $u_0$ is a branch covered  of the trivial cylinder. Then we still can construct a pregluing curve $u_* : C_* \to W$. Given sections $\psi_{\pm} \in \Gamma(N_{u_{\pm}})$ and a complex function $\psi_{\Sigma}$, we can deform $u_*$ as in Gluing 2, the result is denoted by $e_{u_*}(\beta_+ \psi_+, \beta_{\Sigma} \psi_{\Sigma}, \beta_- \psi_- )$.    For any small $\tau$,  $\bar{\partial}_{J_{\tau_0 +\tau}}e_{u_*}(\beta_+ \psi_+, \beta_{\Sigma} \psi_{\Sigma}, \beta_- \psi_- )=0$ if and only if
		
		In our situation, we need to glue a $J_{\tau_0}$-holomorphic  PFH-HF curve with $I=\operatorname{ind}=-1$ and a $I=\operatorname{ind}=1$ PFH curves with connectors in  between. The difference between the case here  and  the above  is that  $(J_{\tau},  L_{\tau})$ changes as $\tau$ vary.  This difference only influences  the second steps as the other steps take over the ends of holomorphic curves where $(J_{\tau}, L_{\tau})$  is fixed.  As before, the deformed  curve is $J_{\tau_0+\tau}$-holomorphic if it satisfies the Equation (\ref{eq2}).  Over the region  of $|s| \le R_0$, $\Theta_-$ should be replaced by  $$\Theta_-(\psi_-, \tau)=D\bar{\partial}_{J_{\tau_0}, u_-} \psi_- + \tau \dot{J}_{\tau_0}(u_-) \circ du \circ j + F_-(\psi_-, \tau),$$ where  $F_-(\psi_-, \tau)$ is type 2-quadratic (see Definition 5.1 of \cite{HT2})  with respect to $\psi_-$ and $\tau$.   Over the region $|s| >R_-$, $\Theta_-$ is the same as before. Because  our $\{J_{\tau}\}_{\tau \in [0,1]}$ is generic, the operator  $D\bar{\partial}_{J_{\tau_0}, u_-} \psi_- + \tau \dot{J}_{\tau_0}(u_-) \circ du \circ j $ is a bijection between  suitable Sobolev spaces. Hence, for sufficiently small $\psi_{0}$, we can find a unique solution $(\psi_-, \tau)$ depending on $\psi_{0}$. The rest of argument is the same as before.
	\end{proof}

	\subsection{Proof of the partial invariance}
	To prove the  partial invariance  of Theorem \ref{thm2}, we want to apply the  neck stretching  argument. However,  at present   the PFH cobordism maps  only can be defined via Seiberg-Witten equations which is not compatible with our setting here. To ensure that the cobordism maps can be defined by holomorphic curves, we assume that $\varphi_{H}$ and $\varphi_G $ satisfy (\ref{assumption1}) and (\ref{assumption2}) respectively.

	\begin{comment}
	\begin{itemize}
	\item
	The periodic orbits of $\omega_H$ with degree less than  or equals to $d$ are either hyperbolic or $d$ negative elliptic.
	\item
	The periodic orbits of $\omega_G$ with degree less than  or equals to $d$ are either hyperbolic or $d$ positive elliptic.
	\end{itemize}
	\end{comment}

	\begin{comment}
	\begin{theorem}
	Let $(X, \omega_X)$ be a fiberwise symplectic cobordism from $(Y, \omega_H)$ to $(Y, \omega_G)$.  Suppose that $\varphi_H$ and $\varphi_G$ satisfies the conditions \ref{assumption1}  and \ref{assumption2}. Let $J_H \in \mathcal{J}(W, \omega_{\varphi_H})$  and $J_G \in \mathcal{J}(W, \omega_{\varphi_G})$ be generic admissible almost complex structures. Then  we have the following diagram:
	\begin{displaymath}
	\xymatrix{
	\widetilde{PFH}_*(\mathbb{S}^2,  \varphi_H, d)_{J_H} \ar[d]^{HP_{sw}(X, \omega_X) } \ar[r]^{\mathcal{T}_{r*}^+} & HF_{*}(\mathbb{S}^2,  \varphi_H,  \Lambda)_{J_H}  \ar[d]^{ I(H_s)_* }\\
	\widetilde{PFH}_*(\mathbb{S}^2,  \varphi_G, d)_{J_G} \ar[r]^{\mathcal{T}_{r*}^-} &HF_{*}(\mathbb{S}^2,  \varphi_G,  \Lambda)_{J_G} }
	\end{displaymath}
	\end{theorem}
	\end{comment}

	Let $(W, \Omega_{\varphi_G}, L_{\Lambda_{G}})$ be a closed-open  cobordism defined by a Hamiltonian $G$ as before.   Let $(M, \Omega,  L_0, L_1)$ be  the  Lagrangian cobordism from $(\varphi_G(\underline{\Lambda}), \underline{\Lambda})$ to $(\varphi_H(\underline{\Lambda}), \underline{\Lambda})$ provided by Lemma \ref{lem10}.  Let $(X, \Omega_X)$  be the symplectic cobordism (\ref{eq9}) from $(Y_{\varphi_H}, \omega_{\varphi_H})$ to  $(Y_{\varphi_G}, \omega_{\varphi_G})$. For simplicity, assume that  the constant $R_0=1$ in Lemma \ref{lem10} and (\ref{eq9}). We glue  $(X, \Omega_X)$, $(W, \Omega_{G}, L_{\Lambda_{G}})$  and $(M, \Omega,  L_0, L_1)$ in the following way:  Let $s_+$, $s_0$ and  $s_-$ be the $\mathbb{R}$-coordinates of $X$, $W$ and $M$ respectively.  Define
	\begin{equation*}
		(W_{r, R}, \Omega_{r, R}) := (M, \Omega) \vert_{s_- \le r} \cup_{s_-=r \sim s_0=-r } (W, \Omega_{\varphi_G}) \vert_{ -r \le s_0 \le R} \cup_{s_0=R \sim s_+=-R}  (X, \Omega_X) \vert_{s_+ \ge -R}.
	\end{equation*}
	Note that $\pi: W_{r, R} \to B$ still is a surface  bundle over $B$.
	We glue the Lagrangian submanifolds together similarly. Define
	\begin{equation*}
		L_{r, R}:=   (L_0\cup L_1) \vert_{s_- \le r} \cup_{s_-=r \sim s_0=-r }   L_{\Lambda_{G}} \vert_{ -r \le s_0}.
	\end{equation*}
	%Define $W_R= (W_-, \omega_{W_-}) \cup   (W_G, \omega_{W_G}) \cup ([0, R] \times  Y_G ) \cup (X_+, \omega_{X_+}) $. The symplectic form $\Omega_R$ on $W_R$ is defined by gluing the symplectic forms  $\omega_{X_+} + ds\wedge dt $, $\omega_{W_G} + ds \wedge dt$ and  $\omega_{W_-}+ ds \wedge dt$ in the natural way.
	Then $L_{r, R}$ is  a $\Omega_{r, R}$-Lagrangian submanifold and  it consists of $d$ disjoint connected components. Moreover,  the Lagrangian $L_{r, R} \subset \pi^{-1} ( \partial B)$  satisfies
	\begin{itemize}
		\item
		$L_{r, R} \vert_{s \le -2r-1 \times \{0\} } =\mathbb{R}_{s \le -2r-1} \times \{0\}\times \varphi_H(\underline{\Lambda})$,
		\item
		 $L_{r, R} \vert_{s \le -2r-1 \times \{1\} } =\mathbb{R}_{s \le -2r-1} \times \{1\}\times \underline{\Lambda}$,
		\item
		and $L \vert_{|s| \le r} = L_{\Lambda_G}$.
	\end{itemize}
	% $L_{s \le -10 \times \{0\} } =\mathbb{R}_{s \le -10} \times \phi_H(\Lambda)$, $L_{s \le -10 \times \{0\} } =\mathbb{R}_{s \le -10} \times \Lambda$ and $L \vert_{|s| \le 5} = L_{\Lambda_G}$.% The Lagrangian $L$ can be constructed by gluing  $L_{\Lambda_G}$ with the Lagrangian cobordism in Section 1.3.
	%  Let $(W'=W_- \cup W_G, \Omega_{W'})$ with Lagrangian $L$.  For   generic $J', J$, we still can define the open--closed morphisms $\Phi_{A_0, J'}(W', \Omega_{W'}, L)$  and $\Phi_{A_0, J}(W_G, \Omega_{W_G}, L_{\Lambda_G})$ as before.Le

Define
\begin{equation*}
 \begin{split}
    &(W_R, \Omega_R, L_R) : = (W_{r, R}, \Omega_{r, R}, L_{r, R}) \vert_{r=-3},\\
       &(M \circ W \circ X, \Omega_1, L_1 ) := (W_R, \Omega_R, L_R) \vert_{R=1},\\
        &(M \circ W,  \Omega', L' ): =   (M, \Omega, L_0 \cup L_1) \vert_{s_- \le 3} \cup_{s_-=3 \sim s_0=-3 } (W, \Omega_{\varphi_G}, L_{\Lambda_G}) \vert_{s_0 \ge -3}.
   \end{split}
\end{equation*}
Let $s$ denote the $\mathbb{R}$ coordinate on $W_{R}$ such that $s=s_0$ over $W_{\varphi_G}$. 	Take a generic path of almost complex structures $\{J_R\}_{R \ge 1}$ on $W_R$ such that
	\begin{enumerate}
		\item
		$J_1=J_{R} \vert_{R=1}$  is generic in the sense of Lemma \ref{lem31};
		%\item
		%$J_1 \vert_{|s| \ge R_0}$ agrees with   $J_H$ on $\mathbb{R} \times Y_H$.
		\item
		$J_R =J_H \in \mathcal{J}(Y_{\varphi_H}, \omega_{\varphi_H})$ when $s  \ge 2R+10$ and $s \le -10$.
		\item
		$J_R =J_G \in \mathcal{J}(Y_{\varphi_G}, \omega_{\varphi_G})$  when  $-3 \le s  \le 2R-10$.
		\item
		For each $R$,  $J_{R}$ is an $\Omega_{R}$-tame almost complex structure such that the projection $\pi_R: W_R  \to B$ is complex linear with respect to $(J_{R},  j_B)$.
		\item
		As $R \to \infty$, $J_R \vert_{M \circ W}$ converges in $C^{\infty}_{loc}$ to a generic  $J'$ and $J_R \vert_{X}$ converges in $C^{\infty}_{loc}$ to a generic  admissible almost complex structure $J_X$.
	\end{enumerate}
As $R \to \infty$, $(W_R, \Omega_R, L_R, J_R)$  splits into  $(X, \Omega_X, J_X)$ and $(M \circ W, \Omega', L', J')$.  We want
 to show that the closed-open morphism defined by $(M \circ W\circ X, \Omega_1, L_1, J_1)$ equals to the
 composition of the PFH cobordism map and the one defined by $(M \circ W, \Omega', L', J')$.

	Define the moduli space $\mathcal{M}^{\{J_R\}}(\alpha, \textbf{y}, \mathcal{Z}) : =  \{ (u, R) \vert u \in \mathcal{M}^{J_R} (\alpha, \textbf{y}, \mathcal{Z}), R\in [1, \infty)\}$.  The following lemma study the boundary of   $\mathcal{M}^{\{J_R\}}(\alpha, \textbf{y}, \mathcal{Z})$.

	\begin{lemma} \label{lem17}
	Suppose that $\varphi_H$ and $\varphi_G$ satisfy \ref{assumption1} and \ref{assumption2} respectively. 	Assume that $I(\mathcal{Z}) =0$.  Let $\alpha$ be a PFH generator.  Let $\{(u_n, R_n)\}_{n=1}^{\infty}$ be a sequence of  curves in  $\mathcal{M}^{\{J_{R}\}}(\alpha, \mathbf{y}, \mathcal{Z}) $. Then after passing to a subsequence,   $\{(u_n, R_n)\}_{n=1}^{\infty}$ converges to a broken holomorphic curve $\mathbf{u}$ in the sense of  SFT.  Furthermore, the limit $\mathbf{u}$ satisfies the following properties:
		\begin{itemize}
			\item
			If $R_n \to \infty$, then $\mathbf{u}=\{ u^{+}, u^1, ..., u^k, u^{-} \}$, where $u^+$ is an embedded  $J_X$-holomorphic curve in $X $ with $I=0$, $u^i$ are  $ \operatorname{ind} =0$ connectors and $u^{-}$ is an embedded $J'$-holomorphic PFH-HF curve in $M\circ W$ with $I=0$.
			\item
			If $R_n \to R_*$ and $J_{R_*} $ is not a generic almost complex structure, then  we have the following to possibilities:
			\begin{enumerate}
				\item
				$\mathbf{u}=\{ u^{+}, u^1, ...,u^k, u^{-} \}$ , where  $u^+$ is a $J_H$-holomorphic curve in $\mathbb{R} \times Y_{\varphi_H}$ with $I(u^+) =1$, $u^i$ are $ \operatorname{ind} =0$ connectors and $u^-$ is a $J_{R_*}$-holomorphic PFH-HF curve in $W_{R_*}$ with $I=-1$.
				
				\item
				$\mathbf{u}=\{ u^{+},  u^{-} \}$ , where  $u^+$ is  a $J_{R_*}$-holomorphic PFH-HF curve in $W_{R_*}$ with $I=-1$ and $u^-$ is a  $J_H$-holomorphic HF curve with $I=1$.
			\end{enumerate}
		\end{itemize}
	\end{lemma}
	\begin{proof}
The convergence of  $\{(u_n, R_n)\}_{n=1}^{\infty}$
  is guaranteed by Lemma \ref{lem29}, SFT
 compactness \cite{FYHKE} and the same argument as in Lemma \ref{lem3}.

		Consider the case that $R_n \to \infty$. Then  $$\mathbf{u}=\{u^{N^+},... ,u^1,   u^{+}, v^1, ...,v^k, u^{-}, v^{-1}, ..., v^{-N_-}\}$$ is a broken holomorphic curve with total zero ECH index,   where $u^i$ are $J_H$-holomorphic  curves in $\mathbb{R} \times Y_{\varphi_H}$,  $u^+$ is  a $J_X$-holomorphic curve in $X $, $v^i$ are $J_G$-holomorphic curves in $\mathbb{R} \times Y_{\varphi_G}$, $u^{-}$ is a  $J'$-holomorphic curve in $M \circ W$, and $v^{-j}$  are $J_H$-holomorphic curves in $M$ with Lagrangian boundary conditions $\mathbb{R} \times \{0\} \times  \varphi_H(\underline{\Lambda}) \cup \mathbb{R} \times \{1\} \times \underline{\Lambda}$.
		
		As we choose  the  almost complex structures so that the projections are complex linear, the same argument in Lemma \ref{lem7}  shows that the bubbles lie in the fibers.  Hence, the inequality (\ref{eq3}) still holds.
 By Lemma 9.5 of \cite{H1}, (\ref{eq59}) and Theorem \ref{thm4},
 $u^i$, $u^-$ and $v^{j}$  have nonnegative ECH index.  By     the assumptions on $\varphi_{H}, \varphi_G$ and Corollary  7.4 in \cite{GHC},  we have  $I(u^+) \ge 0$. In sum, each level of $\mathbf{u}$ has    nonnegative ECH index. Consequently,  each level of $\mathbf{u}$ has   zero ECH index and no bubbles appear.

 The zero ECH index implies that $v^{-j}$ are trivial strips,   $u^i$ and $v^j$ are connectors.  By  Lemma 7.7 in \cite{GHC}, $ \operatorname{ind} u^+ \ge 0.$ Since the total Fredholm index is zero, by \ref{fact3},  each level of $\mathbf{u}$ has   zero Fredholm  index.
 The ECH partition conditions and \ref{fact1} can be used to show that $u^i$ are trivial connectors. In particular, $u^i$ and $v^{-j}$ are ruled out. Then
 we get the statement.

		For the case  $R_n \to R_*$, the proof is the same as  Lemma \ref{lem4}.
	\end{proof}

 We have closed-open maps $\widetilde{\mathcal{CO}}(M \circ W, \Omega', L')_{J'} $ and $\widetilde{\mathcal{CO}}(M \circ W \circ X,\Omega_1,L_1)_{J_1}$
 defined by counting $I = 0$ PFH-HF curves in the corresponding closed-open symplectic cobordisms. Also, we have a holomorphic curve definition of the PFH cobordism
 map $PFH_{Z_{ref}}(X, \Omega_X)_{J_X}$.  Since these maps also depend on the choice of the reference classes, we need to clarify our choice here.  To define $\widetilde{\mathcal{CO}} (M \circ W, \Omega', L')_{J'} $,  we
 use $\mathcal{Z}_{ref} = [\Psi'_G(B \times \mathbf{x})\#F(\mathbb{R} \times [0,1] \times \mathbf{x})] $ as a reference homology class, where $\Psi'_G$
 is the trivialization in Remark \ref{remark5} and $F$ is the diffeomorphism (\ref{eq44}). For $PFH_{Z_{ref}}(X, \Omega_X)_{J_X}$, we take $Z_{ref}=[\mathbb{R} \times S^1 \times \mathbf{x}] \in H_2(X, \gamma_H^{\mathbf{x}},  \gamma_G^{\mathbf{x}}). $ Then the reference
 homology class for defining   $\widetilde{\mathcal{CO}}(M \circ W \circ X,\Omega_1,L_1)_{J_1}$  is $\mathcal{Z}_{ref} \# Z_{ref} = [\Psi_H'(B \times \mathbf{x})]$.   Note that here
 $ [\Psi'_G(B \times \mathbf{x})]$, $[ F(\mathbb{R} \times [0,1] \times \mathbf{x})] $ and $[\mathbb{R} \times S^1 \times \mathbf{x}]$ are the classes used to define $\widetilde{\mathcal{CO}}(\underline{\Lambda}, G)$,
 $I_{G,H}$ and $\mathfrak{I}_{H,G}$ respectively.

  By Lemma  \ref{lem17} and Hutchings-Taubes's obstruction gluing argument \cite{HT1, HT2},  we  have
	\begin{equation*}
		\widetilde{\mathcal{CO}} ( M \circ W\circ X, \Omega_1, L)_{J_1} =\widetilde{\mathcal{CO}} ( M \circ W, \Omega', L')_{J'}\circ PFC_{Z_{ref}}(X,  \Omega_{X})_{J_X}  + d_{J_H} \circ K + K\circ  \partial_{J_H}.
	\end{equation*}
	The morphism $K: \widetilde{PFC}(\Sigma, \varphi_H, d)_{J_H} \to CF(\Sigma,   \varphi_H( \underline{\Lambda}),  \underline{\Lambda})_{J_H} $ is defined by counting the moduli space $\mathcal{M}^{\{J_R\}}(\alpha, \textbf{y}, \mathcal{Z})$ with $I(\mathcal{Z})=-1$.  In particular, we have
	$$\widetilde{\mathcal{CO}}( M \circ W \circ X, \Omega_1, L_1 )_{J_1} = \widetilde{\mathcal{CO}} ( M \circ W, \Omega', L')_{J'} \circ PFH_{Z_{ref}}(X,  \Omega_{X})_{J_X}.  $$
at the  homological level.
	
	Recall that $(\Omega_1, L, J_1)_{|s| \ge 20} =(\Omega_{\varphi_H}, L_{\Lambda_{H}}, J_H)$.   Under the trivialization in Remark \ref{remark5},  one can   find a homotopy $\{(\Omega_{\tau}, L_{\tau}, J_{\tau})\}_{\tau \in [0,1]}$ such that $(\Omega_{\tau}, L_{\tau}, J_{\tau}) \vert_{\tau =0} = (\Omega_H, L_{\Lambda_{H}}, J_H )$ and  $(\Omega_{\tau}, L_{\tau}, J_{\tau}) \vert_{\tau =1} = (\Omega_1, L, J_1)$.   More in details, we can construct  $\{(\Omega_{\tau}, L_{\tau})\}_{\tau \in [0,1]}$ as follows.  Let $H^X_s$, $H_s^M$  denote the functions using for defining (\ref{eq9})
 and (\ref{eq44}) respectively. Let
$$H_s =
\begin{cases}
H_s^X, & s >1 \\
G, & -3 \le s \le 1\\
H_s^M, & s <-3.
\end{cases}$$
Define  $H_{s, \tau} : =(1 -\tau) H_s + \tau H$ and $\tilde{\omega}_{\tau}: = \omega + d(\dot{\chi} H_{s, \tau}  dt)$, where $\chi$ is the cutoff function in Remark \ref{remark5}.  Then $(\Omega_{\tau}, L_{\tau}):= \left( (\Psi_H'^{-1})^*\tilde{\omega}_{\tau} + ds \wedge dt , \Psi_H'(\underline{\Lambda}) \right)$ satisfies our
 requirement.

By Lemma \ref{lem11}, we have
	\begin{equation}\label{eq31}
		 \widetilde{\mathcal{CO}} (\underline{\Lambda}, H)_{J_H}  =\widetilde{\mathcal{CO}} ( M \circ W, \Omega', L')_{J'}  \circ PFH_{Z_{ref}}(X,  \Omega_{X})_{J_X}
	\end{equation}
at the homological level.
	According to Theorem 3 in \cite{GHC},  we can replace  $ PFH_{Z_{ref}}(X,  \Omega_{X})_{J_X}$ by  $PFH^{sw}_{Z_{ref}}(X,  \Omega_{X}) =\mathfrak{I}_{H, G}.$

 The remaining step is to prove the following lemma by applying neck stretching
 argument to $( M \circ W, \Omega', L', J').$
	\begin{lemma} \label{lem24}
		We have $\widetilde{\mathcal{CO}} ( M \circ W, \Omega', L')_{J'} =I_ {G, H} \circ \widetilde{\mathcal{CO}} ( \underline{\Lambda}, G)_{J_G}.$
	\end{lemma}
	\begin{proof}
	For $r\ge 0$, define
  $$(W_{r}, \Omega_{r} , L_{r}):=   (M, \Omega, (L_0\cup L_1)) \vert_{s_- \le r+3} \cup_{s_-=r+3 \sim s_0=-r-3 } (W, \Omega_{\varphi_G},  L_{\Lambda_{G}}) \vert_{ -r -3\le s_0 }.   $$
  Note that  $( M \circ W, \Omega', L') =(W_{r}, \Omega_{r} , L_{r}) \vert_{r=0}.$

		Let $\{J_r\}_{r\ge 0}$ be a generic family of almost complex structures such that
	\begin{enumerate}	
	\item
		$J_{r=0} =J'$, $ J_r \vert_{ W_{\varphi_G}} $ converges to $ J_G$ and $J_r \vert_{M} $ converges to $ J''$. Moreover, $J_G$ and $J''$ are generic.
	\item
	$J_r =J_G$  over the region  $s\ge -2r $ and $J_r=J_H$ over the region $s \le -2r-10$.
	\item
	$J_r$ is $\Omega_r$-tame. Also, the projection  $\pi_r: W_r  \to B$ is complex linear with respect to $(J_{r},  j_B)$.
	\end{enumerate}
As $r \to \infty $,the tuple $(W_r, \Omega_r,L_r,J_r) $  splits into $(M, \Omega, L_0, L_1, J'')$  and $(W_{\varphi_G}, \Omega_{\varphi_G}, L_{\Lambda_G}, J_G)$.
		 Let $\{u_n\}_{n=1}^{\infty}$ be a sequence of $J_{r_n}$-holomorphic  PFH-HF curves in $W_{r_n}$ with fixed relative homology class.   Then $\{u_n\}_{n=1}^{\infty}$
 converges to a broken holomorphic $\mathbf{u}$. Using the same argument  in Lemma \ref{lem17}, then we get the following results:
		\begin{itemize}
			\item
			If $r_n \to \infty$, then $\mathbf{u}=\{ u^{+}, u^{-} \}$, where $u^+$ is an embedded  $J_G$-holomorphic PFH-HF curve in $W_{\varphi_G}$ with $I=0$   and $u^{-}$ is an embedded $J''$-holomorphic curve in $M$ with $I=0$.
			\item
			If $r_n \to r_*$ and $J_{r_*} $ is not a generic almost complex structure, then  we have the following two possibilities:
			\begin{enumerate}
				\item
				$\mathbf{u}=\{ u^{+}, u^1,...,u^k, u^{-} \}$, where  $u^+$ is a $J_G$-holomorphic curve in $\mathbb{R} \times Y_{\varphi_G}$ with $I(u^+) =1$, $u^i$ are $ \operatorname{ind}  =0$ connectors and $u^-$ is a $J_{r_*}$-holomorphic PFH-HF curve in $W_{r_*}$ with $I=-1$.
				
				\item
				$\mathbf{u}=\{ u^{+},  u^{-} \}$, where  $u^+$ is  a $J_{r_*}$-holomorphic PFH-HF curve in $W_{r_*}$ with $I=-1$ and $u^-$ is a  $J_H$-holomorphic HF curve with $I=1$.
			\end{enumerate}
		\end{itemize}
		 By the gluing argument \cite{HT1, HT2} and our choice of the reference homology classes, we finish the proof of the statement.
	\end{proof}
	Combine  Lemma \ref{lem24} and (\ref{eq31}); then  we finish the proof of the partial invariance in Theorem \ref{thm2}.

	\section{Computations} \label{section6}
	In this section, we show that the closed-open morphism $\widetilde{CO}$ is non-vanishing under assumption  \ref{assumption1}. To this end,   we first  fix a Morse function as follows.  Let $f_{\Lambda_i} : \Lambda_i \to \mathbb{R}$ be a perfect Morse function with a minimum $y_i^-$ and  a maximum $y_i^+$. We   extend $\cup_{i=1}^d f_{\Lambda_i}$ to be a Morse function $f: \Sigma\to \mathbb{R} $ such that
	\begin{enumerate}
		\item
		$(f, g_{\Sigma})$ satisfies the Morse-Smale condition, where $g_{\Sigma}$ is a fixed metric on $\Sigma$.
		\item
		$f=f_{\Lambda_i} -\frac{1}{2} y^2$ in a neighbourhood of $\Lambda_i$, where $y$ is the coordinate of the normal direction.
		\item $\{y_i^+\}_{i=1}^d$ are the only   maximum points  of $f$. Also,   $f(y_i^+) =0$ for any $1 \le i \le d$.
	\end{enumerate}
	The picture of $f$ is shown in Figure \ref{figure2}.
	% $(y_i^-, y_i^+)$ are the only critical points restricting $f$ to a neigborhood of $\Lambda_i$.   %Hence, $CF(\Sigma, \Lambda, f)$ is generated by $y=\{y_1^{\epsilon_1},  \cdots y_d^{\epsilon_d}\}$, where $\epsilon_i \in \{+, -\}$.
	Define a function $H_{\epsilon}:=\epsilon f$, where $0<\epsilon<1$ is a small constant.

To prove the non-vanishing result, the idea is to compute the closed-open morphism  for a special Hamiltonian function $G$ and then using the partial invariance in Theorem \ref{thm2} to deduce the non-vanishing for other $H$.

	The Morse function $H_{\epsilon}$ is a nice  candidate  for computation.    In the next subsection, we show that the Reeb chords and periodic orbits of $\varphi_{H_{\epsilon}}$  correspond  to the critical points of $f_{\Lambda_i}$ and $f$ respectively.  Also, the ECH index  and  the energy  of holomorphic curves in $ (W, \Omega_{\varphi_{H_{\epsilon}}}, {L_{\Lambda_{H_{\epsilon}}}} )$  are computable. Combining the energy and index reasons, we could compute   $ \widetilde{\mathcal{CO}}(\underline{\Lambda}, H_{\epsilon})$.  However, to apply the partial invariance in Theorem \ref{thm2}, we require that $H_{\epsilon}$ satisfies assumption \ref{assumption1} or \ref{assumption2} while $H_{\epsilon}$ does not.   Therefore, we need to modify  $H_{\epsilon}$ to  $H_{\varepsilon}'$ satisfying  \ref{assumption2}. After this modification,  the periodic orbits of ${\varphi_{H'_{\varepsilon}}}$ are no longer  corresponding to the critical points any more.  We need to do a bit more works.   For example, it is not clear which  PFH generators  represented the unit of PFH.  Thus,  we need to compute the PFH cobordism maps.

	\subsection{Reeb chords and period orbits of $\varphi_{H_{\epsilon}}$ }  \label{section5.1}
In this section, we make a suitable choice of the Hamiltonian function and study its  Reeb chords and periodic orbits.

	\begin{lemma} \label{lem25}
		If $0< \epsilon \ll 1$ is sufficiently small, then $\mathbf{y}=[0, 1] \times \{y_1^{\epsilon_1}, ..., y_d^{\epsilon_d}\}$ are the only Reeb chords of $\varphi_{H_{\epsilon}}$, where  $\epsilon_i \in \{+, -\}$. %then the Reeb chords of $H$ are 1--1 corresponding to the critical points of $f$.
	\end{lemma}
	\begin{proof}
		Since  $\operatorname{dist}(\varphi_{H_{\epsilon}}(x), x ) \le c_0 \epsilon |\nabla f|_{L^{\infty}}$, we assume that $\varphi_{H_{\epsilon}}(\Lambda_i)$ lies inside a small neighbourhood of $\Lambda_i$, $1\le i \le d $. Let $(x, y)$ be  local coordinates near $\Lambda_i$, where $x$ is the coordinate of  $\Lambda_i$ and $y$ is the coordinate of the normal direction. We  assume that the symplectic form  is $\omega =dx \wedge dy$ under these  coordinates. It is easy to check that the Hamiltonian vector field is $X_{{H_{\epsilon}}}=-\epsilon y \partial_x -\epsilon f_{\Lambda_i}'(x) \partial_y $.  Therefore, the Hamiltonian flow is
		\begin{equation} \label{eq5}
			\begin{cases}
				& \dot{x}=-\epsilon y\\
				& \dot{y}=-\epsilon f_{\Lambda_i}'(x).
			\end{cases}
		\end{equation}

		Take a neighbourhood  $\mathcal{V}_{2\delta_0}=\mathcal{V}_{2\delta_0}^0 \cup \mathcal{V}_{2\delta_0}^1  \subset \Lambda_i$ of  $f_{\Lambda_i}$'s  critical points such that
		\begin{itemize}
			\item
			$|f'_{\Lambda_i}(x)| \ge \delta_0$   away from   a smaller neighbourhood  $\mathcal{V}_{\delta_0} \subset \mathcal{V}_{2\delta_0}$.
			\item
			$f_{\Lambda_i}\vert_{\mathcal{V}_{2\delta_0}^0} = m_p+  \frac{1}{2}x^2$ and $f_{\Lambda_i}\vert_{\mathcal{V}_{2\delta_0}^1} =- \frac{1}{2}x^2$, where  $m_p$ is the minimum value of  $f_{\Lambda_i}$.
		\end{itemize}
Suppose that for any $\epsilon_n \to 0$, we can  find   nonconstant solutions $(x_n, y_n)$  to (\ref{eq5}) satisfying $(x_n(0), 0), (x_n(1), 0) \in \Lambda_i$.  By Arzela-Ascoli Theorem,  $(x_n, y_n)$ converges in $C^{\infty}$ to a constant $ (x_0, 0)$ as $n \to \infty$. If $x_0$ is outside  $\mathcal{V}_{\delta_0}$, then we have $f_{\Lambda_i}'(x_n) \ge \frac{\delta_0}{2}$ or  $f_{\Lambda_i}'(x_n) \le - \frac{\delta_0}{2}$ for $n$ large enough. Equations (\ref{eq5}) implies that $y_n(1) \le - \frac{\delta_0{\epsilon_n}}{2}$ or $y_n(1) \ge  \frac{\delta_0{\epsilon_n}}{2}$. This contradicts with $y_n(1)=0$.
		
		If $x_0$ lies in the $\delta_0$-neighbourhood of the maximum, we may assume that $(x_n, y_n)$ lies inside $\mathcal{V}^1_{2\delta_0} \times (-2\delta_0, 2\delta_0)$ where  $f_{\Lambda_i}(x) = -\frac{1}{2}x^2$.  Then the solution to the ODE (\ref{eq5}) is
\begin{equation*}
x_n(t)+iy_n(t) =(x_n(0) + iy_n(0)) e^{i\epsilon_n t}.
\end{equation*}
 Obviously, $(0,0)$ is the only solution satisfying $y_n(0)=y_n(1)=0$.
		
		If $(x_n, y_n)$  lies inside $\mathcal{V}^0_{2\delta_0} \times (-2\delta_0, 2\delta_0)$   where $f_{\Lambda_i}(x) =m_p+ \frac{1}{2}x^2$.  Then the solution to the ODE (\ref{eq5}) is
		\begin{equation*} \
			\begin{split}
				&  x_n(t) = \cosh(\epsilon_n t) x_n(0) - \sinh(\epsilon_n t) y_n(0)\\
				& y_n(t) = -\sinh(\epsilon_n t) x_n(0) + \cosh(\epsilon_n t) y_n(0).
			\end{split}
		\end{equation*}	
		Again, $(0,0)$ is the only solution satisfying $y_n(0)=y_n(1)=0$. These contradict with  our assumption on $(x_n(t),y_n(t))$.

	\end{proof}
	
	\begin{comment}
	\begin{remark} \label{remark2}
		For the above $H_{\epsilon}$, the   homology  $HF_{*}(Sym^d \varphi_{H_{\epsilon}}(\Lambda), Sym^d\Lambda, \textbf{x})$ is well defined.   Let $s: B_0 \to B_0 \times Sym^d  \Sigma$ be a holomorphic section contributed to the cobordism map (\ref{eq30}).  It is easy to  deduce  formulas for  the  Fredholm index and the  energy of $s$.    From these formulas, we know that the negative end of $s$ must asymptotic to $\textbf{y}_{\heartsuit}$. Therefore,  the unit $\mathbf{1}^{\textbf{x}}_{H_{\epsilon}}$ is represented by $(\textbf{y}_{\heartsuit}, \mathcal{S}_{\mathbf{1}})=(y_1^+,..y_d^+)$ for some capping $\mathcal{S}_{\mathbf{1}}$.   Let $A_e$ be the capping  such that $\Psi_*([A_e]) =\mathcal{S}_{\mathbf{1}}$.   Then $(\textbf{y}_{\heartsuit}, [A_e])$ is a representative of   the unit of $HF(\Sigma, \varphi_{H_{\varepsilon}}, \Lambda, \textbf{x})$.
	\end{remark}
	\end{comment}

	\paragraph{First modification on $H_{\epsilon}$:} The purpose of this modification is to ensure that there exist periodic orbits with large period near the local minimum of $H_{\epsilon}$. Even though these orbits do not contribute to  the chain complex $\widetilde{PFC}(\Sigma, \varphi_{H_{\epsilon }}, \gamma_{H_{\epsilon}}^{\mathbf{x}})$, they are useful when
 we compute the cobordism maps and closed-open morphisms.

	Let $(x,y )$ be the local coordinates near a  local minimum  $p$ of $f$ such that  $\omega =dx \wedge dy$.  Under these coordinates, we choose $f$ such that  $f= m_p +x^2+y^2$.   Let $(r, \theta)$ be the polar coordinates.    Let $\mathcal{U}^{\delta}_p=\{(r, \theta) \vert r\le \delta\}$ denote  a $\delta$-neighbourhood of $p$.  Define $\mathcal{U}^{\delta}:= \cup_p \mathcal{U}^{\delta}_p $, where $p$ run over all the local minimums  of $f$. Fix  positive constants  $\delta, \delta_0>0$.
	Take a smooth  function $\varepsilon(r) $ such that
	\begin{itemize}

		\item
		$\varepsilon(r) =\epsilon$ on the regions $\mathcal{U}^{\delta + \delta_0}$ and $\Sigma \setminus  \mathcal{U}^{\delta +2 \delta_0} $. Here $\epsilon >0$ is a small constant. Also,   $\varepsilon(r) \ge \epsilon$ on the whole $\Sigma$.
		
		\item
		
		Over the region $\{ \delta + \delta_0 \le  r \le \delta + 2 \delta_0\}$, $\varepsilon$ is a function that  only depends on $r$.  The maximum of $\varepsilon(r) $ is denoted by $\epsilon_0$. We choose $\varepsilon(r)$ such that $0 <\frac{\epsilon_0 -\epsilon}{\delta_0} \ll \epsilon $ is small enough.
	\end{itemize}
	
	%We replace the small constant $\epsilon$ by a function $\epsilon(r)$ such that $\epsilon(r) =\epsilon$ outside the region  $ \delta + \delta_0 \le r \le \delta + 2\delta_0 $.  The minimum of $\epsilon(r) $ is denoted by $\epsilon_0$.
	%By taking $\epsilon_0$ sufficiently close to $\epsilon$,
	  Replace $H_{\epsilon}$ by  $H_{\varepsilon}: =\varepsilon(r) f$.  Note that  $H_{\varepsilon}=H_{\epsilon}$ outside the region $\{ \delta + \delta_0 \le r \le \delta + 2\delta_0  \}$. On the region  $\{ \delta + \delta_0 \le r \le \delta + 2\delta_0  \}$,  we have
	  \begin{equation}  \label{eq35}
	  (\varepsilon(r) f(r))' =\varepsilon'(r) (m_p+r^2)  +2\varepsilon(r) r \ge  2\epsilon (\delta+ \delta_0)  - c_0\frac{\epsilon_0 -\epsilon}{\delta_0}>0.
	  \end{equation}
 Therefore,  the replacement  does not  create  any new critical point and $H_{\varepsilon}$ is still a Morse function satisfying the Morse-Smale condition.  The Hamiltonian vector field of $H_{\varepsilon}$ over $\{r\le \delta+ 2\delta_0 \}$ is
	\begin{equation*}
		X_{H_{\varepsilon}}=-\frac{1}{r}H_{\varepsilon}'(r) \partial_{\theta}.
	\end{equation*}
	For $r_0 \in (\delta + \delta_0,      \delta + 2\delta_0)$ such that $\frac{1}{r_0}H_{\varepsilon}'(r_0)=\frac{p}{q} \ne 0$ is a rational number ($p, q$ are relatively prime), then $H_{\varepsilon}$  has  a family of periodic orbits of the form
	\begin{equation} \label{eq27}
		\gamma_{r_0, \theta_0}(\tau) =(\tau, r_0, \theta_0 - \frac{p}{q} \tau),
	\end{equation}
	where $\tau \in  \mathbb{R}/(q\mathbb{Z})$.     Since $\varepsilon(r) $ and $\varepsilon'(r)$ are very small,  the period $q >d$.

 Provided that  $\epsilon \le \varepsilon   $ is sufficiently small, then the periodic orbits of $\varphi_{H_{\varepsilon}}$ with period less than or equal   $d$ are still  one-to-one  corresponding to the iterations of the constant orbits at critical points of $H_{\varepsilon}$.  The critical points of $H_{\varepsilon}$ are  described   in Figure \ref{figure2}.
	\begin{figure}[h]
		\begin{center}
			\includegraphics[width=10cm, height=5cm]{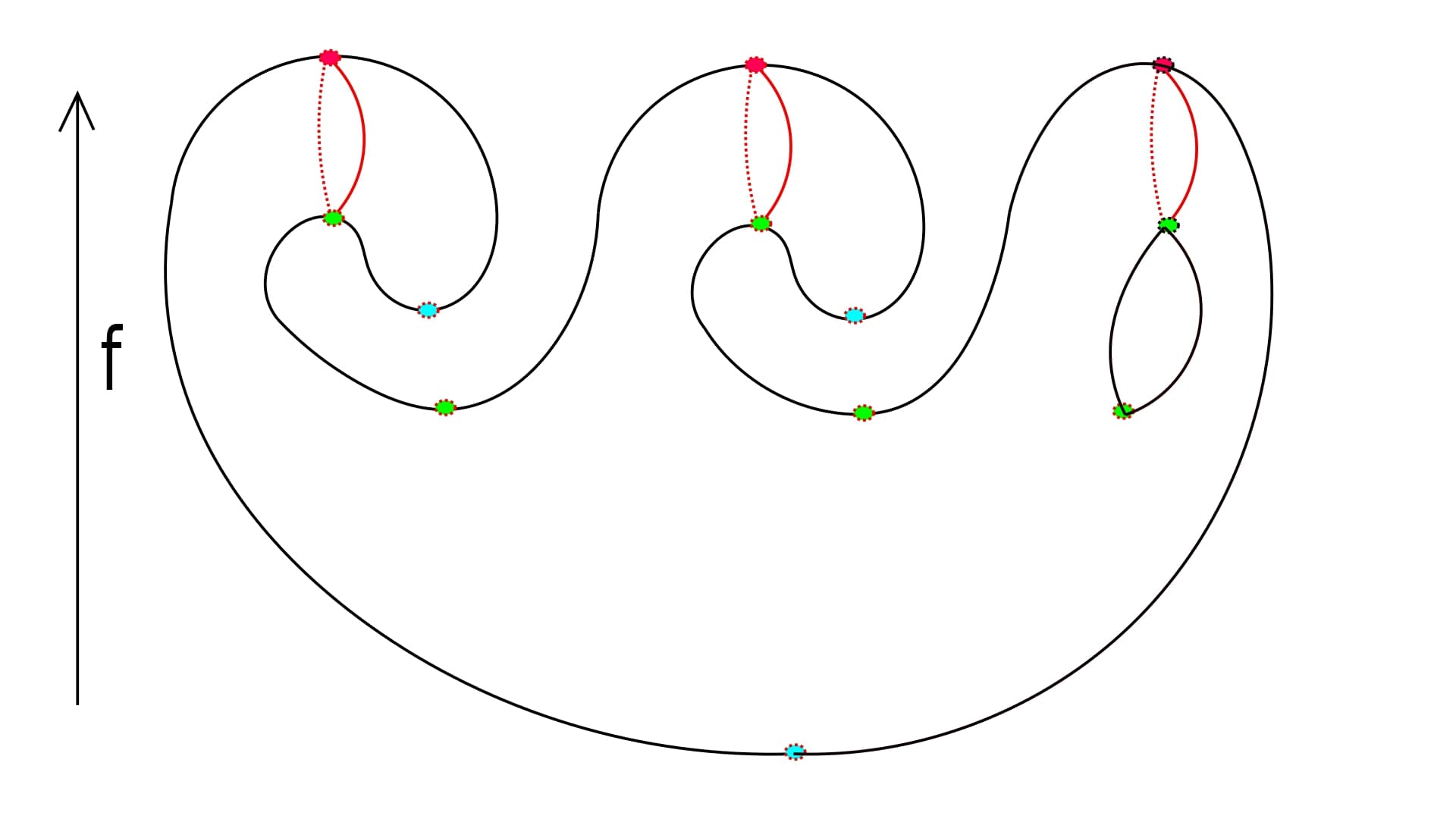}
		\end{center}
		\caption{The red points are the index 2 critical points. The green points are saddle points. The blue points are index 0 critical points. The union of the red circles is the link.  The function $f$ is a small  perturbation of  the height function.}
		\label{figure2}
	\end{figure}
	
 Let $p\in Crit(H_{\varepsilon})$ and $\gamma_p$ denote the constant simple periodic orbit at $p$.	Note that the degree of $\gamma_p$ is $1$. Let $\alpha=\gamma_{p_1} \cdots \gamma_{p_d}$ and  $\beta=\gamma_{q_1} \cdots \gamma_{q_d}$ be two orbit sets. Here $\{p_i\}, \{q_i\}$  could   repeat. Let $\eta: \sqcup_{i=1}^d \mathbb{R}_{s_i} \to \Sigma$ be a path with $d$ components such that $\lim_{s_i \to \infty}\eta(s_i) =p_i$ and $\lim_{s_i \to -\infty}\eta(s_i) =x_i$. Define  relative homology classes
	\begin{equation} \label{eq19}
		Z_{\alpha}:=[\Psi_{H_{\varepsilon}}(S^1 \times \eta)] \in H_2(Y_{\varphi_{H_{\varepsilon}}}, \alpha, \gamma_{H_{\varepsilon}}^{\mathbf{x}}) \mbox{ and } Z_{\alpha, \beta} :=Z_{\alpha} -Z_{\beta} \in H_2(Y_{\varphi_{H_{\varepsilon}}}, \alpha, \beta).
	\end{equation}

	Define a chain $\mathfrak{c}_{\heartsuit} : =\sum (\alpha_I, [Z_I]) \in \widetilde{PFC}(\Sigma, \varphi_{H_{\varepsilon}},  d) $. The notations  are explained as follows:
	\begin{enumerate} [label=\textbf{C.\arabic*}]
		\item \label{class1}
		$\alpha_{I} =\gamma_{y^+_{i_1}} \cdots \gamma_{y^+_{i_d}}$  and   $I=(i_1, ..., i_d)$. The labels  $i_1 , ..., i_d$  could    repeat.
		\item \label{class2}
		$Z_I := Z_{\alpha_I} \in H_2 (Y_{\varphi_{H_{\varepsilon}}},  \alpha_I, \gamma_{H_{\varepsilon}}^{\mathbf{x}})$.
	\end{enumerate}
	When $I=(1,...,d)$, we write $\alpha_I$ as  $\alpha_{\heartsuit} =\gamma_{y^+_{1}} \cdots \gamma_{y^+_{d}}$.

	\begin{lemma}
		The chain	$\mathfrak{c}_{\heartsuit}=\sum_I (\alpha_I, [Z_I])$ is a cycle.  Moreover, it represents a nonzero class $\mathfrak{e}_{H_{\varepsilon}} \in   \widetilde{PFH}(\Sigma, \varphi_{H_{\varepsilon}},  d) $.
	\end{lemma}
	\begin{proof}
		As remarked in Pages 15--16 of  \cite{EH},  each  holomorphic current in $\mathbb{R} \times Y_{\varphi_{H_{\varepsilon}}}$  contributing to  the differential is a union of trivial cylinders and a unique non-trivial holomorphic  cylinder.  For sufficiently small $0<\varepsilon \ll 1$, such  a non-trivial  cylinder   corresponds  to a  Morse flow line of $H_{\varepsilon}$. Therefore, the computation of the PFH differential   reduces to the computation of the Morse differential.  The details are essentially the
 same as the computation of ECH for prequantization bundles by J. Nelson and M.
 Weiler  \cite{NW}.  The Morse differential can be figured out easily for our  $H_{\varepsilon}$, and this shows that $\mathfrak{c}_{\heartsuit}$  is a cycle and it is not exact.  (In Figure \ref{figure2},  for any two maximums adjacent to each other, there is a unique flow line  from the maximum to the saddle point in the middle.)
	\end{proof}
	
	\begin{remark} \label{remark4}
		The class $\mathfrak{e}_{H_{\varepsilon}}$ here can be regarded as the ``unit" of the periodic Floer homology. Similar to the unit of the quantitative Heegaard Floer homology, the class $\mathfrak{e}_{H_{\varepsilon}}$
should be equal to $PFH_{Z_{ref}}(X_+, \Omega_{X_+})(1)$. Here $X_+=B_+\times \Sigma$ is a symplectic cobordism from empty set to $Y_{\varphi_{H_{\varepsilon}}}$; the surface  $B_+$  is a  punctured sphere with a negative cylindrical end. The holomorphic curve definition of the cobordism map
 $PFH_{Z_{ref}}(X_+, \Omega_{X_+})(1)$ should be well defined by the index computation as in \cite{GHC2}. To
 see why we should have  $\mathfrak{e}_{H_{\varepsilon}}=PFH_{Z_{ref}}(X_+, \Omega_{X_+})(1)$,   the idea is to compute the ECH
 index of relative homology class in $B_+\times \Sigma$  by the methods   in \cite{GHC2}.   From ECH index
 formula, we can find that  $I=0$  only can happen when $\alpha=\alpha_I$.  For suitable choices
 of symplectic form and almost complex structure,  $B_+ \times \{y_{i_1}^+,..., y_{i_d}^+\}$  is a holomorphic
 curve in  $\mathcal{M}_0^J(\alpha_I)$.   Such a holomorphic curve is the unique curve in   $\mathcal{M}_0^J(\alpha_I)$ due to the
 energy restriction as in \cite{GHC2}. As a result,  $\mathfrak{c}_{\heartsuit}=PFC_{Z_{ref}}(X_+, \Omega_{X_+})(1)$ is a cycle. The
 argument is similar to the proof of Lemma \ref{lem26}.
	\end{remark}

For the reasons in the above remark, we make the following definition. 	
\begin{definition} \label{definition6}
Define $\mathfrak{e} : = \mathfrak{j}_{H_{\varepsilon}}(e_{H_{\varepsilon}})\in \widetilde{PFH}(\Sigma, d)$. We call this class  PFH unit.
\end{definition}
	
	\paragraph{Second modification on $H$:}
	%Suppose that $0< \epsilon \le \varepsilon \ll 1$ is sufficiently small. Then all the periodic orbits of $H_{\varepsilon}$ with degree less than or equals to $d$ are iteration of the constant orbits at critical points of $H_{\varepsilon}$.
	We expect that the closed-open morphism maps $\mathfrak{c}_{\heartsuit}$ to a cycle representing  the unit  of HF.  However, to apply the partial invariance in Theorem \ref{thm2}, we need to perturb $\varphi_{H_{\varepsilon}}$  so that it satisfies  the condition \ref{assumption1} or \ref{assumption2}.   This can be done by using the following result. It is    Proposition 3.7 of \cite{GHC} for  Hamiltonian case.

\begin{prop} \label{lem37}
Fix a Hamiltonian function $H$ and a metric $g_Y$ on $S^1 \times \Sigma$.   For any  positive numbers  $\delta>0$ and $d$, we can modify $H$ to a  Hamiltonian function $H'$ satisfying the following properties:
		\begin{enumerate}
			\item
			The symplectomorphism $\varphi_{H'}$ satisfies the condition (\ref{assumption1}).
			\item
			$|H-H'| \le \delta$ and $|dH-dH'|_{g_Y} \le \delta$.
		\end{enumerate}
		The condition (\ref{assumption1}) can be replaced by (\ref{assumption2}) depending on the context.
	\end{prop}
	% in  $HF(Sym^d \mathbb{S}^2, Sym^d \varphi_{H_{\varepsilon} (\Lambda)}$
	Applying Proposition  \ref{lem37} to $H_{\varepsilon}$, we  obtain a Hamiltonian function  $\varphi_{H'_{\varepsilon}}$ satisfying  the condition \ref{assumption2}.
	According to the  proof of Proposition 3.7 of \cite{GHC},  such a modification takes place in an arbitrarily small neighborhood of the periodic orbits with degree less than  $d$  that   violate  \ref{assumption2}.  This modification  may create new periodic orbits that   violate \ref{assumption2} in the neighbourhood. But these newly created orbits have larger period.   We repeat the modification in a small neighbourhood of the newly
 created orbits that  violate   \ref{assumption2} until the newly created orbits have period  greater than $d$.
	
	  The modifications are local and the only orbits violating \ref{assumption2} are the  constant orbits at  the local minima of $H_{\varepsilon}$, hence  we  can arrange that $X_{H_{\varepsilon}}=X_{H_{\varepsilon}'}$ outside $\mathcal{U}^{\delta}$.   In particular, $ \gamma_{r_0, \theta_0}$  in (\ref{eq27}) are still periodic orbits of $\varphi_{H_{\varepsilon}'}$.  According to the construction, $H'_{\varepsilon}$ satisfies the estimates in Proposition \ref{lem37}.

	\subsection{Some computations on the PFH cobordism maps}
	In this section, we find a cycle which represents the  unit   in  $\widetilde{PFH}(\Sigma, \varphi_{H'_{\varepsilon}}, d)$.  To
 this end, we do some computations on  $\mathfrak{I}_{H'_{\varepsilon}, H_{\varepsilon}}$.
	
	Let $(X_+, \Omega_{X_+})$ be the symplectic cobordism from $(Y_{\varphi_{H'_{\varepsilon}}}, \varphi_{H_{\varepsilon}'})$ to $(Y_{\varphi_{H_{\varepsilon}}}, \varphi_{H_{\varepsilon}})$ defined by (\ref{eq9}).    Keep in mind that  $\Omega_{X_+}$ is   $\mathbb{R}$-invariant  in    the region $\mathbb{R} \times (Y_{\varphi_{H'_{\varepsilon}}} \setminus (S^1 \times \mathcal{U}^{\delta})) \subset X_+ $.  This region is called a \textbf{product region}. Here we identify  $S^1 \times \mathcal{U}^{\delta} $ as a subset of  $Y_{\varphi_{H_{\varepsilon}}}$ implicitly by using the trivialization  (\ref{eq11}).
	
 Given orbit sets $\alpha_{\pm}$ and $Z \in H_2(X_+, \alpha_+, \alpha_-)$, let $\overline{\mathcal{M}^{J_+}}(\alpha_+, \alpha_-, Z )$  denote the
 moduli space of broken holomorphic currents with relative homology class $Z$.

		\begin{lemma} \label{lem12}
		Let $J_+$ be an admissible almost complex structure on $X_+$ such that  it is $\mathbb{R} $-invariant in   the product region $\mathbb{R} \times (Y_{\varphi_{H'_{\varepsilon}}} \setminus  (S^1 \times \mathcal{U}^{\delta})) \subset X_+.   $ Let  $\alpha_{I'} =\gamma_{y^+_{i'_1}} \cdots \gamma_{y^+_{i'_d}}$     and  $Z_{\alpha_I, \alpha_{I'}} \in H_2(X_+, \alpha_I, \alpha_{I'})$ be the relative homology  class  defined in  \ref{class2}.   Then the moduli space $\overline{\mathcal{M}^{J_+}} (\alpha_I, \alpha_{I'}, Z_{\alpha_I, \alpha_{I'}} ) \ne \emptyset$ if and only if  $I'=I$. Moreover, $\mathbb{R} \times \alpha_I$ is the unique element in $\overline{\mathcal{M}^{J_+} }(\alpha_I, \alpha_I, Z_{\alpha_I, \alpha_{I}} )$.
	
	\end{lemma}
	\begin{proof}
		Let $\mathcal{C} \in \overline{\mathcal{M}^{J_+} }(\alpha_I, \alpha_{I'}, Z_{\alpha_I, \alpha_{I'}} )$  be a broken holomorphic current.  Recall the periodic orbit $ \gamma_{r_0, \theta_0}$ in   (\ref{eq27}). Since it is disjoint from $\alpha_I, \alpha_{I'}$, the intersection number $\#(\mathcal{C} \cap  \mathbb{R} \times \gamma_{r_0, \theta_0})$ is well defined. Moreover, the intersection number only depends on the relative homology class of $\mathcal{C}$.   Therefore,  we have  $$\#(\mathcal{C} \cap \mathbb{R} \times \gamma_{r_0, \theta_0}) = \# ( (S^1 \times \eta)\cap \mathbb{R} \times \gamma_{r_0, \theta_0} )=0. $$
		By our choice of $J_+$,   $\mathbb{R} \times \gamma_{r_0, \theta_0}$ is a $J_+$-holomorphic curve.  According to the  intersection  positivity of holomorphic curves, $\mathcal{C}$ does not  intersect $\mathbb{R} \times \gamma_{r_0, \theta_0} $. Since $\theta_0$ is arbitrary, $\mathcal{C}$ must lie inside
		$\mathbb{R}   \times(  Y_{\varphi_{H'_{\varepsilon}}} \setminus (S^1 \times \mathcal{U}^{\delta + \delta_0})  )$.
		
 Note that the $\omega_{X_+}$-energy  $E_{\omega_{X_+}}(\mathcal{C})=E_{\omega_{X_+}}(Z_{\alpha_I, \alpha_{I'}} )  =0$. Since $J_+$ is $\mathbb{R} $-invariant over $\mathbb{R}   \times(  Y_{\varphi_{H'_{\varepsilon}}}\setminus (S^1 \times \mathcal{U}^{\delta + \delta_0})  )$, we can conclude that $\mathcal{C} =\mathbb{R} \times \alpha_I$ (see Proposition 9.1 of \cite{H1}).  In particular, $\alpha_{I'} =\alpha_I$.
	\end{proof}
	
	\begin{lemma} \label{lem13}
 Let $\alpha_{I}$ be the PFH  generators of $(Y_{\varphi_{H'_{\varepsilon}}}, \omega_{\varphi_{H'_{\varepsilon}}})$ defined in \ref{class1}.   Let $J_+$ be an admissible almost complex structure on $X_+$ such that  it is $\mathbb{R} $-invariant in   the product region $\mathbb{R} \times (Y_{\varphi_{H_{\varepsilon}}'} \setminus ( S^1 \times \mathcal{U}^{\delta})) \subset X_+.   $   Let $Z_{ref}^+=[\mathbb{R} \times S^1 \times \mathbf{x}] \in H_2(X_+, \gamma^{\mathbf{x}}_{H_{\varepsilon}'},  \gamma^{\mathbf{x}}_{H_{\varepsilon}})$	 be
 the reference relative homology class. 	Then we have  $$PFC_{Z_{ref+}}^{sw}(X_+, \Omega_{X_+})_{J_+}(\alpha_I, Z_I) =(\alpha_I, Z_I) . $$
	\end{lemma}
	%By index reason, we have $CP_{sw}(X_+, \omega_+, J_+)(\gamma_-, Z) =(\gamma_-, Z) + \sum_I n_I(\gamma_{-I}, Z_I )$. %Here $\gamma_{-I} =\gamma_{-i_1} \cdots \gamma_{-i_d}$ and at least two $i_k=i_l$. $Z_I$ are relative homology class such that $E(Z_- -Z_I)=0$.  $n_I$ are unknown number and they  may be zero.  Note that the right hand side is a cycle.
	\begin{proof}
		Let $\alpha_-=\gamma_{p_1} \cdots \gamma_{p_d}$ be a PFH generator  of $\varphi_{H_{\varepsilon}}$.    Let  $Z_{\alpha_I, \alpha_-}=[S^1 \times \eta] \in H_2(X_+, \alpha_I, \alpha_-)$  be a reference relative homology class defined as in (\ref{eq19}). Any other relative homology  class $Z \in H_2(X_+, \alpha_I, \alpha_-)$ can be written as $Z=Z_{\alpha_I, \alpha_-}+  m_Z [\Sigma]+[S_Z]$, where $[S_Z] \in H_1(S^1, \mathbb{Z}) \otimes H_1(\Sigma, \mathbb{Z})$.  Then the ECH index of   $Z \in H_2(X_+, \alpha_I, \alpha_-)$ is
		\begin{equation} \label{eq32}
			I(\alpha_I, \alpha_-, Z) = 2d-h(\alpha_-) -2e_+(\alpha_-) +2m_Z (k+1),
		\end{equation}
		where $h(\alpha_-)$ is the number of hyperbolic  orbits   in $\alpha_-$ and $e_+(\alpha_-)$ is the total multiplicities of periodic orbits  $\gamma_{y_i^+}$ in $\alpha_-$.  One can check this formula by using the argument
 in Proposition 1.3 of  \cite{NW}.
The $\omega_{X_+}$-energy of $Z$ is $$E_{\omega_{X_+}}(Z)=\int_Z \omega_{X_+} =H_{\varepsilon}(\alpha_I) -H_{\varepsilon}(\alpha_-)+ m_Z. $$
		
		Suppose  that $\langle PFC_{Z_{ref+}}^{sw}(X_+, \Omega_{X_+})_{J_+}((\alpha_I, [Z_I])), (\alpha_-, [Z_-]) \rangle \ne 0$. Then the holomorphic curve axiom in Theorem \ref{thm3}  provides us a $J_+$-holomorphic current $\mathcal{C}$ with $I(\mathcal{C}) =0$.  Moreover, we have the following estimates:
		\begin{equation} \label{eq6}
			\begin{split}
				E_{\omega_{X_+}}(\mathcal{C})=&\int_{\mathcal{C}} \omega + d_{\Sigma} H_s \wedge dt + \partial_s H_s ds \wedge dt\\
				\ge &\int_{\mathcal{C}} \omega + d_{\Sigma} H_s \wedge dt - d\int_{S^1} \max_{\Sigma} |H_{\varepsilon}'-H_{\varepsilon}|dt\\
				\ge & -dc_0 \delta.
			\end{split}
		\end{equation}
		The last step follows from  the fact  that the term $\int_{\mathcal{C}} \omega + d_{\Sigma} H_s \wedge dt$ is nonnegative (see Lemma 3.8 in \cite{CHS1}).

		Let $Z$ be the relative homology class of $\mathcal{C}$. Then we have  $m_Z \ge 0$; otherwise,  $E_{\omega_{X_+}}(Z) \le -1+O(d\epsilon)< -d c_0 \delta$ which contradicts (\ref{eq6}).
	  By (\ref{eq32}),  $I(Z) = 0$ implies that $m_Z=h(\alpha_-) =0$  and $e_+(\alpha_-)=d$. In other words,   $Z=Z_{\alpha_I, \alpha_{I'}} $ and $\alpha_-$ only consists of periodic orbits  $\gamma_{y_i^+}$.
By Lemma \ref{lem12}  and the holomorphic axioms in Theorem
\ref{thm3}, we get the results.
	\end{proof}

	Let $(X_-, \Omega_{X_-})$ be the symplectic cobordism from  $(Y_{\varphi_{H_{\varepsilon}}}, \omega_{\varphi_{H_{\varepsilon}}})$ to $(Y_{\varphi_{H'_{\varepsilon}}}, \omega_{\varphi_{H_{\varepsilon}'}})$ defined as in (\ref{eq9}).   Let $J_-$ be an admissible almost complex structure such that $J_-=J_+$ on  the product region   $\mathbb{R} \times (Y_{\varphi_{H'_{\varepsilon}}} \setminus (S^1 \times \mathcal{U}^{\delta + \delta_0})) \subset X_-$.  Let $Z_{ref-}=[\mathbb{R} \times S^1 \times \mathbf{x}] \in H_2(X_-, \gamma_{H_{\varepsilon}}^{\mathbf{x}},  \gamma_{H'_{\varepsilon}}^{\mathbf{x}})$ be the reference relative homology class. By the composition law of the  PFH cobordism maps (Theorem \ref{thm3} ), we have
	\begin{equation} \label{eq33}
		PFC_{Z_{ref-}}^{sw}(X_-, \Omega_{X_-})_{ J_-} \circ PFC_{Z_{ref+}}^{sw}(X_+, \Omega_{X_+})_{ J_+} =\operatorname{Id} + K \circ \partial' + \partial' \circ K,
	\end{equation}
 where $\partial'$ is the PFH differential of $\varphi_{H'_{\varepsilon}}$ and $K : \widetilde{PFC}_*(\Sigma, \varphi_{H'_{\varepsilon}}, d) \to  \widetilde{PFC}_{*-1}(\Sigma, \varphi_{H'_{\varepsilon}}, d)$ is a homomorphism.
	By Lemma \ref{lem13} and (\ref{eq33}), we get a cycle
 \begin{equation}\label{eq63}
\mathfrak{c}_{\heartsuit} + K \circ \partial' \mathfrak{c}_{\heartsuit}\in \widetilde{PFC}(\Sigma, \varphi_{H_{\varepsilon}'},  d),
\end{equation}
	where $\mathfrak{c}_{{\heartsuit}} =\sum(\alpha_I, [Z_I])$  is defined by  \ref{class1}, \ref{class2}.
	 Also, it represents a non-vanishing class of   $\widetilde{PFH}(\Sigma, \varphi_{H_{\varepsilon}'},  d)$ because the cobordism map $\mathfrak{I}_{H_{\varepsilon}, H'_{\varepsilon}}=PFH_{Z_{ref-}}^{sw}(X_-, \Omega_{X_-})_{ J_-} $ is an isomorphism.

	\subsection{Unit of HF} \label{section6.3}
	In this subsection, we describe the unit  for  in $HF(\operatorname{Sym}^d\varphi_{H_{\varepsilon}}(\underline{\Lambda}),    \operatorname{Sym}^d\underline{\Lambda})$ (Lemma \ref{lem26}). Before we do that, let us recall the definition of unit as follows.
	
 Let $(D_0, \omega_{D_0}
 , j_0)$ be a disk with one boundary puncture, where $( \omega_{D_0}
 , j_0)$  is a K\"ahler structure. 	Assume that  $( \omega_{D_0}
 , j_0)$   is a disk with a negative strip-like end, i.e., there exists  a neighbourhood $U$ of the puncture of such that
 \begin{equation} \label{eq45}
(U, \omega_{D_0}
 , j_0) \cong \left( (-\infty, 0]_s \times [0,1]_t, ds\wedge dt, j\right),
 \end{equation}
 where $j$ is the standard complex structure that maps $\partial_s$ to $\partial_t$.

 Let $\mathbb{E}:=D_0 \times \mathbb{M}$, where $\mathbb{M}= \operatorname{Sym}^d \Sigma$.    Let $\Omega_{\mathbb{E}}$ be a symplectic form  such that $\Omega_{\mathbb{E}} =\omega_V + ds \wedge dt $ over the strip-like end.   Let   $\mathbb{L} \subset \partial D_0 \times \mathbb{M}$ be  a Lagrangian submanifold such that $$ \mathbb{L} \vert_{s\le -R} = \mathbb{R}_{s\le -R} \times (\{0\}\times  \operatorname{Sym}^d \varphi_{H}(\underline{\Lambda}) \cup \{1\}\times  \operatorname{Sym}^d \underline{\Lambda}). $$

Similar as defining the PFH cobordism maps, we need to fix a reference homology class $\mathcal{S}_{ref} \in H_2(\mathbb{E}, \emptyset, \textbf{x}_H)$. %We require that $\pi_{\mathbb{M}}(\mathcal{S}_{ref})$ doesn't bound a nontrivial class in $H_2( \mathbb{M}, \mathbb{L}_z; \mathbb{Z})$,  where $\mathbb{L}_z =\mathbb{L} \cap \pi_{\mathbb{E}}^{-1}(z)$ and $z \in \partial D_0$.
Let us clarify  the choice of  $\mathcal{S}_{ref}$. Let $\overline{D}_0$ be punctured disk with a positive strip-like end.  Let $\overline{\mathbb{E}} = \overline{D}_0 \times \mathbb{M}$ and $\overline{\mathbb{L}} = \partial \overline{D}_0  \times   \operatorname{Sym}^d \underline{\Lambda}$. %Take $\overline{\mathbb{S}}
 %=[\overline{D}_0 \times \mathbf{x}] \in H_2(\overline{\mathbb{E}}, \mathbf{x})$.
 Let $(\mathbb{R}_{s_0} \times [0,1] \times \mathbb{M}, \mathbb{L}_0)$
 be a Lagrangian cobordism from $( \operatorname{Sym}^d \varphi_H(\underline{\Lambda}),  \operatorname{Sym}^d\underline{\Lambda})$ to  $( \operatorname{Sym}^d\underline{\Lambda},  \operatorname{Sym}^d\underline{\Lambda})$ given by Lemma \ref{lem10}. Then we define the composition of this three pairs
 \begin{equation*}
	(\mathbb{E}_R, \mathbb{L}_R): = (\mathbb{E}, \mathbb{L}) \vert_{s \ge -R} \cup_{s=-R \sim s_0=R }  ( \mathbb{R}_{s_0} \times [0,1] \times\mathbb{M}, \mathbb{L}_0) \vert_{|s_0| \le R} \cup_{s_0=-R \sim s_- =R} (\overline{\mathbb{E}}, \overline{\mathbb{L}}) \vert_{s_- \le  R}.
\end{equation*}
Let  $\mathcal{S}_{\mathbf{x}} \in H_2(\mathbb{M}, \mathbf{x}_H, \mathbf{x})$  be  the relative homology class defined in Proposition \ref{lem30}. The composition $ \mathcal{S}_{ref} \# \mathcal{S}_{\mathbf{x}} \#[\overline{D}_0 \times \{\mathbf{x}\}] $ gives a class in  $H_2 (\mathbb{E}_R, \mathbb{L}_R, \mathbb{Z}).$
Note that  $H_2 (\mathbb{E}_R, \mathbb{L}_R, \mathbb{Z}) \cong  H_2 (D \times \mathbb{M}, \partial D \times  \operatorname{Sym}^d \underline{\Lambda}, \mathbb{Z})$, where $D$ is a closed disk.   Under this identification, we choose $\mathcal{S}_{ref}$ such that
 \begin{equation*}
  \mathcal{S}_{ref} \# \mathcal{S}_{\mathbf{x}} \#[\overline{D}_0 \times \{\mathbf{x}\}]   = D \times \{\mathbf{x}\} \in H_2 (D \times \mathbb{M}, \partial D \times  \operatorname{Sym}^d \underline{\Lambda}, \mathbb{Z}).
\end{equation*}

The triple  $(\mathbb{E}, \Omega_{\mathbb{E}}, \mathbb{L})$  induces a cobordism map
 \begin{equation}\label{eq30}
		 HF_{\mathcal{S}_{ref} }(\mathbb{E}, \Omega_{\mathbb{E}},  \mathbb{L}): \mathbb{F}[T^{-1}, T] \to  HF_{*}(  \operatorname{Sym}^d \varphi_H(\underline{\Lambda}),   \operatorname{Sym}^d \underline{\Lambda})
	\end{equation}
The cobordism map  $ HF_{\mathcal{S}_{ref} }(\mathbb{E}, \Omega_{\mathbb{E}},  \mathbb{L})$  is induced by
 \begin{equation*}
		 CF_{\mathcal{S}_{ref} }(\mathbb{E}, \Omega_{\mathbb{E}},  \mathbb{L}) (1) : =\sum_{(\textbf{y}, \mathcal{S})} \# \mathcal{M}^J_0(\emptyset, \textbf{y}, \mathcal{S}_{ref}\#\mathcal{S}) (\textbf{y},  \mathcal{S})
		 	\end{equation*}
at the  chain level, where $ \mathcal{M}^J_0(\emptyset, \textbf{y}, \mathcal{S}_{ref}\#\mathcal{S}) $ is the moduli space of $ \operatorname{ind}=0$ holomorphic sections $s : D_0 \to  \mathbb{E}$  such that $s \vert_{\partial D_0} \subset \mathbb{L} $ and $[s]= \mathcal{S}_{ref}\#\mathcal{S}$.    Also, the negative end of $s$ is asymptotic to $\textbf{y}\in  \operatorname{Sym}^d \varphi_H(\underline{\Lambda}) \cap  \operatorname{Sym}^d \underline{\Lambda}$.  The cobordism maps  $HF_{\mathcal{S}_{ref} }(\mathbb{E}, \Omega_{\mathbb{E}},  \mathbb{L})$ only depends on  $( \operatorname{Sym}^d \varphi_H(\underline{\Lambda}),  \operatorname{Sym}^d \underline{\Lambda})$.

\begin{definition}
For any $H$, define $\textbf{1}_H:=  HF_{\mathcal{S}_{ref} }(\mathbb{E}, \Omega_{\mathbb{E}},  \mathbb{L})(1) \in  HF_{*}( \operatorname{Sym}^d\varphi_H(\underline{\Lambda}),  \operatorname{Sym}^d\underline{\Lambda}).
$  The unit $\mathbf{1}_{\underline{\Lambda}} \in HF( \operatorname{Sym}^d \underline{\Lambda})$ is defined by  $\mathbf{1}_{\underline{\Lambda}}: = \textbf{j}_H(\textbf{1}_H)$.
\end{definition}
The above definition is well defined because of the functorial properties of   Lagrangian Floer homology and the diagram (\ref{eq41}).  From (\ref{eq17}), we know that the  quantitative  Heegaard Floer homology is non-vanishing, so is the  unit $\mathbf{1}_{\underline{\Lambda}}$.

 Now we describe the unit when $H$ is a small Morse function. Let $D_0$ be a disk with
 a negative strip-like end (see (\ref{eq45})) and $\mathbb{E}=D_0 \times \mathbb{M}$.  Obviously, $\pi_{\mathbb{E}}: \mathbb{E} \to D_0$ is a fiber bundle over $D_0$. To begin with, let us construct a symplectic form and  a Lagrangian explicitly  over $\mathbb{E}$. Let $\mathbb{K}$ be a time independent  Hamiltonian function on $\mathbb{M}$.  Let $(s,t)$ be the coordinates on the end of $D_0$.  Let $\chi: \mathbb{R} \to\mathbb{R}$ be a nonincreasing  cutoff function such that $\chi=1$ when $s \le -R_0$ and $ \chi=0$ when $s \ge -1$. Define a 2-form $$\omega_{0}: = \omega_V+ d(\chi(s) \mathbb{K}\wedge dt).$$  Define a diffeomorphism
\begin{equation*}
\begin{split}
\Phi: &\mathbb{R}_- \times [0,1] \times \mathbb{M} \to   \mathbb{R}_- \times [0,1] \times \mathbb{M}\\
&(s, t, \mathbf{x}) \to (s, t, (\varphi_{\mathbb{K}}^{\chi(s)t})^{-1}( \mathbf{x})).
\end{split}
\end{equation*}
Because $\Phi=\operatorname{Id}$ when $s \ge -1$, we   extend it to be $\operatorname{Id}$ over the rest of $\mathbb{E}$. Let $\varphi^t=\varphi_{\mathbb{K}}^t$.  Note that $(\varphi^t)^* \mathbb{K} =\mathbb{K} $ because $\mathbb{K}$ is time independent. By a direct computation, we have
\begin{equation*}
\begin{split}
&\Phi^{-1}_*(\partial_s)=\partial_s + t\dot{\chi}(s) X_{ \mathbb{K} } \circ\varphi^{\chi(s)t} \\
&\Phi^{-1}_*(\partial_t)=\partial_t + \chi(s)X_{ \mathbb{K} } \circ\varphi^{\chi(s)t}\\
&  \Phi^{-1}_*(v) = \varphi^{\chi(s)t}_*(v) \mbox{ for } v \in T\mathbb{M}.
\end{split}
\end{equation*}
Combining these ingredients, we get  a 2-form
\begin{equation}\label{eq16}
\omega_{\mathbb{E}}: =(\Phi^{-1})^* \omega_0 = \omega_V + t\dot{\chi}(s) ds \wedge d\mathbb{K} + \dot{\chi}(s) \mathbb{K} ds\wedge dt.
\end{equation}
 Note that $\omega_{\mathbb{E}}= \omega_V$ when $s \le -R_0$.  The symplectic form on $\mathbb{E}$ is defined by $\Omega_{\mathbb{E}}:= \omega_{\mathbb{E}} + \omega_{D_0}$.
Define $\mathbb{L} := \Phi(\partial D_0 \times \varphi_{\mathbb{K}}( \operatorname{Sym}^d\underline{\Lambda}))$ is a  $\Omega_{\mathbb{E}}$-Lagrangian such that $$\mathbb{L} \vert_{s \le -R_0} = \mathbb{R}_{s\le -R_0} \times ((\{0\} \times  \operatorname{Sym}^d\underline{\Lambda}\cup \{1\} \times \varphi_{\mathbb{K}}( \operatorname{Sym}^d\underline{\Lambda}) ).$$

Now we take $\mathbb{K}$ to be a Hamiltonian which is compatible with $ \operatorname{Sym}^d H_{\varepsilon}$, i.e., $\mathbb{K}$ satisfies  the properties (\ref{eq70}).%the following properties:
%\begin{equation*}
%\begin{split}
%&\mathbb{K}= \operatorname{Sym}^d H_{\varepsilon} \mbox{ outside the   diagonal $%\Delta$, }\\
%&\mathbb{K}   \mbox{ is a constant near diagonal.  }
%\end{split}
%\end{equation*}
We  choose $\mathbb{K}$ to be  a time independent Hamiltonian function such that $\mathbb{K} \le 0$ and the only maximum point of  $\mathbb{K}$  is $\mathbf{y}_{\heartsuit} =[y_1^+, ...,y_d^+]$.  This is feasible  due to the construction (Remark 6.8 of \cite{CHMSS}).
Note that  $\varphi_{\mathbb{K}}( \operatorname{Sym}^d \underline{\Lambda}) =  \operatorname{Sym}^d \varphi_{H_{\varepsilon}} (\underline{\Lambda})$. The intersection points of  $\varphi_{\mathbb{K}}( \operatorname{Sym}^d \underline{\Lambda}) \cap  \operatorname{Sym}^d \underline{\Lambda}$ are still described by Lemma  \ref{lem25}.

Let $\mathcal{J}_{\mathbb{E}}$ denote the set of  $\Omega_{\mathbb{E}}$-compatible almost complex structures on $\mathbb{E}$ satisfying the following conditions:
\begin{enumerate}
\item
$d\pi_{\mathbb{E}} \circ J =j_0 \circ  d\pi_{\mathbb{E}}$.
\item
$J$ preserves the vertical bundle  $T\mathbb{M}$. Moreover, $J \vert_{T\mathbb{M}}$ is compatible with $\Omega_{\mathbb{E}} \vert_{T\mathbb{M}}$.
\item
Over the end of $\mathbb{E}$, $J$ is $\mathbb{R}$-invariant, and $J(\partial_s) =\partial_t$. %$J$ preserves $TSym^d \Sigma$

\end{enumerate}

For each  $\mathbf{y}=[0,1] \times [y_1, ...,y_d] $, we construct a relative homology class $\mathcal{S}_{\mathbf{y}}$ as follows: Let $\eta = \cup_{i=1}^d\eta_i : \sqcup_i[0,1]_{s} \to \underline{\Lambda}$ be a $d$ union of paths in $\underline{\Lambda}$, where  $\eta_i \subset \Lambda_i$ satisfies $\eta_i(0)=  y_i$ and $\eta_i(1) =x_i$.  Let  $u_i(s, t) = (s, t, (\varphi_{H_{\varepsilon}}^t)^{-1}\circ \varphi_{H_{\varepsilon}}(\eta_i(s)))$. Then $u=\cup_{i=1}^d u_i$ is a $d$-multisection and it represents a class $A_{\textbf{y}} \in H_2(M, \textbf{x}_{H_{\varepsilon}}, \textbf{y})$.  Using the tautological correspondence, $u$ gives arise a relative homology class  $\mathcal{S}_{\textbf{y}} \in H_2(\mathbb{M}, \textbf{x}_{\mathbb{K}}, \textbf{y})$.  Note that $u_i \cap u_j =\emptyset$ for $i\ne j $. Hence,   $\mathcal{S}_{\textbf{y}} \cdot \Delta =0 $.   By the same  computation  in (\ref{eq13}), we know that $\int_{\mathcal{S}_{\textbf{y}}} \omega =\mathbb{K}(\textbf{x}) -\mathbb{K}(\textbf{y})= H_{\varepsilon}(\textbf{x}) - H_{\varepsilon}(\textbf{y})$.

\begin{lemma} \label{lem26}
Take $\mathbb{K}$ to be the above function. Let $\mathbf{y}_{\heartsuit} = [0,1 ] \times [y_1^+, ..., y_d^+]$.  Let $J \in \mathcal{J}_{\mathbb{E}}$ be a generic almost complex structure.  Let $\mathcal{S}_{ref} \in H_2(\mathbb{M}, \emptyset,  \mathbf{x}_H)$ be the reference class that is represented by $\Phi(D_0 \times \varphi_{\mathbb{K}}(\mathbf{x})) $.  Then  we have
$$CF_{\mathcal{S}_{ref}}(\mathbb{E},   \Omega_{\mathbb{E}}, \mathbb{L})_J(1) =( \mathbf{y}_{\heartsuit}, [\mathcal{S}_{\mathbf{y}_{\heartsuit}}]). $$
In particular, $( \mathbf{y}_{\heartsuit}, [\mathcal{S}_{\mathbf{y}_{\heartsuit}}])$ is a cycle representing   the unit.
\end{lemma}
\begin{proof}
 Let $ \mathcal{M}^J(\emptyset, \textbf{y}, \mathcal{S})$ be the moduli space of holomorphic sections in  $\mathbb{E}$ with Lagrangian boundary condition $\mathbb{L}$.   Let $u \in \mathcal{M}^J(\emptyset, \textbf{y}, \mathcal{S})$ be a curve contributing  to $CF_{\mathcal{S}_{ref}}(\mathbb{E},   \Omega_{\mathbb{E}}, \mathbb{L})_J(1)$. Note that
  \begin{equation*}
\begin{split}
&\int_{\mathcal{S}_{ref}} \omega_{\mathbb{E}} =  \int_{D_0 \times \varphi_{\mathbb{K}}(\textbf{x})} \omega_V+ d(\chi(s) \mathbb{K}\wedge dt) =-\mathbb{K}(\textbf{x}) \mbox{ and } \mathcal{S}_{ref} \cdot \Delta =0.
\end{split}
\end{equation*}
Let $\mathcal{S}_0 \in H_2(\mathbb{M}, \textbf{x}_{\mathbb{K}}, \textbf{y})$ be the class determined  by $\mathcal{S} =\mathcal{S}_{ref} \# \mathcal{S}_0$. Then
\begin{equation} \label{eq39}
\begin{split}
&\int u^*\omega_{\mathbb{E}} =  \int_{\mathcal{S}}  \omega_{\mathbb{E}} =\int_{\mathcal{S}_{ref}}  \omega_{\mathbb{E}} + \int_{\mathcal{S}_{0}}  \omega_{\mathbb{E}}  =-\mathfrak{A}_{\mathbb{K}}( \textbf{y}, \mathcal{S}_0), \\
& u\cdot \Delta = \mathcal{S}_{ref}\cdot \Delta + \Delta \cdot \mathcal{S}_0 = \Delta \cdot \mathcal{S}_0.
\end{split}
\end{equation}
On the other hand, we have
\begin{equation} \label{eq62}
\int u^*\omega_{\mathbb{E}} = \int |d^{vert} u|^2 + \omega_{\mathbb{E}}(d^{hor}u, J^{hh} d^{hor}u),
\end{equation}
where $d^{vert} u \in T^{vert}{\mathbb{E}} \cong T\mathbb{M}$ and $d^{hor} u \in T^{hor}{\mathbb{E}}: =\{ v \in T \mathbb{E}: \omega_{\mathbb{E}} (v, w)=0, \forall w\in T^{hor}{\mathbb{E}}\}$ are respectively the vertical and horizontal components of $du$.
By (\ref{eq16}), one can  check that $T^{hor}\mathbb{E}=span\{\partial_s -t\dot{\chi} X_{\mathbb{K}}, \partial_t\} $.  Therefore, $\omega_{\mathbb{E}} \vert_{T^{hor}\mathbb{E}} = \dot{\chi} \mathbb{K} \omega_{D_0}$.  Recall that we choose $\mathbb{K} \le 0$. Hence,  $\int u^*\omega_{\mathbb{E}} = \int |d^{vert} u|^2 + \dot{\chi} \mathbb{K} |d^{hor}u|^2 \ge 0. $  Furthermore, we  choose $J$ such that $J=\operatorname{Sym}^d j_z$ near the diagonal. Then,  $ u\cdot \Delta \ge 0$ by intersection positivity and $\mathbb{K}$ is a constant near the diagonal.  Therefore, (\ref{eq39}) and  (\ref{eq62}) imply that
\begin{equation}  \label{eq47}
\begin{split}
&\int u^*\omega_{\mathbb{E}} + \eta u\cdot \Delta  =  -\mathfrak{A}^{\eta}_{\mathbb{K}}( \mathbf{y}, \mathcal{S}_0) \ge 0.
\end{split}
\end{equation}
Write $\mathcal{S}_0= \mathcal{S}_{\mathbf{y}} + \sum_{i=1}^{k+1} c_i \Psi_*([B_i])$.  By Remark \ref{remark8}, we have
\begin{equation}  \label{eq40}
\begin{split}
&0= \operatorname{ind} u = n(\textbf{y}) + \sum_{i=1}^{k+1} 2c_i\\
&\mathfrak{A}^{\eta}_{\mathbb{K}}( \textbf{y}, [\mathcal{S}_0]) =\mathfrak{A}_{\mathbb{K}}( \textbf{y}, [\mathcal{S}_0])- \eta \Delta \cdot \mathcal{S}_0\\
=&  H_{\varepsilon}(\textbf{y}) - \sum_{i=1}^k c_i \lambda -c_{k+1}\left(\int_{B_{k+1}} \omega +2\eta(d+g-1)\right)\\
=&H_{\varepsilon}(\textbf{y}) - \lambda\sum_{i=1}^k c_i,
\end{split}
\end{equation}
where $n(\textbf{y})$ is the number of $y_i^-$-components.   By  (\ref{eq47}) and (\ref{eq40}), we know that  $ \mathbf{y} =\mathbf{y}_{\heartsuit}$,  $\int u^*\omega_{\mathbb{E}} =0$   and $d^{vert} u=0$.     Therefore, the  horizontal section $u=D_0 \times \{\textbf{y}_{\heartsuit}\}$ is the only holomorphic curve contributing to $CF_{\mathcal{S}_{ref}}(\mathbb{E},   \Omega_{\mathbb{E}}, \mathbb{L})_J(1)$.
\end{proof}

By the construction of the isomorphism $\Phi_{H_{\varepsilon}}$ in Theorem \ref{thm1},   we have $$\Phi_{H_{\varepsilon}}((\textbf{y}_{\heartsuit}, [A_{\textbf{y}_{\heartsuit}}])) =(\textbf{y}_{\heartsuit}, \mathcal{S}_{\textbf{y}_{\heartsuit}})$$
at the chain level.  Then, the above lemma implies that $(\mathbf{y}_{\heartsuit}, [A_{\mathbf{y}_{\heartsuit}}])$ is a cycle and represents a  non-zero class  $e_{H_{\varepsilon}} \in HF(\Sigma,  \varphi_{H_{\varepsilon}}(\underline{\Lambda}), \underline{\Lambda})$.  Moreover,  $j_{H_{\varepsilon}}(e_{H_{\varepsilon}})$ is the unit defined in Definition \ref{definition4}.

	\subsection{Proof of the non-vanishing result }
 With the preparations in the last three sections, now we can compute the closed-open morphism $\widetilde{\mathcal{CO}}(\underline{\Lambda}, H'_{\varepsilon})$.	
 
 \begin{comment}
 Before we prove the result, first note that for two different base $\mathbf{x}, \mathbf{x}'$, we have  the following diagram:
	 \begin{equation*}
 \begin{CD}
				\widetilde{PFH}_*(\Sigma,  \varphi_H, \gamma_H^{\mathbf{x}})_{J_H} @> \widetilde{\mathcal{CO}}(\underline{\Lambda}, H)_{J_H}>> HF_{*}(\Sigma,   \underline{\Lambda}, \varphi_H,  \mathbf{x})_{J_H} \\
				@VV \Psi^{pfh}_{H, {\textbf{x}, \textbf{x}'}} V @VV \Psi_{H, {\textbf{x}, \textbf{x}'}}V\\
				\widetilde{PFH}_*(\Sigma,  \varphi_H,  \gamma_H^{\mathbf{x}'})_{J_H} @>\widetilde{\mathcal{CO}}(\underline{\Lambda}, H)_{J_H}>> HF_{*}(\Sigma, \underline{\Lambda}, \varphi_H,    \mathbf{x}')_{J_H},
			\end{CD}
				\end{equation*}
where $\widetilde{\mathcal{CO}}(\underline{\Lambda}, H)_{J_H}$
on the upper arrow is defined by $\mathbf{x}$ and the one on the lower arrow
 is defined by $\mathbf{x}'$. Therefore, the choice of the base point is not important. % In this  subsection, we take $\mathbf{x}=\mathbf{y}_{\heartsuit}$ throughout.
\end{comment}

	Reintroduce the   closed-open symplectic cobordism $(W, \Omega_{\varphi_H}, L_{\Lambda_{H}})$.
	Let $ \mathcal{J}(W, \Omega_{\varphi_H})  \subset  \mathcal{J}_{tame}(W, \Omega_{\varphi_H}) $  be the set of almost complex structures which are the restriction of admissible almost complex structures in  $\mathcal{J}(Y_{\varphi_{H}}, \omega_{\varphi_{H}} )$. We compute the closed-open morphism using the almost complex structures in  $ \mathcal{J}(W, \Omega_{\varphi_H})$ instead of  $ \mathcal{J}_{tame}(W, \Omega_{\varphi_H}) $.

	Let $u_{y_i^{\epsilon}} =B \times \{y^{\epsilon}_i\} / (0, y_i^{\epsilon}) \sim (2, y_i^{\epsilon})$, where $\epsilon \in \{+, -\}$.  Take  a   $J \in \mathcal{J}(W, \Omega_{\varphi_{H'_{\varepsilon}}}) $. Then  $u_{y_i^{\epsilon}}$  is a $J$-holomorphic PFH-HF curve in $\mathcal{M}^J(\gamma_{y^{\epsilon}}, y^{\epsilon})$.  It is called a \textbf{horizontal section of    $(W, \Omega_{\varphi_{H_{\varepsilon}'}}, L_{\Lambda_{H'_{\varepsilon}}}, J)$}. Moreover, it is to check  that 
\begin{equation*}	
	  \operatorname{ind} u_{y_i^{+}} =0 \mbox{ and } \int  (u_{y_i^{\epsilon}})^* {\omega_{\varphi_{H'_{\varepsilon}}}}  =0 
	\end{equation*} from the definition.

The strategy of proving the non-vanishing result is the same as Lemma \ref{lem26}. We
 will show that the horizontal sections are the only PFH-HF curves contributing to
$\widetilde{\mathcal{CO}} (\underline{\Lambda}, H'_{\varepsilon})_J$. We begin with the following two lemmas concerning the basis properties
 of the horizontal sections.
	\begin{lemma} \label{lem22}
	Given 	$J \in \mathcal{J}(W, \Omega_{\varphi_{H}}) $,  let $u$ be a $J$-holomorphic PFH-HF curve in   $(W,  \Omega_{\varphi_H}, L_{\Lambda_{H}})$.  Then $$E_{\omega_{\varphi_{H}} }(u)=\int_{\dot{F}} u^* \omega_{\varphi_H}  \ge 0.$$ Moreover, when $H=H'_{\varepsilon}$, $E_{\omega_{\varphi_{H'_{\varepsilon}}}} (u) =0$ if and only if $u$ is a union of  the horizontal sections $\cup_i u_{y_i^{\epsilon_i}}$.
	\end{lemma}
	\begin{proof}
		Note that the almost complex structure $J=\begin{bmatrix}
			J^{hh} & 0 \\
			0 & J^{vv}
		\end{bmatrix}
		$ with respect to the splitting $TW=TW^{hor} \oplus TW^{vert}$,  where $TW^{vert} =\ker (\pi_W)_*$  and $TW^{hor}$ is the $\Omega_{\varphi_H}$-orthogonal complement. By direct computations,  we have
		\begin{equation*}
			\int_{\dot{F}} u^* \omega_{\varphi_H} = \int_{\dot{F}} |d^{vert}u|^2 \ge 0.
		\end{equation*}
		In the case that $H=H'_{\varepsilon}$, then $E_{\omega_{\varphi_{H'_{\varepsilon}}}} (u_{y_i^{\epsilon}}) =0$ by definition.  Conversely, if  $E_{\omega_{\varphi_{H'_{\varepsilon}}}} (u) =0$,  we have $du \in T^{hor}W$. Therefore, the negative ends of $u$ must lie inside the trivial strips. This implies that $u=u_{y_i^{\epsilon}}$ provided that  $u$ is irreducible.
	\end{proof}

 Now we use a smaller set of almost complex structures	$J \in \mathcal{J}(W, \Omega_{\varphi_{H}}) $. Lemma \ref{lem33}  needs to be replaced by the following lemma.

 \begin{lemma} \label{lem34}
We have the following statements about the transversality:
\begin{enumerate}
  \item
  There is a Baire subset of  $
 \mathcal{J}(W, \Omega_{\varphi_{H}}) $, denoted by $  \mathcal{J}^{reg}(W, \Omega_{\varphi_{H}}) $. For any $J\in \mathcal{J}(W, \Omega_{\varphi_{H}})$,
 if $u$ is not a horizontal section, then $u$ is Fredholm regular.
  \item
  For any $J\in \mathcal{J}(W, \Omega_{\varphi_{H}})$,
 if $u$ is   a horizontal section with $ \operatorname{ind} u=0$, then $u$ is Fredholm regular.
\end{enumerate}
 \end{lemma}
 \begin{proof}
  The proof of the first statement is the same as Lemma 9.12 of \cite{H1}.

  To see the second statement, we can follow the argument in Lemma 2.11 of  \cite{PS}.   Let $u$ be a horizontal section and $u_{\tau}:=\operatorname{exp}_u(\tau \psi)$, where $\psi \in W^{1,p}(u^*T^{vert}W)$ ($p>2$) such that $\psi  \in u^* TL_{\Lambda_H} $ along $\partial \dot{F}$. By $\int u_{\tau}^* \omega_{\varphi_H} =0$ and a direct  computation, we have
  \begin{equation*}
    \int |d^{vert} u_{\tau}|^2 = \int |\bar{\partial}_J^{vert} u_{\tau}|^2,
  \end{equation*}
  where  $\bar{\partial}_J^{vert} u_{\tau}:=d^{vert}u_{\tau} + J^{vv}\circ d^{vert} u_{\tau} \circ j.  $ Differentiate the above equation twice at
 $\tau=0$, we have  $\int |\nabla \psi|^2 =\int |D_u \psi|^2$, where   $D_u$ is the vertical part of the linearization
 of $\bar{\partial}_J$ at $u$. Therefore, if $ \psi \in \ker D_u$, then $\nabla \psi =0$ . Then,  the norm of $\psi$ is a constant.
 As a result, $\psi$ cannot belong to $W^{1, p}$ unless $\psi \equiv 0$. Therefore, $ \ker D_u=0$. If $ \operatorname{ind}  u = 0$,
 then $\operatorname{coker} D_u = 0$. In other words, $u$ is Fredholm regular.
 \end{proof}

	\begin{remark} \label{remark9}
  Define the closed-open morphism $\widetilde{\mathcal{CO}}(\underline{\Lambda}, H)_{J'}$ by using a generic almost complex structure $J'\in \mathcal{J}_{tame}(W,\Omega_{\varphi_{H}}
 ) $ that is sufficiently close to  $J\in \mathcal{J}^{reg}(W,\Omega_{\varphi_{H}}
 ) $.
 Lemma \ref{lem34} tells us that $\#\mathcal{M}^{J'}(\alpha,\mathbf{y}, \mathcal{Z}) = \# \mathcal{M}^{J}(\alpha,\mathbf{y}, \mathcal{Z})$ provided that $I(\mathcal{Z})=0$.
 Therefore, we  can also define   $\widetilde{\mathcal{CO}}(\underline{\Lambda}, H)_{J} $ by using $J\in \mathcal{J}^{reg}(W,\Omega_{\varphi_{H}})$.
\end{remark}
	
	%We modify $\phi_H$ such that  all the elliptic orbits are $d$-negative, denoted by $\phi_{H'}$. Note that the modification  is supported in a neighbourhood of  the local maximum of $H$.
	\begin{comment}
	\begin{lemma}
	The closed open morphism  $\Phi_{\mathcal{Z}_0}(W, \Omega_{H_{\varepsilon}'}, L_{\Lambda_{H'_{\varepsilon}}})_J  $ is non--vanishing.
	\end{lemma}
	\begin{proof}
	Note that the horizontal section $u_y =\cup u_{y_i}$ has zero ECH index. Also, we have $E_{\omega_{H'}}(u_y)=0$.  By the above lemma, the horizontal section is the unique element in $\mathcal{M}^J(\gamma_{y}, y, A_{tri}) $. Hence,  $\#\mathcal{M}^J(\gamma_{y}, y, A_{tri}) =1$, where $A_{tri}$ is represented by a union of horizontal sections.  Therefore,
	\begin{equation*}
	\Phi_{A_0, J}(W, \Omega_{f'}, {L_{\Lambda_{f'}}}) (\gamma_y, Z) = (y, B) + \sum_{I(A)=0, E(A)>0}\#\mathcal{M}^J(\gamma_{y}, y', A)    (y', B')t^{E(A)}.
	\end{equation*}
	By Seidel's Lemma, $\Phi_{A_0, J}(W, \Omega_{H'}, {L_{\Lambda_{H'}}}) $ is surjective.
	
	% Let $\gamma_-=\{(\gamma_{y^-_{i}}, 1)\}$ be the orbit set.   Let $\alpha_-$ be the image of $(\gamma_-, Z_-)$ under the cobordism maps from $f$ to $f'$. We claim that $\Phi_{A_0, J}(W, \Omega_{f'}, {L_{\Lambda_{f'}}}) _*([\alpha_-)  )\ne 0$.
	
	%Let $I$ be the cobordism maps from $HF(f')$ to $HF(f)$.
	\end{proof}
	\end{comment}

 In the next lemma, we compute the ECH index, $J_0$ index and energy of PFH-HF curves. These information helps us determining what kinds of holomorphic curves
 contributing to $\widetilde{\mathcal{CO}}(\underline{\Lambda}, H'_{\varepsilon})$.  Let  $\eta : \mathbb{R} \to \Lambda_i$ be a path in the hemicycle such that $\lim_{s\to \infty}\eta(s) = y_i^+$ and $\lim_{s\to -\infty}\eta(s) = y_i^-$. Define $u_i^+(s,t) = (s,t, \varphi_{H'_{\varepsilon}}\circ (\varphi_{H'_{\varepsilon}}^t
 )^{-1}(\eta(s)))$.
 Let  $u_i^-$ be the trivial strip from $y_i^+$ to itself. Let $u_{\mathbf{y}} = \cup_{i=1}^d u_i^{\epsilon_i} \in \mathcal{M}^J(\mathbf{y}_{\heartsuit}, \mathbf{y})$, where
 $\epsilon_i \in\{+,-\}$ and they are determined by $\mathbf{y}$.  Let $u_{\heartsuit}: =\cup_{i=1}^d u_{y^+_i} \in \mathcal{M}^J(\alpha_{\heartsuit}, y_{\heartsuit})$ be a union
 of the horizontal sections from  $\alpha_{\heartsuit}$ to $ y_{\heartsuit}$. Then $Z_{\alpha_I, \alpha_{\heartsuit}} \# u_{\heartsuit} \#u_{\mathbf{y}}$ represents a relative
 homology class $\mathcal{Z}_0 \in H_2(W_{\varphi_{H'_{\varepsilon}}} ,\alpha_I,\mathbf{y})$.
	
	\begin{lemma} \label{lem14}
		Let  $\mathcal{Z} \in H_2(W_{\varphi_{H'_{\varepsilon}}}, \alpha_I, \mathbf{y})$ be a relative homology class such that $\mathcal{Z} =\mathcal{Z}_0 +m[\Sigma] + \sum_{i=1}^{k+1} c_i[B_i] + [S]$, where $[S] \in H_1(S^1, \mathbb{Z}) \otimes H_1(\Sigma, \mathbb{Z})$.  Then the ECH index, $J_0$ index  and the energy are respectively
		\begin{equation*}
			\begin{split}
				&	I(\mathcal{Z}) =n(\mathbf{y}) + 2m(k+1) + 2 \sum_{i=1}^{k+1} c_i, \\
&J_0(\mathcal{Z}) = 2m(d+g-1)+2c_{k+1}(d+g-1), \\
		\mbox{and }		&	E_{\omega_{\varphi_{H'_{\varepsilon}}}}(\mathcal{Z}) = H_{\varepsilon}'(\alpha_I)-H'_{\varepsilon}(\mathbf{y}) + m+  \sum_{i=1}^k \lambda c_i + c_{k+1} \int_{B_{k+1}} \omega,\\
			\end{split}
		\end{equation*}
		where $n(\mathbf{y})$ is the number of $y_i^-$ in $\mathbf{y}$.
		Moreover, if $I(\mathcal{Z})=0$, then $E_{\omega_{\varphi_{H'}}}(\mathcal{Z}) + \eta J_0(\mathcal{Z})\le 0$ and ``$=$" holds if and only if $n(\mathbf{y})=0$.
	\end{lemma}
	\begin{proof}
		The relative homology classes  $Z_{\alpha_I, \alpha_{I'}} \in H_2(Y_{\varphi_{H'_{\varepsilon}}}, \alpha_I, \alpha_{I'})$ satisfy $$I(Z_{\alpha_I, \alpha_{I'}})=J_0(Z_{\alpha_I, \alpha_{I'}})=E_{\omega_{\varphi_{H'_{\varepsilon}}}}( Z_{\alpha_I, \alpha_{I'}}) =0.$$ By the additivity of ECH index, $J_0$ index and energy, it suffices to prove the statement for $\alpha_I=\alpha_{\heartsuit}$.
		
  By the similar computation in (\ref{eq13}), the energy of $u^+_i$ is
		$$E(u^+_i) =\int (u^+_i)^*\omega =H'_{\varepsilon}(y^+_i) -H_{\varepsilon}'(y_i^-).$$
		%	\begin{equation*}
		%		\begin{split}
		%		E(u^+_i) &=\int_{\mathbb{R} \times [0,1]} {u_i^+}^* \omega =\int_{\mathbb{R} \times [0,1]} \omega (X_{H'_{\varepsilon}} \circ \eta(s), \partial_s\eta) ds \wedge dt\\ & =\int_{\mathbb{R} \times [0,1]} dH'_{\varepsilon}(\partial_s \eta) ds \wedge dt =H'_{\varepsilon}(y^+_i) -H_{\varepsilon}'(y_i^-).
		%	\end{split}
		%	\end{equation*}
		Moreover, it is easy to check that  $ \operatorname{ind} u^+_i= 1$, $ \operatorname{ind}  u^-_i=0$, and $J_0(u_i^{\pm})=0$ by definition.	
		
   Then,  we have
    \begin{equation*}
   I(u_{\mathbf{y}}) = \sum_{i=1}^d I(u_i^{\epsilon_i}) =n(\mathbf{y}), J_0(u_{\mathbf{y}}) =0, \mbox{ and } E_{\omega}(u_{\mathbf{y}}) = H_{\varepsilon}'(\mathbf{y}_{\heartsuit})-H'_{\varepsilon}(\mathbf{y}).
 \end{equation*}
By definition, we also have $ I(u_{\heartsuit}) = \sum_{i=1}^d I(u_{y_i^+}) =0$, $J_0(u_{\heartsuit}) =0$,  and  $E_{\omega_{\varphi_{H'_{\varepsilon}}}}(u_{\heartsuit}) =0. $  By the additivity of the ECH index, $J_0$ index  and energy, we have $$ I(\mathcal{Z}_0) =n(\textbf{y}), J_0(\mathcal{Z}_0) =0, \mbox{ and }
		E_{\omega_{\varphi_{H'_{\varepsilon}}}}(\mathcal{Z}_0) = H_{\varepsilon}'(\textbf{y}_{\heartsuit})-H'_{\varepsilon}(\textbf{y}). $$ The first statement of the lemma follows from  Lemma \ref{lem38}, Lemma \ref{lem29}, and  $\mathcal{Z}=\mathcal{Z}_0 + m [\Sigma] +\sum_{i=1}^{k+1}  c_i[B_i] +[S] $.

		If $I(\mathcal{Z}) =0$,   then formulas on $I(\mathcal{Z})$,  $J_0(\mathcal{Z})$, 	$E_{\omega_{\varphi_{H'_{\varepsilon}}}}(\mathcal{Z}) $ and assumption \ref{assumption4} imply
 that $$ E_{\omega_{\varphi_{H'_{\varepsilon}}}}(\mathcal{Z}) +\eta J_0(\mathcal{Z}) =H_{\varepsilon}'(\textbf{y}_{\heartsuit})-H_{\varepsilon}'(\textbf{y}) - \frac{\lambda n(\mathbf{y})}{2}. $$ If $n(\textbf{y}) \ge 1$, then  $E_{\omega_{\varphi_{H'_{\varepsilon}}}}(\mathcal{Z}) + \eta J_0(\mathcal{Z}) <0$.  If $n(\textbf{y})=0$, then $\textbf{y}=\textbf{y}_{\heartsuit}$ and  $ E_{\omega_{\varphi_{H'_{\varepsilon}}}}(\mathcal{Z}) + \eta J_0(\mathcal{Z}) =0$.
	\end{proof}

	\begin{lemma} \label{lem36}
		We have $\widetilde{\mathcal{CO}}(\underline{\Lambda}, H'_{\varepsilon})_J( K\circ \partial' (\alpha_I, Z_I)) =0$,  where $\partial'$ is the differential
 of $\widetilde{PFH}(\Sigma, \varphi_{H'_{\varepsilon}}, d)$ and $K$ is the chain homotopy in (\ref{eq63}). In particular, we have
 $\widetilde{\mathcal{CO}}(\underline{\Lambda}, H'_{\varepsilon})_J(K\circ \partial' \mathfrak{c}_{\heartsuit})=0.$
	\end{lemma}
	\begin{proof}
		Suppose that the statement is not true, then we get a chain of  holomorphic  currents $\mathcal{C}=( \mathcal{C}_1, \mathcal{C}_2, \mathcal{C}_3)$  from $\alpha_I$ to $\textbf{y}$ with total zero ECH index, where $\mathcal{C}_1$ is a  holomorphic current in $\mathbb{R} \times Y_{\varphi_{H_{\varepsilon}'}}$ with $I=1$, $\mathcal{C}_2$ is a $J_+ \circ_R J_-$ broken holomorphic current in $X_+ \circ_R X_-$  with $I=-1$, and $\mathcal{C}_3$ is a PFH-HF curve with $I=0$.   Here $\mathcal{C}_1$ comes from the definition of
 $\partial'(\alpha_I,Z_I)$, $\mathcal{C}_2$ comes from $K$ and the holomorphic curve axioms (Theorem \ref{thm4}), and $\mathcal{C}_3$
 comes from the definition of closed-open morphisms. The holomorphic current $\mathcal{C}_1\#\mathcal{C}_2\#\mathcal{C}_3$ represents  a
 class $\mathcal{Z} \in H_2(W,\alpha_I, \mathbf{y})$.

		Reintroduce the  periodic orbits $\gamma^i_{r_0, \theta_0}$ (\ref{eq27}) near the local minimums  of $H_{\varepsilon}$.  Here the superscript ``$i$'' indicates that the local minimum lies in the  domain $\mathring{B}_i$ ($1\le i\le k+1$). In particular,     $\gamma^i_{r_0, \theta_0}$ lies in $S^1 \times \mathring{B}_i$.  Let $v_i:= (\mathbb{R}\times \gamma_{r_0, \theta_0}^i) $.      Then  for any relative class $\mathcal{Z}' \in H_2(W, \alpha_I, \textbf{y})$, we have a well-defined intersection number
		\begin{equation*}
			n_i(\mathcal{Z}') : =\# (\mathcal{Z}' \cap v_i).
		\end{equation*}
Here we regard  $\mathcal{Z}' \subset W \subset \mathbb{R} \times Y_{\varphi_{H'_{\varepsilon}}}.$
	Recall that the relative homology class $\mathcal{Z} \in H_2(W, \alpha_I, \textbf{y})$ in the first paragraph can be written as
\begin{equation*}
\mathcal{Z}=Z_{\alpha_I, \alpha_{\heartsuit}} \#\mathcal{Z}_{hor} \#[u_{\mathbf{y}}] + \sum_{i=1}^{k+1} c_i[B_i] +m[\Sigma] + [S],
 \end{equation*}
 where $\mathcal{Z}_{hor}$ is the class represented by the union of horizontal sections and $u_{\mathbf{y}}$ are the curves defined in Lemma \ref{lem14}.
	 Let $q_i$ denote the period of $\gamma^i_{r_0, \theta_0}$. Note that the period $q_i$ is determined by the formula  (\ref{eq35}) and $r_0$,  we can choose $\varepsilon$  and $r_0$ such that $q_i=q$ for any $i$.   By  definition, we have
	 \begin{equation} \label{eq12}
	 	n_i(Z_{\alpha_I, \alpha_{\heartsuit}} \#\mathcal{Z}_{hor} \#  [u_{\mathbf{y}}] )=0, \ \  n_i([B_j])=\delta_{ij}q, \ \ n_i([S])=0 \ \  \mbox{and } n_i([\Sigma])=q
	 \end{equation}
	 for $1\le i, j \le k+1$.
	 From (\ref{eq12}), we know that
 \begin{equation*}
 \#(\mathcal{C} \cap (\sqcup_{i=1}^{k+1} v_i) )=\sum_{i=1}^{k+1} n_i(\mathcal{Z})=\sum_{i=1}^{k} c_i q +(k+1)mq=-\frac{qn(\mathbf{y})}{2}.
\end{equation*}
The last step is because %$E_{\omega_{\varphi_{H'_{\varepsilon}}}} (\mathcal{Z})=0$.
$I(\mathcal{Z})=0$. Note that $v_i$ are $J$-holomorphic for $J \in \mathcal{J}(W, \Omega_{\varphi_{H'_{\varepsilon}}})$. By the intersection positivity of holomorphic curves, we have $n(\mathbf{y})=0$ and $\mathbf{y} = \mathbf{y}_{\heartsuit}$. Moreover, $\mathcal{C}_2$ does not  intersect  $\mathbb{R} \times \gamma^i_{r_0, \theta_0}$.   Since the choice of $\theta_0$ is arbitrary, $\mathcal{C}_2$ lies inside the product region of $X_+ \circ_R X_-$ and $ E_{\omega_{\varphi_{H'_{\varepsilon}}}}(\mathcal{C}_2) \ge 0$.    $ E_{\omega_{\varphi_{H'_{\varepsilon}}}}(\mathcal{C}_1) > 0$  because the only holomorphic currents
 in  $\mathbb{R} \times Y_{\varphi_{H'_{\varepsilon}}}$ with zero energy
are trivial cylinders.  By Lemma \ref{lem22},  we have
$$ E_{\omega_{\varphi_{H'_{\varepsilon}}}} (\mathcal{Z})= E_{\omega_{\varphi_{H'_{\varepsilon}}}}(\mathcal{C}_1)  +   E_{\omega_{\varphi_{H'_{\varepsilon}}}}(\mathcal{C}_2) +  E_{\omega_{\varphi_{H'_{\varepsilon}}}}(\mathcal{C}_3) >0.$$
By  Lemma \ref{lem27}, we have $J_0(\mathcal{C}_1), J_0(\mathcal{C}_2) \ge 0$, because $\mathcal{C}_2$ lies in the product region.  By Lemma \ref{lem29}, we have $J_0(\mathcal{Z}) = {J}_0(\mathcal{C}_1) + {J}_0(\mathcal{C}_2)+ {J}_0(\mathcal{C}_3) \ge 0$. Consequently, we obtain $$E_{\omega_{\varphi_{H'_{\varepsilon}}}} (\mathcal{Z}) + \eta J_0(\mathcal{Z}) >0.$$
This  contradicts with Lemma \ref{lem14}.

	\end{proof}
	
	\begin{lemma} \label{lem35}
We have
 $$\widetilde{\mathcal{CO}}(\underline{\Lambda}, {H'_{\varepsilon}})_J(  (\alpha_I, [Z_I])) =0$$  for $\alpha_I \ne  \alpha_{\heartsuit}$. Moreover, we have
		$$\widetilde{\mathcal{CO}}(\underline{\Lambda}, {H'_{\varepsilon}})_J(  (\alpha_{\heartsuit}, [Z_{\alpha_{\heartsuit}, \alpha_{\heartsuit}}])) =(\mathbf{y}_{\heartsuit}, [A_{\mathbf{y}_{\heartsuit}}]). $$
	%	The class $A_+$ is determined  by the relation
	%	\begin{equation}
	%%	\end{equation}
		%where  $\mathcal{Z}_{hor}$ is the class represented by the horizontal section.
	\end{lemma}
	\begin{proof}
		By  Lemmas \ref{lem29}, \ref{lem22}, \ref{lem14},  if $\langle \widetilde{\mathcal{CO}}(\underline{\Lambda}, {H'_{\varepsilon}})_J(  (\alpha_I, [Z_I])) , (\mathbf{y}, [A]) \rangle \ne 0 $, then $\mathbf{y}$ must be $\mathbf{y}_{\heartsuit}$ and any PFH-HF curve $u$ contributing  to the closed-open morphism has  $E_{\omega_{\varphi_{H'_{\varepsilon}}}} (u) + \eta J_0(u) =0$.  According to Lemmas \ref{lem29}, \ref{lem22},  $E_{\omega_{\varphi_{H'_{\varepsilon}}}} (u) $ and $J_0(u) $ are nonnegative.  Consequently,  $E_{\omega_{\varphi_{H'_{\varepsilon}}}} (u) =  J_0(u) =0$. Again by  Lemma \ref{lem22}, $u=u_{\heartsuit}=\cup_{i=1}^d u_{y_i^+}$.   Therefore,  we have  $\alpha_I= \alpha_{\heartsuit}$ and
		\begin{equation*}
		%	\mathcal{CO}_{\mathcal{Z}_0}(W_{H_{\varepsilon}'}, L_{\Lambda_{H_{\varepsilon}'}}, \Omega_{H_{\varepsilon}'})_J ( \mathfrak{c}_{\heartsuit})=
		\widetilde{\mathcal{CO}}(\underline{\Lambda}, {H'_{\varepsilon}})_J((\alpha_{\heartsuit}, [Z_{\alpha_{\heartsuit}, \alpha_{\heartsuit}}]))=(\textbf{y}_{\heartsuit}, [A_{\textbf{y}_{\heartsuit}}])
		\end{equation*}
	\end{proof}
	
	In sum, by Lemma \ref{lem36} and Lemma \ref{lem35}, we have
	\begin{equation}\label{eq72}
	\widetilde{\mathcal{CO}}(\underline{\Lambda}, {H_{\varepsilon}'})_J (\mathfrak{c}_{\heartsuit}+  K\circ \partial'  \mathfrak{c}_{\heartsuit})=(\mathbf{y}_{\heartsuit}, [A_{\textbf{y}_{\heartsuit}}]).
\end{equation}
	Since $\varphi_{H'_{\varepsilon}} = \varphi_{H_{\varepsilon}} $  outside the region $\mathcal{U}^{\delta}$,  $HF(\Sigma,    \varphi_{H'_{\varepsilon}}(\underline{\Lambda}), \underline{\Lambda}) $ is canonically  isomorphic to  $HF(\Sigma,  \varphi_{H_{\varepsilon}}(\underline{\Lambda}), \underline{\Lambda})$.  By Lemma \ref{lem26}, $(\textbf{y}_{\heartsuit}, [A_{\textbf{y}_{\heartsuit}}])$ represents the unit $j_{H_{\varepsilon}'}^{-1}(e_{\underline{\Lambda}})$.  By definition, $[\mathfrak{c}_{\heartsuit}+  K\circ \partial'  \mathfrak{c}_{\heartsuit}]=\mathfrak{I}_{H_{\varepsilon}, H'_{\varepsilon}}(\mathfrak{j}^{-1}_{H_{\varepsilon}}(\mathfrak{e}))=\mathfrak{j}^{-1}_{H'_{\varepsilon}}(\mathfrak{e})$.  Therefore, (\ref{eq72}) means 
	\begin{equation}\label{eq73}
	\widetilde{\mathcal{CO}}(\underline{\Lambda}, {H_{\varepsilon}'})_J(\mathfrak{j}^{-1}_{H'_{\varepsilon}}(\mathfrak{e})) = j_{H_{\varepsilon}'}^{-1}(e_{\underline{\Lambda}}).
	\end{equation} 
 By
 the partial invariance, we have $\widetilde{\mathcal{CO}}(\underline{\Lambda}, {H})_J(\mathfrak{j}^{-1}_{H}(\mathfrak{e})) = j_{H}^{-1}(e_{\underline{\Lambda}})$ provided that $H$ satisfies \ref{assumption1}. 

	Up to now, we finish  the proof of Theorem \ref{thm5}.

 \section{Proof of Theorem \ref{thm2}}
 Now we prove Theorem \ref{thm2} by using Theorem \ref{thm5}.
 \begin{proof}[Proof of Theorem \ref{thm2}]
 Fix a Hamiltonian function $H'_{\varepsilon}$ defined in Section \ref{section5.1}. Fix a generic almost complex structure $J_0 \in    \mathcal{J}^{reg}(W, \Omega_{\varphi_{H'_{\varepsilon}}}).$ For any Hamiltonian function $H$ such that $\varphi_H$ is nondegenerate, define
 \begin{equation}\label{eq69}
\mathcal{CO}(\underline{\Lambda}, H ) = I_{H'_{\varepsilon}, H} \circ \widetilde{\mathcal{CO}}(\underline{\Lambda}, {H_{\varepsilon}'})_{J_0} \circ \mathfrak{I}_{H, H'_{\varepsilon}}.
 \end{equation}
 Then the invariance follows from the properties of the continuous morphisms  of  continuous morphisms ((\ref{eq66}) and Proposition of \ref{lem30}).   The invariance property and (\ref{eq73}) implies that  $\widetilde{\mathcal{CO}}(\underline{\Lambda}, {H})_J(\mathfrak{j}^{-1}_{H}(\mathfrak{e})) = j_{H}^{-1}(e_{\underline{\Lambda}})$. 
 %By Lemma \ref{lem36} and Lemma \ref{lem35}, we know that $\widetilde{\mathcal{CO}}(\underline{\Lambda}, {H_{\varepsilon}'})_{J_0}$ is   non-vanishing, so is  $\mathcal{CO}(\underline{\Lambda}, H )$, beacuse   $ \mathfrak{I}_{H, H'_{\varepsilon}}$ and   $ I_{H'_{\varepsilon}, H} $ are isomorphisms.

Suppose that the link $\underline{\Lambda}$ is $0$-admissible.  Now we prove the estimate (\ref{eq68}).  We first assume that $H$ satisfies \ref{assumption1}.  Following  from  the partial invariance in Theorem \ref{thm2},  $\mathcal{CO}(\underline{\Lambda}, H)=\widetilde{\mathcal{CO}}(\underline{\Lambda}, H)_{J}$.  Here we can take $J \in \mathcal{J}^{reg}(W, \Omega_{\varphi_H})$ by the reasons in Remark \ref{remark9}.

Suppose that   $\mathcal{CO}(\underline{\Lambda}, H) (\mathfrak{j}_H^{-1}(\sigma)) =j_H^{-1}(a) \ne 0$.    For any $n>0$, we   find a cycle $\mathfrak{c}= \sum(\alpha, [Z])$ representing $\mathfrak{j}_H ^{-1}(\sigma) $ such that  $\mathbb{A}_H(\alpha, [Z]) \le c_{d}^{pfh}(H, \sigma) + 1/n $.     Then,   $\widetilde{\mathcal{CO}}(\underline{\Lambda}, H)_J ( \mathfrak{c}) =\sum (\mathbf{y}, [A])$  is a cycle representing  $j_H^{-1}(a)$.   By the definition of $\widetilde{\mathcal{CO}}(\underline{\Lambda}, H)_J $,  there is a PFH-HF curve $u$ from $\alpha$ to $\mathbf{y}$ such that  $[u]=Z\# \mathcal{Z}_{ref} \# A$.   Note that $\int_{\mathcal{Z}_{ref}} \omega_{\varphi_H} =0$.
By Lemma \ref{lem22}, we have
		\begin{equation*}
			\begin{split}
				0\le \int_{\dot{F}} u^*\omega_{\varphi_H} &= \int_Z \omega_{\varphi_H} + \int_{\mathcal{Z}_{ref}} \omega_{\varphi_H} + \int_A  \omega\\
				&= \mathbb{A}_H(\alpha, [Z]) -\mathcal{A}_H(\mathbf{y}, [A]) \\
				&\le   c_{d}^{pfh}(H, \sigma)  + 1/n		  -\mathcal{A}_H(\textbf{y}, [A]).
			\end{split}
		\end{equation*}
		By the definition of the HF spectral invariant, we have
		$$   c^{hf}_{\underline{\Lambda}}(H, a) \le   c_d^{pfh}(H,   \sigma)
 + \frac{1}{n}.$$
		 Then we get (\ref{eq68}) by taking $ n  \to \infty$.
	
For a general Hamiltonian function $H$,  according to Proposition \ref{lem37},  we  have  a sequence of  Hamiltonian function  $\{H_n\}_{n=1}^{\infty}$ such that  satisfy (\ref{assumption1}) and   converge to $H$ in $C^1$-topology.  By the above discussion, the spectral invariant of  $H_n$ satisfies (\ref{eq68}).   Since  the spectral invariants satisfy the Hofer continuity (see Theorem 3.1 of \cite{CHS2} and Theorem 1.13 of \cite{CHMSS}),  (\ref{eq68})  also  holds for $H$.
		\end{proof}
\begin{remark} \label{remark3}
A priori, the definition (\ref{eq69}) could depend on the choice of $H'_{\varepsilon}$ and $J_0$.
\end{remark}

	\appendix
	\renewcommand{\appendixname}{Appendix~\Appendix{section}}
	\section{Seiberg-Witten equations with non-exact perturbation} \label{appendixA}
	 In the appendix, we do not assume that $\varphi$ is Hamiltonian. The map $\varphi \in Symp(\Sigma, \omega)$
 can be any symplectomorphism such that the periodic orbits are nondegenerate. For
 our purpose, we introduce the twisted Seiberg-Witten cohomology and its cobordism
 maps defined by non-exact perturbations. Most of what follows here paraphrases parts
 of the accounts in  \cite{LT, CPZ} and Chapter 29 of \cite{KM}.

\subsection{Seiberg-Witten cohomology}
Given a closed Riemannian $3$-manifold $(Y, g)$, a $Spin^c$ structure $\mathfrak{s}$ on $Y$  is a  pair $(S, \mathfrak{cl}_Y )$, where   $S$ is a rank $2$  Hermitian
vector bundle over $Y$ and  $\mathfrak{cl}_Y: TY \to End(S)$  is a bundle map such that
\begin{eqnarray} \label{e6}
\mathfrak{cl}_Y(u)\mathfrak{cl}_Y(v)+ \mathfrak{cl}_Y(v)\mathfrak{cl}_Y(u)=-2g(u,v),
\end{eqnarray}
for any $u$, $v \in TY$ and $\mathfrak{cl}(e_1)\mathfrak{cl}(e_2)\mathfrak{cl}(e_3)=1$,  where  $\{e_1, e_2, e_3\}$ is an  orthonormal  frame for $TY$. A   \textbf{$Spin^c$ connection} $\nabla_B$ is a connection  on $S$ such that
\begin{eqnarray} \label{e17}
\nabla_B (\mathfrak{cl}_Y(u) \psi) =  \mathfrak{cl}_Y(\nabla^g u)\psi +  \mathfrak{cl} (u) \nabla_B \psi,
\end{eqnarray}
 where $u \in TY$,  $\psi$ is a section of $S$   and $\nabla^g$ is the Levi-Civita connection of $g$. We
 employ the same notation $B$ to denote the corresponding connection on $\det S.$
 The \textbf{Dirac operator} of $B$ is defined by $D_{B} \psi = \sum_{i=1}^3\mathfrak{cl}_Y(e_i) \nabla_{ B, e_i}\psi$.

Let $\varphi \in Symp(\Sigma, \omega)$ and  $J \in \mathcal{J}_{comp}(Y_{\varphi}, \omega_{\varphi})$.  Then $(\varphi, J)$ determines a metric $g$
  $Y_{\varphi}$ such that:
 \begin{equation*}
|R|_g=1 \mbox{ and } *_3\omega_{\varphi} =dt.
 \end{equation*}
 Given $\Gamma \in H_1(Y_{\varphi}, \mathbb{Z})$ and the above metric,  we have a $Spin^c$ structure $\mathfrak{s}_{\Gamma}=(S, \mathfrak{cl}_Y)$ such that $c_1(\mathfrak{s}_{\Gamma}) =c_1(\xi) +2PD(\Gamma).$ Furthermore, we have a decomposition
\begin{equation} \label{e2}
S=E \oplus (E\otimes \xi),
\end{equation}
where $E $ is a Hermitian line bundle satisfying $c_1(E) =PD(\Gamma)$. Given $\Gamma \in H_1(Y, \mathbb{Z})$, we  define a  $Spin^c$  structure $\mathfrak{s}_{\Gamma}$ such that $c_1(E)=PD(\Gamma)$.

 To avoid using the Novikov ring, we assume that $(\varphi, \Gamma)$ is \textbf{monotone}, i.e.,  $c_1(\mathfrak{s}_{\Gamma}) = -\rho [\omega_{\varphi}]$, where $\rho$ is a nonzero constant.  Let $\mathcal{C} (\mathfrak{s}_{\Gamma}) : =\operatorname{Conn}(\det S) \oplus \Gamma(S)$  denote the
 set of $Spin^c$ connections and sections of $S$. A pair in  $\mathcal{C} (\mathfrak{s}_{\Gamma})$ is called a \textbf{configuration}.  Let $\mathcal{G}(Y_{\varphi}) : = C^{\infty} (Y_{\varphi}, S^1)$ be the \textbf{gauge group.}  It acts on $\mathcal{C} (\mathfrak{s}_{\Gamma})$  by
 \begin{equation*}
   u\cdot (B, \Psi) := (B-2u^{-1}du, u\Psi),
 \end{equation*}
 where  $u \in \mathcal{G}(Y_{\varphi})$.  Given $(B, \Psi)$,  the \textbf{Chern-Simon-Dirac functional} is
 \begin{equation}
\begin{split}
 \mathfrak{a}_{\mathfrak{g}}(B, \Psi)&= -\frac{1}{8}\int_{Y_{\varphi}}(B-B_0)\wedge (F_B -F_{B_0})
- \frac{i}{4}\int_{Y_{\varphi}}(B-B_0)\wedge (2r\omega + \wp_3) \\
&+ \frac{1}{2}\int_{Y_{\varphi}} \Psi^*D_B \Psi + \mathfrak{g}(B, \Psi) -  \mathfrak{g}(B_0, \Psi_0),
\end{split}
\end{equation}
where $(B_0, \Psi_0)$  is a fixed reference pair in  $\mathcal{C} (\mathfrak{s}_{\Gamma})$.  Here $\mathfrak{g}: \mathcal{C} (\mathfrak{s}_{\Gamma}) \to \mathbb{R}$  is a $\mathcal{G}(Y_{\varphi})$-invariant
 function which is called an \textbf{abstract perturbation.} Let  $\mathcal{G}_0(Y_{\varphi})$  be the identity component of  $\mathcal{G}(Y_{\varphi})$. Note that the Chern-Simon-Dirac functional is $\mathcal{G}_0(Y_{\varphi})$-invariant.

 Let $\mathcal{S} =(\omega_{\varphi}, J, r, \Gamma, \mathfrak{g})$  denote the parameter set, where   $J \in \mathcal{J}_{comp}(Y_{\varphi}, \omega_{\varphi})$, $\Gamma \in H_1(Y_{\varphi}, \mathbb{Z})$,  and $r$ is a large positive constant.   The $\mathcal{S}$-perturbed Seiberg-Witten
 equations ask that a pair $(B,\Psi)$ obeys
\begin{equation} \label{e1}
  \begin{cases}
   D_{B} \Psi  -\mathfrak{S}_{\mathfrak{g}}(B, \Psi)=0 \\
\frac{1}{2}*_3F_B + \mathfrak{cl}_{Y}^{-1} (\Psi\Psi^*)_0  +ir*_3\omega_{\varphi} + \frac{i}{2}*_3 \wp_3 - \mathfrak{C}_{\mathfrak{g}} (B, \Psi)=0,
  \end{cases}
\end{equation}
where $(\Psi\Psi^*)_0   = \Psi \otimes \Psi^* - \frac{1}{2}  \operatorname{Id}_S$, $\wp_3$ is a closed $2$-form with cohomology class $2\pi c_1(\mathfrak{s}_{\Gamma})$, and $\mathfrak{p}(B, \Psi)  =(\mathfrak{S}_{\mathfrak{g}}(B, \Psi), \mathfrak{C}_{\mathfrak{g}}(B, \Psi))$    is the ``$L^2$ formal gradient” of  $\mathfrak{g}$ at  $(B, \Psi)$. Note that the right hand side of (\ref{e1}) is the $L^2$ formal gradient of the Chern-Simon-Dirac  functional.

Define a  chain complex ${CM}^{*}(Y_{\varphi}, \mathcal{S})$  to be a free module generated by  $\mathcal{G}_0(Y_{\varphi})$  equivalence class of solutions to (\ref{e1}).
Given $[\mathfrak{c}_{\pm} ]^{\circ}\in {CM}^{*}(Y_{\varphi}, \mathcal{S})$, then the
 differential $<\delta [\mathfrak{c}_+]^{\circ}, [\mathfrak{c}_-]^{\circ}>$  is defined by counting index one solutions to the following  flow line equations:
\begin{equation}
  \begin{cases}
  \frac{\partial}{\partial s}\Psi(s) + D_{B(s)} \Psi(s) = \mathfrak{C}_{\mathfrak{g}} (B(s), \Psi(s)) \\
\frac{\partial}{\partial s}B(s) + *_3F_{B(s)}= -2\mathfrak{cl}_{Y}^{-1} (\Psi\Psi^*)_0  -ir*_3dt  -\frac{i}{2}*_3 \wp_3 +  \mathfrak{C}_{\mathfrak{g}} (B(s), \Psi(s)) \\
 \lim\limits_{s \to \pm\infty}(B(s), \Psi(s))=u_{\pm} \cdot (B_{\pm}, \Psi_{\pm}) \mbox{ for some } u_{\pm} \in \mathcal{G}_0(Y_{\varphi}),
  \end{cases}
\end{equation}
 modulo  the $\mathcal{G}_0(Y_{\varphi})$-equivalence  and the $\mathbb{R}$-action.   The homology of $({CM}^{*}(Y_{\varphi}, \mathcal{S} ), \delta)$ is denoted by ${HM}^{*}(Y_{\varphi}, \mathcal{S} )$.

\begin{remark}
 For sufficiently large $r > 0$, by By Lemma 29.12 of  \cite{KM}, the solutions
 to (\ref{e1}) are irreducible (even without the monotone condition). Therefore, here we do
 not need to work with the blow-up version of the Seiberg-Witten equations.
\end{remark}

\subsection{Cobordism maps on Seiberg-Witten cohomology }
Let $(Y_{+}, g_{+})$ and $(Y_{-}, g_{-})$ be two Riemannian $3$-manifolds. Let $X$ be a cobordism from $Y_+$ to $Y_-$, i.e., $\partial X=Y_+ \bigsqcup (-Y_-)$. Given  a Riemannian metric $g$ on $X$ such that $g \vert_{Y_{\pm}} =g_{\pm}$, a $Spin^c$ structure $\mathfrak{s}_X$ on $X$  consists of a rank $4$  Hermitian
vector bundle $S=S_+\oplus S_-$ over $X$,  where $S_+$ and $S_-$ are  rank $2$  Hermitian vector bundles, together with a  map $\mathfrak{cl}_X: TX \to End(S)$ satisfying  (\ref{e6}) and
\begin{equation*}
\mathfrak{cl}_X(e_1)\mathfrak{cl}_X(e_2)\mathfrak{cl}_X(e_3)\mathfrak{cl}_X(e_4)=\left(
  \begin{array}{ccc}
    -1 & 0 \\
    0 & 1 \\
  \end{array}
\right),
\end{equation*}
 whenever $\{e_1, e_2, e_3, e_4\}$ is an orthonormal frame for $TX$.  The definitions of $Spin^c$
 connection, gauge groups and Dirac operator are defined similarly as the 3-dimensional
 case.

Let $\mathfrak{s}_{\pm}=(S_{Y_{\pm}}, \mathfrak{cl}_{Y_{\pm}})\in Spin^c(Y_{\pm})$ and $\mathfrak{s}_X=(S_+\oplus S_-, \mathfrak{cl}_X) \in  Spin^c(X)$. We say that    $\mathfrak{s}_X \vert_{Y_{\pm}}=\mathfrak{s}_{\pm}$  if  $S_{Y_{\pm}}=S_+ \vert_{Y_{\pm}}$ and $\mathfrak{cl}_{Y_{\pm}}(\cdot) = \mathfrak{cl}_X(\nu_{\pm})^{-1}\mathfrak{cl}_X(\cdot)$, where $\nu_+$ is the outward unit normal vector of $Y_+$ and $\nu_-$ is the inward unit normal vector of $Y_-$.

 A symplectic cobordism    from $(Y_{\varphi_+},  \omega_{\varphi_+})$ to $(Y_{\varphi_-},  \omega_{\varphi_-})$ is a compact symplectic manifold $(X, \Omega_X)$ such that  $\partial X= Y_{\varphi_+} \bigsqcup (-Y_{\varphi_-})$ and $\Omega_X  \vert_{Y_{\varphi_{\pm}}} = \omega_{\varphi_{\pm}}$.  Let $(\widehat{X}, \Omega_X)$ be the symplectic completion via adding cylindrical ends (see Section 2.3 of \cite{GHC}). An $\Omega_X$-compatible almost complex structure
 $J$     on $\widehat{X}$ is called \textbf{cobordism-admissible}   if it agrees with some admissible almost complex structures of $\mathbb{R} \times Y_{\varphi_{\pm}}$
over the cylindrical ends.   The set of cobordism-admissible almost complex structures is denoted by $\mathcal{J}_{comp}(X, \Omega_X)$.

 Let $(X_+, \Omega_{X_+})$ be a symplectic cobordism from $(Y_{\varphi_+}, \omega_{\varphi_+})$ to  $(Y_{\varphi_0}, \omega_{\varphi_0})$ and  $(X_-, \Omega_{X_-})$
 be another symplectic cobordism from  $(Y_{\varphi_0}, \omega_{\varphi_0})$   to   $(Y_{\varphi_-}, \omega_{\varphi_-})$.  For $R \ge  0$, we define
 \textbf{$R$-stretched composition} by
 \begin{equation}\label{e5}
   X_+ \circ_R X_- =X_- \cup_{Y_{\varphi_0}} \left( [-R, R] \times Y_{\varphi_0}\right)  \cup_{Y_{\varphi_0}} X_+.
 \end{equation}
Given $J_{\pm} \in \mathcal{J}_{comp}(X_{\pm}, \Omega_{X_{\pm}})$ such that $J_+ \vert_{\mathbb{R}_- \times Y_{\varphi_0}} =J_+ \vert_{\mathbb{R}_+ \times Y_{\varphi_0}}$, then we can obtain a
 cobordism-admissible almost complex structure  $J_+ \circ_R J_-$ on $\widehat{X_+ \circ_R X_-}$ by gluing them
 along $Y_{\varphi_0}$. (See Section 2.5 of \cite{GHC}.)

Let $(X, \Omega_X)$ be a symplectic cobordism    from $(Y_{\varphi_+},  \omega_{\varphi_+})$ to $(Y_{\varphi_-},  \omega_{\varphi_-})$ and $J \in \mathcal{J}_{comp}(X, \Omega_X)$. Define a metric $g$ on $\widehat{X}$ by $g(\cdot, \cdot) := \Omega_X(\cdot, J \cdot)$.  Then $\Omega_X$ satisfies $*\Omega_X = \Omega_X$ and $|\Omega_X|_g =\sqrt{2}. $ Moreover, $g$ agrees with  $ds^2+ g_{\pm}$ on the ends $[0, +\infty) \times Y_{\varphi_+}$
and $(-\infty, 0] \times Y_{\varphi_-}$
respectively. Similar to the 3-dimensional case, there is a decomposition
\begin{equation}\label{e3}
S_+= E \oplus E\otimes K_X^{-1},
\end{equation}
where $E$ is a Hermitian line bundle and $K_X^{-1}$ is the canonical line  bundle of $(\widehat{X},J)$. Furthermore,  the splitting (\ref{e3}) agrees with the splitting (\ref{e2}) \ on the ends $[0, \infty) \times Y_{\varphi_+}$ and $(-\infty, 0] \times Y_{\varphi_-}$.

Fix   $2$-forms  $\wp_{3,\pm}$ on $Y_{\varphi_{\pm}}$  such that $[\wp_{3, \pm}]= 2\pi c_1(\mathfrak{s}_{\pm}) $, and   $\wp_4$ be a closed $2$-form on $\widehat{X}$ which agrees with   $\wp_{3,+}$ and $\wp_{3,-}$  on the cylindrical  ends. Let $\mathfrak{p}$ be an
 abstract perturbation supported on the cylindrical ends.  The $(\Omega_X, r, J, \mathfrak{p})$ Seiberg-Witten  equations is
\begin{equation} \label{e4}
  \begin{cases}
D_{A} \Phi =\mathfrak{S}(A, \Phi) \\
\frac{1}{2}F_{A}^+=  \mathfrak{cl}_X^{-1}(\Phi\Phi^*)_0 -\frac{ir}{2}\Omega_X -i\wp_r^+ + \mathfrak{C}(A, \Phi),
  \end{cases}
\end{equation}
where  $A$ is a connection  $\det S_+$ and $\Phi \in \Gamma(S_+)$.  After a $\mathcal{G}_0(X)$-transformation, we  have  $\nabla_A= \partial_s + \nabla_{B(s)}$ (temporal gauge) on the ends, where $B(s) $ a  connection on   $ \det S_+ \vert_{\{s\} \times {Y_{\varphi_{\pm}}}}$.  Let  $\mathcal{S}_{\pm}$ be SW-parameter sets of $Y_{\varphi_{\pm}}$ with a common $r$-parameter.
 Given $[\mathfrak{c}_{\pm}]^{\circ} \in {CM}^{*}(Y_{\varphi_{\pm}},  \mathcal{S}_{\pm})$, let $\mathfrak{M}_{X, \operatorname{ind}=0}([\mathfrak{c}_+]^{\circ}, [\mathfrak{c}_-]^{\circ})$ be  the  moduli space of solutions to (\ref{e4}) with boundary conditions
\begin{equation*}
\lim\limits_{s \to \pm\infty}(B(s), \Phi(s))=u_{\pm} \cdot (B_{\pm}, \Psi_{\pm}) \mbox{ for some } u_{\pm} \in \mathcal{G}_0(Y_{\varphi_{\pm}}),
\end{equation*}
 and modulo  the  $\mathcal{G}_0(X)$-equivalence.  A solution $(A, \Phi)$ to (\ref{e4}) satisfying the above
 boundary condition is called a \textbf{SW-instanton}.

The cobordism map at the chain level is given by
\begin{equation*}
{CM}(X,  \Omega_X, J, r,  \mathfrak{p}, \mathfrak{s}_X)  = \sum_{[\mathfrak{c}_-]^{\circ}}  \#\mathfrak{M}_{X,  \operatorname{ind}=0}( [\mathfrak{c}_+]^{\circ}, [\mathfrak{c}_-]^{\circ})  [\mathfrak{c}_-]^{\circ}.
\end{equation*}
 To simplify the exposition, we assume that there is only finitely many $[\mathfrak{c}_-]^{\circ} $ such that  $\mathfrak{M}_{X,  \operatorname{ind}=0}( [\mathfrak{c}_+]^{\circ}, [\mathfrak{c}_-]^{\circ})  \ne \emptyset$; otherwise, we need to define the cobordism maps by using
 Novikov ring.
This map induces a homomorphism
$${HM}(X,  \Omega_X, J, r,  \mathfrak{p}, \mathfrak{s}_X):  {HM}^{*}(Y_{\varphi_+}, \mathcal{S}_+) \to {HM}^{*}(Y_{\varphi_-}, \mathcal{S}_-).$$
 that is called a \textbf{Seiberg-Witten cobordism map.}  The total Seiberg-Witten cobordism map is given by
 $${HM}(X,  \Omega_X, J, r,  \mathfrak{p})=\sum_{\mathfrak{s}_X \in Spin^c(X), \mathfrak{s}_X \vert_{Y_{\varphi_{\pm}}} =\mathfrak{s}_{\pm}} {HM}(X,  \Omega_X, J, r,  \mathfrak{p}, \mathfrak{s}_X).$$

 \begin{theorem} \label{thmA}
 The Seiberg-Witten cobordism map ${HM}(X,  \Omega_X, J, r,  \mathfrak{p})$ is well defined.
 Furthermore, it satisfies the following properties:
 \begin{itemize}
   \item
   (Homotopy invariance) Suppose that we have two tuples $(\Omega_i,J_i, \mathfrak{p}_i) $ on $\widehat{X}$ such
 that $(\Omega_1,J_1, \mathfrak{p}_1) =(\Omega_0,J_0, \mathfrak{p}_0) $ outside a compact subset of $\widehat{X}$  and $[\Omega_1 - \Omega_0] =0 $.
 Then
 $${HM}(X,  \Omega_0, J_0, r,  \mathfrak{p}_0, \mathfrak{s}_X)= {HM}(X,  \Omega_1, J_1, r,  \mathfrak{p}_1, \mathfrak{s}_X).$$
   \item
   (Composition rule)  Let $(X_+, \Omega_{X_+})$ be a symplectic cobordism from $(Y_{\varphi_+}, \omega_{\varphi_+})$ to  $(Y_{\varphi_0}, \omega_{\varphi_0})$, and  $(X_-, \Omega_{X_-})$
a symplectic cobordism from  $(Y_{\varphi_0}, \omega_{\varphi_0})$   to   $(Y_{\varphi_-}, \omega_{\varphi_-})$. Then
   \begin{equation*}
     \begin{split}
    & {HM}(X_+\circ X_-,  \Omega_{X_+\circ X_-}, J_+\circ J_-, r,  \mathfrak{p})\\
    =&{HM}(X_-,  \Omega_{X_-}, J_-, r,  \mathfrak{p}_-) \circ {HM}(X_+,  \Omega_{X_+}, J_+, r,  \mathfrak{p}_+).
      \end{split}
   \end{equation*}
 \end{itemize}
 \end{theorem}
To illustrate the proof of Theorem \ref{thmA}, we need to review some functionals on the
 Seiberg-Witten equations. Let
 \begin{equation*}
  \mathcal{SW}_{\mathfrak{p}}(A, \Phi)=\big( \frac{1}{2}F^+_A - \mathfrak{cl}_X^{-1}(\Phi \Phi^*)_0  + \frac{ir}{2}\Omega_X+ \frac{i}{2} \wp_4^+-\mathfrak{C}(A, \Phi), D_A \Phi - \mathfrak{S}(A, \Phi) \big).
 \end{equation*}
  As in \cite{KM}, we define the \textbf{analytic energy} and \textbf{topological energy} as follows:
\begin{itemize}
 \item (Analytic energy)
\begin{align*}
\mathscr{E}_{anal}(X_R)(A,\Phi)&= \frac{1}{4} \int_{X_R} |F_A|^2 + \int_{X_R} |\nabla_A \Phi|^2  + \int_{X_R} 2| \frac{ir}{2}\Omega_X + \frac{i}{2} \wp^+_4  -\mathfrak{cl}_{X_R}^{-1}(\Phi \Phi^*)_0 |^2 \\
&+\frac{1}{4}\int_{X_R} R_g |\Phi|^2 - i \int_{X_R} F_A \wedge    \frac{1}{2}*\wp_4,
\end{align*}
 where $X_R: =\{x \in \widehat{X} \vert | s(x)| \le R \}$.

 \item (Topological energy)
\begin{align*}
\mathscr{E}_{top}(X_R)(A,\Phi)&= \frac{1}{4} \int_{X_R} F_{A}\wedge F_{A}   - \int_{\partial X_R} \langle\Phi, D_A \Phi \rangle +  i \int_{X_R} F_A \wedge (r\Omega_X + \frac{1}{2}\wp_4)\\
& + \int_{\partial X_R} \frac{H}{2} |\Phi|^2,
\end{align*}
 where   $H$ is mean curvature of $\partial X_R$.
\end{itemize}
 By the same computation as in Section 4.5 of  \cite{KM}, we have
\begin{equation} \label{e7}
 | \mathcal{SW}(A, \Phi)|^2_{L^2(X_R)} =\mathscr{E}_{anal}(X_R)(A,\Phi)- \mathscr{E}_{top}(X_R)(A,\Phi),
\end{equation}
where  $\mathcal{SW}(A, \Phi)$ is the $\mathfrak{p}=0$-version of $\mathcal{SW}_{\mathfrak{p}}(A, \Phi)$.

\begin{proof} [Sketch of the proof of Theorem \ref{thmA}] To  prove the well-definedness of ${HM}(X,  \Omega_X, J, r,  \mathfrak{p}, \mathfrak{s}_X)$  and its properties, the argument is the same as Lemma 25.3.6 and Proposition 25.3.8
 of \cite{KM}. The key ingredients are to prove the compactness and transversality of the
 moduli space. We illustrate these two points as follows.

\begin{itemize}
  \item
   For a generic choice of $\mathfrak{p}$, the moduli space $ \mathfrak{M}_{X}( [\mathfrak{c}_+]^{\circ}, [\mathfrak{c}_-]^{\circ})$    is a smooth manifold
 of expected dimension. To prove this, the idea is to show that the linearization
 of the Seiberg-Witten equations is surjective (Definition 24.4.2 of \cite{KM}). Note
 that the appearance of the perturbation term $\frac{i}{2} (r\Omega_X + \wp^+_4)$  does not change the
 linearization of the Seiberg-Witten equations. Therefore, one can repeat the argument in Proposition 24.3.1 and Proposition 24.4.7 of \cite{KM} to prove the transversality. Let $ \mathfrak{M}_{X}( [\mathfrak{c}_+]^{\circ}, [\mathfrak{c}_-]^{\circ})_P$ be the moduli space of families of $(\Omega_p, J_p, r, \mathfrak{p}_p)$-
perturbed SW-instantons, where  $p \in P$ and $P$ is a manifold.  Furthermore, we
 require that $(\Omega_p, J_p,\mathfrak{p}_p)$ are fixed  over the cylindrical ends and $[\Omega_p-\Omega_{p_0}] = 0$ for all
$p \in P$.  The second condition implies that $\Omega_p =\Omega_{p_0} + d\mu_p$, where $\mu_p \in \Omega^1_c(X)$.  As
 the unparameterization case, $\Omega_{p_0}$
does not contribute to the linearization of the
 Seiberg-Witten equations while $d\mu_p$ is a family of of exact perturbations. Thus,
 repeat the argument of Proposition 24.4.10 in \cite{KM}, the transversality results can
 be extended to  $ \mathfrak{M}_{X}( [\mathfrak{c}_+]^{\circ}, [\mathfrak{c}_-]^{\circ})_P$.

 \item
 To prove the compactness, we need to estimate $|\nabla_A \Phi|^2_{L^2(X_R)}$, $|F_A|^2_{L^2(X_R)}$ and $|\Phi|^4_{L^4(X_R)}$, where $\mathfrak{d} \in \mathfrak{M}_{X}( [\mathfrak{c}_+]^{\circ}, [\mathfrak{c}_-]^{\circ})$  is a SW-instanton. By the
 Cauchy-Schwarz inequality, we have
 \begin{equation*}
|\nabla_A \Phi|^2_{L^2(X_R)} + |F_A|^2_{L^2(X_R)} + |\Phi|^4_{L^4(X_R)} \le c_0  \mathscr{E}_{anal}(X_R)(A,\Phi) + C(r\Omega_X, \wp, g)R,
 \end{equation*}
  where $C(r\Omega_X, \wp, g)$ is a constant depending on $r\Omega_X$, $\wp$, $g$, and $c_0$ is a universal
 constant.  Therefore, it suffices to estimate the analytic energy $\mathscr{E}_{anal}(X_R)(A,\Phi).$

 By (\ref{e7}) and Cauchy-Schwarz inequality, we have
 \begin{equation*}
    \begin{split}
   0= | \mathcal{SW}_{\mathfrak{p}}(A, \Phi)|^2_{L^2(X_R)} & \ge  \frac{1}{2} | \mathcal{SW}(A, \Phi)|^2_{L^2(X_R)} - 4  |\mathfrak{p}|^2_{L^2(X_R)}\\
   &=\frac{1}{2}\mathscr{E}_{anal}(X_R)(A,\Phi)- \frac{1}{2}\mathscr{E}_{top}(X_R)(A,\Phi)-4   |\mathfrak{p}|^2_{L^2(X_R)}.
    \end{split}
 \end{equation*}
 On the other hand, a direct computation shows that
 \begin{equation*}
\mathscr{E}_{top}(X_R)(A,\Phi)= \frac{1}{4} \int_{X_R} F_{A_0}\wedge F_{A_0}   + \frac{i}{4} \int_{X_R} F_{A_0} \wedge (r\Omega_X + \frac{1}{2}\wp_4) - 2\mathfrak{a}( \mathfrak{d} \vert_{\partial X_R}) + \int_{\partial X_R} \frac{H}{2}|\Phi|^2,
\end{equation*}
where $A_0$ is a reference connection such that $A_0 = B_0$ over the cylindrical ends.
 Then, we have
 \begin{equation*}
   \mathscr{E}_{anal}(X_R)(A,\Phi) \le C(r\Omega_X, \wp, g)  R + 4  |\mathfrak{p}|^2_{L^2(X_R)} - 2\mathfrak{a}_{\mathfrak{g}}( \mathfrak{d} \vert_{\partial X_R}) +2\mathfrak{g}(\mathfrak{d} \vert_{\partial X_R})
 \end{equation*}
 By Stokes’s theorem and Cauchy-Schwarz inequality, we have
 \begin{equation*}
    \begin{split}
\mathfrak{g}(\mathfrak{d} \vert_{\partial X_R}) &=\int_0^R \frac{d}{ds} \mathfrak{g}_s(\mathfrak{d}(s)) ds\\
&= \int_0^R \langle \mathfrak{p}_s, \partial_s \mathfrak{d}(s)\rangle ds + \int_0^R  (\frac{d}{ds} \mathfrak{g}_s)(\mathfrak{d}(s)) ds\\
&\le c_0|\mathfrak{p}|^2_{L^2(X_R)} + c_0^{-1} |\partial_s \mathfrak{d}(s)|_{L^2(X_R\setminus X)}^2 +  \int_0^R  (\frac{d}{ds} \mathfrak{g}_s)(\mathfrak{d}(s)) ds.
    \end{split}
 \end{equation*}
 The term  $|\partial_s \mathfrak{d}(s)|_{L^2(X_R\setminus X)}^2$ is bounded by   $|\nabla_A \Phi|^2_{L^2(X_R)} + |F_A|^2_{L^2(X_R)}$. Apply the
 argument in the proof of Lemma 24.5.1 of \cite{KM} to estimate $ \int_0^R  (\frac{d}{ds} \mathfrak{g}_s)(\mathfrak{d}(s)) ds$, we obtain
 \begin{equation*}
    \begin{split}
     \mathfrak{g}(\mathfrak{d} \vert_{\partial X_R}) & \le  c_0|\mathfrak{p}|_{L^2(X_R)}^2 + \mathfrak{g}(B_0, 0) R\\
     & + \left(c_0^{-1}|\nabla_A \Phi|^2_{L^2(X_R)} + c_0^{-1}|F_A|^2_{L^2(X_R)}  + \epsilon R|\Phi|^4_{L^4(X_R)}  + \epsilon R|F_A|^2_{L^2(X_R)}  \right) + \frac{1}{c_0\epsilon^2} R.
    \end{split}
 \end{equation*}
  Choose sufficiently small  $\epsilon>0$, the terms on  $|\nabla_A \Phi|^2_{L^2(X_R)}$, $|F_A|^2_{L^2(X_R)}$, $|\Phi|^4_{L^4(X_R)}$  are absorbed by the analytic energy.  By (iv) of Definition 10.5.1 in \cite{KM}, $|\mathfrak{p}|^2_{L^2(X_R)} \le c_0^2R + \frac{1}{c_0^2}|\Phi|^4_{L^4(X_R)}$. The term  $\frac{1}{c_0^2}|\Phi|^4_{L^4(X_R)}$  also can be absorbed by the analytic
 energy. In sum, we have
 \begin{equation}\label{e8}
 \begin{split}
\mathscr{E}_{anal}(A, \Phi)& \le C(r\Omega_X, \wp_4, B_0, g) R^3 -2\mathfrak{a}_{\mathfrak{g}}(\mathfrak{d} \vert_{\partial X_R})\\
 &\le C(r\Omega_X, \wp_4, B_0, g) R^3 -2\mathfrak{a}_{\mathfrak{g}_+} ([\mathfrak{c}_+]^{\circ}) + 2\mathfrak{a}_{\mathfrak{g}_-} ([\mathfrak{c}_-]^{\circ}).
 \end{split}
 \end{equation}
  The last inequality is because the Chern-Simon-Dirac functional is decreasing.
 Recall that  $\mathfrak{a}_{\mathfrak{g}_{\pm}} $
are  $\mathcal{G}_0(Y_{\varphi_{\pm}})$-invariant. Therefore, for any fixed $R$, we get a
 uniform bound on $|\nabla_A \Phi|^2_{L^2(X_R)}$, $|F_A|^2_{L^2(X_R)}$ and $|\Phi|^4_{L^4(X_R)}$.

 \item
 The bound (\ref{e8}) can be used to prove a local compactness of the moduli space as in Theorem 24.5.2 of \cite{KM}. Combining this with compactness results of SW
instantons on $\mathbb{R} \times Y$, the local compactness can be upgraded to the global compactness, i.e., for a sequence of SW-instantons $\{\mathfrak{d}_n \}_{n=1}^{\infty} \subset \mathfrak{M}_{X}( [\mathfrak{c}_+]^{\circ}, [\mathfrak{c}_-]^{\circ})$, there
 exists a subsequence converging to a broken trajectory in Definition 24.6.1 of
\cite{KM}.
\end{itemize}
\end{proof}

 \section{ Twisted cobordism maps on PFH} \label{appendixB}
In this appendix, we present the statements regarding the holomorphic curve axioms
 for PFH cobordism maps in twisted setting.   They are slight refinements of Theorem 1 in  \cite{GHC}.  The statements are parallel to those for ECH  (Proposition 6.2 of \cite{H6} and \cite{HT}).    For a nondegenerate symplectomorphism $\varphi \in Symp(\Sigma, \omega)$,  the twisted period Floer homology  $\widetilde{PFH}_*(\Sigma, \varphi, \gamma^{ref})$ is still well
 defined for a fixed reference 1-cycle $\gamma^{ref}$.

	Let  $\phi   \in  Hom( \widetilde{PFC}_*(\Sigma_+, \varphi_+, \gamma_+^{ref}), \widetilde{PFC}_*(\Sigma_-, \varphi_-, \gamma^{ref}_-))$ be a linear map with decomposition
\begin{equation*}
\phi =\sum_{\Gamma_X \in H_2(X, \partial X, \mathbb{Z}), \partial_{\pm} \Gamma_X =\Gamma_{\pm}} \phi_{\Gamma_X}.
\end{equation*}
Fix a reference relative homology class $Z_{ref} \in H_2(X, \gamma_+^{ref}, \gamma_-^{ref})$. Then $\Gamma_X= [Z_{ref}] + j(S) \in  H_2(X, \partial X, \mathbb{Z})$ for some $S\in H_2(X, \mathbb{Z})$, where $j:  H_2(X,  \mathbb{Z}) \to H_2(X, \partial X, \mathbb{Z})$  is the homomorphism in
 the exact sequence of relative homology. Following the terminology in \cite{H6},  we say
 that $\phi$ \textbf{counts J-holomorphic currents}  if $$\langle \phi_{\Gamma_X}(\alpha_+, [Z_+]), (\alpha_-, [Z_-]) \rangle \ne 0$$  implies
 that the moduli space $\overline{\mathcal{M}^J}(\alpha_+, \alpha_-, [Z_+\#(Z_{ref} +j(S)) \# (-Z_-)])$ is non-empty.  Here the
 bracket ``[ ]” denote the equivalence class of the relative homology class modulo out the
 following relation: $Z_1 \sim Z_2$ if and only if $\int_{Z_1} \Omega_X = \int_{Z_2} \Omega_X. $

  Following Hutchings and Taubes’ idea for ECH \cite{HT}, the author defines the PFH
 cobordism maps via the isomorphism ``SWF=PFH" \cite{LT} and Seiberg-Witten theory
 \cite{KM}. (See Theorem 1 of \cite{GHC}.) The following theorem covers three situations which do
 not appear in Theorem 1 of \cite{GHC}, but they can be obtained by the same method. The
 parallel statement for ECH cobordism maps can be found in Proposition 6.2 of \cite{H6}.

	\begin{theorem}\label{thm3}
		Let  $(X, \Omega_X) $ be a symplectic   cobordism  from $(Y_{\varphi_+}, \omega_{\varphi_+})$ to  $(Y_{\varphi_-}, \omega_{\varphi_-})$.  Assume that $d(\gamma_{\pm}) > g(\Sigma_{\pm})$. Fix a reference  relative homology class $Z_{ref} \in H_2(X, \gamma_+^{ref}, \gamma_-^{ref})$.    Then   $(X,   \Omega_X, Z_{ref}) $  induces a module homomorphism
		\begin{eqnarray*}
			PFH_{Z_{ref}}^{sw}(X, \Omega_X ): \widetilde{PFH}_*(\Sigma_+, \varphi_+,\gamma^{ref}_+) \to \widetilde{PFH}_*(\Sigma_-, \varphi_-, \gamma^{ref}_-)
		\end{eqnarray*}
		satisfying the following holomorphic curve axioms:
		\begin{enumerate}
			\item (Holomorphic curves)
			Given  a cobordism-admissible almost complex structure $J \in \mathcal{J}_{comp}(X,  \Omega_X)$ such that $J_{\pm} =J \vert_{\mathbb{R}_{\pm} \times Y_{\pm}} $ are  generic, then there is a chain map
			$$	PFC_{Z_{ref}}^{sw}(X, \Omega_X ): \widetilde{PFC}_*(\Sigma_+, \varphi_+,\gamma^{ref}_+) \to \widetilde{PFC}_*(\Sigma_-, \varphi_-, \gamma^{ref}_-)$$ inducing $PFH_{Z_{ref}}^{sw}(X, \Omega_X )$   and $PFC_{Z_{ref}}^{sw}(X, \Omega_X )_J$ counts $J$-holomorphic currents with $I=0$.

	\item (Homotopy invariance)
			Let $\{(\Omega_X^t, J_t)\}_{t \in [0,1]}$ be a family of symplectic  2-forms and cobordism-admissible almost complex structures such that $\Omega_X=\omega_{\varphi_{\pm}} + ds \wedge dt$ and $J=J_{\pm} \in \mathcal{J}(Y_{\varphi_{\pm}}, \omega_{\varphi_{\pm}})$ on the cylindrical ends. Suppose that $[\Omega_t-\Omega_0] =0 \in H^2_c(X, \mathbb{R})$ for each $t\in [0,1]. $  Then there exists a chain homotopy
			\begin{equation*}
					K: \widetilde{PFC}_*(\Sigma_+, \varphi_+,\gamma^{ref}_+) \to \widetilde{PFC}_*(\Sigma_-, \varphi_-, \gamma^{ref}_-)
			\end{equation*}
			such that
			\begin{equation*}
				PFC_{Z_{ref}}^{sw}(X, \Omega_ 0)_{J_0} =	PFC_{Z_{ref}}^{sw}(X, \Omega_1)_{J_1} + K \circ \partial_+ + K\circ \partial_-,
			\end{equation*}
			and $K$ counts $J_t$-holomorphic currents with $I=-1$.

			\item (Composition rule)
			Let $(X_+,  \Omega_{X+}) $ and $(X_-, \Omega_{X-}) $  be    symplectic  cobordisms   from $(Y_{\varphi_+},  \omega_{\varphi_+})$ to $(Y_{\varphi_0},  \omega_{\varphi_0})$  and  from $(Y_{\varphi_0},  \omega_{\varphi_0})$   to $(Y_{\varphi_-}, \omega_{\varphi_-})$ respectively.  If $(X,  \Omega_X)$ is the composition of $(X_+, \, \Omega_{X_+})$ and $(X_-, \Omega_{X_-})$, then there exists a chain homotopy
			\begin{equation*}
					K: \widetilde{PFC}_*(\Sigma_+, \varphi_+,\gamma^{ref}_+) \to \widetilde{PFC}_*(\Sigma_-, \varphi_-, \gamma^{ref}_-)
			\end{equation*}
			such that
			\begin{equation*}
				PFC_{Z_{ref-}}^{sw}(X_-, \Omega_{X_-} )_{J_-} \circ 	PFC_{Z_{ref+}}^{sw}(X_+, \Omega_{X_+} )_{J_+} =	PFC_{Z_{ref}}^{sw}(X, \Omega_X )_J + K \circ \partial_+ + K\circ \partial_-,
			\end{equation*}
			and $K$ counts $J_+ \circ_R J_-$-holomorphic currents with $I=-1$. Here $Z_{ref}=Z_{ref-}\#Z_{ref+}$.

			\item (Trivial cylinders)
			Suppose that $\mathbb{R} \times \mathcal{U}$ is a product region contained in $\widehat{X}$, i.e., $\Omega_X \vert_{\mathbb{R} \times \mathcal{U}} =\omega_{\varphi} + ds \wedge dt$.  Let $J$ be a cobordism-admissible almost complex structure on $\widehat{X}$ such that $J$ is $\mathbb{R}$-invariant over  $\mathbb{R} \times \mathcal{U}$. Suppose that the trivial cylinder  $\mathbb{R}\times \alpha$ is  the only element in $\overline{\mathcal{M}_0^J}(\alpha, \alpha)$, then
			\begin{equation*}
				\langle PFC_{Z_{ref}}^{sw}(X, \Omega_X )_J(\alpha, Z_+), (\alpha, Z_- ) \rangle=1,
			\end{equation*}
	 where $Z_+$ and $Z_-$ are relative homology classes such that $Z_+ -Z_- = [\mathbb{R}\times \alpha].$		
			
		\end{enumerate}
	\end{theorem}
	\begin{proof}[Sketch of proof]
Assume that  $(\varphi_{\pm}, J_{\pm})$ are $d$-$\delta$ flat approximations satisfying the properties in
 Lemma 2.1 and Lemma 2.4 of \cite{LT}.  For each orbit set  $\alpha$, Lee and Taubes construct
 a solution $\mathfrak{c}_{\alpha}(r)$ to (\ref{e1}), and this induces an identification between the PFH chain
 complex and Seiberg-Witten Floer chain complex (Theorem 3.1 of \cite{LT}).
As explained
 in Section 3.2 of  \cite{CPZ}, the above identification can be modified for the twisted setting by
 \begin{equation*}
   \begin{split}
     \mathcal{T}_r:& \widetilde{PFC}(\Sigma, \varphi,\gamma^{ref}) \to CM^{-*}(Y_{\varphi}, \mathcal{S}) \\
        & (\alpha, Z) \to [u_Z \cdot \mathfrak{c}_{\alpha}(r)]^{\circ},
    \end{split}
 \end{equation*}
 where $u_Z \in \mathcal{G}(Y_{\varphi})$  is uniquely determined by $Z$.

To define the twisted PFH cobordism maps, we need some preparations. First,
 we identify the relative homology classes with the $Spin^c$ structures via the following bijection
 \begin{equation*}
 \begin{split}
     \{\Gamma_X  \in H_2(X, \partial X, \mathbb{Z}): \partial_{\pm}\Gamma_X = \Gamma_{\pm}\}& \to \{\mathfrak{s}_X \in Spin^c(X) :  \mathfrak{s}_X \vert_{Y_{\varphi_{\pm}}} =\mathfrak{s}_{\Gamma_{\pm}}\}\\
     \Gamma_X & \to \mathfrak{s}_{\Gamma_X},
    \end{split}
 \end{equation*}
 where $\mathfrak{s}_{\Gamma_X}$
is the $Spin^c$ structure with $c_1(\mathfrak{s}_{\Gamma_X}) =2PD(\Gamma_X ) +c_1(K_X^{-1}).$

 Let $C_{ref} $ be a proper embedded surface such that  $C_{ref}  = \mathbb{R}_{\pm} \times \gamma^{ref}_{\pm}$  over the
 cylindrical ends and representing the class  $Z_{ref}$.  Let $(B^c_{\pm}(r), (\alpha_{\pm}^c(r), 0))$ be  the $
\gamma^{ref}_{\pm}$ concentrated family defined in Section 3.2 of \cite{CPZ}. We choose a family of configurations $\mathfrak{d}^c=(A^c(r), \Phi^c(r) =(\alpha^c(r), 0))$ such that
\begin{enumerate}
                               \item
                               $(A^c(r), \Phi^c(r) )= u_{\pm}^c \cdot (B^c_{\pm}(r), (\alpha_{\pm}^c(r), 0))$ over the cylindrical ends for some $u_{\pm}^c \in \mathcal{G}(Y_{\varphi_{\pm}})$;
                               \item
                               $(\alpha^c(r))^{-1}(0) \cap X =C_{ref} \cap X$  representing   $Z_{ref} \in H_2(X, \gamma_+^{ref}, \gamma_-^{ref})$;
                               \item
                                With respect to the decomposition (\ref{e3}),  we have $A^c(r) =2A^c_E(r) +A_{K^{-1}}$. Furthermore, $\frac{i}{2\pi} F_{A_E^c(r)}$ converge weakly to $C_{ref}$ in current sense as $r\to \infty$.
                             \end{enumerate}

  With the above preparations, the PFH cobordism map at the chain level is defined by
  \begin{equation*}
  \begin{split}
    &PFC^{sw}_{Z_{ref}}(X, \Omega_X, \Gamma_X)_J(\alpha_+, Z_+)\\
    =& \sum_{\alpha_-} \lim_{r\to \infty}\#\mathfrak{M}_{X, \operatorname{ind} =0}( [u^c_+\cdot u_{Z_+} \cdot \mathfrak{c}_{\alpha_+} ]^{\circ}, [u^c_-\cdot u_{Z_-} \cdot  \mathfrak{c}_{\alpha_-} ]^{\circ}) (\alpha_-, Z_-)
      \end{split}
  \end{equation*}
   It induces a homomorphism $PFH^{sw}_{Z_{ref}} (X, \Omega_X, \Gamma_X)$
  at the homological level. Furthermore,  by Theorem \ref{thmA},  $PFH^{sw}_{Z_{ref}} (X, \Omega_X, \Gamma_X)$ is independent of the choices of small $\mathfrak{p}$ and $J$. If $(\varphi_{\pm}, J_{\pm})$ are not   $d$-$\delta$ flat approximations, we can find   $(\varphi_{\pm}', J'_{\pm})$ such that they ar  $d$-$\delta$ flat.  By
 Proposition 2.1 of  \cite{LT}, the PFH of  $(\varphi_{\pm}', J'_{\pm})$ are canonically isomorphic to those of  $(\varphi_{\pm}, J_{\pm})$.
 Then we perturb $(\Omega_X,J)$ so that it becomes a cobordism between  $(Y_{\varphi_{\pm}'},  \omega_{\varphi_{\pm}'}, J_{\pm}')$ (Lemma 5.22 of  \cite{GHC}),  which we denote by  $(\Omega_X',J')$. Then, we define   $PFH^{sw}_{Z_{ref}} (X, \Omega_X, \Gamma_X)$  as
 the composition of   $PFH^{sw}_{Z_{ref}} (X, \Omega'_X, \Gamma_X)$  with the canonically isomorphisms between
 PFH of   $(\varphi_{\pm}', J'_{\pm})$ and those of  $(\varphi_{\pm}, J_{\pm})$ (see (5-32) of \cite{GHC}). The total PFH cobordism
 map  $PFH^{sw}_{Z_{ref}} (X, \Omega_X)$ is defined to be the sum of $PFH^{sw}_{Z_{ref}} (X, \Omega_X, \Gamma_X)$ over all the
 $\Gamma_X$ with $\partial_{\pm}\Gamma_X = \Gamma_{\pm}$.

 Now we prove the holomorphic curve axioms. We only prove the case that $(\varphi_{\pm}, J_{\pm})$
 are $d$-$\delta$  flat approximations.  The general case can be achieved by the Gromov-Taubes
 compactness argument as in Section 6.3 of \cite{HT}.
 Suppose that  $PFC^{sw}_{Z_{ref}} (X, \Omega_X, \Gamma_X)_J$
 is non-trivial. By definition,  we obtain a SW-instanton
 \begin{equation*}
   \mathfrak{d}(r) \in \mathfrak{M}_{X, \operatorname{ind} =0}( [u^c_+\cdot u_{Z_+} \cdot \mathfrak{c}_{\alpha_+} ]^{\circ}, [u^c_-\cdot u_{Z_-} \cdot  \mathfrak{c}_{\alpha_-} ]^{\circ})
 \end{equation*}
 for $r\gg 1$.  As $r \to \infty$, Proposition 5.12 in \cite{GHC} shows that $\mathfrak{d}(r)$ converge to a broken
 $J$-holomorphic current $\mathcal{C} =\{\mathcal{C}^{-N_-}, ..., \mathcal{C}^0, ..., \mathcal{C}^{N_+}\} \in \overline{\mathcal{M}^J}(\alpha_+, \alpha_-)$   in current sense. Moreover, $I(\mathcal{C}) =0$.  For a more general version of Taubes’s ``$SW \Rightarrow Gr$” degeneration, please see Theorem 1.10 of  \cite{L}.

  Now we verify that $[\mathcal{C}] = [Z_+\#(Z_{ref} +j(S)) \#(-Z_-)]$.  It suffices to verify that
  \begin{equation*}
  \begin{split}
    & \int_{\mathcal{C}^0 \cap \mathbb{R}_+ \times Y_{\varphi_+}}  \omega_{\varphi_+} + \int_{\mathcal{C}^0 \cap \mathbb{R}_- \times Y_{\varphi_-}}  \omega_{\varphi_+} + \int_{\mathcal{C}^0 \cap X}  \Omega_X + \sum_{i=-N-}^{-1} \int_{\mathcal{C}^i}  \omega_{\varphi_-} +   \sum_{j=1}^{N_+} \int_{\mathcal{C}^j}  \omega_{\varphi_+} \\
     =& \int_{Z_+\#(Z_{ref} +j(S)) \#(-Z_-)} \Omega_X.
     \end{split}
  \end{equation*}
   The left hand side of the above equation is the energy of the holomorphic current,
 denoted by $\mathcal{E}_{\Omega_X} (\mathcal{C})$. Similarly,  for a configuration $\mathfrak{d}= (A,\Phi)$ which is asymptotic to $[u^c_{\pm} \cdot u_{Z_{\pm}} \cdot \mathfrak{c}_{\alpha_{\pm}}(r)]^{\circ}$, we define its energy by
 \begin{equation*}
   \mathcal{F}_{\Omega_X} (\mathfrak{d}) : = \frac{i}{2\pi} \int_{ \mathbb{R}_+ \times Y_{\varphi_+}} F_{A_E(r)}  \wedge \omega_{\varphi_+} + \frac{i}{2\pi} \int_{ \mathbb{R}_- \times Y_{\varphi_-}} F_{A_E(r)}  \wedge \omega_{\varphi_-} +\frac{i}{2\pi}   \int_{X}  F_{A_E(r)}  \wedge\Omega_X.
 \end{equation*}
  By the argument in the Proof of Theorem 1.10  \cite{L},   we have  $   \lim_{r \to \infty }\mathcal{F}_{\Omega_X} (\mathfrak{d}(r)) = \mathcal{E}_{\Omega_X} (\mathcal{C})$ (the parallel result is (7.36) of \cite{L}). Then, it suffices to show that  $$   \lim_{r \to \infty }\mathcal{F}_{\Omega_X} (\mathfrak{d}(r)) =\int_{Z_+\#(Z_{ref} +j(S)) \#(-Z_-)} \Omega_X.$$
 By Stokes’s theorem, we have
 \begin{equation} \label{e9}
    \mathcal{F}_{\Omega_X} (\mathfrak{d}(r))  -\mathcal{F}_{\Omega_X} (\mathfrak{d}_+(r)\#\mathfrak{d}^c(r)\#(-\mathfrak{d}_-(r))) =\int_{\Gamma_X- [Z_{ref}]} \Omega_X= \int_{j(S)} \Omega_X,
 \end{equation}
 where  $\mathfrak{d}_{\pm}(r)$ is a family of configurations of $Y_{\varphi_{\pm}}$ such that  $\lim_{s\to \infty}\mathfrak{d}_{\pm}(r) = u_{\pm}^c \cdot u_{Z_{\pm}} \cdot \mathfrak{c}_{\alpha_{\pm}} $ and  $\lim_{s\to -\infty}\mathfrak{d}_{\pm}(r) = u_{\pm}^c \cdot (B^c_{\pm}(r), (\alpha^c_{\pm}(r), 0)) $.

 Using Stokes’s theorem again and the property of $u_{Z_{\pm}}$, we get
 \begin{equation}  \label{e10}
   \begin{split}
    \lim_{r\to \infty}     \mathcal{F}_{\omega_{\varphi_{\pm}}} (\mathfrak{d}_{\pm}(r)) &=  \lim_{r\to \infty}   \frac{i}{4\pi} \int_{\mathbb{R} \times Y_{\varphi_{\pm}}} ds \wedge \partial_s B_{\pm} \wedge \omega_{\varphi_{\pm}}\\
    &=   \lim_{r\to \infty}   \frac{i}{4\pi} \int_{ Y_{\varphi_{\pm}}} (u_{Z_{\pm}} \cdot B_{\pm}(r) - B_{\pm}^c(r)) \wedge \omega_{\varphi_{\pm}} = \int_{Z_{\pm}}  \omega_{\varphi_{\pm}}.
   \end{split}
 \end{equation}
 Since $\mathfrak{d}^c(r)$ is constantly equal to $u^c_{\pm} \cdot \mathfrak{c}_{\pm}^c$ over the cylindrical ends, it follows that $    \mathcal{F}_{\Omega_X} (\mathfrak{d}^c(r)) = \frac{i}{2\pi}  \int_X F_{A_E^c(r)} \wedge \Omega_X$.   The convergence of $\frac{i}{2\pi} F_{A_E^c(r)}$  to $C_{ref}$ as currents
 implies that
 \begin{equation} \label{e11}
\lim_{r\to \infty }\mathcal{F}_{\Omega_X} (\mathfrak{d}^c(r)) = \lim_{r\to \infty } \frac{i}{2\pi}  \int_X F_{A_E^c(r)} \wedge \Omega_X =\int_{C_{ref} \cap  X} \Omega_X = \int_{Z_{ref}}{\Omega_X}.
 \end{equation}
 By (\ref{e9}), (\ref{e10}),  and (\ref{e11}), we have   $\mathcal{E}_{\Omega_X} (\mathcal{C}) =  \int_{Z_+\#(Z_{ref} +j(S)) \#(-Z_-)} \Omega_X. $ Therefore, $\mathcal{C}  \in \overline{\mathcal{M}^J}(\alpha_+, \alpha_-, [Z_+\#(Z_{ref} +j(S)) \#(-Z_-)]).$

 The second and the third statements follow from the composition rule and homotopy
 invariance of the Seiberg-Witten cobordism maps. By the proof of Theorem  \ref{thmA}, we
 have
 \begin{equation*}
   CM(X,\Omega_1, J_1, r, \mathfrak{p}_1, \mathfrak{s}_{\Gamma_X}) -   CM(X,\Omega_0, J_0, r, \mathfrak{p}_0, \mathfrak{s}_{\Gamma_X}) =\delta_- \circ K_r^{sw} + K_r^{sw} \circ \delta_+.
 \end{equation*}
 The chain homotopy  $K_r^{sw} $ is defined by counting $(\Omega_t,J_t,r,\mathfrak{p}_t)$-perturbed SW-instantons
 with $\operatorname{ind} = -1.$  To simplify the discussion, assume that $u^c_{\pm} = 1$ and $\Gamma_X = [Z_{ref}]$. Then
 define $K := \lim_{r\to \infty} (\mathcal{T}^-_r)^{-1}  \circ K_r^{sw} \circ \mathcal{T}^+_r$.  Therefore, $K\ne 0$ implies that the existence of
 $(\Omega_t,J_t,r,\mathfrak{p}_t)$-perturbed SW-instanton for $r\gg1$. Again, by Proposition 5.12 in  \cite{GHC}, we
 obtain a holomorphic current $\mathcal{C}$. Moreover, the ECH index of $\mathcal{C}$ follows from Theorem
 5.1 of \cite{DCG}. This prove the homotopy invariance property. The composition rule can be
 proved similarly.

		For the last statement, we need to show that the $I=0$ moduli space of holomorphic curves  is diffeomorphic  to the $\operatorname{ind}=0$ moduli space of solutions to the Seiberg-Witten equations.   A special case that $\widehat{X} =\mathbb{R} \times Y$ and $(\Omega_X ,J )$ is $\mathbb{R}$-invariant  has been proved in \cite{LT}.  For another special case that the holomorphic current $\mathcal{C}$ is embedded,  Taubes's argument can be applied to the cobordism case equally well  \cite{GHC}.   The proofs  are   mirror modifications  of the argument in Taubes' s series papers \cite{T2}, \cite{T3}, \cite{T4}.
		
		Since  Taubes' arguments require many hard analysis on the Seiberg-Witten equations which beyond  scope of this paper.    We suggest the readers    to read \cite{LT} and \cite{GHC} for  the relevant e details on Taubes' s idea. Also  see    Proposition  6.3 of \cite{HT} for the parallel arguments  on  ECH.  Here we just outline the main idea as follow.
		
		a)Let $\mathcal{C} =\mathbb{R} \times \alpha$ be the unique holomorphic current with $I=0$.    We follow  Section 5a in \cite{T2}   to build  a complex line bundle $E$ by gluing the normal bundle  of $\mathcal{C}$ with the trivial line bundle away from $\mathcal{C}$.  The $Spin^c$ structure on $\widehat{X}$ is defined by $S_+= E \oplus E K_X^{-1}$. We can construct an approximation solution $(A_r^*, \psi_r^*)$  ( closed to solve $\mathfrak{p} = 0$)-version
 of $\ref{e4}$)  associated to $\mathcal{C}$(see Section 5a of \cite{T2}) . Away from the trivial cylinders $\mathcal{C}$, $(A_r^*, \psi_r^*)$  is just the trivial solution. Let $(\gamma, m)$ be a component of $\alpha$. Near the trivial cylinder $\mathbb{R} \times \gamma$,  $(A_r^* ,\psi_r^*)$ is determined by a map $\mathfrak{v}: \mathbb{R} \times \gamma \to \mathfrak{C}_m$, where $\mathfrak{C}_m$ is  the moduli space of $m$-vortices. The vortex  equations can be regarded as the 2-dimensional Seiberg-Witten equations.   When  $\mathcal{C}$ is the trivial cylinders, there is a canonical way to choose the map $\mathfrak{v}$.   The analysis in \cite{T2} can be used to perturb $(A_r^*, \psi_r^*)$ to be a true solution $(A_r, \psi_r)$ to (\ref{e4}).
		
		b)  The argument in \cite{T3} can be used to show that $(A_r ,\psi_r)$ is non-degenerate. By  Theorem 5.1 of \cite{DCG},  the index  of $(A_r ,\psi_r)$  is zero.
		
		c) Let $(A_r', \psi_r')$ be another solution to the $r\Omega_X$-perturbed Seiberg-Witten equations with $\operatorname{ind}=0$.  By Proposition 5.12 of \cite{GHC},   $(A_r', \psi_r')$  converges to the trivial cylinders $\mathcal{C} =\mathbb{R} \times \alpha$ as $r  \to \infty $ because it is the unique holomorphic current with $I=0$.   For any $\delta>0$, the convergence implies that $1-|\psi'_r|  < \delta$ away from $\mathcal{C}$.      Intuitively, this means     $(A_r', \psi_r')$  is close to the trivial solution away from $\mathcal{C}$. The arguments in Section 6 of \cite{T4} can be  carried  over  to show that  $(A_r', \psi_r')$ is gauge equivalent to $(A_r, \psi_r)$. 	
	\end{proof}

Shenzhen University

\verb| E-mail address: mathseu1234@outlook.com |

\end{document}